\documentclass[11pt,letterpaper]{amsart}

\usepackage{amsmath}
\usepackage{amssymb}
\usepackage{amsthm}
\usepackage{amsxtra}
\usepackage{graphicx}
\usepackage[all,cmtip]{xy}
\usepackage{hyperref}
\usepackage{enumitem}
\usepackage{color}
\usepackage{bm}

\DeclareFontFamily{U}{mathx}{}
\DeclareFontShape{U}{mathx}{m}{n}{<-> mathx10}{}
\DeclareSymbolFont{mathx}{U}{mathx}{m}{n}
\DeclareMathAccent{\widehat}{0}{mathx}{"70}
\DeclareMathAccent{\widecheck}{0}{mathx}{"71}

\def \mod{{\textup{mod}}}
\def \Mod{{\textup{Mod}}}

\newcommand{\la}{\langle}
\newcommand{\ra}{\rangle}

\renewcommand{\Re}{\operatorname{Re}\,}
\renewcommand{\Im}{\operatorname{Im}\,}

\newcommand{\sgn}{\operatorname{sgn}}

\def\R{{\mathbb R}}

\def\e{{\varepsilon}}
\def\g{{\gamma}}

\def\s{{\sigma}}
\def\r{{\rho}}

\def\a{{\alpha}}
\def\b{{\beta}}
\def\d{{\delta}}
\def\l{{\lambda}}

\def\z{{\zeta}}

\def\m{{\mu}}
\def\n{{\nu}}
\def \x{{\xi}}
\def\Uue{\underline{U_{e}}}
\def\Uud{\underline{U_{d}}}
\def\Hu{\mathcal H_{\underline{\omega}}}
\def\Pue{{\underline{P_e}}}
\def\Pud{{\underline{P_d}}}
\def\uo{\underline{\omega}}

\def\pu{\underline{p}}
\def\Uu{\underline{U}}

\makeatletter
\@namedef{subjclassname@2020}{\textup{2020} Mathematics Subject Classification}
\makeatother

\newtheorem{theorem}{Theorem}
\newtheorem{lemma}[theorem]{Lemma}
\newtheorem{corollary}[theorem]{Corollary}
\newtheorem{proposition}[theorem]{Proposition}

\theoremstyle{definition}
\newtheorem{definition}[theorem]{Definition}

\theoremstyle{remark}
\newtheorem{remark}[theorem]{Remark}

\numberwithin{equation}{section}
\numberwithin{theorem}{section}
\numberwithin{problem}{section}

\title[Asymptotic stability of solitary waves for 1D NLS]{Asymptotic stability of solitary waves for one dimensional nonlinear Schr\"odinger equations}



\begin{document}
	
\author[Charles Collot]{Charles Collot}
\address{Charles Collot, Laboratoire AGM, CY Cergy Paris Universit\'e, 2 avenue Adolphe Chauvin, 95300 Pontoise, France}
\email{ccollot@cyu.fr}

\author[P. Germain]{Pierre Germain}
\address{Pierre Germain, Department of Mathematics, Huxley Building, South Kensington Campus,
Imperial College London, London SW7 2AZ, United Kingdom}
\email{pgermain@ic.ac.uk}

\maketitle 	

\begin{abstract}
We show global asymptotic stability of solitary waves of the nonlinear Schr\"odin-ger equation in space dimension 1. Furthermore, the radiation is shown to exhibit long range scattering if the nonlinearity is cubic at the origin, or standard scattering if it is higher order. We handle a general nonlinearity without any vanishing condition, requiring that the linearized operator around the solitary wave has neither nonzero eigenvalues, nor threshold resonances. Initial data are chosen in a neighborhood of the solitary waves in the natural space $H^1 \cap L^{2,1}$ (where the latter is the weighted $L^2$ space). The proof relies on the analysis of resonances as seen through the distorted Fourier transform, combined for the first time with modulation and renormalization techniques.
\end{abstract}

\tableofcontents

	\section{Introduction}
	
	\subsection{1D nonlinear Schr\"odinger equations and their solitons}
	
	We consider the Cauchy problem for the nonlinear Schr\"odinger equation
	\begin{equation} \tag{NLS}
	\label{NLS}
	i \partial_t v - \partial_x^2 v - F'(|v|^2) v= 0
	\end{equation}
	with prescribed data
	$$
	v(t=0) = v_0.
	$$
	It is stemming from the Hamiltonian
	$$
	H(v) = \int [|\partial_x v|^2 - F(|v|^2)]\,dx.
	$$
	The interaction potential $F$ will only be assumed to be smooth and to have a non-degenerate local minimum at zero. Stationary waves of the type
	$$
	v(t) = e^{-it\omega} \Phi_\omega 
	$$
	are given by solutions of 
	\begin{equation} \label{eq:soliton}
	\partial_x^2 \Phi_\omega - \omega \Phi_\omega + F'(\Phi_\omega^2) \Phi_\omega = 0.
	\end{equation}
	Under our assumptions on $F$, there exists a unique solution of the above equation on an interval $\omega \in (0,\omega^*)$, for some $\omega^*>0$. Furthermore, $\Phi_\omega$ is even, positive, decreasing on $x>0$, and exponentially decreasing at infinity, along with its derivatives (see~\cite{BL} for these facts and a full characterization of the interval of existence).
	
For $p, \gamma, y \in \mathbb{R}$, Galilean, phase and translation symmetries
$$
v(t,x) \mapsto e^{i(p x + p^2 t + \gamma)} v(t,x+2p  t - y)
$$
leave the set of solutions of~\eqref{NLS} invariant. In particular, this gives the family of traveling waves
\begin{equation} \label{id:family-solitons}
e^{i( p x +( p^2-\omega) t + \gamma)} \Phi_\omega (x+2p t - y).
\end{equation}

Our aim in this paper is to establish asymptotic stability of this family of solitary waves, under appropriate spectral assumptions on the linearized operator around them.

\subsection{Long range scattering for small initial data} If $F''(0) \neq 0$, proving decay or deriving the aymptotics of solutions of~\eqref{NLS} for small data is already challenging. This is because a cubic nonlinearity is long range in dimension 1; in more technical terms, $|u|^2 u \not \in L^1_t L^2_x$ if $u$ is a solution of the linear Schr\"odinger equation. Small solutions do not scatter, but there is a logarithmic correction in the phase in Fourier, which is refered to as long range, or modified, scattering. More precisely,
$$
v(t) \sim \mathcal F^{-1}\left( e^{i\left(\xi^2 t-\frac{L}{2}|\widehat f|^2 \ln t\right)}\widehat f\right)  \qquad \mbox{as }t\to \infty,
$$
for a profile $f$, $L = F''(0)$, and where $\mathcal F$ and $\mathcal F^{-1}$ denote the Fourier transform and its inverse.

This was proved first by Hayashi and Naumkin~\cite{HN}, after which other approaches were proposed~\cite{IT,LS}, see the instructive review~\cite{Murphy}. In the present paper, we will follow the approach proposed by Kato and Pusateri~\cite{KP} in the spirit of the space-time resonance method.

Modified scattering for small solutions also holds in the presence of a potential
$$
i \partial_t v - \partial_x^2 v + V v - |v|^2 v= 0,
$$
but it is more difficult to prove. Over the last few years, a number of proofs appeared~\cite{Delort,GPR,Naumkin,CP1,CP2}; the article by the second author, Pusateri and Rousset~\cite{GPR} proceeds via the distorted Fourier transform, which will also be our approach.

Finally, yet another approach is possible in the case of the cubic~\eqref{NLS}, corresponding to $F(z)=z^2$. Taking advantage of the integrability of this equation, Deift and Zhou~\cite{DZ} were able to show that modified scattering holds for any data in $H^1 \cap L^{2,1}$ in the defocusing case.

The nonlinear Schr\"odinger equation can be linked to the binormal flow of curves. Modified scattering is then related to the self-similar motion of vortex lines, see \cite{VB}.

\subsection{Stability and instability of the solitary waves} 
The first notion of stability is~\textit{orbital stability}; by the general theory of Grillakis-Shatah-Strauss~\cite{GSS,GSS2}, see also~\cite{CL,SS,Weinstein, Weinstein2}, it is completely understood: it holds if and only if
$$
c_\omega = \frac{d}{d\omega} \int |\Phi_\omega|^2\,dx > 0
$$
(leaving aside the limiting case $c_\omega = 0$). For homogeneous power nonlinearities, more can be said when $c_\omega<0$: namely, there exists initial states arbitrarily close to $\Phi_\omega$ which lead to finite-time blow up~\cite{BC}.

Asymptotic stability was first obtained by Buslaev and Perelman~\cite{BP} for functions $F$ vanishing to order $\geq 5$ at the origin, in the absence of a resonance at the edge of the continuous spectrum (see below for a definition), and in the absence of internal modes. Besides modulation, a key idea was to use improved local decay estimates. This restriction on the absence of internal modes was later lifted in~\cite{BS}, where the authors were able to show that radiation damping occurs.

The next important development was due to Krieger and Schlag~\cite{KS}, who were able to construct finite-codimension stable manifolds around the soliton in the monic supercritical case  $F(x) = |x|^p$, $p>3$. This improvement relied on the use of Strichartz estimates and sharper dispersive estimates; an additional difficulty occurs because of unstable modes. Furthermore, the authors give a precise description of the spectral resolution of the linearized operator, which is foundational for the present paper.

In the case of \textit{small} solitons arising from an exterior potential, asymptotic stability was obtained by Mizumachi~\cite{Mizumachi} through dispersion estimates, and more recently by Chen~\cite{Chen} who combined modulation and an analysis of resonances, thus pointing towards the methods applied in the present paper.

The results on asymptotic stability which have been cited so far rely mainly on the dispersive properties of the (linearized) problem, which are exploited through harmonic and spectral analysis. An alternative approach consists in the virial method; it enabled Martel~\cite{Martel} to prove asymptotic stability (on  arbitrarily large compact sets) for the cubic-quintic problem in the absence of resonances or internal modes. This method typically only requires data of finite energy, and might give optimal results for slowly decaying perturbations; but it fails to give sharp decay or asymptotics for the radiation and might be difficult to extend to quasilinear problems.

Finally, in the case $F(z) = |z|^2$, asymptotic stability can be proved by taking advantage of the completely integrable structure: this was achieved by Cuccagna and Pelinovsky~\cite{CP}. 

\subsection{Spectrum of the linearized operator}

Soliton stability is, of course, tightly related to the spectral properties of the linearization. Recasting \eqref{NLS} as a vector equation for $(v,\bar v)^\top$, the linearized operator around the soliton $\Phi$ is
$$
\mathcal{H}_\omega = \begin{pmatrix}
-\partial_x^2 + \omega & 0  \\
0 &  \partial_x^2 - \omega
\end{pmatrix} - \begin{pmatrix}
V_+  & V_-
\\
-V_- & -V_+
\end{pmatrix}
$$
where $V_- = F''(\Phi^2) \Phi^2$ and $V_+=F'(\Phi^2)+V_-$. The spectral properties of $\mathcal H_\omega$ have been studied in \cite{Gr,Weinstein,BP,ErSc,CPV,CGNT} and references therein.

Under the general assumption that $V_\pm$ decay exponentially fast, the essential spectrum of $\mathcal H$ equals $(-\infty,-\omega]\cup [\omega,\infty)$, and the rest of the spectrum consists of finitely many eigenvalues of finite algebraic multiplicity. If moreover $\Phi_\omega$ is the ground state with $\Phi_\omega\geq 0$, these eigenvalues can only be located on $\mathbb R\cup i\mathbb R$. Finally, under the additional condition of orbital stability $c_\omega>0$, then eigenvalues are in fact necessarily real. For these three facts, one can consult for example \cite{CPV}, Section 2 of \cite{ErSc} and the proof of Proposition 9.2. in \cite{KS} and references therein respectively.

In the present case, nonzero eigenvalues are thus real and are known as \textit{internal modes}. The kernel of $\mathcal H$ is never null, since the set of solitons \eqref{id:family-solitons} is of dimension four. Indeed, by differentiating with respect to each parameter, we obtain four functions
\begin{equation}\label{id:generalized-kernel}
\Xi_0 =\begin{pmatrix} \Phi  \\ -\Phi \end{pmatrix} , \quad \Xi_1=\begin{pmatrix} \partial_\omega \Phi  \\ \partial_\omega \Phi \end{pmatrix}, \quad \Xi_2=\begin{pmatrix} \partial_x \Phi  \\ \partial_x \Phi \end{pmatrix}, \quad \Xi_3=\begin{pmatrix} x \Phi  \\ -x \Phi \end{pmatrix} \quad 
\end{equation}
that are always elements of the generalised kernel $\mathcal K=\cup_{n\geq 1} \textup{Ker}(\mathcal H^n)$ as
\begin{equation} \label{id:generalized-kernel-2}
\mathcal H \Xi_0=0, \quad \mathcal H \Xi_1=-\Xi_0, \quad \mathcal H \Xi_2=0, \quad \mathcal H \Xi_3=-2\Xi_2.
\end{equation}
If $c_\omega > 0$, we learn from Proposition 1.2.2 in~\cite{BP} that the generalized kernel equals the tangent space of the full set of solitons
\begin{equation} \label{id:assumption-spectral}
\mathcal K=\textup{Span}\left\{\Xi_j\right\}_{j=0,1,2,3}.
\end{equation}

Finally, the tip of the continuous spectrum might exhibit \textit{edge resonances}; they are absent if 
\begin{equation} \label{id:assumption-spectral-2}
\mbox{there is no nonzero bounded solution $f$ to }\mathcal H f=\omega f.
\end{equation}

\subsection{Main result}
We introduce the space $L^{2,1}$ associated to the weighted norm $\| u\|_{L^{2,1}}=\| \langle x \rangle u \|_{L^2}$. We recall that the interaction potential $F$ is assumed to be smooth, with a non-degenerate local minimum at zero. By \cite{BL}, there exists then a branch of ground state solitary waves $(\Phi_{\omega})_{\omega \in (0,\omega^*)}$.

\begin{theorem} \label{mainthm}
Assume that for $\omega \in (0,\omega^*)$:
\begin{itemize}
\item The orbital stability condition $c_\omega = \frac{d}{d\omega} \int |\Phi_\omega|^2\,dx > 0$ is satisfied.
\item The only eigenvalue of $\mathcal{H}$ is $0$, i.e. there are no internal modes.
\item $\mathcal{H}$ does not have a resonance in the sense of \eqref{id:assumption-spectral-2}.
\end{itemize}
Then, for each $\omega_0\in (0,\omega^*)$, there exists $\epsilon_0> 0$ such that if the data
$$
v_0 = \Phi_{\omega_0}+u_0 
$$
are sufficiently close to the soliton: 
$$
\| u_0 \|_{ H^1 } + \|  u_0 \|_{L^{2,1}}  =\epsilon < \epsilon_0,
$$
then the solution to \eqref{NLS} is global and 
\begin{itemize}
\item[(i)] \emph{Asymptotic stability.} The solution $v$ can be written as 
$$
\displaystyle v(t,x)=e^{i(p x+\gamma)}\Phi_{\omega}(x-y)+u(t,x)
$$
where the parameters $\omega(t) , \gamma(t) ,p(t) ,y(t)  \in  \mathbb R$ enjoy the asymptotics
\begin{align*}
&  \omega(t) = \underline \omega +O(\epsilon \langle t \rangle^{-1-\nu}), \qquad \qquad \qquad   p(t) = \underline p+O(\epsilon \langle t \rangle^{-1-\nu}) \\
&  \gamma(t) = (\underline p^2 - \underline \omega) t + \underline \gamma +O( \epsilon \langle t \rangle^{-\nu}), \qquad  y(t) = - 2\underline p t + \underline y +O( \epsilon \langle t \rangle^{-\nu }) ,
\end{align*}
for constants $\underline \omega, \underline \gamma, \underline p, \underline y$ and some $\nu>0$, and where the radiation $u$ satisfies $\| u\|_{H^1}\lesssim \epsilon$ and disperses:
$$
\| u(t,\cdot)\|_{L^\infty}\lesssim \epsilon \langle t \rangle^{-1/2}.
$$
\smallskip
\item[(ii)] \emph{Modified scattering for the radiation.} There exists a profile function $f$ with $\| f \|_{H^1}+\| \widehat f\|_{L^\infty}\lesssim \epsilon $ such that
$$
u =  \mathcal F^{-1}\left( e^{i\left(\xi^2 t-\frac{L}{2}|\widehat f|^2 \ln t\right)}\widehat f\right) +\widetilde u,
$$
with
$$
\| \widetilde u(t,\cdot) \|_{L^2} \to 0 \qquad \mbox{as }t\to \infty.
$$

\item[(iii)] \emph{Continuity.} The dependance of the asymptotic parameters and profile function on the initial data $u_0\mapsto (\omega,\gamma,p,y,f)$ is continuous from $H^1\cap L^{2,1}$ into $\mathbb R^4 \times H^1\cap \mathcal F^{-1}L^\infty$.
\end{itemize}

\end{theorem}

When are the spectral assumptions of the above theorem satisfied? As explained in Section~\ref{concrete}, they are satisfied in the case of the cubic-quintic nonlinear Schr\"odinger equation, corresponding to $F(x) = x^3 - x^5$, see \cite{PKA}. The spectral analysis which is needed to understand whether the conditions hold or not is quite involved, which explains that this is the only known case; but the expectation is that many other examples exist.

\subsection{Related model: nonlinear Klein-Gordon equation}
The stability theory for the nonlinear Klein-Gordon equation 
\begin{equation}
\label{NLKG} \tag{NLKG}
\partial_t^2 u - \Delta u = G'(u)
\end{equation}
set on the line is very close to that of~\eqref{NLS}. Before we review the parallel development of the different results which were recalled above for~\eqref{NLS}, it might be useful to stress first the differences. The first one is that solitary waves of~\eqref{NLKG} are pinned under appropriate symmetry assumptions (evenness): under such assumptions, it is possible to prove asymptotic stability without modulating, which is the framework of all the papers which will be reviewed below. On the contrary, modulation is necessary for~\eqref{NLS} irrespective of the symmetries which are imposed. The second difference is that much effort has been devoted to understanding the stability of non-localized solitary waves (known as kinks or topological solitons) while such objects have attracted less attention for~\eqref{NLS}.

We now turn to reviewing the stability theory for solitons of~\eqref{NLKG}, omitting the questions of modified scattering and orbital stability, which are very similar to~\eqref{NLS}. Asymptotic stability for high-power nonlinearities is due to~\cite{KK1,KK2}, and the supercritical case was the object of~\cite{KriegerNakanishiSchlag}.

More recently, a remarkable result was the proof by Kowalczyk, Martel and Munoz~\cite{KowalczykMartelMunoz1,KowalczykMartelMunoz2,KMMV} of the asymptotic stability of solitary waves through the virial method, allowing edge resonances, but not internal modes for the linearized operator.

If one does not rely on the virial approach, but rather on dispersive estimates, the main difficulty when proving asymptotic stability of solitary waves is the following: prove decay for Klein-Gordon equations including a potential and quadratic terms:
$$
\partial_t^2 u - \Delta u + u + V u = u^2.
$$
This difficulty was first addressed in~\cite{LLS1,LLS2,LLSS}. The second author with Pusateri~\cite{GP}, see also~\cite{GPZ,KairzhanPusateri} resorted to an approach via the distorted Fourier transform to obtain sharp decay and asymptotics, excluding edge resonances and internal modes. In the presence of internal modes, Delort and Masmoudi~\cite{DelortMasmoudi} used a semiclassical approach to obtain long-time existence. In the presence of resonance, an almost global result was obtained by Luhrmann and Schlag~\cite{LuhrmannSchlag2}

The equation~\eqref{NLKG} becomes integrable if $G=\sin$: this is the so-called Sine-Gordon equation. Asymptotic stability was obtained using the integrable structure by Chen, Liu and Lu~\cite{ChenLiuLu}. Noticing a crucial cancellation, Luhrmann and Schlag~\cite{LuhrmannSchlag1} were able to prove asymptotic stability through PDE means.

Finally, we refer to the surveys \cite{CuMa,KowalczykMartelMunoz1bis,Soffer} for results on the asymptotic stability of solitary waves in different contexts (other equations, higher dimensions).

\subsection{Challenges and ideas of the proof}

This article is the first to combine two fundamental ideas in nonlinear dispersive equations. On the one hand, the \textit{modulation analysis} allows to identify the dynamics of the soliton parameters. On the other hand,  the \textit{analysis of resonances} plays a key role in the dynamics of the radiation as well as the soliton parameters.

Here, resonances can be understood in the broad sense of space-time resonances, see~\cite{Germain} for an introduction, and also~\cite{GMS1,GMS2,GNT}. Their analysis allows for a more precise understanding of the dynamics, beyond the dispersive and Strichartz estimates, which have classically been used in conjunction with modulation, but which only record decay for the amplitude $|u|$ and do not detect further nonlinear cancellations due to oscillations. This improvement explains that we are able to reach the natural class $H^1 \cap L^{2,1}$ for the data, and that we do not require any vanishing of the field potential $F$ at the origin.

Let us now discuss how the two techniques, that of modulation and the analysis of resonances, can be brought together - details will be omitted, and approximations made, but we hope to convey the main idea. 

\medskip

\noindent \emph{Joint renormalization by the symmetries of the equation and the linearized group.} The solution $v$ is sought under the form
$$
V (t,x)= e^{i \sigma_3 (px-\gamma)}(\Phi_\omega + U(t))(x+y),
$$
where $U$ is the radiation, and $(\gamma,y,p,\omega)$ are the parameters of the soliton. To analyze in an optimal fashion the radiation, we filter it by the group generated by $\mathcal{H}_\omega$ (the linearization around the soliton) to obtain the profile
\begin{equation}
\label{formulaprofile}
f(t) = e^{-it\mathcal H_\omega}(e^{i\sigma_3 (px-\gamma) }U(t,x+y)).
\end{equation}
How should the parameters $(\gamma,y,p,\omega)$ be chosen? Should they be chosen on the spur of the moment, to approximate optimally at each time $t$ the solution $v$ by the soliton? Or should one rather keep the end in mind, set them to their value at $+ \infty$, since the formula~\eqref{formulaprofile} is quite sensitive to changes in the parameters? 

To shed some light on this debate, it is useful to introduce the following operators, which correspond to the conjugation of the infinitesimal generators of the soliton symmetries by the group, and are natural in view of~\eqref{formulaprofile}
\begin{align*}
& \partial_\omega f= e^{-it\mathcal H_\omega} \partial_\omega ( e^{it\mathcal H_\omega} f), \quad \partial_p f= i e^{-it\mathcal H_\omega}\sigma_3 x e^{it\mathcal H_\omega} f, \\
& \partial_y f = e^{-it\mathcal H_\omega}\partial_x e^{it\mathcal H_\omega} \quad \mbox{ and } \quad \partial_\gamma f = -ie^{-it\mathcal H_\omega}\sigma_3 e^{it\mathcal H_\omega}.
\end{align*}
The operators in the first line turn out to grow linearly in time, while the operators in the second line can be controlled. As a conclusion, we shall set $\omega$ and $p$ to their final values (which is technically achieved through a bootstrap argument) while $y$ and $\gamma$ are left to fluctuate freely. This amounts to linearizing around a soliton that is intermediate between the one at time $t$ and the final one. Due to the very slow decay of the modulation parameters towards their final values, this choice appears to be the only possible one!

\medskip

\noindent \emph{Improved modulation equations:} The evolution of the modulation parameters $(\omega,\gamma,p,y)$ of the solitary wave is affected by the radiation to leading order by quadratic terms. We estimate these interaction terms carefully by analyzing the space-time resonances that appear, which allows us to obtain a bound that goes beyond the one obtained by the sole use of the improved local decay for the linearized group.

\medskip

This gives the general idea of the combination of modulation and resonance analysis. However, its implementation requires a very careful analysis since modulation, resonances, and modified scattering are interlocked. At a technical level, this requires to develop further (after~\cite{BP} and~\cite{KS}) the theory of the distorted Fourier transform associated to $\mathcal{H}_\omega$, and to understand how it acts on various linear and nonlinear expressions. This also requires a functional setup which accomodates our minimal regularity assumptions; it is inspired by~\cite{GPR} and~\cite{GP}, and will not be detailed in this introduction. 

\subsection{Perspectives} The methods presented in this article seem versatile and powerful, and should apply to a number of problems, in dimension one and higher.

Focusing on the case of dimension one, it would be very interesting to relax the spectral assumptions made in the main theorem. In the case of eigenvalues with nonzero imaginary part, the construction of stable and unstable manifolds seems within reach. However, dealing with edge resonances or internal modes seems much harder, inevitably leading to a singularity in distorted Fourier space whose exact structure remains unclear.

Another interesting open problem would be to deal with the $L^2$-critical quintic Schr\"odinger equation, for which $c_\omega = 0$, and the dynamic in a neighborhood of the soliton is much richer, including the possibility of finite time blow up.

Finally, it would also be of interest to extend some classical constructions of the theory of solitary waves to the context considered here: multi-soliton configurations and their stability, interaction between solitons.

\subsection{Organization of the article} The first part of the paper is dedicated to the analysis of the linearized problem. Classical facts on the spectrum of vector Schr\"odinger operators are recalled in Section~\ref{sectionspectrum}; the scattering theory for the linearized operator of Krieger-Schlag~\cite{KS} is sharpened in Section~\ref{sectionscattering}. Turning to the corresponding distorted Fourier transform, its linear theory (in particular, conjugation of various linear operators) is developed in Section~\ref{sectiondistorted}, and its nonlinear theory (in particular, action on products of functions) in Section~\ref{sectionnonlinear}. Finally, linear estimates on the group $e^{it\mathcal{H}_\omega}$ are established in Section~\ref{sectionlinear}.

The heart of the paper consists of Section~\ref{sectionwriting}, where the equation is set in the appropriate form to allow for estimates, and of Section~\ref{sectionbootstrap}, where the central bootstrap argument is set up.

The rest of the paper is dedicated to the control of the various nonlinear terms through the analysis of resonances: the equation governing modulation parameters is treated in Section~\ref{sectionmodulation}. The equation controlling the radiation is the object of sections~\ref{sectionquadratic},~\ref{sectionremaining} and~\ref{sectionpointwise}, which address,  the weighted $L^2$ norm for the quadratic terms, the weighted $L^2$ norm for the remaining terms, and the pointwise bound in Fourier, respectively. Finally, Section~\ref{sectionmodified} shows modified scattering in the classical Fourier sense.

\subsection*{Acknowledgements} When preparing this article, Pierre Germain was supported by the Simons Foundation Collaboration on Wave Turbulence, a start up grant from Imperial College, and a Wolfson fellowship. 

Charles Collot was supported by the CY Initiative of Excellence Grant "Investissements d'Avenir" [ANR-16-IDEX-0008 to Charles Collot], and by the Chaire Professeur Junior grant [ANR-22-CPJ2-0018-01 to Charles Collot].

\section{Notations}

\subsection{Miscellaneous}

Given two quantities $A$ and $B$, the notation $A \lesssim B$ means that, for a universal constant $C>0$, there holds $ B \leq CA$. If the constant $C$ is allowed to depend on parameters $\alpha_i$, the notation becomes $A \lesssim_{\alpha_i} B$. We will denote $A \sim B$ if $A \lesssim B$ and $B \lesssim A$, . Finally, if $A \leq CB$ for a sufficiently (depending on the context) small constant $C>0$, then we write $A \ll B$.

The real scalar product between vectors $x$ and $y$ in $\mathbb{R}^2$ is denoted by 
$$
x \cdot y = x_1 y_1 + x_2 y_2,
$$ 
while the Hermitian scalar product on $L^2 (\mathbb{R},\mathbb{C}^2)$ is denoted
$$
	\langle f \,,\,g \rangle = \int f(x) \cdot \overline{g(x)} \,dx.
	$$
	The Pauli matrices are
	\begin{equation} \label{S-v-eq4}
	\begin{split}
	\sigma_1 = 
	\left( \begin{array}{ccc}
	0 \quad 1 \\
	1 \quad 0
	\end{array} \right) , \
	\sigma_2 = 
	\left( \begin{array}{ccc}
	0 \ \ -i \\
	i \qquad 0
	\end{array} \right) ,  \
	\sigma_3 = 
	\left( \begin{array}{ccc}
	1 \qquad 0 \\
	0 \ \ -1
	\end{array} \right) . 
	\end{split}
	\end{equation}
	We will also write
	\begin{equation} \label{S-v-eq6}
	\begin{split}
	e_1 = \begin{pmatrix} 1 \\ 0 \end{pmatrix}, \
	e_2 = \begin{pmatrix} 0 \\ 1 \end{pmatrix}, \
	p = 
	\left( \begin{array}{ccc}
	1 \quad 0 \\
	0 \quad 0
	\end{array} \right) , \
	q = 
	\left( \begin{array}{ccc}
	0 \quad 0 \\
	0 \quad 1
	\end{array} \right)  . 
	\end{split}
	\end{equation}
	We let
	$$
	\langle \xi \rangle = \sqrt{\xi^2 + 2}.
	$$
Finally, we adopt the following normalization for the Fourier transform on $\mathbb{R}$
$$
\widehat{f}(\xi) = \frac{1}{\sqrt{2\pi}} \int f(x) e^{-ix\xi} \,dx \qquad \Leftrightarrow \qquad f(x) = \frac{1}{\sqrt{2\pi}} \int \widehat{f}(\xi) e^{ix\xi} \,d\xi.
$$

\subsection{Dyadic decomposition}

\label{dyadicdec}

We will work wiht a smooth, inhomogeneous partition of unity $(\varphi_k(x))_{k \geq 0}$ such that
\begin{align*}
&\sum_{k \geq 0} \varphi_k(x) = 1 \;\; \mbox{for all $x \in \R$}\\
&  \operatorname{Supp} \varphi_0 = [-1,1] \\
& \varphi_k(x) = \phi(2^{-k} |x|), \;\; \mbox{with $\operatorname{Supp} \phi \subset [\frac{1}{2}, 2]$}
\end{align*}
We will furthermore denote
\begin{align*}
& \varphi_{< \ell} = \sum_{0 \leq k < \ell} \varphi_k \\
& \varphi_{\leq \ell} = \sum_{0 \leq k \leq \ell} \varphi_k,
\end{align*}
with obvious generalizations to $\varphi_{>\ell}$, $\varphi_{\geq \ell}$.

Finally, we will denote $\varphi_{\sim k}$ a generic smooth cutoff function 
that is supported on an annulus $|\xi| \sim 2^k$.

\section{The spectrum of vector Schr\"odinger operators}

\label{sectionspectrum}

\subsection{General properties}
In the present section and until Section \ref{sectionlinear}, we take $\omega=1$ without loss of generality and consider a general vector Schr\"odinger operator of the form
$$
\mathcal{H} = \mathcal{H}_0 + V,
$$
where
\begin{equation*} 
\begin{split}
\mathcal{H}_0 = 
\begin{pmatrix}
-\partial_x^2 + 1 & 0  \\
0 &  \partial_x^2 - 1
\end{pmatrix}
\qquad 
V = 
\begin{pmatrix}
V_1  & V_2
\\
-V_2 & -V_1
\end{pmatrix}
\end{split}
\end{equation*}
and the potentials $V_1$, $V_2$ are such that:
\begin{equation}
\tag{H1} \label{H1} \mbox{$V_{1}$ and $V_2$ are even} \\
\end{equation}
and
\begin{equation}
\tag{H2} \label{H2} \mbox{For $k \in \mathbb{N}_0$, \;$\left| \partial_x^k V_1(x) \right| + \left| \partial_x^k V_2(x) \right| \lesssim_k e^{-\beta |x|}$ for some $\beta \in (0,1)$}.
\end{equation}
The following symmetries follow from direct computations.

\begin{lemma}[Symmetries]
The operator $\mathcal{H}$ satisfies the commutation relations
\begin{equation} \label{S-v-eq5}
\begin{split}
\sigma_3 \mathcal{H}^* \sigma_3 = \mathcal{H} \qquad \mbox{and} \qquad \sigma_1 \mathcal{H} \sigma_1 =  - \mathcal{H} .
\end{split}
\end{equation}
\end{lemma}

Next, we want to discuss the spectrum of this operator, and choose the natural domain
$$
\mathcal{D}(\mathcal H) = H^2(\mathbb{R}) \times H^2(\mathbb{R}),
$$
on which $\mathcal{H}$ is closed.

\begin{definition} \begin{itemize} 
\item The \textit{resolvent set} of $\mathcal{H}$ is the set of $\lambda$ for which $\mathcal{H}-\lambda$ admits a bounded inverse $L^2 \times L^2 \to \mathcal{D}(\mathcal H)$.
\item The \textit{spectrum} $\sigma(\mathcal{H})$ is the complement of the resolvent set.
\item The Schr\"odinger operator has an \textit{eigenvalue} at energy $E$ if there exists $\phi \in L^2 \setminus \{ 0 \}$ such that $\mathcal{H} \phi = E \phi$.
\item It is an \textit{embedded} eigenvalue if $E \in (-\infty,-1] \cup [1,\infty)$.
\item The Schr\"odinger operator $\mathcal{H}$ has a \textit{resonance} at energy $E$ if there exists $\phi \in L^\infty \setminus L^2$ such that $\mathcal{H}\phi = E \phi$.
\item The \textit{discrete spectrum} $\sigma_d(\mathcal{H})$ is the set of isolated points of the spectrum with finite algebraic multiplicity.
\item The \textit{essential spectrum} $\sigma_e(\mathcal{H})$ equals $\sigma(\mathcal{H}) \setminus \sigma_d(\mathcal{H})$.
\end{itemize}
\end{definition}

General results on the spectral analysis of $\mathcal{H}$ are as follows.
\begin{lemma}[Spectrum of $\mathcal{H}$] \label{lemmaspectrum} Assume \eqref{H1}-\eqref{H2}. Then
	\begin{itemize}
\item[(i)] $\sigma_e(\mathcal{H}) = (-\infty,-1] \cup [1,\infty)$

\item[(ii)] The discrete spectrum $\sigma_d(\mathcal{H})$ consists of eigenvalues with finite algebraic multiplicity.

\item[(iii)] The set of eigenvalues is finite.

 \item[(iv)] Embedded eigenvalues and resonances can exist.
		
	\end{itemize}
\end{lemma}

\begin{proof} 

\noindent  \underline{Proofs of (i) and (ii)}. This is \cite{HL}, Theorem 1.3.\\

\noindent  \underline{Proof of (iii)}. This is proved in \cite{CPV}.\\

\noindent  \underline{Proof of (iv)}. To see that embedded eigenvalues can exist, consider the following scenario, already pointed out in~\cite{KS}: assume $V_2 = 0$, and $V_1$ is such that the scalar Schr\"odinger operator $-\partial_x^2 - V_1$ has an eigenfunction $\varphi$ with a negative eigenvalue. Then the eigenvector $\begin{pmatrix} 0 \\ \varphi \end{pmatrix}$ is associated to an embedded eigenvalue of $\mathcal{H}$.

\end{proof}

The Riesz projector on the discrete spectrum, when it is made of finitely many points, is defined as
$$
P_d = \frac{1}{2\pi i} \oint_\gamma (z \operatorname{Id} - \mathcal{H})^{-1} \,dz,
$$
where $\gamma$ is a simple curve enclosing the whole discrete spectrum, and lying within the resolvent set. Denoting the eigenfunctions by $(\varphi_i)_{i=1 \dots n}$, the range of $P_d$ is
$$
\operatorname{Im} P_d = \operatorname{Span} (\varphi_1, \dots, \varphi_n).
$$
The projector on the essential spectrum, $P_e$, is such that
$$
P_d + P_e = \operatorname{Id}.
$$

The projectors $P_d$ and $P_e$ satisfy the classical properties of spectral projectors (see~\cite{HS})
\begin{align*}
& P_d^2 = P_d, \qquad P_e^2 = P_e \\
& \mathcal{H}P_d = P_d \mathcal{H}, \qquad \mathcal{H} P_e = P_e \mathcal{H}.
\end{align*}

\subsection{The linearized operator around a ground state soliton of~\eqref{NLS}}

In this case,
\begin{align*}
& V_1 = - F'(\Phi^2) + F''(\Phi^2) \Phi^2, \\
& V_2 = - F''(\Phi^2) \Phi^2,
\end{align*}
where $(\Phi_\omega)_{\omega\in (0,\omega^*)}$ is the branch of ground state solutions to \eqref{eq:soliton} with $\Phi=\Phi_\omega>0$.

\begin{lemma}[Generalised kernel] \label{kernel}

The functions $\Xi_j$ for $j=0,1,2,3$ given by \eqref{id:generalized-kernel} satisfy the relations \eqref{id:generalized-kernel-2}. Moreover, if $c_\omega>0$ and $U\in L^2$ is such that $\mathcal H^j U=(0,0)^\top $ for some $j\geq 1$ then $U\in \textup{Span}\{ \Xi_j\}_{j=0,1,2,3}$.

For the adjoint operator $\mathcal H^*$ there hold the relations
$$
\mathcal H^*\sigma_3 \Xi_0=(0,0)^\top, \quad \mathcal H^*\sigma_3 \Xi_1=-\sigma_3 \Xi_0, \quad \mathcal H^*\sigma_3 \Xi_2=(0,0)^\top, \quad  \mathcal H^*\sigma_3 \Xi_3=-2 \sigma_3 \Xi_2.
$$
If $c_\omega>0$ and $U\in L^2$ is such that $(\mathcal H^*)^j U=(0,0)^\top$ for some $j\geq 1$ then $U\in \textup{Span}\{\sigma_3 \Xi_j\}_{j=0,1,2,3}$.

\end{lemma}

\begin{proof}

See Proposition 1.2.2 in~\cite{BP}.

\end{proof}

\begin{lemma}[Localisations of eigenvalues]
The eigenvalues of $\mathcal H$ belong to $\mathbb R\cup i\mathbb R$. Under the additional orbital stability condition $c_\omega>0$, they are real.
\end{lemma}

\begin{proof}

See \cite{Weinstein} or the proof of Proposition 9.2. in \cite{KS}.

\end{proof}

\subsection{Concrete examples}

\label{concrete}

Our framework does not apply to pure power nonlinearities
\begin{itemize}
\item Cubic NLS $F(x) = x^2$, then $c_\omega = \frac{d}{d\omega} \| \Phi_\omega \|_{L^2}^2 > 0$ but there is a resonance and internal modes, see~\cite{CGNT}. 
\item Quintic NLS $F(x) = x^3$, then $c_\omega = 0$ and the generalized kernel of the linearized operator has dimension 6 - see~\cite{Weinstein}.
\item Septic NLS  $F(x) = x^4$ and higher order, then $c_\omega < 0$, hence unstable eigenmodes.
\end{itemize}

To the best of our knowledge, the only other case where the spectral resolution of the linearized system is well-understood is the cubic-quintic case $F(z) = x^2 + \sigma x^3$, which was analyzed by \cite{PKA}. These authors show that an internal mode is present for $\sigma = 1$, but that for $\sigma = -1$, there is neither internal mode nor resonance, and that furthermore $c_\omega >0$.

\section{Scattering Theory}

\label{sectionscattering}
\label{S-v-lst}

This section is devoted to the scattering theory of an operator $\mathcal{H}$ satisfying properties~\eqref{H1} and~\eqref{H2}. We still take $\omega=1$ without loss of generality.

We follow closely the construction in Krieger-Schlag~\cite{KS}, but some of its building blocks have to be revisited in order to obtain sharper estimates which suit our needs; for the parts that do not require improvements, we mostly refer to the original article.

\subsection{Jost solutions} We will construct functions $f_i$, $i=1,\dots,4$, which are analogous to the classical Jost solutions in the scalar case. These are generalized eigenfunctions of $\mathcal{H}$
$$
\mathcal{H}f_j(\cdot,\xi) = (1+ \xi^2) f_j(\cdot,\xi), \qquad j=1,2,3,4,
$$
which, as $x \to \infty$, satisfy
$$
f_1(x,\xi) \approx e^{ix\xi}e_1 , \quad f_2(x,\xi) \approx e^{-ix\xi}e_1 , \quad f_3(x,\xi) \approx e^{-x\langle \xi \rangle}e_2, \quad f_4(x,\xi) \approx e^{x\langle \xi \rangle}e_2.
$$
	These Jost solutions will provide the basis for the scattering theory and the distorted Fourier transform.
	
	\begin{lemma}  \label{lemma3} 
		For every $\xi \in \mathbb{R}$, there exists a real-valued solution $f_3 (x, \xi)$ of the equation
		\begin{equation}
		\label{eigenfunction} \mathcal{H} f_3 (\cdot , \xi) = (\xi^2 +1 ) f_3 (\cdot , \xi ).
		\end{equation}
with the property that $f_3 (x , \xi) \approx e^{- \langle \xi \rangle x} e_2$ as $x \rightarrow +\infty$. 
		
More precisely, $f_3$ can be written
$$
f_3(x,\xi) = \mathfrak{f}_3(x,\xi) e^{-x \langle \xi \rangle}
$$	
with
\begin{equation}    \label{papillon}
| \partial_\xi^\ell \partial_x^k [ \mathfrak{f}_3(x,\xi)  - e_2] | \lesssim_{k,\ell}   \la \xi \ra^{-1-\ell} e^{-\beta x}
\end{equation}
for all $x \geq -1$, $k$, $\ell \geq 0$.
\end{lemma}
	
	\begin{proof} \underline{Estimate without derivatives.} 
		The eigenfunction equation~\eqref{eigenfunction} with the condition \eqref{papillon} at infinity can be written as a Volterra equation
		\begin{equation*}  
		f_3 (x, \xi)  = e^{- x \langle \xi \rangle } 
		\begin{pmatrix} 0 \\ 1 \end{pmatrix}
		+ \int^{\infty}_x
		\begin{pmatrix}
		\frac{\sin(\xi(y-x))}{\xi}  &  0 \\
		0  &  -\frac{\sinh(\langle \xi \rangle(y-x))}{\langle \xi \rangle}
		\end{pmatrix}
		V(y) f_3 (y, \xi) dy  
		\end{equation*}
		For the unknown function $\mathfrak{f}_3(x,\xi) = e^{x \langle \xi \rangle} f_3(x,\xi)$, this becomes
		\begin{equation}
		\label{integraleq}
		\mathfrak{f}_3(x,\xi) = e_2  + \int^{\infty}_x K(x,y,\xi) V(y) \mathfrak{f}_3(y,\xi) \,dy
		\end{equation}
		with
		$$
		K(x,y,\xi) = \begin{pmatrix}
		\frac{\sin(\xi(y-x))}{\xi}  &  0 \\
		0  &  -\frac{\sinh(\langle \xi \rangle(y-x))}{\langle \xi \rangle}
		\end{pmatrix} e^{(x-y) \langle \xi \rangle}.
		$$
		This kernel is such that
		\begin{equation} \label{chauve-souris}
		\sup_{y \geq x} |K(x,y,\xi)| \lesssim \frac{1}{\langle \xi \rangle}.
		\end{equation}
		Therefore, the operator which to $\mathfrak f_3$ associates the right-hand side of~\eqref{integraleq} is a contraction in $L^\infty(x_0,\infty)$ if $\| V \|_{L^1(x_0,\infty)} \leq c$, for a sufficiently small constant $c$. Choosing $x_0$ accordingly thanks to \eqref{H2}, the Banach theorem gives a solution $\mathfrak{f}_3(x,\xi)$ which is $O(1)$ in $L^\infty(x_0,\infty)$.
		
		In order to extend this bound to $(-1,\infty)$, we derive from~\eqref{integraleq} and \eqref{chauve-souris} the inequality, valid for $x \leq x_0$,
		$$
		|\mathfrak{f}_3(x,\xi)| \lesssim 1 + \int_{x}^{x_0} |\mathfrak{f}_3(y,\xi)| \,dy.
		$$
		Combined with Gronwall's lemma, it implies that $\|\mathfrak{f}_3(\cdot,\xi) \|_{L^\infty(-1,\infty)} = O(1)$.
		
		Finally, using \eqref{integraleq} and $\|\mathfrak{f}_3(\cdot,\xi) \|_{L^\infty(-1,\infty)} \lesssim 1$, we see that, if $x \geq -1$,
		\begin{equation*}
		|\mathfrak{f}_3(x,\xi) - e_2| \lesssim \int_x^\infty \left| K(x,y,\xi) V(y) \right| \,dy \lesssim \int_x^\infty \frac{e^{-\beta y}}{\langle \xi \rangle} \,dy \lesssim \frac{e^{-\beta x}}{\langle \xi \rangle},
		\end{equation*}
		which is estimate~\eqref{papillon} if $k = \ell = 0$.
		
		\medskip
		
		\noindent \underline{Taking derivatives in $x$.} Since $\partial_x K(x,y,\xi) = - \partial_y K(x,y,\xi)$, we see that $\partial_x \mathfrak{f}_3$ satisfies
		$$
		\partial_x \mathfrak{f}_3(x,\xi)  = \int^{\infty}_x K(x,y,\xi)\partial_y V(y) \mathfrak{f}_3(y,\xi) \,dy + \int^{\infty}_x K(x,y,\xi) V(y) \partial_y \mathfrak{f}_3(y,\xi) \,dy.
		$$
		The second term on the right-hand side is a contraction in $L^\infty(x_0,\infty)$ (as a function of $\partial_x f_3$). It is then possible to follow the same reasoning as in bounding $\mathfrak{f}_3(x,\xi)$: first obtain $\|\partial_x \mathfrak{f}_3 \|_{L^\infty(x_0,\infty)} \lesssim 1$ by Banach's fixed point theorem, then $\| \partial_x \mathfrak{f}_3 \|_{L^\infty(-1,\infty))} \lesssim 1$ by Gronwall's lemma, and finally $| \partial_x \mathfrak{f}_3(x,\xi) | \lesssim \frac{e^{-\beta x}}{\langle \xi \rangle}$.
		
		Following the scheme which has been demonstrated for $k=1$, It is then possible to set up an induction procedure to handle more derivatives in $x$, with the induction hypothesis at rank $k$ being that, for $r \leq k$
		$$
		| \partial_x^{r} \mathfrak{f}_3(x,\xi)| \lesssim \frac{e^{-\beta x}}{\langle \xi \rangle}.
		$$
		To prove the statement at rank $k$ assuming it at rank $k-1$, we use the following formula, which can be proved by repeated integration by parts
		$$
		\partial_x^k \mathfrak{f}_3(x,\xi) = \int^{\infty}_x \sum_{j=0}^k \begin{pmatrix} k \\ j \end{pmatrix} K(x,y,\xi) \partial_y^{j} V(y) \partial_y^{k-j} \mathfrak{f}_3(y,\xi) \,dy.
		$$
		
		\medskip

\noindent \underline{Bound on the kernel $K$.} We claim that
\begin{equation} \label{hermine}
\sup_{y \geq x} |\partial_\xi^\ell K(x,y,\xi)| \lesssim \frac{1}{\langle \xi \rangle^{1+ \ell}}.
\end{equation}
To show it, we define for $z=y-x\geq 0$ the function $\psi_z(\xi)=\frac{\sin(z \xi) }{\xi} e^{-z\langle \xi\rangle}$. It extends to a holomorphic function of the complexified variable $\xi $ on $\Omega =\{|\xi|\leq \delta_1\}\cup \{\Re \xi\geq 0 \mbox{ and } |\Im \xi|\leq \delta_2 \Re \xi\}$ for some $\delta_1,\delta_2>0$. Using that $|\sin(z \xi) |\lesssim \min(1,z |\xi|)e^{z|\Im \xi|}$ we obtain $|\psi_z(\xi)|\lesssim \min(|\xi|^{-1},z)e^{-z(\langle \xi\rangle-|\Im \xi|)}$. Hence $\psi$ is uniformly bounded by $\langle \xi \rangle^{-1}$ on $\Omega$ (for some possibly smaller $\delta_1,\delta_2>0$). This implies using the Cauchy integral formula that $|\partial_\xi^{\ell}\psi |\lesssim \langle \xi \rangle^{-1-\ell}$ for all $\xi \in [0,\infty)$. This shows the desired bound \eqref{hermine} for the first component of $K$. The bound for the other component can be proved with the same argument.\\

\noindent \underline{Taking derivatives in $\xi$.} We first claim that
$$
\sup_{y \geq x} |\partial_\xi^\ell K(x,y,\xi)| \lesssim \frac{1}{\langle \xi \rangle^{1+ \ell}}.
$$
Indeed, 
and to use similar estimates to the above, bounding successively $\partial_\xi^\ell \partial_x^k \mathfrak{f}_3(x,\xi)$, for $\ell=1$, $k=0\dots \infty$, then $\ell=2$, $k=0\dots \infty$, etc...

For instance, in the case $k=0, \ell=1$,
$$
\partial_\xi \mathfrak{f}_3(x,\xi) = \int_x^\infty \partial_\xi K(x,y,\xi) V(y) \mathfrak{f}_3(y,\xi) \,dy + \int_x^\infty  K(x,y,\xi) V(y) \partial_\xi \mathfrak{f}_3(y,\xi) \,dy.
$$
The first term on the right-hand side is bounded pointwise by $\frac{1}{\langle \xi \rangle^2}$ for $x \geq -1$, while the second is a contraction on $L^\infty(x_0,\infty)$; the usual arguments apply then.
\end{proof}

\begin{lemma}  \label{lemma12}
For every $\xi \in \mathbb{R}$, there exists a solution $f_1 (x, \xi)$ of the equation
$$ \mathcal{H} f_1 (\cdot , \xi) = (\xi^2 +1 ) f_1 (\cdot , \xi )$$
with the property that $f_1(x,\xi) \approx e^{ix\xi}e_1$ as $x\to \infty$.

More precisely, it is possible to decompose
\begin{equation} \label{stollen}
f_1 (x , \xi) = \mathfrak{f}_{1,1}(x,\xi) e^{ix\xi} + \mathfrak{f}_{1,2}(x,\xi) e^{-x\langle \xi \rangle}, 
\end{equation}
where, for any $x > -1$,
\begin{align}    \label{colvert}
& | \partial_x^k \partial_\xi^\ell (\mathfrak{f}_{1,1} (x, \xi) - e_1) | \lesssim_{k,\ell}   \la \xi \ra^{-1-\ell} e^{-\frac \beta 2 x } \\
& | \partial_x^k \partial_\xi^\ell \mathfrak{f}_{1,2} (x, \xi)  | \lesssim_{k,\ell}  \la \xi \ra^{-1-\ell} \left[ e^{\langle \xi \rangle} \mathbf{1}_{x>2} + 1\right].
\end{align}

\end{lemma}

\begin{definition}
The function $f_2(x,\xi)$ is defined as $f_2(x,\xi) = \overline{f_1 (x, \xi)}$.
\end{definition}

\begin{proof} The case $\xi = O(1)$ is covered in~\cite{KS}, Lemma 5.3, we will therefore restrict our investigations to the case $|\xi| \gg 1$, for which more precise estimates are needed.

\bigskip \noindent
\underline{Choosing the ansatz and writing the associated equation.} We seek a solution of the form
		\begin{equation}   \label{S-v-lst-lm5.3-KS-1-eq2}
		f_1 (x, \xi) = 
		e_1
		v(x, \xi) 
		+ f_3 (x, \xi) u(x, \xi) 
		\end{equation}
		where $v(x, \xi) \sim e^{i x \xi}$ as $x\to \infty$, and $f_3 (x, \xi) = (f_3^{(1)} (x, \xi), f_3^{(2)} (x, \xi))^\top$ is as in Lemma \ref{lemma3}. In order to have $ \mathcal{H} f_1 (\cdot , \xi) = (\xi^2 +1 ) f_1 (\cdot , \xi )$, we must have
		\begin{equation}   \label{S-v-lst-lm5.3-KS-1-eq3}
		\begin{split}
		0 
		& = (\mathcal{H}- (\xi^2+1)) f_1 (\cdot, \xi) \\
		& =  
		\begin{pmatrix}
		(-\partial_x^2 -\xi^2 + V_{1} ) v(\cdot, \xi) \\
		-V_2 v(\cdot, \xi)
		\end{pmatrix}
		+ 
		\begin{pmatrix}
		-\partial_x^2 u (\cdot, \xi) f_3^{(1)} (\cdot, \xi) - 2 \partial_x  u (\cdot, \xi) \partial_x f_3^{(1)} (\cdot, \xi)   \\
		\partial_x^2 u (\cdot, \xi) f_3^{(2)} (\cdot, \xi) + 2 \partial_x  u (\cdot, \xi) \partial_x f_3^{(2)} (\cdot, \xi) 
		\end{pmatrix} 
		.
		\end{split}
		\end{equation}
Integrating out the second line of \eqref{S-v-lst-lm5.3-KS-1-eq3} in the same way as in Lemma 5.3 of \cite{KS}, in particular imposing the boundary conditions $u'(\infty) =0$ and $u(0)=0$, gives
\begin{equation}   
\label{sittelle}
u(x, \xi) = - \int_{\substack{0<y<x \\ y<z}} [ f_3^{(2)} (y, \xi) ]^{-2} f_3^{(2)} (z, \xi)  V_2 (z)  v(z,\xi) \, dz \, dy .
\end{equation}
Note that the formula \eqref{sittelle} is well defined, as from \eqref{papillon} and the fact that $|\xi|\gg 1$ we recall that $ f_3^{(2)} (y, \xi) $ does not vanish for $y\geq 0$. Setting $w(x,\xi)  = e^{-i x \xi} v(x,\xi )$, it satisfies
\begin{equation*}   
\begin{split}
& w (x, \xi) = 1+ \int_x^\infty  D_\xi(y-x) \left[ - V_{2} (y) \frac{f_3^{(1)}(y,\xi)}{f_3^{(2)}(y,\xi)} + V_{1} (y) \right] w (y,\xi)\, dy   \\
& + 2 \int_{x<y<z} D_\xi(y-x) e^{i(z-y)\xi} \left[ - \frac{f_3^{(2)'}(y,\xi)}{f_3^{(2)}(y,\xi)} f_3^{(1)} (y, \xi) + f_3^{(1)'} (y, \xi) \right] \frac{f_3^{(2)}(z,\xi)}{f_3^{(2)} (y, \xi)^2} V_{2} (z) w(z,\xi) \,dz \,dy,
\end{split}
\end{equation*}
where we set
$$
D_\xi(z) = \frac{e^{2i\xi z} - 1}{2i\xi}.
$$
Next, it is useful to replace the occurences of $f_3$ above by $\mathfrak{f}_3(x) = e^{x \langle \xi \rangle} f_3(x,\xi)$, since this is the quantity estimated in Lemma~\ref{lemma3}. The equation becomes (omitting the $\xi$ dependence for simplicity)
\begin{equation*}   
\begin{split}
& w (x) = 1+ \int_x^\infty  D_\xi(y-x) \left[ - V_{2} (y) \frac{\mathfrak{f}_3^{(1)}(y)}{\mathfrak{f}_3^{(2)}(y)} + V_{1} (y) \right] w (y)\, dy   \\
& + 2 \int_{x<y<z} D_\xi(y-x) e^{(z-y)[i\xi - \langle\xi \rangle]} \frac{\mathfrak{f}_3^{(2)}(y)\mathfrak{f}_3^{(1)'} (y) -\mathfrak{f}_3^{(2)'}(y)\mathfrak{f}_3^{(1)} (y) }{\mathfrak{f}_3^{(2)}(y)^2}  \frac{\mathfrak{f}_3^{(2)}(z)}{\mathfrak{f}_3^{(2)} (y)} V_{2} (z) w(z) \,dz \,dy.
\end{split}
\end{equation*}
This can be written in a more compact form as
\begin{equation}
\label{ecureuil}
w(x) = 1+ \int_x^\infty  D_\xi(y-x) A(y,\xi) w (y)\, dy + 2  \int_{x<y<z} D_\xi(y-x) e^{(z-y)[i\xi - \langle\xi \rangle]} B(y,z,\xi) w(z) \,dz \,dy
\end{equation}
with by \eqref{papillon} and \eqref{H2}
\begin{equation}
\label{bernache}
\begin{split}
& | \partial_y^k  \partial_\xi^\ell A(y,\xi) | \lesssim_{k,\ell} \langle \xi \rangle^{-\ell} e^{-\beta y}, \\
& | \partial_y^j  \partial_z^k \partial_\xi^\ell  B(y,z,\xi) | \lesssim_{j,k,\ell} \langle \xi \rangle^{-\ell} e^{-\beta z}.
\end{split}
\end{equation}
We claim first that
\begin{equation} \label{rouge-gorge}
| \partial_x^k \partial_\xi^\ell (w(x,\xi)-1) | \lesssim \langle \xi \rangle^{-1-\ell} e^{-\frac \beta 2 x}.
\end{equation}

\bigskip

\noindent \underline{The fixed point argument for $w$.} Simply writing the Volterra equation as
\begin{equation}
\label{scarabee}
w(x) = 1 + \int_x^\infty K(x,y,\xi) w(y) \,dy,
\end{equation}
the kernel $K$ can be bounded by
$$
|K(x,y,\xi)| \lesssim \frac{1}{|\xi|} e^{-\beta y} + \frac{1}{|\xi|} \int_{x<y<z} e^{(y-z) \langle \xi \rangle} e^{-\beta z} \,dz \lesssim \frac{1}{|\xi|} e^{-\beta y}.
$$
Since we are assuming $|\xi| \gg 1$, it is possible to apply Banach's fixed point theorem to obtain a solution $w$ of size $O(1)$ in $L^\infty(-1,\infty)$. Therefore,
$$
|w(x) - 1| \lesssim \int_x^\infty \frac{1}{|\xi|} e^{-\beta y} \,dy \lesssim \frac{e^{-\beta x} }{|\xi|}.
$$

\bigskip

\noindent \underline{Applying $x$-derivatives to $w$.} Applying $\partial_x$ to~\eqref{ecureuil} while keeping in mind that $D_\xi(0)=0$ leads to
\begin{align*}
\partial_x w(x) & = \int_x^\infty \partial_x  D_\xi(y-x) A(y,\xi) w (y)\, dy \\
& \qquad \qquad \qquad + 2  \int_{x<y<z} \partial_x D_\xi(y-x) e^{(z-y)[i\xi - \langle\xi \rangle]} B(y,z,\xi) w(z) \,dz \,dy.
\end{align*}
Since $\partial_x  D_\xi(y-x) = - \partial_y  D_\xi(y-x)$, we can integrate by parts to obtain
\begin{align*}
\partial_x w(x) & = \int_x^\infty   D_\xi(y-x) \partial_y A(y,\xi) w (y)\, dy + \int_x^\infty   D_\xi(y-x)  A(y,\xi)\partial_y w (y)\, dy \\
& \qquad \qquad \qquad + 2  \int_{x<y<z}  D_\xi(y-x) e^{(z-y)[i\xi - \langle\xi \rangle]}  \partial_y B(y,z,\xi) w(z) \,dz \,dy \\
& \qquad \qquad \qquad + 2  \int_{x<y<z}  D_\xi(y-x) \partial_y e^{(z-y)[i\xi - \langle\xi \rangle]} B(y,z,\xi) w(z) \,dz \,dy \\
& \qquad \qquad \qquad - 2  \int_{x<z}  D_\xi(z-x) B(z,z,\xi) w(z) \,dz.
\end{align*}
Once again, we use that $\partial_y e^{(z-y)[i\xi - \langle\xi \rangle]} = - \partial_z e^{(z-y)[i\xi - \langle\xi \rangle]} $ to integrate by parts in $z$, which yields
\begin{align*}
\partial_x w(x) & = \int_x^\infty K(x,y,\xi) \partial_y w (y)\, dy + \int_x^\infty   D_\xi(y-x) \partial_y A(y,\xi) w (y)\, dy \\
& \qquad \qquad \qquad + 2  \int_{x<y<z}  D_\xi(y-x) e^{(z-y)[i\xi - \langle\xi \rangle]}  \partial_y B(y,z,\xi) w(z) \,dz \,dy \\
& \qquad \qquad \qquad + 2  \int_{x<y<z}  D_\xi(y-x) e^{(z-y)[i\xi - \langle\xi \rangle]} \partial_z B(y,z,\xi) w(z) \,dz \,dy
\end{align*}
Here, $K$ is the kernel defined in~\eqref{scarabee}, so that the first term on the above right-hand side is a contraction. As for other terms on the right-hand side, they are pointwise bounded by $O(1)$. Banach's fixed point theorem gives a solution of size $O(1)$ in $L^\infty(-1,\infty)$. It is then easy to deduce as before that $|\partial_x w(x)| \lesssim \frac{1}{|\xi|} e^{-\beta x}$.

Following the same pattern, higher order derivatives in $x$ can be treated inductively.

\bigskip

\noindent \underline{Applying $\xi$-derivatives to $w$.} We now apply $\partial_\xi$ to~\eqref{ecureuil}, and use the identities
\begin{align*}
& \partial_\xi D_\xi(y-x) = 2i(x-y) \frac{e^{2i\xi (x-y)}}{\xi} - \frac{1}{\xi} D_\xi(y-x) \\
& \partial_\xi  e^{(z-y)[i\xi - \langle\xi \rangle]} = (z-y)\left(i - \frac{\xi}{\langle \xi \rangle} \right) e^{(z-y)[i\xi - \langle\xi \rangle]}
\end{align*}
to obtain
\begin{equation}
\label{bernache-1}
\begin{split}
\partial_\xi w(x) & = \int_x^\infty K(x,y,\xi) \partial_\xi w (y)\, dy + W(x)+ \int_x^\infty 2i(x-y) \frac{e^{2i\xi (x-y)}}{\xi} A(y,\xi) w (y)\, dy \\
& \qquad  + 2 \int_{x<y<z} 2i(x-y) \frac{e^{2 i \xi(x-y)}}{\xi} e^{(z-y)[i\xi - \langle\xi \rangle]} B(y,z,\xi) w(z) \,dz \,dy \\
& \qquad + 2 \int_{x<y<z} D_\xi(y-x) (z-y)\left(i - \frac{\xi}{\langle \xi \rangle} \right) e^{(z-y)[i\xi - \langle\xi \rangle]} B(y,z,\xi) w(z) \,dz \,dy .
\end{split}
\end{equation}
Here, $K$ is the kernel defined in~\eqref{scarabee} and the "source" term $W$ is defined by
\begin{align*}
W(x)  & = \int_x^\infty D_\xi(y-x) \partial_\xi A(y,\xi) w (y)\, dy + 2  \int_{x<y<z} D_\xi(y-x) e^{(z-y)[i\xi - \langle\xi \rangle]}  \partial_\xi B(y,z,\xi) w(z) \,dz \,dy \\
& \quad - \int_x^\infty \frac{1}{\xi} D_\xi(y-x) A(y,\xi) w (y)\, dy - 2  \int_{x<y<z} \frac{1}{\xi} D_\xi(y-x) e^{(z-y)[i\xi - \langle\xi \rangle]} B(y,z,\xi) w(z) \,dz \,dy.
\end{align*}
so that, by \eqref{bernache}, it can be bounded as follows:
$$
\mbox{if $x>-1$,} \qquad |W(x)| \lesssim \frac{e^{-\beta x}}{|\xi|^2}.
$$
There remains to show that the third, fourth, and fifth terms on the right-hand side of \eqref{bernache-1} enjoy the same estimate.  We will focus on the third term, since the fourth and the fifth can be bounded through the same manipulations, but involve lengthier expressions. Integrating by parts in $y$, the third term can be written as
\begin{align*}
& \int_x^\infty 2i(x-y) \frac{e^{2i\xi (x-y)}}{\xi} A(y,\xi) w (y)\, dy = \int_x^\infty  \frac{y-x}{\xi^2} \partial_y e^{2i\xi (x-y)} A(y,\xi) w (y)\, dy \\
& \qquad = \frac{1}{\xi^2} \int_x^\infty e^{2i\xi (x-y)} \left[- A(y,\xi) w (y) - (y-x)\partial_y(A(y,\xi) w (y)) \right] \,dy.
\end{align*}
It is now clear that the above is $O(\frac{e^{-\frac \beta 2 x}}{\xi^2})$ for $x>-1$. Proceeding similarly for the fourth and fifth terms on the right-hand side of \eqref{bernache-1} leads to
$$
\partial_\xi w(x) = \int_x^\infty K(x,y,\xi) \partial_\xi w (y)\, dy + \widetilde{W}(x),
$$
where
$$
\mbox{if $x>-1$,} \qquad |\widetilde{W}(x)| \lesssim \frac{e^{-\frac \beta 2 x}}{|\xi|^2}.
$$
Thus, Banach's fixed point theorem gives a solution of \eqref{bernache-1} such that $|\partial_\xi w| \lesssim \frac{1}{\xi^2}$. One can then go back to the above expression to obtain
$$
| \partial_\xi w(x,\xi) | \lesssim \frac{e^{-\frac \beta 2 x}}{\xi^2}.
$$
Higher order derivatives in $\xi$ can be handled by the same token.

\bigskip

\noindent \underline{Estimates on $u'$.} By~\eqref{sittelle}, $u$ can be written
$$
u'(x,\xi) =  - e^{x \langle \xi \rangle + i x \xi} \underbrace{\mathfrak{f}_3^{(2)} (x,\xi)^{-2} \int_x^\infty e^{(x-y) \langle \xi \rangle + i(y-x) \xi} \mathfrak{f}_3^{(2)} (y,\xi) V_2(y) w(y,\xi) \,dy}_{\displaystyle \mathfrak{u}(x,\xi)}
$$
We claim that 
\begin{equation}\label{balbuzard}
| \partial_\xi^k \partial_x^\ell \mathfrak{u}(x,\xi) | \lesssim e^{-\beta x} \langle \xi \rangle^{-\ell-1} \quad \mbox{if $x>-1$}.
\end{equation}
For $k=\ell=0$, this simply follows from the inequality $\int_x^\infty e^{(x-y) \langle \xi \rangle} e^{-\beta y} \,dy \lesssim e^{-\beta x} \langle \xi \rangle^{-1}$.

When applying derivatives in $x$, we use the identity, valid for any smooth function $f$,
\begin{equation}
\label{tetras}
\partial_x \int_x^\infty e^{(x-y) \langle \xi \rangle + i(y-x) \xi} f(y)\,dy =  \int_x^\infty e^{(x-y) \langle \xi \rangle + i(y-x) \xi} f'(y)\,dy.
\end{equation}
 
When applying derivatives in $\xi$, it suffices to note that
\begin{equation}
\label{lagopede}
\mbox{if $y>x$,} \qquad | \partial_\xi^k e^{(x-y) \langle \xi \rangle + i(y-x) \xi}| \lesssim_k \langle \xi \rangle^{-k} e^{(x-y) [\langle \xi \rangle - 1]}
\end{equation}
to obtain the desired estimate.

\bigskip

\noindent \underline{Estimates on $uf_3$, the case $x>2$.} In this case, we write
\begin{align*}
& u(x,\xi) f_3(x,\xi) \\
& =  e^{i x \xi} \mathfrak{f}_3(x,\xi) \underbrace{\int_0^x (1-\chi(y)) e^{(y-x) \langle \xi \rangle + i(y-x) \xi} \mathfrak{u}(y,\xi) \,dy}_{\displaystyle \widetilde{\mathfrak{u}}_1(x,\xi)}+e^{-x\langle \xi \rangle}\mathfrak{f}_3(x,\xi) \underbrace{\int_0^1 \chi(y) e^{y\langle \xi \rangle + iy\xi} \mathfrak{u}(y,\xi) \,dy}_{\displaystyle \widetilde{\mathfrak{u}}_2(\xi)},
\end{align*}
where $\chi$ is a smooth function with compact support in $[-1,1]$, equal to $1$ on $[-\frac{1}{2},\frac{1}{2}]$. 

It follows from~\eqref{balbuzard}, from an analogue of the identity~\eqref{tetras}, and from~\eqref{lagopede} that
$$
|\partial_x^k \partial_\xi^\ell \widetilde{\mathfrak{u}}_1(x,\xi) | \lesssim \langle \xi \rangle^{-\ell-1} e^{-\beta x}.
$$
Turning to $\widetilde{\mathfrak{u}}_2(x,\xi)$, using the identity
\begin{equation*}
\partial_\xi e^{y \langle \xi \rangle + iy \xi} = y \left(\frac{\xi}{\langle \xi \rangle} + i \right) \frac{1}{\langle \xi \rangle + i \xi} \partial_y e^{y \langle \xi \rangle + iy\xi},
\end{equation*}
each derivative in $\xi$ gives a gain of $\sim \langle \xi \rangle^{-1}$ after integrating by parts in $y$ (notice that the factor $y$ and the support of $\chi$ ensures the absence of boundary terms!). This leads to the estimate
$$
| \partial_x^k \partial_\xi^\ell \widetilde{\mathfrak{u}}_2(\xi)| \lesssim \langle \xi \rangle^{-\ell-1} e^{\langle \xi \rangle}
$$

Therefore, setting $\mathfrak u_{1,>}(x,\xi)=\mathfrak f_3(x,\xi)\widetilde{\mathfrak u}_1(\xi)$ and $\mathfrak u_{2,>}(x,\xi)=\mathfrak f_3(x,\xi)\widetilde{\mathfrak u}_2(x,\xi)$ we have obtained using \eqref{papillon} that for $|x|>2$:
$$
u(x,\xi) f_3(x,\xi)=e^{ix\xi}\mathfrak u_{1,>}(x,\xi)+e^{-x\langle \xi \rangle}\mathfrak u_{2,>}(x,\xi),
$$
with
\begin{equation} \label{ecureuil-gris}
| \partial_x^k \partial_\xi^\ell \mathfrak u_{1,>}(x,\xi)| \lesssim_{k,\ell}   \la \xi \ra^{-1-\ell} e^{-\beta x }, \qquad  | \partial_x^k \partial_\xi^\ell \mathfrak u_{2,>}(x,\xi)| \lesssim_{k,\ell}   \la \xi \ra^{-1-\ell} e^{\langle \xi \rangle}
\end{equation}

\bigskip

\noindent \underline{Estimates on $uf_3$, the case $-3< x < 3$.} Expanding $\mathfrak{u}(x,\xi)$ in Taylor series at the origin gives, using \eqref{balbuzard},
$$
\mathfrak{u}(x,\xi) = \sum_{k=0}^M \alpha_k (\xi) x^k + R_M(x,\xi).
$$
Here, the Taylor coefficients and the remainder satisfy 
$$
|\partial_\xi^\ell \alpha_k (\xi)| \lesssim \langle \xi \rangle^{-\ell-1},  \qquad |\partial_x^k \partial_\xi^\ell R_M(x,\xi) | \lesssim \langle \xi \rangle^{-\ell-1}.
$$
Furthermore, the remainder vanishes to order $M+1$ at the origin $| R_M(y,\xi) | \lesssim |y|^{M+1}$. The polynomial terms in $x$ in the above expansion contribute to $u(x,\xi) f_3(x,\xi)$ terms of the type
$$
e^{-x\langle \xi \rangle}\mathfrak{f}_3(x,\xi) \alpha_k(\xi) \int_0^x y^k e^{y \langle \xi \rangle + iy\xi} \,dy,
$$
which, after repeated integrations by parts using the identity $\frac{1}{\langle \xi \rangle + i\xi} \partial_y e^{y \langle \xi \rangle + iy\xi} = e^{y \langle \xi \rangle + iy\xi}$, reduces to a linear combination of terms of the type
$$
\frac{1}{(\langle \xi \rangle + i \xi)^{k'}}\mathfrak{f}_3(x,\xi) \alpha_k(\xi) x^{k-k'+1} e^{ix\xi}, \quad 1\leq k' \leq k+1, \quad \mbox{and} \quad \frac{1}{(\langle \xi \rangle + i \xi)^{k+1}}\mathfrak{f}_3(x,\xi) \alpha_k(\xi) e^{-x\langle \xi \rangle}.
$$
The remainder term contributes
$$
e^{-x\langle \xi \rangle}\mathfrak{f}_3(x,\xi) \int_0^x R_M(y,\xi) e^{y \langle \xi \rangle + iy\xi} \,dy,
$$
which, after repeated integrations by parts, reduces to a linear combination of terms of the type
$$
\frac{1}{(\langle \xi \rangle + i \xi)^{k'}} \mathfrak{f}_3(x,\xi) R_M^{(k'-1)}(x,\xi)e^{ix\xi}, \quad 0\leq k' \leq k, \quad \mbox{and} \quad \frac{1}{(\langle \xi \rangle + i \xi)^{M+1}}\mathfrak{f}_3(x,\xi) R_M^{(M+1)}(0,\xi) e^{-x\langle \xi \rangle},
$$
along with the integrated term
$$
\frac{1}{(\langle \xi \rangle + i \xi)^{M+1}} e^{-x\langle \xi \rangle}\mathfrak{f}_3(x,\xi) \int_0^x R_M^{(M+2)}(y,\xi) e^{y \langle \xi \rangle + iy\xi} \,dy.
$$

Choosing $M$ sufficiently large, we see that $uf_3$ can be written for $-3<x<3$ under the form
$$
u(x,\xi) f_3(x,\xi) = e^{i x \xi} \mathfrak{u}_{2,<}(x,\xi) +e^{-x\langle \xi \rangle} \mathfrak{u}_{1,<}(x,\xi) 
$$
with
\begin{equation} \label{ecureuil-roux}
|\partial_x^k \partial_\xi^\ell \mathfrak{u}_{1,<}(x,\xi) | + |\partial_x^k \partial_\xi^\ell \mathfrak{u}_{2,<}(x,\xi) | \lesssim \langle \xi \rangle^{-\ell-1}.
\end{equation}

\noindent \underline{End of the proof}. We take $\bar \chi$ a smooth cut-off function with $\bar \chi(x)=1$ for $|x|\leq 2$ and $\bar \chi(x)=0$ for $|x|>3$. We define $\mathfrak{u}_{j}=\bar \chi \mathfrak{u}_{j,<}+(1-\bar \chi)\mathfrak{u}_{j,>}$ for $j=1,2$. Using \eqref{ecureuil-gris} and \eqref{ecureuil-roux} we have $uf_3=  e^{i x \xi} \mathfrak{u}_{2}+e^{-x\langle \xi \rangle} \mathfrak{u}_{1}$ with $|\partial_x^k \partial_\xi^\ell \mathfrak{u}_{1} | \lesssim\langle \xi \rangle^{-\ell-1}e^{-\beta x}$ and $|\partial_x^k \partial_\xi^\ell \mathfrak{u}_{2} | \lesssim\langle  \la \xi \ra^{-1-\ell} \left[ e^{\langle \xi \rangle} \mathbf{1}_{x>2} + 1\right]$. Combining with \eqref{rouge-gorge} this ends the proof of the lemma.

\end{proof}

\begin{lemma}   \label{S-v-lst-lm5.5-KS}
For every $\xi \in \mathbb{R}$, there exists a real-valued solution $f_4^\dagger (x, \xi)$ of the equation
$$ \mathcal{H} f_4^\dagger (\cdot , \xi) = (\xi^2 +1 ) f_4^\dagger (\cdot , \xi ) $$
with the property that $f_4(x,\xi) \sim e^{x \langle \xi \rangle} e_2$ as $x \to \infty$. More precisely, it can be written
$$
f_4^\dagger(x,\xi) = e^{x\langle \xi \rangle} \mathfrak{f}_4^\dagger(x,\xi)
$$
with
\begin{equation}    
| \partial_x^k \partial_\xi^\ell [\mathfrak{f}_4^\dagger(x,\xi) - e_2 ] | \lesssim \langle \xi \rangle^{-\ell - 1} e^{-\beta x}.
\end{equation}
for all $x>-1$, and $k,\ell \geq 0$.
\end{lemma}

	\begin{proof}
		Let $\mathfrak{f}_4^{\dagger}(x, \xi) =  e^{- x \la \xi \ra} f_4^\dagger (x, \xi)$. From the proof of Lemma 5.5 in \cite{KS}, we know that $\mathfrak{f}_4^{\dagger} (x, \xi)$ satisfies
		\begin{equation}   \label{integraleq-1}
		\begin{split}
		\mathfrak{f}_4^{\dagger} (x, \xi)  
		& =  e_2
		+
		\int^{+\infty}_x
		\begin{pmatrix}
		0 &  0 \\
		0  &  -\frac{1}{2 \langle \xi \rangle}
		\end{pmatrix}
		V(y) \mathfrak{f}_4^{\dagger} (y, \xi) dy  
		\\
		& \qquad 
		+ 
		\int^{x}_{x_1}
		\begin{pmatrix}
		\frac{\sin(\xi(x-y))}{\xi} e^{- \langle \xi \rangle(x-y)}  &  0 \\
		0  &  -\frac{e^{-2 \langle \xi \rangle(x-y)}}{2\langle \xi \rangle}
		\end{pmatrix}
		V(y) \mathfrak{f}_4^{\dagger} (y, \xi) dy  \\
		& = e_2 +  \int^{+\infty}_x A(\xi) V(y) \mathfrak{f}_4^{\dagger} (y, \xi) dy  + \int^{x}_{x_1} B(x, y, \xi) V(y) \mathfrak{f}_4^{\dagger} (y, \xi) dy  ,  \\
\end{split}
		\end{equation}
where $x_1\in [-\infty,\infty)$ remains to be determined depending on $\xi$.

\medskip

\noindent \textbf{Step 1}. \emph{The case $|\xi|\lesssim 1$}. We claim that there exists a large enough $x_1>0$ such that, given any $R>0$, if $|\xi|<R$ then
\begin{equation}    
| \partial_x^k \partial_\xi^\ell [\mathfrak{f}_4^\dagger(x,\xi) - e_2 ] | \lesssim_{k,\ell,R} e^{-\beta x} \qquad \forall x>-1.
\end{equation}

\noindent \underline{Estimate without derivatives.} From the kernel bounds
\begin{equation} \label{rat musque}
|B(x,y,\xi)| \lesssim 
\frac{1}{\langle \xi \rangle} e^{(\frac{1}{10}-\langle \xi \rangle)(x-y)} \qquad \mbox{if $x_1<y<x$}
\end{equation}
and
\begin{equation} \label{geai}
|A(\xi)V(y)|\lesssim \la \xi \ra^{-1}e^{-\beta y},
\end{equation}
it is easy to see that the right-hand side of~\eqref{integraleq-1} is a contraction in $L^\infty(x_1,\infty)$, if $x_1$ is chosen large enough (only depending on $V$). Banach's fixed point theorem gives a solution $\mathfrak{f}_4^{\dagger} $ of size $O(1)$ in $L^\infty(x_1,\infty)$. 

Since $|\xi|<R$, we have $|B(x,y,\xi)| \lesssim 1$ for $-1<x<y<x_1$, which allows to extend the solution $\mathfrak{f}_4^{\dagger} $ all the way to $-1$, with the same bound $O(1)$, by applying Gronwall's lemma to
$$
\mbox{if $x<x_1$}, \qquad |\mathfrak{f}_4^{\dagger} (x,\xi)| \lesssim 1+ \int_x^{x_1} | \mathfrak{f}_4^{\dagger} (y,\xi)| \,dy.
$$
Finally, from the bounds on $B(x,y,\xi)$ and $\| \mathfrak f_4^\dagger\|_{L^\infty([-1,\infty)}\lesssim 1$:
\begin{align*}
|\mathfrak{f}_4^{\dagger} (x,\xi) - e_2| & \lesssim \int_x^\infty |A(y,\xi)| e^{-\beta y} \,dy + \int_{x_1}^x |B(x,y,\xi)| e^{-\beta y} \,dy\\
& \lesssim e^{-\beta x} .
\end{align*}

\medskip

\noindent \underline{Taking derivatives in $x$.} Applying the identity 
$$
\partial_x \int_{x_1}^x K(x-y) f(y)\,dy = \int_{x_1}^x K(x-y) f'(y)\,dy + K(x-x_1) f(x_1),
$$ 
we have
\begin{align} \label{raclette}
\partial_x \mathfrak{f}_4^{\dagger}  (x,\xi)  &= - A(\xi) V(x) \mathfrak{f}_4^{\dagger} (x, \xi)  + \int^{x}_{x_1} B(x, y, \xi) \partial_y V(y) \mathfrak{f}_4^{\dagger} (y, \xi) dy \\
\nonumber & \qquad+ \int^{x}_{x_1} B(x, y, \xi) V(y) \partial_y \mathfrak{f}_4^{\dagger} (y, \xi) dy+ B(x,x_1,\xi) V(x_1)\mathfrak{f}_4^{\dagger} (x_1 , \xi) .
\end{align}
Note that for $x> -1$ and $|\xi|<R$ there holds $|\partial_x^k B(x,x_1,\xi)|\lesssim 1$. Solving again for $\partial_x \mathfrak{f}_4^{\dagger} $ by Banach's fixed point theorem, we obtain the bound
$$
\left| \partial_x \mathfrak{f}_4^\dagger(x,\xi) \right|\lesssim e^{-\beta x} \qquad \forall x>-1.
$$
One can then prove inductively that $| \partial_x^{k} \mathfrak{f}_4^{\dagger} (x,\xi)| \lesssim_{k} e^{-\beta x}$.

\medskip
		
\noindent \underline{Taking derivatives in $\xi$.} 
It suffices to observe that for $|\xi|<R$,
$$
|\partial_\xi^\ell A(\xi)| \lesssim 1, \qquad \sup_{y \geq x} |\partial_\xi^\ell B(x,y,\xi)| \lesssim e^{(\frac{1}{10}-\langle \xi \rangle)(x-y)} \qquad \mbox{if $x_1<y<x$,}
$$
and to use similar estimates to the above, bounding successively $\partial_\xi^\ell \partial_x^k \mathfrak{f}_4^{\dagger} (x,\xi)$, for $\ell=1$, $k=0\dots \infty$, then $\ell=2$, $k=0\dots \infty$, etc.. For instance, in the case $k=0, \ell=1$,
\begin{equation} 
\begin{split}
&  \partial_\xi \mathfrak{f}_4^{\dagger}  (x,\xi)  =   \int^{+\infty}_x \partial_\xi A(\xi) V(y) \mathfrak{f}_4^{\dagger} (y, \xi) dy +  \int^{+\infty}_x A(\xi) V(y)  \partial_\xi  \mathfrak{f}_4^{\dagger} (y, \xi) dy  \\
&\quad  + \int^{x}_{x_1} \partial_\xi B(x, y, \xi) V(y) \mathfrak{f}_4^{\dagger} (y, \xi) dy + \int^{x}_{x_1} B(x, y, \xi) V(y) \partial_\xi  \mathfrak{f}_4^{\dagger} (y, \xi) dy .  \\
\end{split}
\end{equation}
The first and third terms on the right-hand side are bounded pointwise by $1$ for $x \geq -1$, while the second and last terms are contractions on $L^\infty(x_0,\infty)$; the previous arguments apply then.

\medskip

\noindent \textbf{Step 2}. \emph{The case $|\xi|\gg 1$}. We claim that for $|\xi|$ large enough then choosing $x_1=-\infty$ in~\eqref{integraleq-1} one has:
\begin{equation}    
| \partial_x^k \partial_\xi^\ell [\mathfrak{f}_4^\dagger(x,\xi) - e_2 ] | \lesssim_{k,\ell}  \langle \xi \rangle^{-1-\ell}e^{-\beta x} \qquad \forall x>-1.
\end{equation}

\noindent \underline{Estimate without derivatives.} Recall the kernel bounds \eqref{rat musque} and \eqref{geai}. Then, for $x_1=-\infty$ if $|\xi|$ large enough the right-hand side of~\eqref{integraleq-1} is a contraction in $L^\infty(\mathbb R)$. This gives a solution $\mathfrak f_4^\dagger$ that is $O(1)$ in $L^\infty$. Injecting back this bound in \eqref{integraleq-1} shows:
\begin{align*}
|\mathfrak{f}_4^{\dagger} (x,\xi) - e_2| & \lesssim \int_x^\infty |A(y,\xi)| e^{-\beta y} \,dy + \int_{-\infty}^x |B(x,y,\xi)| e^{-\beta y} \,dy\\
& \lesssim \langle \xi \rangle^{-1} e^{-\beta x} 
\end{align*}
for all $x>-1$ by the kernel bound \eqref{rat musque}.

\medskip

\noindent \underline{Taking derivatives in $x$.} Since $x_1=-\infty$ the identity \eqref{raclette} becomes
\begin{align} \label{fondue}
\partial_x \mathfrak{f}_4^{\dagger}  (x,\xi)  &= - A(\xi) V(x) \mathfrak{f}_4^{\dagger} (x, \xi)  + \int^{x}_{-\infty} B(x, y, \xi) \partial_y V(y) \mathfrak{f}_4^{\dagger} (y, \xi) dy \\
\nonumber & \qquad+ \int^{x}_{-\infty} B(x, y, \xi) V(y) \partial_y \mathfrak{f}_4^{\dagger} (y, \xi) dy.
\end{align}
Solving once more by Banach's fixed point theorem using \eqref{rat musque}, we obtain the bound
$$
\left| \partial_x \mathfrak{f}_4^\dagger(x,\xi) \right|\lesssim \frac{e^{-\beta x}}{\langle \xi \rangle},
$$
and can then prove inductively that $| \partial_x^{k} \mathfrak{f}_4^{\dagger} (x,\xi)| \lesssim \frac{e^{-\beta x}}{\langle \xi \rangle}$.

\medskip
		
\noindent \underline{Taking derivatives in $\xi$.} Since for all $\ell \geq0$,
$$
|\partial_\xi^\ell A(\xi)| \lesssim \frac{1}{\langle \xi \rangle^{1+ \ell}}, \quad \sup_{y \geq x} |\partial_\xi^\ell B(x,y,\xi)| \lesssim \frac{1}{\langle \xi \rangle^{1+ \ell}} e^{(\frac{1}{10}-\langle \xi \rangle)(x-y)} \qquad \mbox{if $y<x$,}
$$
one can bound successively $\partial_\xi^\ell \partial_x^k \mathfrak{f}_4^{\dagger} (x,\xi)$, for $k=0$, $\ell=0\dots \infty$, then $k=1$, $\ell=0\dots \infty$, etc... 

\medskip

\noindent \textbf{Step 3}. \emph{Connecting the cases $|\xi|<R$ and $|\xi|$ large.} We saw that $x_1$ can be chosen to be some positive constant, for the case $|\xi|\lesssim 1$, or $-\infty$, for the case $|\xi|$ large. Denoting these two solutions by $f_4^{\dagger,1}$ and $f_4^{\dagger,2}$, respectively, we set
$$
f_4^{\dagger}(x,\xi) = \chi(\xi)f_4^{\dagger,2}(x,\xi)  + [1-\chi(\xi)] f_4^{\dagger,1}(x,\xi),
$$
for an appropriately chosen cutoff function $\chi$. It satisfies the desired estimates.

\end{proof}

The solution $f_4^\dagger$ which was constructed in the previous lemma needs to be slightly modified in order to achieve independence, in the sense of the Wronskian, from $f_1$ and $f_2$. In the present context, the Wronskian is complex valued, and defined by
$$
W[f,g](x) = f'(x) \cdot g(x) - f(x) \cdot g'(x)
$$
(recall that $\cdot$ denotes the real scalar product); it is independent of $x$ if $\mathcal{H} f - \lambda f = \mathcal{H} g - \lambda g = 0$. We learn from the asymptotic behavior of $f_1$, $f_2$ and $f_3$ that
$$
W(f_1,f_2) = 2i\xi, \quad \mbox{and} \quad W(f_1,f_3) = W(f_2,f_3) = 0.
$$

\begin{lemma}
There exists a unique pair of complex-valued functions $c_1 (\xi) $, $c_2 (\xi)$, such that
\begin{equation}
f_4  (x , \xi) := f_4^\dagger (x , \xi) - c_1 (\xi) f_1  (x , \xi) - c_2 (\xi) f_2  (x , \xi)
\end{equation}
satisfies the Wronskian identities
\begin{equation}
W [f_1 (\cdot, \xi), f_4 (\cdot, \xi) ] = W [f_2 (\cdot, \xi), f_4 (\cdot, \xi) ] =0.
\end{equation}
For $|\xi| \geq 1$, the coefficients $c_1$ and $c_2$ are such that
$$
| \partial_\xi^k c_1(\xi)| + | \partial_\xi^k c_2(\xi)| \lesssim \langle \xi \rangle^{-k-1}.
$$
For $|\xi| \leq 1$, there holds
$$
| \partial_x^k \partial_\xi^\ell  f_4 (x, \xi) | \lesssim 1.
$$
\end{lemma}

\begin{proof} The requirement that $W [f_1 (\cdot, \xi), f_4 (\cdot, \xi) ] = W [f_2 (\cdot, \xi), f_4 (\cdot, \xi) ] =0$ gives the formulas
$$
c_1(\xi) = - \frac{W(f_2,f_4^\dagger)}{2i\xi}, \qquad c_2(\xi) = \frac{W(f_1,f_4^\dagger)}{2i\xi}.
$$ 
The estimates for $|\xi| \geq 1$ are then consequences of the estimates on $f_1$, $f_2$ and $f_4$. Finally, setting
$$
F(x,\xi) = c_1 (\xi) f_1  (\cdot , \xi) + c_2 (\xi) f_2  (\cdot , \xi),
$$
it is such that for $|\xi|\leq 1$
$$
 | \partial_x^k \partial_\xi^\ell F (x, \xi) | \lesssim 1,
$$
from which the desired estimates for $|\xi| \leq 1$ follow.
\end{proof}

\subsection{Scattering theory for $\mathcal{H}$}
	
	Setting
	\begin{equation}
	g_j (x, \xi) := f_j (-x, \xi) , \quad j = 1, 2, 3, 4,
	\end{equation}
	we get, since $V$ is even,
	\begin{equation}
	\mathcal{H} g_j (\cdot, \xi) = (1+ \xi^2) g_j (\cdot, \xi) , \quad j = 1, 2, 3, 4,
	\end{equation}
	and as $x \rightarrow \pm \infty$, $g_j$ has the same asymptotic behavior  as $f_j$ when $x \rightarrow \mp \infty$.
	
We then let, for each $\xi \in \mathbb{R}$,
\begin{equation}     \label{S-v-lst-def-eq1}
\begin{split}
& F_1 (\cdot, \xi) = (f_1 (\cdot, \xi), f_3 (\cdot, \xi)) , \ F_2 (\cdot, \xi) := (f_2 (\cdot, \xi), f_4 (\cdot, \xi)) , \\ 
& G_1 (\cdot, \xi) = (g_2 (\cdot, \xi), g_4 (\cdot, \xi)) , \  G_2 (\cdot, \xi) := (g_1 (\cdot, \xi), g_3 (\cdot, \xi)).
\end{split}
\end{equation}

Finally, given $2\times 2$ matrices $F(x)$ and $G(x)$, the matrix Wronskian $\mathcal{W}(F,G)$ is given by
$$
\mathcal{W}(F,G)(x) = (F'(x))^\top G(x) - (F(x))^\top G'(x) .
$$
If $F$ and $G$ are solutions of $\mathcal{H} F = \lambda F$ and $\mathcal{H} G = \lambda G$, for $\lambda \in \mathbb{C}$, the value of $\mathcal{W}(F,G)$ is independent of $x$.

\begin{lemma} \label{lemmaAB}
(i) For all $\xi \in \mathbb{R}$, 
		\begin{equation}
		\begin{split}
		& G_1 (x, \xi) = F_2 (-x, \xi) , \  G_2 (x, \xi) = F_1 (-x, \xi) , \\
		& \overline{F_1 (x, \xi)} = F_1 (x, - \xi) , \  \overline{F_2 (x, \xi)} = F_2 (x, - \xi) , \\
		& \overline{G_1 (x, \xi)} = G_1 (x, - \xi) , \  \overline{G_2 (x, \xi)} = G_2 (x, - \xi) . \\
		\end{split}
		\end{equation}
		(ii) For every $\xi \in \mathbb{R} \setminus \{ 0 \}$, there exist unique constant $2 \times 2$ matrices $A = A(\xi)$, $B = B(\xi)$ with complex entries so that
		\begin{equation} \label{pinson}
		F_1 (x, \xi) = G_1 (x, \xi) A(\xi) +  G_2 (x, \xi) B(\xi) .   
		\end{equation}  
		Then $A(-\xi) = \overline{A(\xi)}$, $B(-\xi) = \overline{B(\xi)}$, and
		\begin{equation} \label{scarabeerhino}
		\begin{split}
		& G_2 (x, \xi) =  F_2 (x, \xi) A(\xi) +  F_1 (x, \xi) B(\xi) , \\
		&  \mathcal{W} [ F_1 (\cdot , \xi), G_2 (\cdot , \xi) ] = A(\xi)^\top (2i \xi p - 2 \la \xi \ra q) , \\
		&  \mathcal{W} [ F_1 (\cdot , \xi), G_1 (\cdot , \xi) ] = - B(\xi)^\top (2i \xi p - 2 \la \xi \ra q) . \\
		\end{split}
		\end{equation}
		Moreover, $A(\xi)$ and $B(\xi)$ are smooth for $\xi \neq 0$. Furthermore, $\xi p A(\xi)$, $q A(\xi)$, $\xi p B(\xi)$, $qB(\xi)$ are smooth functions of $\xi \in \mathbb{R}$. Finally, if $|\xi| \geq 1$,
		\begin{equation} \label{merle}
		\partial_\xi^\ell ( A (\xi) - I ) = O( \la \xi \ra^{-1-\ell} ) , \  \partial_\xi^\ell B (\xi)   = O( \la \xi \ra^{-1-\ell} ) . 
		\end{equation}
	\end{lemma}

	\begin{proof}
Most of the above statement appears in Lemma 5.13, Lemma 5.14 and Corollary 5.15 of \cite{KS}. 
The only statement which requires a proof is~\eqref{merle}. By~\eqref{scarabeerhino}, $A(\xi)$ is given by the formula
$$
A(\xi)^\top = \mathcal{W} [ F_1 (\cdot , \xi), G_2 (\cdot , \xi) ] (2i \xi p - 2 \la \xi \ra q)^{-1}.
$$
In order to evaluate $\mathcal{W} [ F_1 (\cdot , \xi), G_2 (\cdot , \xi) ]$, we note that, as a consequence of lemmas~\ref{lemma3} and~\ref{lemma12}, for any $\xi$, and for any $\ell \geq 0$,
\begin{align*}
& |\partial_\xi^\ell(f_1(0,\xi) - e_1)| \lesssim \langle \xi \rangle^{-1-\ell}, \qquad
|\partial_\xi^\ell(f_1'(0,\xi) - i\xi e_1)| \lesssim \langle \xi \rangle^{-\ell}, \\
& |\partial_\xi^\ell(f_3(0,\xi) - e_2)| \lesssim \langle \xi \rangle^{-1-\ell}, \qquad
|\partial_\xi^\ell(f_3'(0,\xi) + \langle \xi \rangle e_2)| \lesssim \langle \xi \rangle^{-\ell}, \\
& |\partial_\xi^\ell(g_1(0,\xi) - e_1)| \lesssim \langle \xi \rangle^{-1-\ell}, \qquad
|\partial_\xi^\ell(g_1'(0,\xi) + i\xi e_1)| \lesssim \langle \xi \rangle^{-\ell}, \\
& |\partial_\xi^\ell(g_3(0,\xi) - e_2)| \lesssim \langle \xi \rangle^{-1-\ell}, \qquad
|\partial_\xi^\ell(g_3'(0,\xi) - \langle \xi \rangle e_2)| \lesssim \langle \xi \rangle^{-\ell}.
\end{align*}
As a consequence, 
$$
\mathcal{W} [ F_1 (\cdot , \xi), G_2 (\cdot , \xi) ] = (2i \xi p - 2 \la \xi \ra q) + R(\xi), \quad \mbox{where $|\partial_\xi^\ell R(\xi)| \lesssim \langle \xi \rangle^{-\ell}$},
$$
which leads to the desired statement. The proof of the estimate for $B(\xi)$ is similar.
\end{proof}

The link between the scattering theory of $\mathcal{H}$ and its spectral properties will be provided by the matrix Wronskian
\begin{equation}     \label{S-v-lst-def-eq2}
D(\xi) = \mathcal{W} [ F_1 (\cdot , \xi), G_2 (\cdot , \xi) ].
\end{equation}

\begin{lemma} \label{lemmaD}
The following properties are equivalent: 
\begin{itemize}
\item[(a)] There are no embedded eigenvalues and $E = 1$ is not a resonance of $\mathcal{H}$. 
\item[(b)] $D(\xi)$ is invertible for all $\xi \in \mathbb{R}$. 
In this case,
\begin{equation} \label{loriot}
A(\xi)^{-1} = D(\xi)^{-1} (2i \xi p - 2 \la \xi \ra q)
\end{equation}
for all $\xi \in \mathbb{R}$. 
\end{itemize}
\end{lemma}
	
\begin{proof}
This is Corollary 5.21 of~\cite{KS}.
\end{proof}

The generalized eigenfunctions $(f_j)_{j=1,2,3,4}$ coincide with that of the article \cite{KS}, as they are solutions of the same fixed point scheme. This allows us to directly apply certain results of \cite{KS} in this section and the next one. We shall however use slightly different eigenfunctions at some point in the next section.

\begin{lemma} \label{lem:improved-eigenfunctions}

For every $\xi \in \mathbb{R}$, there exist three solutions $f_j^*$ for $j=1,2,3$ of the equation $\mathcal{H} f_j^*= (\xi^2 +1 ) f_j^* $ that are of the following form
\begin{equation} \label{id:decomposition-f*}
\begin{array}{l l}
& f_1^* (x , \xi) = \mathfrak{f}_{1,1}^*(x,\xi) e^{ix\xi} + \mathfrak{f}_{1,2}^*(x,\xi) e^{-x\langle \xi \rangle}, \\
 & f_2^*=\overline{f_1^*},\\
 & f_3^*(x,\xi) = \mathfrak{f}_3^*(x,\xi) e^{-x \langle \xi \rangle},
\end{array}
\end{equation}
with, for all $x \geq -10$, $k$, $\ell \geq 0$
\begin{equation} \label{bd:f*}
\begin{array}{l l}
 & | \partial_x^k \partial_\xi^\ell (\mathfrak{f}_{1,1}^* (x, \xi) - e_1) | \lesssim_{k,\ell}   \la \xi \ra^{-1-\ell} e^{-\beta x }, \\
& | \partial_x^k \partial_\xi^\ell \mathfrak{f}_{1,2}^* (x, \xi)  | \lesssim_{k,\ell} e^{-50 \langle \xi \rangle},\\
& | \partial_x^k \partial_\xi^\ell  [ \mathfrak{f}_3^*(x,\xi)  - e_2] | \lesssim_{k,\ell}   \la \xi \ra^{-1-\ell} e^{-\beta x}.
\end{array}
\end{equation}

\end{lemma}

\begin{proof}

Consider the operator with translated potential
$$
\check{\mathcal{H}} = \mathcal{H}_0 + \check{V}, \qquad \check V(x)=V(x-100).
$$
We still have $|\partial_x^k \check V|\lesssim_k e^{-\beta |x|}$. We can thus apply Lemmas \ref{lemma3} and \ref{lemma12} to $\check{\mathcal{H}}$ (indeed, only Hypothesis \eqref{H2} is used in their proof, and not Hypothesis \eqref{H1}). We obtain that there exist two solutions $\check f_j$ for $j=1,3$ of $\check{\mathcal{H}} \check{f_j}= (\xi^2 +1 ) \check{f_j}$ of the form
\begin{align*}
& \check{f_1} (x , \xi) = \check{\mathfrak{f}}_{1,1}(x,\xi) e^{ix\xi} + \check{\mathfrak{f}}_{1,2}(x,\xi) e^{-x\langle \xi \rangle},\\
& \check{f_3}(x,\xi) = \check{\mathfrak{f}_3}(x,\xi) e^{-x \langle \xi \rangle} 
\end{align*}
where, for any $x > -1$ and $k$, $\ell \geq 0$,
\begin{align*}   
& | \partial_x^k \partial_\xi^\ell (\check{\mathfrak{f}}_{1,1} (x, \xi) - e_1) | \lesssim_{k,\ell}   \la \xi \ra^{-1-\ell} e^{-\beta x }, \\
& | \partial_x^k \partial_\xi^\ell \check{\mathfrak{f}}_{1,2} (x, \xi)  | \lesssim_{k,\ell}  \la \xi \ra^{-1-\ell} \left[ e^{\langle \xi \rangle} \mathbf{1}_{x>2} + 1\right],\\
& | \partial_\xi^\ell \partial_x^k [ \check{\mathfrak{f}}_3(x,\xi)  - e_2] | \lesssim_{k,\ell}   \la \xi \ra^{-1-\ell} e^{-\beta x}.
\end{align*}
We now define $f_1^*(x)=\check{f_1}(x+100)$, $f_2^*(x)=\overline{f_1^*}$ and $f_3^*(x)=e^{100\langle \xi \rangle}\check{f_3}(x+100)$. These are solutions to $\mathcal{H} f_j^*= (\xi^2 +1 ) f_j^* $, and the desired identities \eqref{id:decomposition-f*}, and inequalities \eqref{bd:f*} for $x>-10$ are direct consequences of the ones above for $x>-1$. As a side note, notice that we have $f_3^*=f_3$ due to the uniqueness of their asymptotic behaviour as $x\to \infty$, but we shall not need this fact.

\end{proof}

\section{The distorted Fourier Transform}

\label{sectiondistorted}
\label{S-v-dft}

In this section we assume in addition to (H1) and (H2) that
\begin{equation} \label{H3}
\tag{H3}	\mbox{The operator $\mathcal{H}$ does not have embedded eigenvalues or resonances at energy 1.}
\end{equation}

\subsection{The basis of generalized eigenfunctions}

By Assumption \eqref{H3} and Lemma \ref{lemmaD}, $D(\xi)$ is invertible for all $\xi\in \mathbb R $. In analogy with the scalar case, we define the transmission and reflection coefficients $s(\xi)$ and $r(\xi)$ by
\begin{equation} \label{S-v-dft-eq7}
s(\xi) e_1 = 2i \xi p D(\xi)^{-1} e_1 \quad \text{ and } \quad r(\xi) e_1 = 2i \xi p B(\xi) D(\xi)^{-1} e_1.
\end{equation}
(the precise analogy is the following: if $V_2 = 0$, then the spectral theory of $\mathcal{H}$ is equivalent to that of a scalar Schr\"odinger operator, and $s(\xi)$ and $r(\xi)$ are given by the classical transmission and reflection operators $T(\xi)$ and $R(\xi)$ respectively).

\begin{lemma} \label{lemmars}
	The matrix  
	\begin{equation*}
	\begin{split}
	S (\xi) = 
	\left( \begin{array}{ccc}
	s(\xi) \quad  r(\xi)
	\\
	r(\xi) \quad s(\xi)
	\end{array} \right) 
	\end{split}
	\end{equation*}
	is unitary. In fact, $S(\xi)^* = S(\xi)^{-1} = S(-\xi)$ for all $\xi \in \mathbb{R}$, which can also be written
	\begin{equation}
	\label{relationsrs}
	\begin{split}
	&|r(\xi)|^2 + |s(\xi)|^2 = 1 \\
	&\overline{r(\xi)} s(\xi) + \overline{s(\xi)} r(\xi) = 0 \\
	&r(-\xi) = \overline{r(\xi)} \\
	&s(-\xi) = \overline{s(\xi)}.
	\end{split}
	\end{equation}
	Finally, the values of $s$ and $r$ at $0$ are given by
\begin{equation} \label{rsat0}
s(0) =0, \qquad r(0) = -1,
\end{equation}
and $s$ and $r$ enjoy the estimates
	\begin{equation} \label{estimatesrs}
	| \partial_\xi^k (s(\xi) - 1) | + | \partial_\xi^k r(\xi) | \lesssim \langle \xi \rangle^{-k-1}.
	\end{equation}
\end{lemma}

\begin{proof} For the proof that $S(\xi)$ is Hermitian,  see the proof of Lemma 6.3 in~\cite{KS}.

Evaluating \eqref{S-v-dft-eq7} at $\xi=0$ shows $s(0)=0$. Multiplying on the right \eqref{pinson} by $2i\xi D(\xi)^{-1}e_1$, using \eqref{loriot} and \eqref{S-v-dft-eq7}, we have
 $$
 2i\xi F_1(x,\xi) D(\xi)^{-1} e_1 =  g_2(x,\xi) +r(\xi) g_1(x,\xi) +2i\xi g_3(x,\xi) e_2^\top q B(\xi) D(\xi)^{-1}e_1.
 $$
Letting $\xi \to 0$ yields $0=g_2(x,0) +r(0) g_1(x,0)$ (we recall that $q B(\xi)$ is continuous at $0$ by Lemma \ref{lemmaAB}). Equivalently, $0=f_2(x,0) +r(0) f_1(x,0)$. By our construction of Lemma \ref{lemma12}, one notices that $f_1(\cdot,\xi)$ is for $\xi=0$ a real valued function. Recalling that $f_2=\overline{f_1}$ we obtain for $\xi=0$ that  Hence $r(0)=-1$.

As for the estimates~\eqref{estimatesrs}, they follow from~\eqref{merle} and~\eqref{loriot}.
\end{proof}

The following lemma collects some properties of the generalized eigenfunctions
\begin{align} 
& \label{S-v-dft-eq1} \mathcal{F}_+ (x, \xi) = 2 i \xi F_1 (x, \xi) D(\xi)^{-1} e_1 \\
& \label{S-v-dft-eq2}
\mathcal{G}_+ (x, \xi) = 2 i \xi G_2 (x, \xi) D(\xi)^{-1} e_1 
\end{align}
which will serve as the basic building blocks of the distorted Fourier transform. Note that 
$$
\mathcal{G}_+ (x, \xi) =\mathcal{F}_+ (-x, \xi)
$$
and that, under the assumption (H3),
$$
\mathcal{F}_+ (x,0) = \mathcal{G}_+ (x,0) = 0. 
$$

Before stating the next lemma, we need a piece of notation: let $\chi_{+}$ be a smooth, non-negative function, equal to $0$ on $[-\infty,-\frac{1}{2}]$, equal to $1$ on $[\frac{1}{2},\infty]$, and such that,
$$
\mbox{setting $\chi_-(x) = \chi_+(-x)$,} \qquad \chi_+(x) + \chi_-(x) = 1 \qquad \mbox{for all $x$}.
$$

\begin{lemma}[Decomposition of generalized eigenfunctions] \label{heroncendre}
For all $\xi \in \mathbb{R}$, $\mathcal{F}_+ (x, \xi)$ and $\mathcal{G}_+ (x, \xi)$ constitute a basis of the space of bounded solutions of $\mathcal{H} f = (1+\xi^2) f$.

Furthermore, they can be decomposed as follows
\begin{align} 
\label{meringue} & \mathcal{F}_+ (x, \xi) = \mathcal{F}_+^S(x,\xi) + \mathcal{F}_+^R(x,\xi), \\
\nonumber & \mathcal{G}_+ (x, \xi) = \mathcal{G}_+^S (x, \xi) + \mathcal{G}_+^R (x, \xi).
\end{align}
Here, the non-decaying singular parts $\mathcal{F}_+^S$, $\mathcal{G}_+^S$ are given by
\begin{align*}
& \mathcal{F}_+^S (x, \xi) = \chi_+ (x) s(\xi) e^{i x \xi} e_1 + \chi_- (x) [ e^{i x \xi} + r(\xi) e^{-i x \xi} ]e_1\\
& \mathcal{G}_+^S (x, \xi) =\chi_+ (x) [ e^{-i x \xi} + r(\xi) e^{i x \xi} ]e_1 + \chi_- (x) s(\xi) e^{-i x \xi} e_1,
\end{align*}
while the decaying regular parts $\mathcal{F}_+^R$ can be decomposed as
\begin{equation} \label{guimauve}
\mathcal{F}^R_+ =m^+(x,\xi) e^{ix\xi}+ m^-(x,\xi) e^{-ix\xi} ,
\end{equation}
where the coefficients enjoy the estimates
\begin{equation}
\label{carambar}  | \partial_x^k \partial_\xi^\ell m^{\pm}(x,\xi) |  \lesssim e^{-\beta |x|} \langle \xi \rangle^{-\ell-1}, \qquad \forall x\in \mathbb R,
\end{equation}
and similarly for $\mathcal G_+^R(x)=\mathcal F_+^{R}(-x)$.

\end{lemma}

\begin{proof} The first assertion of this lemma is Lemma 6.4 in ~\cite{KS}; the other statements correspond to Lemma 6.3 in~\cite{KS}, but we provide here a more precise decomposition. Since $\mathcal{G}_+(x,\xi)=\mathcal{F}_+(-x,\xi)$, it suffices to consider $\mathcal F_+$.

\medskip

\noindent \textbf{Step 1}. \emph{First formulas for $\mathcal{F}_+$ for $x \geq 0$ and $x\leq 0$}. On the one hand, we can simply use the definition \eqref{S-v-dft-eq7} of $s(\xi)$ to write
\begin{equation} \label{formulaF++}
\begin{split}
\mathcal{F}_+(x,\xi) & = 2i\xi F_1(x,\xi) D(\xi)^{-1} e_1 \\
& = 2i\xi 	F_1(x,\xi) p D(\xi)^{-1} e_1 + 2i\xi 	F_1(x,\xi) q D(\xi)^{-1} e_1 \\
& = F_1(x,\xi) s(\xi) e_1 + 2i\xi (0,f_3(x,\xi)) D(\xi)^{-1} e_1 \\
& = s(\xi) f_1(x,\xi) + \left(2i\xi e_2^\top D(\xi)^{-1} e_1\right)  f_3(x,\xi) 
\end{split}
\end{equation}
which will be useful for $x\geq 0$. On the other hand, starting from the definition \eqref{S-v-dft-eq1}  of $\mathcal{F}_+$, and using successively~\eqref{pinson},~\eqref{loriot}, and the definition of $r(\xi)$, gives
\begin{equation} \label{formulaF+-}
\begin{split}
& \mathcal{F}_+(x,\xi) = 2i\xi F_1(x,\xi) D(\xi)^{-1} e_1 \\
& \quad = 2i\xi G_1(x,\xi) A(\xi)D(\xi)^{-1} e_1+ 2i\xi G_2(x,\xi) B(\xi) D(\xi)^{-1} e_1 \\
& \quad = 2i\xi G_1(x,\xi)\left(\frac{1}{2i\xi} p - \frac{1}{2 \langle \xi \rangle} q \right) e_1 + 2i\xi G_2(x,\xi) p B(\xi) D(\xi)^{-1} e_1 +  2i\xi G_2(x,\xi) q B(\xi) D(\xi)^{-1} e_1 \\
& \quad = G_1(x,\xi) e_1 + r(\xi) G_2(x,\xi) e_1 + 2i \xi(0 , g_3) B(\xi) D(\xi)^{-1} e_1 \\
& \quad = g_2(x,\xi) + r(\xi) g_1(x,\xi) +\left( 2 i \xi e_2^\top B(\xi) D(\xi)^{-1} e_1\right)g_3(x,\xi)
\end{split}
\end{equation}
which will be useful for $x\leq 0$.

\medskip

\noindent \textbf{Step 2}. \emph{Bound on $\mathcal{F}_+$ at $x=\pm 5$}. We now prove that
$$
| \partial_\xi^{\ell} \mathcal F_+ (\pm 5,\xi)|\lesssim_\ell 1 
$$
 is uniformly bounded for $\xi \in \mathbb R$. The bound $|\partial_\xi^\ell \mathcal{F}_+(5,\xi)| \lesssim 1$ follows immediately from~\eqref{formulaF++}, the estimates \eqref{estimatesrs} on $s$ and $r$, \eqref{merle} and \eqref{loriot} which imply $|\partial_\xi^\ell(\xi D(\xi)^{-1})|\lesssim 1$, and the bounds on $f_1$ and $f_3$ stated in lemmas~\ref{lemma12}  and~\ref{lemma3} respectively.

The bound $|\partial_\xi^\ell \mathcal{F}_+(-5,\xi)| \lesssim 1$ follows from \eqref{formulaF+-} and the same set of estimates, recalling that $g_i(x)=f_i(-x)$ and $f_2=\overline{f_1}$.

\medskip

\noindent \textbf{Step 3}. \emph{Second formula for $\mathcal{F}_+$}. To obtain better estimates in the zone $|x|\leq 5$, we convert the expressions of Step 1 into expressions involving the eigenfunctions $f^*_1$, $f^*_2$ and $f^*_3$ of Lemma \ref{lem:improved-eigenfunctions}. As they form a basis for solutions to $\mathcal H f=(1+\xi^2)f$ that are bounded as $x\to \infty$, $\mathcal{F}_+$ is a linear combination of these three functions. Matching the identity \eqref{formulaF++} with the asymptotics \eqref{id:decomposition-f*}-\eqref{bd:f*} shows
$$
\mathcal{F}_+(x,\xi)= s(\xi) f_1^*(x,\xi)+q(\xi)f_3^*(x,\xi)
$$
for some $q(\xi)\in \mathbb C$, which will account for $\mathcal F_+$ for $x\geq -5$. For $|\xi|$ large, $f_3^*(-5,\xi)$ does not vanish from \eqref{bd:f*} so that $q= (\mathcal{F}_+(-5) -s f_1^*(-5)) / f_3^*(-5) $. Using then that both $\mathcal{F}_+(-5)$, $ f_1^*(-5)$ and all their derivatives in $\xi$ are uniformly bounded in $\xi$ from Step 2 and \eqref{bd:f*}, and that $f_3^*(-10)\sim e^{5 \langle \xi \rangle}$ from \eqref{bd:f*} shows
$$
| \partial_\xi^\ell q(\xi)| \lesssim_\ell e^{-5\langle \xi \rangle}.
$$
We now define $g_j^*(x,\xi)=f_j^*(-x,\xi)$ for $j=1,2,3$. Using \eqref{formulaF+-}, the same argument shows
$$
\mathcal F_+(x,\xi)= g_2^*(x,\xi) + r(\xi) g_1^*(x,\xi) +\widetilde q(\xi) g_3^*(x,\xi),
$$
where $|\partial_\xi^\ell \widetilde q(\xi)|\lesssim e^{-5\langle \xi \rangle}$, which will account for $\mathcal F_+$ for $x\leq 5$. We eventually decompose

\begin{equation}
\label{chardonneret}
\begin{split}
\mathcal{F}_+(x,\xi) & = \left[  \chi_{+} (x) + \chi_{-} (x ) \right] \mathcal{F}_+(x,\xi)\\
& =  \chi_{-} (x ) g_2^*(x,\xi) + r(\xi)  \chi_{-} (x ) g_1^*(x,\xi)+ s(\xi)  \chi_{+} (x) f_1^*(x,\xi)\\
&\qquad +  q(\xi) \chi_{+} (x) f_3^*(x,\xi)) +\widetilde q(\xi)  \chi_{-} (x) g_3^*(x,\xi)
\end{split}
\end{equation}
with for all $\xi\in \mathbb R$
\begin{equation}
\label{chardonneret2}
|\partial_\xi^\ell q(\xi)|+|\partial_\xi^\ell \widetilde q(\xi)|\lesssim_\ell  e^{-5\langle \xi \rangle}.
\end{equation}

\medskip

\noindent \textbf{Step 4}. \emph{Decomposition of $\mathcal{F}_+$ into singular and regular parts.} We now prove the decomposition \eqref{meringue}-\eqref{guimauve} of $\mathcal{F}_+$ and the corresponding estimates. Injecting the decompositions \eqref{id:decomposition-f*} in \eqref{chardonneret}, using that $g_j^*(x)=f_j^*(-x)$ shows
\begin{equation}
\label{chardonneret3}
\begin{split}
\mathcal{F}_+(x,\xi) & =  \chi_{-} (x ) e^{ix\xi} + r(\xi)  \chi_{-} (x ) e^{-ix\xi }e_1 + s(\xi)  \chi_{+} (x) e^{i x\xi }e_1 \\
& \quad +\left(\chi_-(x) (\mathfrak g_{2,1}^*-e_1)+ s(\xi)  \chi_{+} (x) (\mathfrak f_{1,1}^*-e_1)\right)e^{ix\xi}+ r(\xi)  \chi_{-} (x ) (\mathfrak g_{1,1}^*(x,\xi)-e_1)e^{-ix \xi} \\
&\qquad + \chi_{-} (x )( \mathfrak g_{2,2}^*+r(\xi)\mathfrak g_{1,2}^*+\widetilde q(\xi)\mathfrak g_{3}^*)  e^{\langle \xi \rangle x}+\chi_+(x)(s(\xi) \mathfrak f_{1,2}^* +q(\xi)\mathfrak f_3^*)e^{-\langle \xi \rangle x},
\end{split}
\end{equation}
where the notation for the $\mathfrak g^*$'s naturally adapts that for $\mathfrak f^*$'s. This shows the identity \eqref{meringue} with
\begin{align*}
& m^+_1(x,\xi) = \chi_-(x) (\mathfrak g_{2,1}^*-e_1)+ s(\xi)  \chi_{+} (x) (\mathfrak f_{1,1}^*-e_1),\\
& m^-_1(x,\xi) = m^-_{1,1}(x,\xi)+m^-_{1,2}(x,\xi),\\ 
& m^-_{1,1}(x,\xi)= r(\xi)  \chi_{-} (x ) (\mathfrak g_{1,1}^*(x,\xi)-e_1), \\
& m^-_{1,2}(x,\xi) = \chi_{-} (x )( \mathfrak g_{2,2}^* +r(\xi)\mathfrak g_{1,2}^*+\widetilde q(\xi)\mathfrak g_{3}^*)e^{(\langle \xi \rangle+i\xi) x}  +\chi_+(x)(s(\xi) \mathfrak f_{1,2}^* +q(\xi)\mathfrak f_3^*)e^{(-\langle \xi \rangle+i\xi) x},
\end{align*}
The desired bound \eqref{carambar} for $m^+_1$ and $m^-_{1,1}$ is a direct consequence of \eqref{bd:f*} and of the localising properties of $\chi_\pm$. We now turn to showing the bound \eqref{carambar} for $m^-_{1,2}$. By \eqref{bd:f*} and \eqref{chardonneret2} we have
$$
\left| \partial_{x}^k \partial_\xi^\ell \left( \mathfrak g_{2,2}^* +r(\xi)\mathfrak g_{1,2}^*+\widetilde q(\xi)\mathfrak g_{3}^*\right)\right|\lesssim_{k,\ell} e^{-5\langle \xi \rangle}
$$
for all $x\leq 10$. For $x\leq 10$ we also have $|\partial_x^k \partial_\xi^\ell (e^{(\langle \xi \rangle+i\xi) x})|\lesssim e^{-\frac 12 x+\frac12 \langle \xi \rangle+\langle \xi \rangle x}$. As $\chi_-$ is localised for $x\leq 1/2$ this implies that for the first term
\begin{align*}
& \left| \partial_{x}^k \partial_\xi^\ell \left( \chi_{-} (x )( \mathfrak g_{2,2}^* +r(\xi)\mathfrak g_{1,2}^*+\widetilde q(\xi)\mathfrak g_{3}^*)e^{(\langle \xi \rangle+i\xi) x}\right)\right| \\
& \lesssim e^{-5\langle \xi \rangle -\frac 12  x+\frac12 \langle \xi \rangle +\langle \xi \rangle x}  = e^{- \frac 12 |x| -\langle \xi \rangle} e^{(\langle \xi \rangle-\frac 12+\frac 12 \textup{sgn}(x))x -\frac 32 \langle \xi \rangle}\leq e^{- \frac 12 |x| -\langle \xi \rangle} e^{\langle \xi \rangle \frac 12 -\frac 32 \langle \xi \rangle}\leq e^{- \frac 12 |x| -\langle \xi \rangle}
\end{align*}
where we used that for each $\xi\in \mathbb R$, the function $x\mapsto (\langle \xi \rangle-\frac 12+\frac 12 \textup{sgn}(x))x$ attains its maximum on $(-\infty,\frac 12 ]$ at $x=\frac 12$.

The second term can be dealt with the exact same way. Hence $m^-_{1,2}$ also satisfies \eqref{carambar}. This ends the proof of the lemma.

 \end{proof}

\begin{lemma}[Vanishing at zero frequency]
For any $k,\ell \in \mathbb{N}_0$, the generalized eigenfunctions $\mathcal{F}_+$ and $\mathcal{G}_+$ satisfy for $|\xi| \leq 1$
\begin{equation}
\label{chouette}
\begin{split}
& | \mathcal{F}_+(x,\xi)| + |\mathcal{G}_+(x,\xi) | \lesssim \min(\langle x \rangle |\xi|,1)  \\
& \left| \partial_x^k \partial_\xi^\ell \left[ \frac{\mathcal{F}_+(x,\xi)}{\xi} \right] \right|  + \left| \partial_x^k \partial_\xi^\ell \left[ \frac{\mathcal{G}_+(x,\xi)}{\xi} \right] \right| \lesssim \langle x \rangle^{1+\ell}.
\end{split}
\end{equation}
The singular parts $\mathcal{F}_+^S$ and $\mathcal{G}_+^S$ enjoy the same bounds, and the regular parts $\mathcal{F}_+^R$ and $\mathcal{G}_+^R$ are such that, if $|\xi| \leq 1$,
\begin{equation}
\label{chouette2}
\begin{split}
& | \mathcal{F}^R_+(x,\xi)| + |\mathcal{G}^R_+(x,\xi) | \lesssim |\xi| e^{-\frac \beta 2 |x|}  \\
& \left| \partial_x^k \partial_\xi^\ell \left[ \frac{\mathcal{F}^R_+(x,\xi)}{\xi} \right] \right|  + \left| \partial_x^k \partial_\xi^\ell \left[ \frac{\mathcal{G}^R_+(x,\xi)}{\xi} \right] \right| \lesssim e^{-\frac{\beta}{2}|x|}.
\end{split}
\end{equation}

\end{lemma}

\begin{proof} \underline{The full functions $\mathcal{F}_+$ and $\mathcal{G}_+$.} It suffices to prove the second inequality in~\eqref{chouette}, since it implies the first one. Furthermore, we will prove the second inequality only for $\mathcal{F}_+$, since $\mathcal{G}_+(x)=\mathcal F_+(-x)$. From the identities \eqref{formulaF++} and \eqref{formulaF+-} in the proof of Lemma~\ref{heroncendre}, we learn that
\begin{align*}
\frac{1}{\xi} \mathcal{F}_+(x,\xi) & = \chi_+(x) \left[ \frac{s(\xi)}{\xi} f_1(x,\xi) + \left(2i e_2^\top D(\xi)^{-1} e_1\right)  f_3(x,\xi) \right] \\
&  + \chi_-(x) \left[  \frac{g_2(x,\xi)-g_1(x,\xi)}{\xi} + \frac{r(\xi)+1}{\xi} g_1(x,\xi) +\left( 2 i e_2^\top B(\xi) D(\xi)^{-1} e_1\right)g_3(x,\xi)  \right].
\end{align*}
Recalling that $s(0)=0$ while $r(0)=-1$ and that $qB(\xi)$ is smooth (Lemma~\ref{lemmaAB}), most terms in the above expression are readily seen to satisfy the desired bound~\eqref{chouette}. But the term $\frac{g_2(x,\xi) - g_1(x,\xi)}{\xi}$ requires a further argument: it can be written using \eqref{stollen} as
\begin{align*}
&  g_1(-x,\xi) - g_2(-x,\xi) = f_1(x,\xi) - \overline{f_1(x,\xi)} \\
& \quad = \mathfrak{f}_{1,1}(x,\xi) [e^{ix\xi} - e^{-ix\xi}] + [\mathfrak{f}_{1,1}(x,\xi)  -\overline{\mathfrak{f}_{1,1}(x,\xi)}] e^{-ix\xi} + [\mathfrak{f}_{1,2}(x,\xi) - \overline{\mathfrak{f}_{1,2}(x,\xi)} ] e^{-x \langle \xi \rangle} .
\end{align*}
From this expression, the bounds on $\mathfrak{f}_{1,1}$ and $\mathfrak{f}_{1,2}$ in Lemma~\ref{lemma12}, and the fact that $\mathfrak{f}_{1,1}(x,0)$ and $\mathfrak{f}_{1,2}(x,0)$ are real, we get, for $x<-1$
$$
\left| \partial_x^k \partial_\xi^\ell \left[ \frac{g_2(x,\xi) - g_1(x,\xi)}{\xi} \right] \right| \lesssim \langle x \rangle^{1+\ell},
$$
which completes the proof.

\medskip

\noindent \underline{The singular parts $\mathcal{F}_+^S$ and $\mathcal{G}_+^S$.} They can be dealt with as the full functions $\mathcal{F}_+$ and $\mathcal{G}_+$.

\medskip

\noindent \underline{The regular parts $\mathcal{F}_+^R$ and $\mathcal{G}^+_R$.} Once again, it suffices to prove the second inequality in~\eqref{chouette2} since it implies the first one. The first step is to observe that
$$
\mathcal{F}_+^R(x,0) = \mathcal{F}_+(x,0) - \mathcal{F}_+^S(x,0) = 0.
$$
This implies that the coefficients in the decomposition~\ref{guimauve} satisfy
$$
m_+(x,0) + m_-(x,0) = 0.
$$
Therefore, writing $\frac{\mathcal{F}^R_+(x,\xi)}{\xi}$ as
$$
\frac{\mathcal{F}^R_+(x,\xi)}{\xi} = \frac{m_+(x,\xi) + m_-(x,\xi)}{\xi} + m^+(x,\xi) \left[ \frac{e^{ix\xi} - 1}{\xi} \right] + m^-(x,\xi) \left[ \frac{e^{-ix\xi} - 1}{\xi} \right]
$$
leads to the desired estimate for $\mathcal{F}_+^R$, and then $\mathcal{G}^R_+$.
\end{proof}

While $\mathcal{F}_+$ and $\mathcal{G}_+$ span the continuous spectrum, their counterparts $\mathcal{F}_-$ and $\mathcal{G}_-$ will span the negative spectrum. They are defined by
\begin{align*} 
& \mathcal{F}_- (x, \xi) = \sigma_1 \mathcal{F}_+ (x, \xi) \\
& \mathcal{G}_- (x, \xi) = \sigma_1 \mathcal{G}_+ (x, \xi).
\end{align*}
and are bounded solutions of $\mathcal{H} f = - (\xi^2 + 1) f$.

We can now define the functions $\psi_\pm$, through which the operator $\mathcal{H}$ will be diagonalized:
$$
\psi_\pm (x, \xi) := 
\begin{cases}
\mathcal{F}_\pm (x, \xi) \ \text{ if } \xi \geq 0  \\
\mathcal{G}_\pm (x, - \xi) \ \text{ if } \xi \leq 0 ,  \\
\end{cases}
$$
In a similar way to $\mathcal{F}_+$ and $\mathcal{G}_+$, the functions $\psi_{\pm}$ can 
be decomposed into a singular and a regular part:
\begin{equation} \label{alouette}
\psi^S_\pm (x, \xi) :=
\begin{cases}
\mathcal{F}^S_\pm (x, \xi) \ \text{ if } \xi \geq 0 , \\
\mathcal{G}^S_\pm (x, -\xi) \ \text{ if } \xi \leq 0 .  \\
\end{cases}
, \quad 
\psi^R_\pm (x, \xi) :=
\begin{cases}
\mathcal{F}^R_\pm (x, \xi) \ \text{ if } \xi \geq 0 , \\
\mathcal{G}^R_\pm (x, -\xi) \ \text{ if } \xi \leq 0 .  \\
\end{cases}
\end{equation}

\begin{remark}
At this point, it is instructive to examine the case $V=0$: then
\begin{align*}
& \psi_{+}(x,\xi) =  \psi_{+}^S(x,\xi) = e^{ix\xi} e_1 \\
& \psi_{-}(x,\xi) =  \psi_{-}^S(x,\xi) = e^{ix\xi} e_2 .
\end{align*}
\end{remark}

\subsection{Definition and first properties of the distorted Fourier transform}

The following proposition describes the spectral resolution of $\mathcal{H}$ which can be obtained in terms of $\psi_{\pm}$.
\begin{proposition} \label{S-v-dft-prop1}
For every $f$, $g \in \mathcal{S}$,
\begin{equation} \label{S-v-dft-eq16}
\la P_e f, g \ra = \frac{1}{2 \pi} \sum_{\epsilon = \pm } \epsilon \int \la f, \sigma_3 \psi_\epsilon (\cdot, \xi) \ra \overline{\la g, \psi_\epsilon (\cdot, \xi) \ra} d \xi. 
\end{equation}
The integrals on the right-hand side are absolutely convergent, since the integrand is rapidly decaying.
\end{proposition}

\begin{proof} This is Proposition 6.9 in \cite{KS}, up to a small misprint: the $\epsilon$ factor on the right-hand side was omitted there. To clarify why it should appear, we come back to the formula for the resolvent jump on the real axis: it was proved in Lemma 6.7 of~\cite{KS} that
\begin{equation}
\label{bruant1}
(\mathcal{H} - (\xi^2 + 1 + i0))^{-1}(x,y) - (\mathcal{H} - (\xi^2 + 1 - i0))^{-1}(x,y) = -\frac{1}{2i\xi} \mathcal{E_+}(x,\xi) \mathcal{E}_+(y,\xi)^* \sigma_3,
\end{equation}
for $\xi \geq 0$, where we used the notation
$$
\mathcal{E}_\pm(x,\xi) = [\psi_\pm(x,\xi)\,,\, \psi_\pm(x,-\xi)]
$$
(so that $\mathcal{E}_\pm(x,\lambda)$ is a $2 \times 2$ matrix). Using successively~\eqref{S-v-eq5}, $\sigma_1\sigma_1= \operatorname{Id}$, formula~\eqref{bruant1} and the anticommutation relation $\sigma_3 \sigma_1 = - \sigma_1 \sigma_3$, we find, for $\xi \geq 0$,
\begin{equation}
\label{bruant2}
\begin{split}
&(\mathcal{H} - (-\xi^2 - 1 + i0))^{-1}(x,y) - (\mathcal{H} - (-\xi^2 - 1 - i0))^{-1}(x,y) \\
& \qquad = (-\sigma_1 \mathcal{H} \sigma_1 + (\xi^2 + 1 - i0))^{-1}(x,y) - (- \sigma_1 \mathcal{H} \sigma_1 +(\xi^2 + 1 +i0))^{-1}(x,y) \\
& \qquad = - \sigma_1 (\mathcal{H} - (\xi^2 +1 - i0))^{-1}\sigma_1(x,y) + \sigma_1(\mathcal{H} - (\xi^2+ 1 + i0))^{-1}\sigma_1(x,y) \\
& \qquad = \sigma_1 (\mathcal{H} - (\xi^2 +1 +i0))^{-1}\sigma_1(x,y) - \sigma_1(\mathcal{H} - (\xi^2 +1 - i0))^{-1}\sigma_1(x,y) \\
& \qquad = -\frac{1}{2i\xi} \sigma_1 \mathcal{E_+}(x,\xi) \mathcal{E}_+(y,\xi)^* \sigma_3  \sigma_1 \\
& \qquad = \frac{1}{2i\xi} \sigma_1 \mathcal{E_+}(x,\xi) \mathcal{E}_+(y,\xi)^*  \sigma_1 \sigma_3 \\
& \qquad = \frac{1}{2i\xi}  \mathcal{E_-}(x,\xi) \mathcal{E}_-(y,\xi)^* \sigma_3,
\end{split}
\end{equation}
where for the last line we used that $\psi_-=\sigma_1 \psi_+$. Notice the sign difference between~\eqref{bruant1} and~\eqref{bruant2}. Using these two formulas and following the argument in the proof of Proposition 6.9 of~\cite{KS} leads to the desired result.

\end{proof}

\begin{corollary} \label{S-v-dft-cor1}
For every $f,g \in \mathcal{S}$, 
\begin{equation} \label{S-v-dft-eq17}
\la e^{it \mathcal{H}}P_e f, g \ra = \frac{1}{2\pi}\sum_{\epsilon = \pm } \epsilon \int e^{i \epsilon t (\xi^2+1)} \la f, \sigma_3 \psi_\epsilon (\cdot, \xi) \ra \overline{\la g, \psi_\epsilon (\cdot, \xi) \ra} d\xi. 
\end{equation}
with absolutely convergent integrals.
\end{corollary}

\begin{proof} This is Corollary 6.10 in~\cite{KS}.
\end{proof}

The above can be expressed as
\begin{equation*} \label{S-v-dft-eq18}
e^{it \mathcal{H}} P_e f  = \frac{1}{2\pi}\sum_\epsilon \epsilon \int e^{i \epsilon t (\xi^2 + 1)} \la f, \sigma_3 \psi_\epsilon (\cdot, \xi) \ra   \psi_\epsilon (\cdot, \xi)  \,d \xi. 
\end{equation*}

We now come to defining the distorted Fourier transform associated to the operator $\mathcal{H}$:
\begin{equation}
\label{defFT}
\begin{split}
& (\widetilde{\mathcal{F}} f) (\xi) :=   (\widetilde{f}_+ (\xi),  \widetilde{f}_-(\xi))^\top , \\
& \text{where } \widetilde{f}_{\pm} (\xi) := (\widetilde{\mathcal{F}}_\pm f) (\xi) :=\frac{1}{\sqrt{2 \pi}}  \la f, \sigma_3 \psi_{\pm} (\cdot, \xi) \ra    .  \\
\end{split}
\end{equation}

\begin{remark} The terminology we choose (distorted Fourier transform) can be misleading. On the one hand, the operator we define is in many ways to $\mathcal{H}$ what the Fourier transform is to $-\partial_x^2$ in the scalar case. Furthermore, in the semi-classical limit $\xi \to \infty$, the distorted Fourier transform approaches the flat Fourier transform. On the other hand, $\mathcal{H}$ is not self-adjoint, so that $\widetilde{\mathcal{F}}$ does not enjoy some of the basic properties of the Fourier transform such as being unitary.
\end{remark}

For future use, we collect some basic formulas associated to the distorted Fourier transform:
\begin{itemize}
\item A function $f \in \mathcal{S}$ can be reconstructed from its Fourier transform through the formula
\begin{equation}
\label{inverseFT}
P_e f(x) = \frac{1}{\sqrt{2\pi}} \sum_\epsilon \epsilon \int \widetilde{f}_\epsilon(\xi) \psi_\epsilon(x,\xi) \,d\xi. 
\end{equation}
Indeed, starting from the identity~\eqref{S-v-dft-eq16} and applying Fubini's theorem, we get for $f,g \in \mathcal{S}$
$$
\langle P_e f, g \rangle = \frac{1}{\sqrt{2\pi}} \sum_\epsilon \epsilon \int \widetilde{f}_\epsilon(\xi) \overline{ g(y)} \cdot \psi_\epsilon(y,\xi)  \, dy \, d\xi = \langle  \frac{1}{\sqrt{2\pi}} \sum_\epsilon \epsilon \int \widetilde{f}_\epsilon(\xi) \psi_\epsilon(y,\xi) \,d\xi \, , \, g \rangle.
$$	
\item Considering now the action of $\mathcal{H}$, equation~\ref{S-v-dft-eq18} implies that
$$
\mathcal{H} = \widetilde{\mathcal{F}}^{-1} (1 + |\xi|^2) \widetilde{\mathcal{F}}
$$
\item The Schr\"odinger group is then given in Fourier space by
$$
\widetilde{\mathcal{F}}_{\pm}(e^{it \mathcal{H}} P_e f)(\xi) = e^{\pm i(1+\xi^2)t} \widetilde{f}_{\pm}(\xi),
$$
and in physical space by
\begin{equation} \label{bibliotheque}
e^{it \mathcal{H}} P_e f = \frac{1}{\sqrt{2\pi}}\sum_{\epsilon = \pm 1} \epsilon  \int e^{i \epsilon t (1+\xi^2)} \widetilde{f}_\epsilon(\xi) \psi_\epsilon (\cdot, \xi) d\xi. 
\end{equation}
\item The spectral projectors satisfy
\begin{equation} \label{id:relation-distorted-fourier-and-projectors}
\widetilde{\mathcal F} P_e=\widetilde{\mathcal F} \qquad \mbox{and} \qquad \widetilde{\mathcal F} P_d=0 .
\end{equation}
\end{itemize}

The last assertion directly follows from \eqref{defFT} and \eqref{thym citrone}. We introduce the space $L^{2,1}$ with norm $\| (\widetilde f_+,\widetilde f_-)\|_{L^{2,1}}^2=\sum_{\pm }\int_{\mathbb R} |\widetilde f_\pm(\xi)|^2\langle \xi \rangle^2 d\xi$.

\begin{proposition} 
\label{tourterelle}
[Boundedness properties of the distorted Fourier transform and its inverse]
The Fourier transform defined in~\eqref{defFT} and its inverse defined in~\eqref{inverseFT} are such that
\begin{itemize}
\item[(i)] $\widetilde{\mathcal{F}}$ and $\widetilde{\mathcal{F}}^{-1}$ are bounded on $L^2$. 
\item[(ii)] $\widetilde{\mathcal{F}}$ and $\widetilde{\mathcal{F}}^{-1}$ are bounded operators from $L^1$ to $\mathcal{C}_0$, the set of continuous functions decaying to zero at infinity. Furthermore, $\widetilde{f}(0) = 0$ if $f \in L^1$.
\item[(iii)] $\widetilde{\mathcal{F}}$ and $\widetilde{\mathcal{F}}^{-1}$ map boundedly $L^p$ to $L^{p'}$ if $1 \leq p \leq 2$.
\item[(iv)] $\widetilde{\mathcal{F}}$ maps boundedly $H^1$ to $L^{2,1}$, and $L^{2,1}$ to $H^1\cap \{\widetilde f_\pm(0)=0\}$; and $\widetilde{\mathcal{F}}^{-1}$ maps boundedly $H^1\cap \{\widetilde f_\pm(0)=0\}$ to $L^{2,1}$, and $L^{2,1}$ to $H^1$.

\end{itemize}
\end{proposition}

\begin{remark}
The four above statements are analogous to classical results for the Fourier transform on $\mathbb{R}^d$, namely the Plancherel theorem, Riemann-Lebesgue lemma, Hausdorff-Young inequality, and the fact that the Fourier transform is an isomorphism between $H^1$ and $L^{2,1}$.
\end{remark}

\begin{proof}
\noindent (i) can be proved through the decomposition $\psi_{\pm}(x,\xi) = \psi_{\pm}^S(x,\xi) + \psi_{\pm}^R(x,\xi)$. Considering first the map $f \mapsto \langle f,\psi^S(x,\xi) \rangle$, it is bounded on $L^2$ by Plancherel's theorem. Turning to $f \mapsto \langle f,\psi^R(x,\xi) \rangle$, it is the sum of two operators with kernels $\overline{m^\pm(x,\xi)}e^{\mp i\xi x}$. The bounds in Lemma~\ref{heroncendre} implies that these operators are Hilbert-Schmidt, hence bounded on $L^2$. This shows the boundedness in $L^2$ of $\widetilde{\mathcal F}$. Since the previous bounds remain true if the roles of $x$ and $\xi$ are exchanged, we also get the boundedness in $L^2$ of $\widetilde{\mathcal F}^{-1}$.

\medskip

\noindent (ii) is a consequence of the facts that $\psi_{\pm}(x,\xi)$ is uniformly bounded, depends continuously on $x$ and $\xi$, and satisfies $\psi_{\pm}(x,0) = 0$.

\medskip

\noindent (iii) follows by interpolating between the two previous assertions.

\medskip

\noindent (iv) will be proved by decomposing $\xi \int \psi(x,\xi) f(x) \,dx$ according to Lemma~\ref{heroncendre}.  We aim at bounding this term in $L^2$ if $f$ belongs to $H^1$. The contribution of $\psi^S(x,\xi)$ is easily treated, and we skip it. We consider next terms of the type (using the notations of Lemma~\ref{heroncendre} and simplifying it by omitting indexes and complex conjugation) $\xi \int  m (x,\xi) e^{ ix\xi} f(x) \,dx$ where $|\partial_x^k \partial_\xi^\ell m|\lesssim \langle \xi \rangle^{-1-\ell}e^{-\beta |x|}$. They reduce, after integration by parts, to
\begin{align*}
&  i \int m (x,\xi) e^{ ix\xi} \partial_x f(x) \,dx+ i \int  \partial_x m(x,\xi) e^{ ix\xi} f(x) \,dx.
\end{align*}
These two summands can be treated as in $(i)$, hence belonging to $L^2$.

The continuity of $\widetilde{\mathcal F}:L^{2,1}\rightarrow H^1\cap \{\widetilde f_\pm(0)=0\}$, and that of $\widetilde{\mathcal F}^{-1}$ from $H^1\cap \{\widetilde f_\pm(0)=0\}$ to $L^{2,1}$ and from $L^{2,1}$ to $H^1$ can be proved very similarly, we omit the details.

\end{proof}

\begin{definition}[Wave operator and distorted Fourier multipliers] \label{chevalier}
The \textit{distorted Fourier multiplier} $\mathfrak{a}(\widetilde{D})$ is the operator defined by
$$
\mathfrak{a}(\widetilde{D}) = \widetilde{\mathcal{F}}^{-1} \mathfrak{a}(\xi) \widetilde{\mathcal{F}}.
$$
The \textit{wave operator} $\mathcal{W}$ is the operator defined by
\begin{equation} \label{grimpereau}
\mathcal{W}_\rho = \widehat{\mathcal{F}}^{-1} \widetilde{\mathcal{F}_\rho}
\end{equation}
(for scalar potentials, this operator is usually denoted $\mathcal{W}^*$, but in order to avoid a discussion on its invertibility, we simply set the above to be equal to $\mathcal{W}$).
\end{definition}

\begin{corollary} \label{pluvier}
\begin{itemize} 
\item[(i)] (Boundedness of distorted Fourier multipliers on $L^2$) If $\mathfrak{a}(\xi) \in L^\infty$, then the distorted Fourier multiplier $\mathfrak{a}(\widetilde{D})$ is bounded on $L^2$.
\item[(ii)] (Boundedness of the wave operator on $L^2$) The wave operator $\mathcal{W}$ is bounded on $L^2$.
\end{itemize}
\end{corollary}

\subsection{Linear operations seen on the Fourier side}

\begin{lemma}[Exact symmetries of the distorted Fourier transform: $\sigma_1$ and complex conjugation] 
\label{gelinotte}
If $f \in \mathcal{S}$, then for any $\xi \in \mathbb{R}$,
\begin{equation} \label{id:symmetry-FT-sigma1}
(\widetilde{\mathcal{F}} \sigma_1 f)_{\pm}(\xi) = - (\widetilde{\mathcal{F}} f)_{\mp}(\xi) 
\end{equation}
and for any $\xi \in \mathbb{R}$,
\begin{equation} \label{id:symmetry-FT-complex-conjugation}
\begin{pmatrix} (\widetilde{\mathcal{F}} \overline{f})_{\pm}(\xi) \\ (\widetilde{\mathcal{F}} \overline{f})_{\pm}(-\xi) \end{pmatrix} = S(-|\xi|) \overline{\begin{pmatrix} \widetilde f_{\pm}(-\xi) \\ \widetilde f_{\pm}(\xi)\end{pmatrix}}
\end{equation}
\end{lemma}

\begin{proof}
To prove the first assertion, we use successively that $\sigma_1$ is Hermitian, that $\sigma_1$ and $\sigma_3$ anticommute, and that $\sigma_1 \psi_{\pm}(x,\xi) = \psi_{\mp}(x,\xi)$ to obtain
\begin{align*}
\sqrt{2\pi} (\widetilde{\mathcal{F}}\sigma_1 f)_{\pm} & = \langle \sigma_1 f\,,\, \sigma_3 \psi_{\pm} \rangle = \langle f \,,\, \sigma_1 \sigma_3 \psi_{\pm} \rangle \\
& = - \langle f \,,\, \sigma_3 \sigma_1 \psi_{\pm} \rangle = - \langle f \,,\, \sigma_3 \psi_{\mp} \rangle = - \sqrt{2\pi} (\widetilde{\mathcal{F}} f)_{\mp}.
\end{align*}

To prove the second assertion, it suffices to consider the case $\xi>0$. By the first assertion in Lemma~\ref{heroncendre}, the functions $\overline{\psi_\pm(x,\xi)}$ and $\overline{\psi_\pm(x,-\xi)}$ can be expressed as a linear combination of ${\psi_\pm(x,\xi)}$ and ${\psi_\pm(x,-\xi)}$. The coefficients in this linear combination can be identified thanks to Lemma~\ref{heroncendre}, which asserts that
$$
\mathcal{F}_+(x,\xi) \sim s(\xi) e^{ix\xi} e_1, \qquad \mathcal{G}_+(x,\xi) \sim e^{-ix\xi} e_1 + r(\xi) e^{ix\xi} e_1, \qquad \mbox{as $x \to \infty$}.
$$
Using this in combination with the identities~\eqref{relationsrs}, we obtain for $\xi>0$ that
$$
\begin{pmatrix} \overline{\psi_\pm(x,\xi)} \\ \overline{\psi_\pm(x,-\xi)} \end{pmatrix}
=  S(-\xi) \begin{pmatrix} {\psi_\pm(x,-\xi)} \\ {\psi_\pm(x,\xi)} \end{pmatrix};
$$
here, we are slightly abusing notations by denoting $S(\xi)$ the $4 \times 4$ matrix $\begin{pmatrix} s(\xi) \operatorname{Id} &  r(\xi) \operatorname{Id} \\  r(\xi) \operatorname{Id} & s(\xi) \operatorname{Id} \end{pmatrix}$.

Using the definition \eqref{defFT} of the distorted Fourier transform, as well as the convention for the Hermitian scalar product, we infer for $\xi>0$ that
\begin{align*}
\sqrt{2\pi} (\widetilde{\mathcal F}\bar f)_{\pm}(\xi) = \langle \bar f, \sigma_3\psi_\pm (\cdot,\xi)\rangle & = \overline{\langle  f, \sigma_3\overline{\psi_\pm (\cdot,\xi)}\rangle} \\
&= \overline{\langle  f, \sigma_3(s(-\xi) \psi_\pm (\cdot,-\xi)+r(-\xi)\psi_\pm (\cdot,\xi))\rangle} \\
&= s(-\xi) \overline{\langle  f, \sigma_3 \psi_\pm (\cdot,-\xi) \rangle}  +r(-\xi) \overline{\langle  f, \sigma_3 \psi_\pm (\cdot,\xi) \rangle}\\
&= s(-\xi) \sqrt{2\pi} \overline{\widetilde f_\pm(-\xi)} +r(-\xi)\sqrt{2\pi} \overline{\widetilde f_\pm(\xi)},
\end{align*}
and similarly $(\widetilde{\mathcal F}\bar f)_{\pm}(-\xi) =r(-\xi)\overline{\widetilde f_\pm(-\xi)}+s(-\xi)\overline{\widetilde f_\pm (\xi)}$. This shows \eqref{id:symmetry-FT-complex-conjugation}.

\end{proof}

\begin{lemma}[Approximate symmetry of the distorted Fourier transform: $\sigma_3$] \label{bartavelle}
There holds
$$
\widetilde{\mathcal{F}} \sigma_3 P_e f (\xi) = \sigma_3 \widetilde{f} (\xi) +\mathcal L_{R} \widetilde f,
$$
where $\mathcal L_{R}$ is a lower order operator of the form
$$
\mathcal L_{R} \widetilde f (\xi) = \sum_{\lambda \in \{ \pm \}} \int_{\mathbb R} \mathfrak{s}_{\rho \lambda} (\xi,\eta) \widetilde{f}_\lambda (\eta) \,d\eta
$$
for some smooth symbols $(\mathfrak{s}_{\rho \lambda})$ for $\rho,\lambda=\pm $ that satisfy the estimates, for $\eta \neq 0$,
\begin{align}
& \label{grandduc1} \left| \partial_\xi^a \partial_\eta^b \mathfrak{s}_{\rho\lambda}(\xi,\eta) \right| \lesssim_{a,b} \sum_{\pm} \frac{1}{\langle \xi \pm \eta \rangle^2} \\
& \label{grandduc2} \left| \partial_\xi^a \partial_\eta^b \frac{\mathfrak{s}_{\rho\lambda}(\xi,\eta)}{\eta} \right| \lesssim_{a,b} \frac{1}{\langle \eta \rangle}\sum_{\pm}  \frac{1}{\langle \xi \pm \eta \rangle^2} .
\end{align}
\end{lemma}
\begin{proof} The starting point is the identity $\sigma_3 \psi_\rho^S(x,\xi) = \rho \psi_\rho^S(x,\xi)$ from \eqref{alouette}-\eqref{meringue}. It implies that
\begin{align*}
\widetilde{\mathcal{F}}_\rho (\sigma_3 f) & = \frac{1}{\sqrt{2\pi}} \langle f\,,\,\sigma_3 \psi_\rho(\xi) \rangle =  \frac{1}{\sqrt{2\pi}} \langle f\,,\,\rho \psi^S_\rho(\xi) + \sigma_3 \psi^R_\rho(\xi) \rangle \\
& = \frac{1}{\sqrt{2\pi}} \langle f\,,\,\rho \psi_\rho(\xi) - \rho \psi^R_\rho(\xi) + \sigma_3 \psi^R_\rho(\xi) \rangle = \rho  \widetilde{f}_\rho(\xi) + \frac{1}{\sqrt{2\pi}} \langle f\,,\, \sigma_3 \psi^R_\rho(\xi) - \rho \psi^R_\rho(\xi) \rangle.
\end{align*}
Expanding $f$ from its distorted Fourier transform, the second term becomes
$$
\frac{1}{\sqrt{2\pi}} \langle f\,,\, \sigma_3 \psi^R_\rho(\xi) - \rho \psi^R_\rho(\xi) \rangle 
= \frac{1}{\sqrt{2\pi}}  \sum_\lambda \lambda \int \widetilde{f}_\lambda(\eta) \langle \psi_\lambda(\eta) \,,\, \sigma_3 \psi^R_\rho(\xi) - \rho \psi^R_\rho(\xi) \rangle \,d\eta.
$$
Therefore, the symbol $\mathfrak{s}_{\rho \lambda}$ is given by
$$
\mathfrak{s}_{\rho \lambda} (\xi,\eta) = \frac{\lambda}{\sqrt{2\pi}} \langle \psi_\lambda(\eta) \,,\, \sigma_3 \psi^R_\rho(\xi) - \rho \psi^R_\rho(\xi) \rangle
$$
(thanks to the rapid decay of $\psi^R_\rho$, this inner product is well-defined). Expanding $\psi_{\pm}$ as in Lemma~\ref{heroncendre}, and dropping unnecessary indices, we obtain a linear combination of terms of the type $\langle \psi^A(\eta) \,,\, \psi^R(\xi) \rangle$, where $A$ might be $S$, $R$.
\begin{itemize}
\item If $A = S$, then we can assume without loss of generality that $\psi^R(x,\xi) = e^{\pm i x \xi} \theta(x,\xi)$, where 
\begin{align*}
& |\partial_x^k\partial_\xi^\ell \theta(x,\xi) | \lesssim  \frac{e^{-\beta |x|}}{\langle \xi \rangle^{1+\ell}}.
\end{align*}
Then, $\langle \psi^A(\eta) \,,\, \psi^{R}(\xi) \rangle$ can be written as a sum of terms of the type
$$
\mathfrak{a}(\eta) \int e^{i x (\pm \xi \pm \eta)} \theta(x,\xi) \,dx,
$$
for a bounded coefficient $\mathfrak{a}$. This can be bounded by $O(1)$, or, after two integrations by parts in $x$:
$$
\left| \int e^{i x (\pm \xi \pm \eta)} \theta(x,\eta) \,dx \right| \lesssim \frac{1}{|\pm \xi \pm \eta|^2} \left| \int e^{i x (\pm \xi \pm \eta)} \partial_x^2 \theta(x,\xi) \,dx \right| \lesssim \frac{1}{|\pm \xi \pm \eta|^2} .
$$.
\item If $A = R$, a very similar argument applies.
\end{itemize}

Gathering the two cases above leads to the estimate~\eqref{grandduc1}. In order to obtain~\eqref{grandduc2}, we have to bound terms of the form $ \langle \frac{\psi_\lambda(\eta)}{\eta} \,,\,  \psi^R_\rho(\xi)  \rangle$ for $|\eta|\leq 1$. This is achieved by combining the above considerations and the estimate~\eqref{chouette}. 

\end{proof}

\begin{lemma}[Approximate symmetry of the distorted Fourier transform: $\partial_x$] \label{lem:symmetry-partialx}
There holds
$$
\widetilde{\mathcal F}(\partial_x P_e f)(\xi)=i\xi \widetilde{\mathcal F}f(\xi)+\mathcal L_0 \widetilde f,
$$
where $\mathcal L_0$ is a lower order operator that can be decomposed as
$$
\mathcal L_0=\mathcal L_{0,\delta}+\mathcal L_{0,p.v.}+\mathcal L_{R}'
$$
where the Dirac part is
$$
\mathcal L_{0,\delta}\widetilde f(\xi)=-i\xi |r(\xi)|^2 \widetilde f(\xi)-i\xi \overline{r(|\xi|)}s(|\xi|)\widetilde f(-\xi),
$$
the principal value part
$$
(\mathcal L_{0,p.v.}\widetilde f)_\rho(\xi)= - \sqrt{\frac 2 \pi} |\xi|\overline{r(|\xi|)} \sum_{\alpha=\pm} \int \widetilde f_\rho(\eta) \mathfrak a^{-\text{sgn}(\xi)}_\alpha
(\eta)\frac{\widehat \phi(\alpha \eta+\xi)}{\alpha \eta+\xi}d\eta 
$$
where $\phi$ is a Schwartz function with $\int \phi=1$, and where
\begin{equation} \label{id:coefficients-a-singular-cubic-term}
\mbox{for $\xi>0$,} \quad \left\{
\begin{array}{l}
\mathfrak{a}^{+}_+(\xi) = s(\xi) \\
\mathfrak{a}^{+}_-(\xi) = 0 \\
\mathfrak{a}^{-}_+(\xi) =  1\\
\mathfrak{a}^{-}_-(\xi) = r(\xi),
\end{array} \right.
\quad \quad \mbox{while for $\xi < 0$}, \quad \left\{
\begin{array}{l}
\mathfrak{a}^{+}_+(\xi) =  1\\
\mathfrak{a}^{+}_-(\xi) = r(-\xi) \\
\mathfrak{a}^{-}_+(\xi) = s(-\xi) \\
\mathfrak{a}^{-}_-(\xi) = 0,
\end{array} \right.
\end{equation}
and the regular part
$$
(\mathcal L_{R}'\widetilde f)_\rho(\xi)= \sum_{\epsilon=\pm } \int \widetilde f_\epsilon (\eta) \mathfrak l_{\epsilon \rho}(\xi,\eta)d\eta .
$$
for some symbols $(\mathfrak{l}_{\rho \epsilon})$ for $\rho,\epsilon=\pm $ that are smooth outside $\{\eta \xi=0\}$ and satisfy the estimates, for $\eta,\xi \neq 0$,
\begin{align}
\label{bd:mathfrakl-1}& \left| \partial_\xi^a \partial_\eta^b \mathfrak{l}_{\rho\epsilon}(\xi,\eta) \right| \lesssim_{a,b} \sum_{\pm} \frac{1}{\langle \xi \pm \eta \rangle^2} \\
\label{bd:mathfrakl-2}& \left| \partial_\xi^a \partial_\eta^b \frac{\mathfrak{l}_{\rho\epsilon}(\xi,\eta)}{\xi \eta } \right| \lesssim_{a,b} \frac{1}{\langle \eta \rangle \langle \xi \rangle}\sum_{\pm}  \frac{1}{\langle \xi \pm \eta \rangle^2} .
\end{align}
\end{lemma}

\begin{proof}

Integrating by parts, and then using the inverse distorted Fourier transform shows that
$$
\mathcal L_0\widetilde f(\xi)=\frac{1}{2\pi} \sum_\epsilon \epsilon \int \widetilde f_\epsilon(\eta) \langle \psi_\epsilon(\eta), \sigma_3( i\xi \psi_\rho(\xi) -\partial_x \psi_\rho (\xi) ) \rangle d\eta .
$$
We split according to singular and regular parts $\psi=\psi^S+\psi^R$. We further decompose $\psi^S$ into two pieces localized on the left and right half-lines
\begin{align}\nonumber
\psi_\lambda^S(x,\xi) & =  \chi_+ (x) \sum_{\substack{\alpha \in \{\pm  \} \\ r \in \{1,2 \}}} \mathfrak{a}^{\lambda, r, +}_\alpha (\xi) e^{i \alpha x \xi } e_r + \chi_- (x) \sum_{\substack{\alpha \in \{\pm  \} \\ r \in \{1,2 \}}} \mathfrak{a}^{\lambda, r, -}_\alpha (\xi) e^{i \alpha x \xi } e_r \\
\label{becasse} & = \chi_+ (x) \psi_\lambda^{S, +} (x, \xi) + \chi_- (x) \psi_\lambda^{S, -} (x, \xi) .
\end{align}
where
\begin{equation} \label{id:def-mathfraka-lambda-r-epsilon-alpha}
\mathfrak{a}^{\lambda,r,\epsilon}_\alpha(\xi) = \delta_{\lambda,r} \mathfrak{a}^{\epsilon}_\alpha(\xi),
\qquad \delta_{\lambda,r} = 
\left\{
\begin{array}{ll}
1 & \mbox{if $r=1$ and $\lambda = +$, or $r=2$ and $\lambda = -$} \\
0 & \mbox{otherwise},
\end{array}
\right.
\end{equation}
and $\mathfrak a^\epsilon_\alpha(\xi)$ is given by \eqref{id:coefficients-a-singular-cubic-term}. This gives
$$
\mathcal L_0 \widetilde f(\xi) = \frac{1}{2\pi}\sum_\epsilon \int \widetilde f_\epsilon(\eta)\mathfrak l^S_{\rho,\epsilon}(\xi,\eta)d\eta+ \frac{1}{2\pi} \sum_\epsilon \int \widetilde f_\epsilon(\eta)\mathfrak l^R_{\rho,\epsilon}(\xi,\eta)d\eta
$$
where the singular part is
\begin{align*}
\mathfrak l_{\rho,\epsilon}^S & = \epsilon   \langle  \chi_+ \psi_\epsilon^{S,+}(\eta),\chi_+\sigma_3 (i\xi \psi_\rho^{S,+}(\xi)-\partial_x \psi_\rho^{S,+}(\xi))\rangle\\ 
& \qquad + \epsilon \langle  \chi_- \psi_\epsilon^{S,-}(\eta),\chi_-\sigma_3 (i\xi \psi_\rho^{S,-}(\xi)-\partial_x \psi_\rho^{S,-}(\xi))\rangle
\end{align*}
and the regular part is (omitting the $\eta$ and $\xi $ dependence in the functions)
\begin{align*}
\mathfrak l_{\rho,\epsilon}^R & = \epsilon   \langle \psi_\epsilon^R,\sigma_3 (i\xi \psi_\rho^S -\partial_x \psi_\rho^S )\rangle+ \epsilon   \langle \psi_\epsilon^S , \sigma_3 (i\xi \psi_\rho^R -\partial_x \psi_\rho^R )\rangle + \epsilon   \langle \psi_\epsilon^R, \sigma_3 (i\xi \psi_\rho^R -\partial_x \psi_\rho^R )\rangle \\ 
& \qquad + \epsilon   \langle \chi_+ \psi_\epsilon^{S,+} ,\chi_-\sigma_3 (i\xi \psi_\rho^{S,-} -\partial_x \psi_\rho^{S,-} )\rangle+ \epsilon   \langle \chi_- \psi_\epsilon^{S,-} ,\chi_+\sigma_3 (i\xi \psi_\rho^{S,+} -\partial_x \psi_\rho^{S,+} )\rangle \\
& \qquad +\epsilon \langle \psi_\epsilon^S,-\chi_+'\psi_\rho^{S,+}-\chi'_-\psi_\rho^{S,-} \rangle.
\end{align*}

\medskip

\noindent \underline{Computation of the regular part $\mathfrak l^R_{\rho,\epsilon}$}. We use for the first term the identity
\begin{equation}
\label{grandharle}
i\xi \psi_\rho^{S,\pm}(\xi)-\partial_x \psi_\rho^{S,\pm}(\xi)=2i\xi \mathfrak{a}^{\pm}_-(\xi) e^{-ix\xi}e_{r(\rho)}
\end{equation}
with $r(\rho)=1$ if $\rho=+$ and $r(\rho)=2$ if $\rho =-$, and where $|\partial_\xi^l \mathfrak{a}_-^{\pm} (\xi) |\lesssim \langle \xi \rangle^{-1-l}$. Appealing in addition to the identities-\eqref{meringue}-\eqref{guimauve}~\eqref{carambar} and the vanishing at zero frequency as expressed in \eqref{chouette} and~\eqref{chouette2} shows that $\mathfrak l^R_{\rho,\epsilon}$ is of the form
$$
\mathfrak l_{\rho,\epsilon}^R(\eta,\xi) =\sum_{\lambda \mu} \int W_{\lambda \mu}(\eta,\xi,x)e^{i(\lambda \xi+\mu \eta)}dx
$$
for some functions $W_{\lambda \mu}$ satisfying $|\partial_x^k\partial_\eta^l\partial_\xi^m (\frac{W_{\lambda \mu}}{\eta\xi})|\lesssim \langle \eta \ra^{-1-l}\la \xi \ra^{-1-m}e^{-\beta |x|}$. Therefore, after integrating parts, one obtains that the symbol $\mathfrak l^R_{\rho,\epsilon}$ satisfies the estimates \eqref{bd:mathfrakl-1}-\eqref{bd:mathfrakl-2}.

\medskip

\noindent \underline{Computation of the singular part $\mathfrak l^S_{\rho,\epsilon}$}. An examination of the formula for $\psi^S$ shows that $\mathfrak l^S_{\rho,\epsilon}=0$ if $\epsilon \neq \rho$. If $\epsilon=\rho$, as $\sigma_3 \psi_\rho^{S,\pm}=\rho  \psi_\rho^{S,\pm}$ we get
\begin{align*}
\mathfrak l_{\rho,\rho}^S & =   \langle  \chi_+ \psi_\rho^{S,+}(\eta),\chi_+ (i\xi \psi_\rho^{S,+}(\xi)-\partial_x \psi_\rho^{S,+}(\xi))\rangle\\ 
& \qquad +  \langle  \chi_- \psi_\rho^{S,-}(\eta),\chi_-(i\xi \psi_\rho^{S,-}(\xi)-\partial_x \psi_\rho^{S,-}(\xi))\rangle
\end{align*}
By~\eqref{grandharle}, this becomes
$$
\mathfrak l_{\rho,\rho}^S  =  -2i\xi \overline{\mathfrak a_-^+(\xi)} \langle  \chi_+ \psi_\rho^{S,+}(\eta),\chi_+ e^{-ix\xi}e_{r(\rho)}\rangle-2i\xi \overline{\mathfrak a_-^-(\xi)} \langle  \chi_- \psi_\rho^{S,-}(\eta),\chi_- e^{-ix\xi}e_{r(\rho)}\rangle .
$$
By the formula \eqref{id:coefficients-a-singular-cubic-term} giving $\mathfrak a^\pm_-$, we further get
$$
\mathfrak l_{\rho,\rho}^S  =  -2i\xi \overline{r(|\xi|)} \langle  \chi_{-\text{sgn}(\xi)} \psi_\rho^{S,-\text{sgn}(\xi)}(\eta),\chi_{-\text{sgn}(\xi)} e^{-ix\xi}e_{r(\rho)}\rangle
$$
To compute the integral, we decompose by \eqref{becasse} and \eqref{id:def-mathfraka-lambda-r-epsilon-alpha}
$$
\psi^{S,-\text{sgn}(\xi)}_\rho(\eta)= \sum_{\alpha\in \{\pm\}} \mathfrak{a}^{-\text{sgn}(\xi)}_{\alpha} (\eta)e^{i\alpha x\eta} e_{r(\rho)}.
$$
Recalling the identity
\begin{equation} \label{fourier-transform}
\widehat{(\chi_\pm)^2} (\xi) = \sqrt{\frac{\pi}{2}} \delta (\xi) \pm \frac{\widehat{\phi} (\xi)}{i \xi}  +\widehat{\psi}  (\xi) ,
\end{equation}  
where $\phi$ and $\psi$ are even Schwartz functions and $\int \phi=1$ (see Section 4 in~\cite{GP} for a proof), we get
$$
\mathfrak l^S_{\rho,\rho}= - \sum_{\alpha} 2i\xi \overline{r(|\xi|)} \mathfrak a_\alpha^{\text{sgn}(\xi)} (\eta)\left(\pi \delta(\alpha\eta+\xi) +\sqrt{2\pi}\text{sgn}(\xi)\frac{\widehat \phi(\alpha \eta+\xi)}{i(\alpha\eta+\xi)}+\sqrt{2\pi}\widehat \psi(\alpha\eta+\xi)\right).
$$
The first two terms account for $\mathcal L_{0,\delta}$ and $\mathcal L_{0,\operatorname{p.v.}}$. As for the last term, it can be integrated into the regular symbol by defining eventually
$$
\mathfrak l_{\rho,\epsilon}(\eta,\xi)= \mathfrak l_{\rho,\epsilon}^R(\eta,\xi) -\sum_{\alpha} 2i\xi \overline{r(|\xi|)} a_\alpha^{\text{sgn}(\xi)} (\eta) \sqrt{2\pi}\widehat \psi(\alpha\eta+\xi).
$$
Indeed, it is obvious that the second term satisfies \eqref{bd:mathfrakl-1}. It also satisfies \eqref{bd:mathfrakl-2} because of the cancellations
$$
\lim_{\eta\uparrow 0} a^{\lambda}_+(\eta)+\lim_{\eta\uparrow 0} a^{\lambda}_-(\eta)=\lim_{\eta\downarrow 0} a^{\lambda}_+(\eta)+\lim_{\eta\downarrow 0} a^{\lambda}_-(\eta)=0
$$
for any $\lambda$, because of \eqref{id:coefficients-a-singular-cubic-term}.

\end{proof}

\begin{lemma}[Multiplication by a Schwartz function] 
\label{bartavelle2} If $V$ is in the Schwartz class, there exist smooth symbols $(\mathfrak{s}^V_{\rho \lambda})$ for $\rho,\lambda=\pm$ such that
$$
\widetilde{\mathcal{F}}_\rho(V f) (\xi) = \sum_{\lambda \in \{ \pm \}} \int \mathfrak{s}^V_{\rho \lambda}(\xi,\eta) \widetilde{f}_\lambda(\eta) \,d\eta.
$$
These symbols satisfy the estimates~\eqref{grandduc1} and~\eqref{grandduc2}.
\end{lemma}

\begin{proof} Arguing as in the proof of the previous lemma, we find that
$$
\mathfrak{s}^V_{\rho \lambda} (\xi,\eta) = \frac{\lambda}{\sqrt{2\pi}} \langle V \psi_\lambda(\eta) \,,\,  \psi_\rho(\xi) \rangle.
$$
The proof can then proceed similarly; details are omitted.
\end{proof}

\section{Nonlinear spectral distributions}

\label{sectionnonlinear}
\subsection{Quadratic spectral distributions}

\label{S-v-qsd}

\begin{proposition} \label{propmuquad}
For any $W \in \mathcal{S}$, there exists functions $(\mathfrak{m}^W_{jk\ell, \mu \nu \rho})_{\substack{j,k,\ell = 1,2 \\ \mu,\nu,\rho = \pm}}$ such that, if $f = (f_1, f_2)^\top$, $g= (g_1, g_2)^\top \in \mathcal{S}$ are such that $P_e f = f$ and $P_e g = g$,
\begin{equation} \label{S-v-qsd-eq1-1}
\widetilde{\mathcal{F}}_\rho (W f_j g_k e_\ell  ) (\xi) =  \sum_{ \lambda, \mu = \pm}  \int   \widetilde{f}_\lambda  (\eta)  \widetilde{g}_{\mu} (\sigma)
\mathfrak{m}_{jk\ell, \lambda \mu \rho}^W (\xi, \eta, \sigma)  \, d \eta \, d \sigma . 
\end{equation}
The functions $(\mathfrak{m}^W_{jk\ell, \mu \nu \rho})_{\substack{j,k,\ell = 1,2 \\ \mu,\nu,\rho = \pm}}$ are given by
\begin{equation} \label{accenteur-1}
\begin{split}
\mathfrak{m}^W_{jk\ell, \lambda \mu \rho}(\xi,\eta,\sigma) & =   \frac{ (-1)^{\ell-1}}{(2 \pi)^{3/2}} \lambda \mu \int W(x)   (\psi_\lambda (x, \eta) \cdot e_j)  (\psi_\mu (x, \sigma) \cdot e_k) (\overline{\psi_\rho (x, \xi)} \cdot e_\ell) \,dx.
\end{split}
\end{equation}
The symbol $\mathfrak m=\mathfrak{m}^W_{jk\ell, \mu \nu \rho}$ is smooth outside the set $\{\xi \eta \sigma=0\}$. It satisfies
\begin{equation}
\label{hibou1}
\begin{split}
& \left| \partial_\xi^a \partial_\eta^b \partial_\xi^c \mathfrak{m} \right|  \lesssim \sum_{\pm} \frac{1}{\langle \xi \pm \eta \pm \sigma \rangle^2}
\end{split}
\end{equation}
(where it is understood that all combinations of $\pm$ signs are allowed above). It vahishes for $\eta=0$ or $\sigma =0$, which can be captured by the following inequality, valid if $\xi \eta \sigma \neq 0$,
\begin{equation}
\label{hibou2}
\begin{split}
&  \left| \partial_\xi^a \partial_\eta^b \partial_\xi^c \frac{\mathfrak{m}}{\eta \sigma} \right|  \lesssim \frac{1}{\langle \eta \rangle \langle \sigma \rangle}  \sum_{\pm} \frac{1}{\langle \xi \pm \eta \pm \sigma \rangle^2}.
\end{split}
\end{equation}
\end{proposition}

\begin{proof}

Let $f,g,h \in \mathcal{S}$. By definition of the distorted Fourier transform, the formula for the inverse Fourier transform~\eqref{inverseFT}, and using that $\sigma_3 e_\ell = (-1)^{\ell-1} e_\ell$,
\begin{equation*}
\begin{split}
& \mathcal{\widetilde{F}}_\rho(W f_j g_k e_\ell) (\xi) = \frac{(-1)^{\ell-1}}{(2 \pi)^{3/2}} \sum_{\lambda,\mu} \lambda \mu  \int  W (x) (\overline{\psi_\rho (x, \xi)} \cdot e_\ell)\\
& \qquad \qquad \qquad \qquad \qquad \qquad  \left( \int \widetilde{f}_\lambda (\eta) \psi_\lambda (x, \eta) \cdot e_j \, d\eta \right) \left( \int   \widetilde{g}_\mu (\sigma) \psi_\mu (x, \sigma) \cdot e_k \, d\sigma \right) \,dx,
\end{split}
\end{equation*} 
which gives~\eqref{accenteur-1} after switching the order of integration (this is justified since all integrals are absolutely convergent).

\medskip

\noindent \underline{Proof of~\eqref{hibou1}} We now rely on the decomposition of Proposition~\ref{heroncendre} for $\psi_\rho$, $\psi_\lambda$, $\psi_\mu$, to expand the above formula. This gives a variety of different terms, which can be written
$$
\int W(x)  \psi^A(x,\xi) \psi^B(x,\eta) \psi^C(x,\sigma) \,dx
$$
with $A,B,C$ either $S$, or $R$ (we omit various indices to alleviate notations). Writing $\psi^A(x,\xi) = e^{\pm ix\xi} \theta^A(x,\xi)$, the bounds in Proposition~\ref{heroncendre} imply that
$$
|\partial_x \theta^A(x,\xi)| \lesssim 1 \quad \mbox{and} \quad |\partial_x^2 \theta^A(x,\xi)| \lesssim 1 .
$$
For $| \xi \pm \eta \pm \sigma|\leq 1$, the bound
$$
\left|\int W(x)  \psi^A(x,\xi) \psi^B(x,\eta) \psi^C(x,\sigma) \,dx \right|\lesssim 1
$$
is immediate as $W$ is Schwartz. For $| \xi \pm \eta \pm \sigma|\geq 1$, integrating by parts twice, one obtains
\begin{align*}
& \left| \int W(x) e^{i x (\xi \pm \eta \pm \sigma)} \theta^A(x,\xi) \theta^B(x,\eta) \theta^C(x,\sigma) \,dx \right| \\
& \qquad \qquad \lesssim \frac{1}{| \xi \pm \eta \pm \sigma|^2} \int |\partial_x^2 (W(x) \theta^A(x,\xi) \theta^B(x,\eta) \theta^C(x,\sigma))| \,dx \lesssim \frac{1}{| \xi \pm \eta \pm \sigma|^2}.
\end{align*}
Combining, we get the inequality~\eqref{hibou1} when $a=b=c=0$; the case where $a,b,c$ are not zero can be treated identically.

\medskip

\noindent \underline{Proof of~\eqref{hibou2}} In order to prove~\eqref{hibou2}, we need to examine carefully the cancellation of the generalized eigenvectors at zero frequency. The quantity that we want to control is
$$
\int W(x)  \psi(x,\xi) \frac{\psi(x,\eta)}{\eta} \frac{\psi(x,\sigma)}{\sigma} \,dx,
$$
(once again dropping indices), and it suffices here to appeal to~\eqref{chouette} to bound $\frac{\psi(x,\eta)}{\eta}$. With this estimate, the desired bound~\eqref{hibou2} follows by the same argument used to prove~\eqref{hibou1} (notice that powers of $\langle x \rangle$ are not detrimental, since they can be absorbed by the rapidly decaying factor $W(x)$).
\end{proof}

We will also need to characterize quadratic spectral distributions for scalar-valued (as opposed to function-valued) bilinear operators. This is the content of the following proposition, which can be proved in an almost identical way to the above.

\begin{proposition} \label{propmuquadscal}
For any $W \in \mathcal{S}$, there exists functions $(\mathfrak{m}^W_{jk, \mu \nu})_{\substack{j,k = 1,2 \\ \mu,\nu = \pm}}$ such that, if $f = (f_1, f_2)^\top$, $g= (g_1, g_2)^\top \in \mathcal{S}$ are such that $P_e f = f$ and $P_e g = g$,
\begin{equation} \label{S-v-qsd-eq1}
\int W f_j g_k \,dx=  \sum_{ \lambda, \mu = \pm}  \int   \widetilde{f}_\lambda  (\eta)  \widetilde{g}_{\mu} (\sigma)
\mathfrak{m}_{jk, \lambda \mu}^W (\eta, \sigma)  \, d \eta \, d \sigma . 
\end{equation}
The functions $(\mathfrak{m}^W_{jk\ell, \mu \nu \rho})_{\substack{j,k,\ell = 1,2 \\ \mu,\nu,\rho = \pm}}$ are given by
\begin{equation} \label{accenteur}
\begin{split}
\mathfrak{m}^W_{jk\ell, \lambda \mu \rho}(\xi,\eta,\sigma) & =\frac{(-1)^{\ell-1}}{(2 \pi)^{3/2}}   \lambda \mu \int W(x)   (\psi_\lambda (x, \eta) \cdot e_j)  (\psi_\mu (x, \sigma) \cdot e_k) \,dx.
\end{split}
\end{equation}
The symbol $\mathfrak m=\mathfrak{m}^W_{jk\ell, \lambda \mu \rho}(\xi,\eta,\sigma) $ is smooth outside the set $\{\xi\eta=0\}$. It satisfies, for all $b,c$, and $\eta \sigma \neq 0$,
\begin{equation}
\label{hibou3}
\left|  \partial_\eta^b \partial_\xi^c \mathfrak{m}(\eta,\sigma) \right|  \lesssim \sum_{\pm} \frac{1}{\langle \eta  \pm \sigma \rangle^2}  .
\end{equation}
It vahishes for $\eta=0$ or $\sigma =0$, which can be captured by the following inequality, if $\eta \sigma \neq 0$,
\begin{equation}
\label{hibou4}
\begin{split}
&  \left|  \partial_\eta^b \partial_\xi^c \frac{\mathfrak{m}(\eta,\sigma)}{\eta \sigma} \right|  \lesssim \frac{1}{\langle \eta \rangle \langle \sigma \rangle}  \sum_{\pm} \frac{1}{\langle \eta  \pm \sigma \rangle^2} .
\end{split}
\end{equation}
\end{proposition}

\subsection{Cubic spectral distribution}
\label{S-v-csd}

\begin{proposition}[Structure of the cubic spectral distribution] \label{gobemouche}
Let $W(x) $ be a smooth function, such that $W(x) - L$ and its derivatives decay super-polynomially as $x \rightarrow \pm \infty$ for a constant $L$. There exist distributions $(\mathfrak{m}^W_{jk\ell m, \lambda \mu \nu \rho})_{\substack{j,k,\ell,m=1,2 \\ \lambda,\mu,\nu,\rho = \pm}}$ such that, if $f,g,h \in \mathcal{S}$, then
\begin{equation} \label{S-v-csd-eq1}
\widetilde{\mathcal{F}}_\rho (W  f_j g_k h_\ell e_m  ) (\xi) =  \sum_{\lambda,\mu,\nu=\pm}  \iiint   \widetilde{f}_{\lambda}  (\eta)  \widetilde{g}_{\mu} (\sigma) \widetilde{h}_{\nu} (\zeta)
\mathfrak{m}^W_{jk\ell m, \lambda \mu \nu \rho} (\xi, \eta, \sigma, \zeta)  \, d \eta \, d \sigma\, d \zeta . 
\end{equation}
Formally, the distribution $\mathfrak{m}^W_{jk\ell m, \lambda \mu \nu \rho}$ is given by the formula
\begin{equation} \label{albatros}
\begin{split}
& \mathfrak{m}^W_{jk\ell m, \lambda \mu \nu \rho}  (\xi, \eta, \sigma, \zeta) \\
& \qquad =  \frac{(-1)^{m-1}}{4 \pi^{2}} \lambda \mu \nu \int  W(x)
(\psi_\lambda(x, \eta) \cdot e_j) (\psi_\mu(x, \sigma) \cdot e_k) (\psi_\nu(x, \zeta) \cdot e_\ell)   (\overline{\psi_\rho(x,\xi)} \cdot e_m) \, dx  .
\end{split}
\end{equation}
It can be decomposed into
\begin{equation} \label{S-v-csd-eq3}
\mathfrak{m}^W_{jk\ell m, \lambda \mu \nu \rho} (\xi, \eta, \sigma, \zeta) =  \mathfrak{m}^S_{jk\ell m, \lambda \mu \nu \rho} (\xi, \eta, \sigma, \zeta)  + \mathfrak{m}^R_{jk\ell m, \lambda \mu \nu \rho} (\xi, \eta, \sigma, \zeta)  
\end{equation}
where
\begin{itemize}
\item[(i)] The singular part $\mathfrak{m}^S_{jk\ell m, \lambda \mu \nu \rho} (\xi, \eta, \sigma, \zeta) $ can be written
\begin{equation} \label{beccroise}
\begin{split}
\mathfrak{m}^{S}_{jk\ell m, \lambda \mu \nu \rho} (\xi, \eta, \sigma, \zeta) 
& = L \l \m \n\sum_{\epsilon = \pm}  \sum_{\alpha, \beta, \gamma, \delta = \pm }  a^\epsilon_{\substack{ jk\ell m, \lambda \mu \nu \rho \\ \alpha \beta \gamma \delta}} (\xi, \eta, \sigma, 
\zeta) \left[ \frac{1}{4\pi} \delta(p) +\frac{1}{(2\pi)^{\frac 32}} \epsilon  \operatorname{p.v.} \frac{\widehat{\phi} (p)}{ip}  \right] , \\
&  p = \alpha \xi -  \beta \eta -   \gamma \sigma - \delta \zeta .
\end{split}
\end{equation}
Here, $\phi$ is a smooth, even, real-valued, compactly supported function with integral $1$, and the coefficients are given by
\begin{equation} \label{S-v-csd-eq6}
a^ \epsilon_{\substack{ jk\ell m, \lambda \mu \nu \rho \\ \alpha \beta \gamma \delta} }  (\xi, \eta, \sigma, \zeta ) = (-1)^{m-1}
\mathfrak{a}^{\lambda, j,  \epsilon}_\beta (\eta)
\mathfrak{a}^{\mu, k,  \epsilon}_\gamma  (\sigma)
\mathfrak{a}^{\nu, \ell,  \epsilon}_\delta (\zeta)
\overline{\mathfrak{a}^{\rho,m, \epsilon}_\alpha (\xi)}
\end{equation}
where $\mathfrak{a}^{\lambda,r,\epsilon}_\alpha(\xi)$ is defined in \eqref{id:def-mathfraka-lambda-r-epsilon-alpha}.

\item[(ii)] The symbol of the regular part $\mathfrak{m} = \mathfrak{m}^R_{jk\ell m, \lambda \mu \nu \rho} (\xi, \eta, \sigma, \zeta)  $ is smooth outside the set $\{\xi \eta \sigma \zeta=0\}$ and satisfies the inequality 
\begin{align*}
&|\partial_\xi^a \partial_\eta^b \partial_\sigma^c \partial_\zeta^d \mathfrak{m} (\xi,\eta,\sigma,\zeta)| + \langle \eta \rangle \langle \sigma \rangle \langle \zeta \rangle  \left| \partial_\xi^a \partial_\eta^b \partial_\sigma^c \partial_\zeta^d \frac{1}{\eta \sigma \zeta} \mathfrak{m}(\xi,\eta,\sigma,\zeta) \right| \\
& \qquad \qquad \lesssim \sum \frac{1}{\langle \xi \pm \eta \pm \sigma \pm \zeta \rangle^2} 
\end{align*}
where the sum above is over all combinations of $\{\pm\}$.
\end{itemize}
\end{proposition}

\begin{proof}
By definition of the distorted Fourier transform, the formula for the inverse Fourier transform~\eqref{inverseFT}, and using that $\sigma_3 e_\ell = (-1)^{\ell-1} e_\ell$,
\begin{equation*}
\begin{split}
& \mathcal{\widetilde{F}}_\rho(W f_j g_k h_\ell e_m) (\xi) \\
& \qquad = \frac{(-1)^{m-1}}{4 \pi^2} \sum_{\lambda,\mu, \nu}  \int  W (x) (\overline{\psi_\rho (x, \xi)} \cdot e_\ell) \left(\lambda \int \widetilde{f}_\lambda (\eta) \psi_\lambda (x, \eta) \cdot e_j \, d\eta \right) \\
& \qquad \qquad \qquad  \qquad \qquad \left(\mu \int   \widetilde{g}_\mu (\sigma) \psi_\mu (x, \sigma) \cdot e_k \, d\sigma \right)\left(\nu \int   \widetilde{h}_\nu (\zeta) \psi_\mu (x, \zeta) \cdot e_\ell \, d\zeta \right)  \,dx.
\end{split}
\end{equation*} 
In order to exchange the order of integration, we add a smooth cutoff in $x$, making all integrals absolutely convergent: for $w$ a smooth cutoff function, the above is
\begin{equation}
\begin{split}
\dots &= \lim_{R\to \infty} \frac{(-1)^{m-1}}{4 \pi^2} \sum_{\lambda,\mu, \nu} \l\m\n \int  W (x) w \left(\frac{x}{R} \right) (\overline{\psi_\rho (x, \xi)} \cdot e_m) \left( \int \widetilde{f}_\lambda (\eta) \psi_\lambda (x, \eta) \cdot e_j \, d\eta \right) \\
& \qquad \qquad \qquad  \qquad \qquad \left( \int   \widetilde{g}_\mu (\sigma) \psi_\mu (x, \sigma) \cdot e_k \, d\sigma \right)\left( \int   \widetilde{h}_\nu (\zeta) \psi_\nu (x, \zeta) \cdot e_\ell \, d\zeta \right)  \,dx. \\
& = \lim_{R\to \infty} \frac{(-1)^{m-1}}{4 \pi^2} \sum_{\lambda,\mu, \nu} \l\m\n \int \widetilde{f}_\lambda (\eta) \widetilde{g}_\mu (\sigma)  \widetilde{h}_\nu (\zeta) \\
& \qquad \left[ \int W (x) w \left(\frac{x}{R} \right) (\overline{\psi_\rho (x, \xi)} \cdot e_m) ( \psi_\lambda (x, \eta) \cdot e_j) (\psi_\mu (x, \sigma) \cdot e_k) (\psi_\nu (x, \zeta) \cdot e_\ell) \,dx \right] \,d\eta \,d\sigma \, d\zeta.
\end{split}
\end{equation}
This term between brackets converges in the sense of distributions as $R \to \infty$, giving the formula~\eqref{albatros}. We now split the generalized eigenfunctions $\psi$ into their singular and regular parts $\psi^S$ and $\psi^R$ as in~\eqref{alouette}, and split further $\psi^S$ into two pieces localized on the left and right half-lines according to \eqref{becasse} $ \psi_\lambda^S(x,\xi)   = \chi_+ (x) \psi_\lambda^{S, +} (x, \xi) + \chi_- (x) \psi_\lambda^{S, -} (x, \xi) $.

We can now decompose $\mathfrak{m}^W_{jk\ell m,\lambda \mu \nu \rho}$ into a singular and a regular part. The singular part consists of interactions between the singular parts of the generalized eigenfunctions, which furthermore are all concentrated either on the left ($\psi_\lambda^{S,-}$), or on the right half-line ($\psi_\lambda^{S,+}$). Finally, when focusing on either of these half-lines, we can replace $W$ by $L$, the remainder being regular. This leads to the formula
\begin{align*}
& \mathfrak{m}^S_{jk\ell m,\lambda \mu \nu \rho}(\xi,\eta,\sigma,\zeta) \\
& = \frac{(-1)^{m-1}}{4\pi^2} L \lambda \mu \nu \sum_\epsilon \int  (\chi_\epsilon(x))^4 (\overline{\psi_\rho^{S,\epsilon} (x, \xi)} \cdot e_m) ( \psi_\lambda^{S,\epsilon} (x, \eta) \cdot e_j) (\psi_\mu^{S,\epsilon} (x, \sigma) \cdot e_k) (\psi_\nu^{S,\epsilon} (x, \zeta) \cdot e_\ell) \,dx.
\end{align*}

From the formula giving $\psi^{S,\epsilon}_\lambda$, it follows that
\begin{equation} \label{S-v-csd-eq23}
\begin{split} 
\mathfrak{m}_{jk\ell m, \lambda \mu \nu \rho}^{S}
& = \frac{(-1)^{m-1}}{(2\pi)^{\frac 32}}  L \l\m\n \sum_\epsilon \sum_{ \alpha, \beta, \gamma, \delta  }  \widehat{\mathcal{F}} ( (\chi_\epsilon)^4) ( \alpha\xi -   \beta \eta -  \gamma \sigma -\delta \zeta ) a^\epsilon_{\substack{ jk\ell m, \lambda \mu \nu \rho\\ \alpha \beta \gamma \delta} }   (\xi, \eta, \sigma, \zeta ) .
\end{split}
\end{equation}

Finally, formula~\eqref{beccroise} follows from the identity \eqref{fourier-transform} (slightly abusing the notation to ease it by still using $\phi$ and $\psi$) that gives
$$
\widehat{(\chi_\pm)^4} (\xi) = \sqrt{\frac{\pi}{2}} \delta (\xi) \pm \frac{\widehat{\phi} (\xi)}{i \xi}  + \widehat{\psi}  (\xi) .
$$

There remains to discuss the regular part. It consists of all the terms appearing in $\mathfrak{m}^W_{jk\ell m,\lambda \mu \nu \rho}$ which were not kept in $\mathfrak{m}^S_{jk\ell m,\lambda \mu \nu \rho}$, namely
\begin{align*}
&\l\m\n (-1)^{m-1}\mathfrak{m}^R_{jk\ell m,\lambda \mu \nu \rho}(\xi,\eta,\sigma,\zeta) \\
&  =  \sum_{\substack{A, B, C, D \in  \{ S, R \} \\ \{ A,B,C,D \} \neq \{S,S,S,S\} }} \int  (\psi_\lambda^A (x, \eta)  \cdot e_j) (\psi_\mu^B (x, \sigma) \cdot e_k) (\psi_\nu^C (x, \zeta) \cdot e_\ell)  (\overline{ \psi_\rho^D (x, \xi)} \cdot e_m) dx  \\
& + \sum_{\substack{\epsilon_1,\epsilon_2,\epsilon_3,\epsilon_4 \in  \{ \pm \} \\ \mbox{not all $\epsilon_i$ equal} }}  \int \chi_{\epsilon_1} (x) \chi_{\epsilon_2} (x) \chi_{\epsilon_3} (x) \chi_{\epsilon_4} (x) (\psi_\lambda^A (x, \eta)  \cdot e_j) (\psi_\mu^B (x, \sigma) \cdot e_k) (\psi_\nu^C (x, \zeta) \cdot e_\ell)  (\overline{ \psi_\rho^D (x, \xi)} \cdot e_m) dx  \\
& + \sum_\epsilon \int  (\chi_\epsilon(x))^4 (W(x) - L ) (\overline{\psi_\rho^{S,\epsilon} (x, \xi)} \cdot e_m) ( \psi_\lambda^{S,\epsilon} (x, \eta) \cdot e_j) (\psi_\mu^{S,\epsilon} (x, \sigma) \cdot e_k) (\psi_\nu^{S,\epsilon} (x, \zeta) \cdot e_\ell) \,dx.
\end{align*}
This term can be treated like the quadratic spectral distribution; details are omitted.
\end{proof}

\section{Linear estimates}

\label{sectionlinear}

We prove here local and global decay estimates for the vector Schr\"odinger operator $\mathcal{H} = \mathcal{H}_0 + V$, and investigate the effect of conjugating symmetry generators by the linear group.  This is done under the assumptions~\eqref{H1}, \eqref{H2} and \eqref{H3}.

\subsection{Global decay estimates} Projecting a linear solution on the continuous spectrum, it enjoys the expected decay $t^{-1/2}$, as follows from the following proposition, which makes precise which norms of the data have to be controlled.

\begin{lemma} \label{aigrette}
For a symbol $\mathfrak{a}(\xi)$ such that
$$
\left|\partial_\xi^{k} \mathfrak{a}(\xi) \right|\lesssim_k \langle \xi \rangle^{-k},
$$
and under the assumptions~\eqref{H1}, \eqref{H2} and \eqref{H3}, there holds
\begin{equation}\label{bruantjaune} 
{\|e^{it \mathcal{H}} P_e h \|}_{L^\infty} + {\|\mathcal{W}_\rho \mathfrak{a}(\widetilde{D})e^{it \mathcal{H}}  P_e h \|}_{L^\infty} \lesssim \min \left( t^{-1/2} {\|  \widetilde{h} \|}_{L^\infty} + t^{-3/4} {\|   \partial_\xi \widetilde{h} \|}_{L^2}, \| \widetilde{h} \|_{L^1} \right)
\end{equation}
(we recall that the wave operator is defined by $\mathcal{W}_\rho = \widehat{ \mathcal{F}}^{-1} \widetilde{ \mathcal{F}}_\rho$).
\end{lemma}


\begin{proof}
The expressions for $e^{it \mathcal{H}}  h$ and $\mathcal{W}_\rho \mathfrak{a}(\widetilde{D}) e^{it \mathcal{H}}  h$ are similar:
\begin{align*}
& e^{it \mathcal{H}} h = \frac{1}{\sqrt{2 \pi}} \int  e^{it (1+\xi^2)}  \widetilde{h}_+ (\xi) \psi_+(x,\xi) \,d\xi - \frac{1}{\sqrt{2 \pi}} \int  e^{-it (1+ \xi^2)} \widetilde{h}_- (\xi)  \psi_-(x,\xi) \,d\xi  \\
& \mathcal{W}_\rho \mathfrak{a}(\widetilde{D}) e^{it \mathcal{H}} h = \frac{1}{\sqrt{2 \pi}} \int  e^{- it \rho (1+\xi^2)} \mathfrak{a}(\xi) \widetilde{h}_\rho (\xi) e^{ix\xi} e_1 \,d\xi .
\end{align*}

It is immediate that these expressions can be bounded pointwise by $\| \widetilde{h}\|_{L^1}$. The other bound in~\ref{bruantjaune} follows from Lemma~\ref{bruantroux} below. Indeed, this lemma can be applied to the two expressions above, choosing $g = \widetilde{f}$, and then decomposing $\psi_{\pm}(x,\xi)$ as in Lemma~\ref{heroncendre}, and choosing $a(x,\xi)$ and $X$ for each element of the decomposition appropriately.
\end{proof}

\begin{lemma}   \label{bruantroux} 
Consider a function $a(x, \xi)$ defined on $I \times \mathbb{R}_+$ and such that
\begin{equation}    \label{S-v-linear-est-lm-osc-integral-eq0}
|a(x, \xi)|+\la \xi \ra  |\partial_\xi a(x,\xi)| \lesssim 1 ,  \  \forall x \in I ,  \  \forall \xi \in \mathbb{R}_+ ,
\end{equation}
and for every $X \in \mathbb{R}$, consider the oscillatory integral
\begin{equation}
I(t, X, x) = \int_0^{+\infty} e^{it (\xi - X)^2 } a(x, \xi) g (\xi) d \xi , \ t >0, \ x \in I .
\end{equation}
It can be bounded uniformly in $X$, $t$ and $x$ by
\begin{equation}
| I(t, X, x) | \lesssim  t^{-1/2} \|g\|_{L^\infty} +  t^{-3/4} {\|  \partial_\xi g \|}_{L^2}.
\end{equation}
\end{lemma}

\begin{proof}The statement is a variant of the van der Corput lemma, which already appeared in~\cite{GPR}; we give the short proof here for completeness. Let us consider the case $X \geq 0$. We split
\begin{align*}
I(t, X, x) & = \left(  \int_{[X-t^{-1/2}, X+t^{-1/2}] \cap \mathbb{R}_+} + \int_{ \mathbb{R}_+ \setminus [X-t^{-1/2}, X+t^{-1/2}]} \right)  e^{it (\xi - X)^2 } a(x, \xi) g (\xi) d \xi \\
& = I_1 (t, X, x) + I_2 (t, X, x) .
\end{align*}
The first bound is immediate:
\begin{equation}
| I_1   | \lesssim  t^{-1/2} \|g \|_{L^\infty}  . 
\end{equation}
For $I_2$, we split again
\begin{equation}
I_2 (t, X, x) = \left(  \int_{X+t^{-1/2}}^{+\infty} + \int_0^{ X-t^{-1/2} }  \right)  e^{it (\xi - X)^2 } a(x, \xi) g (\xi) d \xi  = I_3 (t, X, x) + I_4 (t, X, x),
\end{equation}
with the convention that $I_4$ is defined only if $X \geq t^{-1/2}$. For $I_3$, we integrate by parts to obtain
\begin{equation*}
\begin{split}
|I_3|
& \lesssim t^{-1/2}  \|g\|_{L^\infty}  + \frac{1}{t}  \int_{X+t^{-1/2}}^{+\infty} \frac{1}{|\xi - X|} |\partial_\xi (a(x, \xi)  g (\xi) )| d \xi +  \frac{1}{t}  \int_{X+t^{-1/2}}^{+\infty} \frac{1}{|\xi - X|^2} | a(x, \xi) g (\xi) | d \xi  \\
& \lesssim t^{-1/2}  \|g\|_{L^\infty}  + \frac{1}{t}  \int_{X+t^{-1/2}}^{+\infty} \frac{1}{|\xi - X| \la \xi \ra} | g (\xi) | d \xi + \frac{1}{t}  \int_{X+t^{-1/2}}^{+\infty} \frac{1}{|\xi - X|} |\partial_\xi g (\xi) | d \xi  \\
& \qquad  \qquad +  \frac{1}{t}  \int_{X+t^{-1/2}}^{+\infty} \frac{1}{|\xi - X|^2} | a(x, \xi)  g (\xi) | d \xi.
\end{split}
\end{equation*}
By Cauchy-Schwarz, this is
\begin{equation*}
\begin{split}
& \lesssim t^{-1/2}  \|g \|_{L^\infty}  + \frac{1}{t} \left( \int_{X+t^{-1/2}}^{+\infty} \frac{1}{|\xi - X|^2}  d \xi \right)^{1/2}  \left( \int_{X+t^{-1/2}}^{+\infty} \frac{1}{ \la \xi \ra^2}  d \xi \right)^{1/2} \| g (\xi) \|_{L^\infty}   \\
& \qquad \qquad + \frac{1}{t} \left( \int_{X+t^{-1/2}}^{+\infty} \frac{1}{|\xi - X|^2  }  d \xi \right)^{1/2}   \|   \partial_\xi g (\xi) \|_{L^2} +  \frac{1}{t}  \int_{X+t^{-1/2}}^{+\infty} \frac{1}{|\xi - X|^2} d \xi  \|g(\xi) \|_{L^\infty}  \\
& \lesssim t^{-1/2} \|g \|_{L^\infty}  + t^{-3/4} \| g(\xi) \|_{L^\infty} + t^{-3/4}  \|   \partial_\xi  g (\xi) \|_{L^2}  + t^{-1/2}  \| g (\xi) \|_{L^\infty} .
\end{split}
\end{equation*}
Turning to $I_4$, if $X\leq 2 t^{-1/2}$, we have the crude estimate
\begin{equation}
|I_4| \lesssim  t^{-1/2}  \| g  \|_{L^\infty} . 
\end{equation}
If $X\geq 2 t^{-1/2}$, we decompose
\begin{equation}
|I_4 (t, X, x) | = \left(  \int_0^{t^{-1/2}} + \int_{t^{-1/2}}^{ X-t^{-1/2} }  \right)  e^{it (\xi - X)^2 } a(x, \xi) \widetilde{f} (\xi) \, d \xi = I_5 (t, X, x) + I_6 (t, X, x) .
\end{equation}
It is easy to see that $|I_5| \lesssim t^{-1/2} \|\widetilde{f} \|_{L^\infty}  $, and $I_6$ can be estimated in the same way as $I_3$.
\end{proof}

\subsection{Improved local decay estimates} Since we are assuming that the linearized problem does not  exhibit a resonance, the local decay is stronger than $t^{-1/2}$. Under the assumption that the distorted Fourier transform of the data is in $H^1$, the local decay rate becomes $t^{-1}$, as is stated in the following proposition.

\begin{lemma} 
\label{bergeronnette}
If $\widetilde{h}(0) =0$, there holds for any $t \neq 0$
$$
\left\| \frac{1}{\langle x \rangle} e^{it \mathcal{H}} P_e h \right\|_{L^\infty} \lesssim \frac{1}{t} \left[  \|\widetilde{h} \|_{L^2}+ \| \partial_\xi \widetilde{h} \|_{L^2} \right].
$$
\end{lemma}

\begin{proof} Recall that
$$
e^{it \mathcal{H}} h = \frac{1}{\sqrt{2 \pi}} \int_{\mathbb R}  e^{it (1+\xi^2)}  \widetilde{h}_+ (\xi) \psi_+(x,\xi) \,d\xi - \frac{1}{\sqrt{2 \pi}} \int_{\mathbb R}  e^{-it (1+ \xi^2)} \widetilde{h}_- (\xi)  \psi_-(x,\xi) \,d\xi .
$$
Further splitting the right-hand side between the contributions of $\xi>0$ and $\xi<0$, there are four summands, which can be treated very similarly. For the sake of concreteness, we will bound
$$
\frac{1}{\langle x \rangle} \int_0^\infty  e^{it (1+\xi^2)}  \widetilde{h}_+ (\xi) \psi_+(x,\xi) \,d\xi = \frac{1}{\langle x \rangle}\int_0^\infty  e^{it (1+\xi^2)}  \widetilde{h}_+ (\xi) \mathcal{F}_+(x,\xi) \,d\xi.
$$
Integrating by parts in $\xi$ using the identity $\frac{1}{2it\xi} \partial_\xi e^{it (1+\xi^2)} = e^{it (1+\xi^2)}$, and using that boundary terms cancel since $\widetilde{h}(0) = 0$ and $\mathcal{F}_+(x,0)=0$, this becomes
$$
-\frac{1}{\langle x \rangle} \int_0^\infty  \frac{e^{it (1+\xi^2)}}{2it\xi} \left[ \partial_\xi \widetilde{h}_+ (\xi) \mathcal{F}_+(x,\xi) + \widetilde{h}_+ (\xi) \partial_\xi  \mathcal{F}_+(x,\xi)\right] \,d\xi.
$$
We now bound $|\mathcal{F}_+(x,\xi)|$ using~\eqref{chouette}, and $|\partial_\xi \mathcal{F}_+(x,\xi)|$ by $\langle x \rangle$ (which follows by inspection of the formula \eqref{meringue} giving $\partial_\xi \mathcal{F}_+$). This gives for the above the bound
$$
\frac{1}{t} \int_0^\infty \left[  |\partial_\xi \widetilde{h}_+ (\xi)| \frac{ \min(\xi \langle x\rangle  , 1) }{\xi \langle x \rangle} + \frac{|\widetilde{h}_+ (\xi)|}{\xi} \right] \,d\xi \lesssim \frac{1}{t} \left[ \| \partial_\xi \widetilde{h}_+ \|_{L^2} + \| \widetilde{h}_+ \|_{L^2} \right].
$$
This is the desired estimate!
\end{proof}

\subsection{Conjugating symmetry generators by the linear group}

A basic idea of the approach pursued in the present article is to filter the nonlinear equation by the linear evolution group. When modulating along the manifold of solitary waves, this leads naturally to conjugating the generators of the symmetries of the equation by the linear group. The following proposition extracts the leading order contribution for such conjugation operations.

\begin{proposition} \label{propsigmadx}
Assume $f=P_ef$. Then
$$
\widetilde{\mathcal F} e^{-it\mathcal H}\sigma_3 e^{it\mathcal H}f= \sigma_3 \widetilde f+e^{-it(\omega+\xi^2)} \mathcal L_{R}e^{it(\omega+\xi^2)}\widetilde f 
$$
and
$$
\widetilde{\mathcal F} e^{-it\mathcal H} \partial_x e^{it\mathcal H}f= i\xi \widetilde f+e^{-it(\omega+\xi^2)} \mathcal L_{0}e^{it(\omega+\xi^2)}\widetilde f \\
$$
where the conjugated lower order operators $e^{-it(\omega+\xi^2)} \mathcal L_{R}e^{it(\omega+\xi^2)} $ and $e^{-it(\omega+\xi^2)} \mathcal L_{0}e^{it(\omega+\xi^2)} $ are bounded from $L^2$ to $L^2$, from $L^{2,1}$ to $L^{2,1}$ and from $H^1\cap\{\widetilde f(0)=0\}$ to $H^1\cap\{\widetilde f(0)=0\}$, uniformly in $t\in \mathbb R$.
\end{proposition}

\begin{proof} 
We have by Lemma \ref{lem:symmetry-partialx}
$$
e^{-it(\omega+\xi^2)} \mathcal L_{0}e^{it(\omega+\xi^2)}\widetilde f = \mathcal L_{0,\delta}\widetilde f+e^{-it(\omega+\xi^2)} \mathcal L_{0,\operatorname{p.v.}}e^{it(\omega+\xi^2)}\widetilde f +e^{-it(\omega+\xi^2)} \mathcal L_{R}'e^{it(\omega+\xi^2)}\widetilde f .
$$
The boundedness of $\mathcal L_{0,\delta}$ in the desired spaces is immediate, since $\xi r$ is uniformly bounded thanks to \eqref{estimatesrs}.

\medskip

\noindent \underline{Boundedness for the regular part}. We have
$$
(e^{-it(\omega+\xi^2)} \mathcal L_{R}'e^{it(\omega+\xi^2)}\widetilde f)_\rho (\xi) =\sum_{\epsilon=\pm}\int e^{it[\epsilon (\omega+\eta^2)-\rho (1+\xi^2)]} \widetilde f_\epsilon (\eta) \mathfrak l(\eta,\xi)d\eta.
$$
where we simplified the notation of the symbol to ease notations. As $\mathfrak l(\eta,0)=0$, the above vanishes at zero frequency $\xi=0$. Since $|\mathfrak l(\xi,\eta)|\lesssim \sum_{\pm} \frac{1}{\langle \xi\pm \eta\ra^2}$, the above is bounded from $L^2$ to $L^2$ and $L^{2,1}$ to $L^{2,1}$ by the Young inequality for convolution. Taking derivatives in $\xi$ gives the expression 
\begin{equation} \label{pancake-souffles}
-2i\rho t \xi \sum_{\epsilon=\pm}\int e^{it[\epsilon (\omega+\eta^2)-\rho (1+\xi^2)]} \widetilde f_\epsilon (\eta) \mathfrak l(\eta,\xi)d\eta+\{\mbox{easier terms}\} .
\end{equation}
In the easier cases, the derivative hits the symbol $\mathfrak l$ what can be bounded as previously thanks to the Young inequality. The first term can be integrated by parts, equalling 
\begin{align*}
&\rho \epsilon \xi  \sum_{\epsilon=\pm}\int e^{it[\epsilon (\omega+\eta^2)-\rho (1+\xi^2)]} \partial_\eta \widetilde f_\epsilon (\eta) \frac{\mathfrak l(\eta,\xi)}{\eta} \,d\eta \\
& \qquad \qquad \qquad \qquad + \rho \epsilon \xi  \sum_{\epsilon=\pm}\int e^{it[\epsilon (\omega+\eta^2)-\rho (1+\xi^2)]} \widetilde f_\epsilon (\eta)  \partial_\eta \left(\frac{\mathfrak l(\eta,\xi)}{\eta} \right)\, d\eta
\end{align*}
(notice that there are no boundary terms since $\widetilde f(0)=0$). For $\widetilde{f} \in H^1\cap\{\widetilde f(0)=0\}$, both terms above can be bounded in $L^2$ uniformly in $t$, by using \eqref{bd:mathfrakl-2} and the same arguments used to prove $L^2$ and $L^{2,1}$ boundedness.

\medskip

\noindent \underline{Boundedness for the $\operatorname{p.v.}$ part}. 
Again, since $\xi r$ is uniformly bounded and vanishes at $\xi=0$, to show the desired bounds for this operator it suffices to show them for the expression
$$
\sum_\alpha \int e^{it\rho (\eta^2-\xi^2)}\widetilde f_\rho(\eta)\mathfrak a^{-\text{sgn}(\xi)}_\alpha(\eta)\frac{\widehat \phi(\alpha \eta+\xi)}{\alpha \eta +\xi}d\eta .
$$
The coefficient $a^{-\text{sgn}(\xi)}_\alpha(\eta)$ is bounded, and can be ignored as far as $L^p$ estimates are concerned. The convolution kernel $\frac{\widehat \phi(\alpha \eta+\xi)}{\alpha \eta +\xi}$ corresponds to a truncated Hilbert transform which is bounded on $L^2$. Due to the rapid decay of $\widehat \phi$, this operator acts like a local operator, and is thus also bounded on $L^{2,1}$.

Changing variables to $\widetilde \eta=\alpha \eta+\xi$, the above equals
$$
\sum_\alpha \int e^{it\rho (-2\xi \widetilde \eta+\widetilde \eta^2)}\widetilde f_\rho(\alpha(\widetilde \eta - \xi))\mathfrak a^{-\text{sgn}(\xi)}_\alpha(\alpha(\widetilde \eta - \xi))\frac{\widehat \phi(\widetilde \eta)}{\widetilde \eta}d\widetilde \eta,
$$
which, differentiated with respect to $\xi$, equals
$$
-2it\rho \sum_\alpha \int e^{it\rho (-2\xi \widetilde \eta+\widetilde \eta^2)}\widetilde f_\rho(\alpha(\widetilde \eta - \xi))\mathfrak a^{-\text{sgn}(\xi)}_\alpha(\alpha(\widetilde \eta - \xi))\widehat \phi(\widetilde \eta)d\widetilde \eta +\{\mbox{easier terms}\}.
$$
Notice that there are no boundary terms at $\xi=\tilde \eta$ thanks to the cancellation $\tilde f(0)=0$. In the easier terms, the derivative hits $\widetilde f_\rho(\alpha(\xi-\widetilde \eta))\mathfrak a^{-\text{sgn}(\xi)}_\alpha(\alpha(\xi-\widetilde \eta))$ which can be treated as previously for the $L^2\to L^2$ bound. To estimate the main term, we go back to original variables $\eta=\alpha(\tilde \eta-\xi)$ and sum over $\alpha$ to write it under the form
$$
-2it\rho \int e^{it\rho (\eta^2-\xi^2)}\widetilde f_\rho(\eta) \mathfrak s(\eta,\xi)d\eta
$$
with symbol $\mathfrak s(\eta,\xi)=\mathfrak a^{-\text{sgn}(\xi)}_+(\eta) \widehat \phi( \eta+\xi)+\mathfrak a^{-\text{sgn}(\xi)}_-(\eta) \widehat \phi( -\eta+\xi)$. This symbol satisfies $|\partial_\eta^j (\frac{\mathfrak s}{\eta})|\lesssim \frac{1}{\langle \eta\rangle}\sum_{\pm}\frac{1}{\langle \xi \pm  \eta\rangle}$ for $j=0,1$ because of the cancellation $\lim_{\eta\uparrow 0} a^{\lambda}_+(\eta)+\lim_{\eta\uparrow 0} a^{\lambda}_-(\eta)=\lim_{\eta\downarrow 0} a^{\lambda}_+(\eta)+\lim_{\eta\downarrow 0} a^{\lambda}_-(\eta)=0$ and as $\phi$ is Schwartz. One can therefore perform the same integration by parts to bound this term as for the previous term \eqref{pancake-souffles} appearing in the regular part.

 Thus, $e^{-it(\omega+\xi^2)} \mathcal L_{0,\operatorname{p.v.}}e^{it(\omega+\xi^2)}$ is bounded from $H^1\cap \{\widetilde f(0)=0\}$ to $H^1\cap \{\widetilde f(0)=0\}$, uniformly in $t\in \mathbb R$.
\medskip
Finally, the proof of the boudedness of $e^{-it(\omega+\xi^2)} \mathcal L_{R}e^{it(\omega+\xi^2)}$ is exactly as that of $e^{-it(\omega+\xi^2)} \mathcal L_{R}'e^{it(\omega+\xi^2)}$.

\end{proof}

\section{Writing the equations}

\label{sectionwriting}

\subsection{Decomposition into radiation and solitary wave}

\subsubsection{Modulation and linearization}

For $\gamma,p,y \in C^1([0,T],\mathbb R)$ and $\omega \in C^1([0,T],(0,\omega^*))$ we write
\begin{equation} \label{id:decomposition-v}
v (t,x)= e^{i(px-\gamma)}(\Phi_\omega + u(t))(x+y).
\end{equation}
Expanding in $u$, \eqref{NLS} becomes
\begin{equation} \label{panettone}
i \partial_t u - \partial_x^2 u +\omega u -  V_+ u - V_- \overline{u} =N(u,\overline{u})+\mod
\end{equation}
where
\begin{align}
\label{id:definition-Nu} N(u,\overline{u}) & = V_{++} u^2 + V_{--}  \overline{u}^2 + V_{+-} |u|^2 + V_{++-} |u|^2 u + V_{+--} |u|^2 \overline{u} + V_{+++} u^3 + V_{---} \overline{u}^3 \\
\nonumber& \qquad \qquad + \{ \mbox{higher order terms} \}  
\end{align}
and
\begin{align*}
\nonumber \mod &= (\omega -p^2-\dot \gamma)(\Phi_\omega+u)-i\dot \omega \partial_\omega \Phi_\omega +i(2p-\dot y)(\partial_x \Phi_\omega+\partial_x u)+\dot p (x\Phi_\omega+xu)
\end{align*}
with the "linear potentials" (writing $\Phi=\Phi_\omega$ below to ease notations) 
\begin{align*}
& V_+ = F'(\Phi^2) + F''(\Phi^2) \Phi^2 \\
& V_- = F''(\Phi^2) \Phi^2,
\end{align*}
the "quadratic potentials"
\begin{align*}
& V_{++} = F''(\Phi^2) \Phi + \frac{1}{2} F'''(\Phi^2) \Phi^3 \\
& V_{--} =  \frac{1}{2} F'''(\Phi^2) \Phi^3 \\
& V_{+-} = 2 F''(\Phi^2)\Phi + F'''(\Phi^2) \Phi^3
\end{align*}
and the "cubic potentials"
\begin{align*}
& V_{++-} = F''(\Phi^2) + 2 F'''(\Phi^2) \Phi^2 + \frac{1}{2} F''''(\Phi^2) \Phi^4\\
& V_{+--} =\frac 32 F'''(\Phi^2) \Phi^2 + \frac{1}{2} F''''(\Phi^2) \Phi^4 \\
& V_{+++} = \frac{1}{2} F'''(\Phi^2) \Phi^2 + \frac{1}{6} F''''(\Phi^2) \Phi^4 \\
& V_{---} =  \frac{1}{6} F''''(\Phi^2) \Phi^4.
\end{align*}
All these potentials decay rapidly, except for $V_{++-}$.

The parameters $\omega,\gamma,p,y$ will be chosen in Subsection \ref{subsec:modulation} so as to perform an orthogonal projection onto the family of solitons. These parameters give thus the \emph{closest soliton} $e^{i(p x-\gamma)}\Phi_\omega (x+y)$ and are time-dependent. Equation \eqref{panettone} is thus the linearization around this closest soliton, and will be used to compute the modulation parameters $\gamma,\omega,p,y$.

The next problem to deal with is that the linear operator $u \mapsto - \partial_x^2 u + \omega u -V_+ u - V_- \overline{u}$ is not complex-linear. In order to handle this difficulty, we turn the equation into a (complex) vector equation by setting
\begin{equation} \label{def:U}
U = \begin{pmatrix} U_1 \\ U_2 \end{pmatrix} = \begin{pmatrix} u \\ \overline{u} \end{pmatrix}.
\end{equation}
The new unknown $U$ satisfies the reality condition
\begin{equation}
\label{grandtetras}
\sigma_1 \overline{U} = U.
\end{equation}
The linearized equation around a time-dependent soliton \eqref{panettone} can now be written
\begin{equation} \label{NLS-vectorial-renormalised-moving}
i \partial_t U + \mathcal{H}_\omega U = \mathcal{N}(U) +\Mod  
\end{equation}
Above, the linearised operator is defined as
\begin{equation} \label{D-H}
\mathcal H_\omega = \mathcal{H}_0 + V_{\omega}, \qquad \mathcal{H}_0 = \begin{pmatrix}
-\partial_x^2 + \omega & 0  \\
0 &  \partial_x^2 - \omega
\end{pmatrix},
\end{equation}
where
$$
V_\omega = 
\begin{pmatrix}
V_1  & V_2
\\
-V_2 & -V_1
\end{pmatrix}, \qquad \begin{pmatrix} V_1 \\ V_2 \end{pmatrix} = - \begin{pmatrix} V_+ \\ V_- \end{pmatrix},
$$
and the nonlinear term $\mathcal{N}$ and the modulation term $\Mod$ are given by
\begin{align}
& \label{id:definition-mathcalNU} \mathcal{N}(U) = \begin{pmatrix} \mathcal{N}_1(U) \\ \mathcal{N}_2(U) \end{pmatrix} = \begin{pmatrix} N(U_1,U_2) \\ - N(U_2,{U_1}) \end{pmatrix}, \\
& \label{id:definition-Mod} \Mod = (\omega-p^2 - \dot \gamma)(\Xi_0+\sigma_3 U)-i\dot \omega \Xi_1  +i(2p-\dot y)(\Xi_2+\partial_x U)+\dot p(\Xi_3+x\sigma_3 U),
\end{align}
where $\Xi_j=\Xi_j[\omega]$.

\subsubsection{A second decomposition, linearizing halfway to the final soliton}

In order to control the radiation $u$, we will use the distorted Fourier transform associated to a time-independent soliton $\Phi_{\underline{\omega}}$, and to a time-independent momentum $\pu$ in the renormalisation. We thus pick two time-independent parameters $(\underline{\omega},\pu )\in (0,\omega^*)\times \mathbb R$ and define the change of unknown
\begin{equation}
\label{defUsouligne}
\Uu = e^{i(p-\pu)\sigma_3 x} U.
\end{equation}
This new unknown still satisfies the reality condition
$$
\sigma_1 \overline{\Uu} = \Uu,
$$
and it solves the following equation
\begin{equation} \label{NLS-vectorial-renormalised-halfway}
i \partial_t \Uu + \mathcal{H}_{\uo}\Uu  =   \mathcal{N}(\Uu) + \Mod_\Phi + \Mod_{\Uu}+\mathcal E
\end{equation}
(here and in the remainder of the text, we use the convention that underlined quantities depend on $\underline{\omega}$ while non-underlined quantities depend on $\omega$) where the two modulation terms are
\begin{align*}
& \Mod_\Phi =e^{i(p-\pu)\sigma_3 x}  \left( (\omega-p^2 - \dot \gamma) \Xi_0-i\dot \omega \Xi_1 +i(2p-\dot y) \Xi_2+\dot p \Xi_3 \right),\\
& \Mod_{\Uu}= \left[(\omega-p^2-\dot \gamma)+(2p-\dot y)(p-\pu)-(\pu-p)^2+\underline{\omega} - \omega \right]\sigma_3 \Uu +i[2p-\dot y+2(\pu-p)]\partial_x \Uu,
\end{align*}
and $\mathcal E$ contains the error in the linearization of the potential term and in the renormalisation of the nonlinear terms
\begin{align}
& \label{id:definition-error} \mathcal{E} = (V_{\underline{\omega}} -  e^{i(p-\pu)\sigma_3 x} V_\omega  e^{i(\pu-p)\sigma_3 x}) \Uu + e^{i(p-\pu)\sigma_3 x} \mathcal N( e^{i(\pu-p)\sigma_3 x} \Uu)  -\mathcal N(\Uu) .
\end{align}

The parameters $\uo$ and $\pu$ will correspond to the limits of $\omega$ and $p$ as $t\to \infty$. Thus, Equation \eqref{NLS-vectorial-renormalised-halfway} is the linearization halfway to the final soliton. It will be used to control the radiation $U$.

\subsubsection{Projection on the continuous and discrete spectrum of the linearized operator}

Following our convention that underlined quantities depend on $\underline{\omega}$, the spectral projectors associated to $\Hu $ are denoted $\Pud=P_d[\underline{\omega}] $ and $\Pue=P_e[\underline{\omega}] $. We decompose the remainder as
\begin{align*}
& \Uu=\Uud+\Uue, \qquad (\Uud,\Uue)=(\Pud \Uu,\Pue \Uu).
\end{align*}
By writing
$$
\Uud= \sum_{j=0}^3 a_j\underline{\Xi_{j}} 
$$
for constants $a_0,a_1,a_2,a_3\in\mathbb R$ depending on time, we further decompose:
\begin{equation} \label{livres}
\Uu=\Uue+\sum_{j=0}^3 a_j\underline{\Xi_{j}} .
\end{equation}
The profile $f$ can then be defined by projecting the solution on the essential spectrum, and filtering it by the linear group
$$
f =  e^{-it\mathcal{H}_{\underline{\omega}} } \Uue
$$
It is such that
$$
\sigma_1 \overline{f} = f.
$$
Indeed, applying successively the identity $\overline{\mathcal{H}_\omega} = \mathcal{H}_\omega$, the commutation relation~\eqref{S-v-eq5}, and the reality condition~\eqref{grandtetras}
$$
\sigma_1 \overline{f} = \sigma_1 \overline{e^{-it\mathcal{H}_{\underline{\omega}}} \Uue} = \sigma_1 e^{it\mathcal{H}_{\underline{\omega}}} \overline{\Uue} = e^{-it\mathcal{H}_{\underline{\omega}}} \sigma_1 \overline{\Uue} =  e^{-it\mathcal{H}_{\underline{\omega}}} \Uue = f.
$$
The profile $f$ satisfies 
\begin{equation} \label{equation-renormalisee-f} 
i \partial_t f  = e^{-it\mathcal{H}_{\underline{\omega}} } \Pue  \left(  \mathcal{N} + \Mod_\Phi + \Mod_{\Uu}+\mathcal E\right).
\end{equation}
The projected modulation of the soliton can be written as
\begin{align}
\label{projectedmodulationterm}
\Pue \Mod_\Phi &= \Pue \Big(  (\omega-p^2 - \dot \gamma)(e^{i(\pu-p)\sigma_3 x}  \Xi_0-\underline{\Xi_0})-i\dot \omega (e^{i(\pu-p)\sigma_3 x}  \Xi_1-\underline{\Xi_1}) \\
\nonumber &\qquad \qquad  +i(2p-\dot y))(e^{i(\pu-p)\sigma_3 x} \Xi_2-\underline{\Xi_2})+\dot p (e^{i(\pu-p)\sigma_3 x} \Xi_3-\underline{\Xi_3}) \Big).
\end{align}
Finally, the nonlinearity can be split into quadratic terms $\mathcal{Q}$ in $\underline{U_e}$, cubic terms $\mathcal{C}$ in $\Uue$, higher order terms $\mathcal{T}$ in $\Uue$, and remainder nonlinear terms $\mathcal{R}$ which involve the discrete spectrum:
$$
\Pue e^{-it\mathcal{H_{\underline{\omega}}}} \mathcal{N} = \mathcal{Q} + \mathcal{C} + \mathcal{T} + \mathcal{R}.
$$
The remainder nonlinear terms can be written
\begin{align*}
e^{it\mathcal{H_{\underline{\omega}}}} \mathcal{R} & = \sum_{j=0,1} \mathcal{Q} (\Uue,a_j \underline{\Xi_j})+\sum_{j,k=0,1} \mathcal{Q} (a_j\underline{\Xi_j},a_k \underline{\Xi_k}) \\
& \quad \quad + \sum_{j=0,1} \mathcal{C} (\Uue,\Uue,a_j\underline{\Xi_j})+\sum_{j,k=0,1} \mathcal{C}  (\Uue,a_j\underline{\Xi_j},a_k\underline{\Xi_k})+\sum_{j,k,l=0,1} \mathcal{C}  (a_j\underline{\Xi_j},a_k \underline{\Xi_k},a_l \underline{\Xi_l}) \\
& \quad \quad + \{ \mbox{higher order terms} \}.
\end{align*}

\subsection{Viewing the equation for the radiation on the Fourier side}
\label{TTDFT}

Taking the distorted Fourier transform with parameter $\underline{\omega}$, the Fourier transform of $f$ satisfies the following reality condition:
\begin{equation}
\label{fauconpelerin}
{\begin{pmatrix} \widetilde f_{\pm}(\xi) \\ \widetilde f_{\pm}(-\xi)\end{pmatrix}} = - S(-|\xi|) \overline{\begin{pmatrix} \widetilde f_{\mp}(-\xi) \\ \widetilde f_{\mp}(\xi)\end{pmatrix}}
\end{equation}
Indeed, by Lemma~\ref{gelinotte},
$$
{\begin{pmatrix} \widetilde f_{\pm}(\xi) \\ \widetilde f_{\pm}(-\xi)\end{pmatrix}} = {\begin{pmatrix} \widetilde{\mathcal{F}} (\sigma_1 \overline{f})_{\pm}(\xi) \\ \widetilde{\mathcal{F}}(\sigma_1 \overline{f})_{\pm}(-\xi)\end{pmatrix}} 
= - {\begin{pmatrix} \widetilde{\mathcal{F}} ( \overline{f})_{\mp}(\xi) \\ \widetilde{\mathcal{F}}( \overline{f})_{\mp}(-\xi)\end{pmatrix}} = - S(-|\xi|) \overline{\begin{pmatrix} \widetilde f_{\mp}(-\xi) \\ \widetilde f_{\mp}(\xi)\end{pmatrix}}.
$$

The equation \eqref{equation-renormalisee-f} for $f$ becomes, on the Fourier side,
\begin{equation}
\label{dtf}
i \partial_t \widetilde{f}_\rho = e^{- it \rho (\uo+ \xi^2)}\left( \widetilde{\mathcal{N}}+\widetilde{\Mod_\Phi}+\widetilde{\Mod_{\Uu}}+\widetilde{\mathcal{E}} \right).
\end{equation}
We shall split the nonlinear terms into
\begin{equation*}
 e^{-it\mathcal{H}} \mathcal{N} = \mathcal{Q}^R + \mathcal{C}^S + \mathcal{C}^R + \mathcal{T} + \mathcal{R}.
\end{equation*}
where
\begin{itemize}
\item The regular quadratic terms $\mathcal{Q}^R$ contain smooth quadratic spectral distributions.
\item The singular cubic terms $\mathcal{C}^S$ contain singular cubic spectral distributions.
\item The regular cubic terms $\mathcal{C}^R$ contain smooth cubic spectral distributions.
\item The remainder terms $\mathcal{T}$ are higher order.
\item The remainder terms $\mathcal{R}$ involve nonlinear interactions with the discrete spectrum.
\end{itemize}.

\bigskip
\noindent \underline{The regular quadratic terms $\mathcal{Q}^R$.} In frequency space, they can be written as
\begin{equation}
\label{defQR}
\widetilde{\mathcal{F}}_\rho \mathcal{Q}^R (\xi) = \sum \int e^{-it\Phi_{\lambda \mu  \rho }(\xi,\eta,\sigma)}  \widetilde{f}_{\lambda}  (\eta) \widetilde{f}_{\mu} (\sigma)
\mathfrak{m}^V_{jkl, \lambda \mu \rho } (\xi, \eta, \sigma) \,d \eta\, d \sigma
\end{equation}
with $V$ localized, and 
$$
\Phi_{\lambda \mu  \rho} =  \rho(\uo+ \xi^2) - \lambda  (\uo+ \eta^2)  - \mu  (\uo+ \sigma^2)   .
$$
We used a slightly imprecise notation above: the sum is taken over indices $\lambda,\mu,j,k,\ell$, with corresponding potentials $V$.

\bigskip
\noindent \underline{The singular cubic terms $\mathcal{C}^S$} originate in the cubic terms in $e^{-it\mathcal{H}} \mathcal{N}$ which do not include a rapidly decaying potential, namely
$$
e^{-it\mathcal{H}}  V_{++-} (U_1^2 U_2 e_1 - U_1 U_2^2 e_2),
$$
which, becomes in frequency space 
$$
e^{-i\rho t(\uo+\xi^2)} \widetilde{\mathcal{F}}_\rho  [ V_{++-} (U_1^2 U_2 e_1 - U_1 U_2^2 e_2)].
$$

By Proposition \ref{gobemouche}, choosing $W = V_{++-}$,
\begin{equation*} 
\begin{split}
& \widetilde{\mathcal{F}}_\rho (V_{++-} U_1^2 U_2  e_1  ) (\xi) =  \frac{1}{4 \pi^{2}} \sum_{\lambda, \mu, \nu = \pm}  \int   \widetilde{U}_{\lambda}  (\eta)  \widetilde{U}_{\mu} (\sigma) \widetilde{U}_{\nu} (\zeta)
\mathfrak{m}^W_{1121, \lambda \mu \nu \rho} (\xi, \eta, \sigma, \zeta) \,d \eta\, d \sigma \,d \zeta ,  \\
& \widetilde{\mathcal{F}}_\rho (V_{++-} U_1 U_2^2   e_2 ) (\xi) =  \frac{1}{4 \pi^{2}} \sum_{\lambda, \mu, \nu = \pm}  \int   \widetilde{U}_{\lambda}  (\eta)  \widetilde{U}_{\mu} (\sigma) \widetilde{U}_{\nu} (\zeta)
\mathfrak{m}^W_{2212, \lambda \mu \nu \rho} (\xi, \eta, \sigma, \zeta)\, d \eta \,d \sigma \,d \zeta ,
\end{split}
\end{equation*}
which can also be written
\begin{equation*} 
\begin{split}
& \widetilde{\mathcal{F}}_\rho (e^{-it\mathcal{H}}V_{++-}  U_1^2 U_2  e_1  ) (\xi) \\
& \qquad \qquad =  \frac{1}{4 \pi^{2}} \sum_{\lambda, \mu, \nu }  \int e^{-it\Phi_{\lambda \mu \nu \rho}(\xi,\eta,\sigma,\zeta)}  \widetilde{f}_{\lambda}  (\eta)  \widetilde{f}_{\mu} (\sigma) \widetilde{f}_{\nu} (\zeta)
\mathfrak{m}^W_{1121, \lambda \mu \nu \rho} (\xi, \eta, \sigma, \zeta) \,d \eta\, d \sigma \,d \zeta ,  \\
& \widetilde{\mathcal{F}}_\rho (e^{-it\mathcal{H}}V_{++-} U_1 U_2^2   e_2 ) (\xi) \\
& \qquad \qquad =  \frac{1}{4 \pi^{2}} \sum_{\lambda, \mu, \nu}  \int  e^{-it\Phi_{\lambda \mu \nu \rho}(\xi,\eta,\sigma,\zeta)} \widetilde{f}_{\lambda}  (\eta)  \widetilde{f}_{\mu} (\sigma) \widetilde{f}_{\nu} (\zeta)
\mathfrak{m}^W_{2212, \lambda \mu \nu \rho} (\xi, \eta, \sigma, \zeta) \, d \eta \,d \sigma \,d \zeta,
\end{split}
\end{equation*}
where
\begin{equation}
\label{fauconcrecerelle}
\Phi_{\lambda \mu \nu \rho} =  \rho(\uo+ \xi^2) - \lambda  (\uo+ \eta^2)  - \mu  (\uo+ \sigma^2)  - \nu  (\uo+ \zeta^2) .
\end{equation}

By formula~\eqref{beccroise}, the cubic spectral distribution has a nontrivial singular part for specific values of $\lambda,\mu,\nu$: the formula simplifies to become 
\begin{equation*}   
\begin{split}
& \widetilde{\mathcal{F}}_+ (e^{-it\mathcal{H}}V_{++-} (U_1^2 U_2  e_1 - U_2^2 U_1  e_2 )) (\xi) \\
& \qquad \qquad =  \frac{1}{4 \pi^{2}} \sum  \int e^{-it\Phi_{++-+}(\xi,\eta,\sigma,\zeta)}  \widetilde{f}_{+}  (\eta)  \widetilde{f}_{+} (\sigma) \widetilde{f}_{-} (\zeta)
\mathfrak{m}^S_{1121, ++-+} (\xi, \eta, \sigma, \zeta) \,d \eta\, d \sigma \,d \zeta ,  \\
& \qquad  \qquad \qquad \qquad + \{ \mbox{regular part} \} \\
& \widetilde{\mathcal{F}}_- (e^{-it\mathcal{H}}V_{++-} (U_1^2 U_2  e_1 - U_2^2 U_1  e_2 ) )_\rho (\xi) \\
& \qquad \qquad =  - \frac{1}{4 \pi^{2}} \sum  \int  e^{-it\Phi_{--+-}(\xi,\eta,\sigma,\zeta)} \widetilde{f}_{-}  (\eta)  \widetilde{f}_{-} (\sigma) \widetilde{f}_{+} (\zeta)
\mathfrak{m}^S_{2212,--+-} (\xi, \eta, \sigma, \zeta) \, d \eta \,d \sigma \,d \zeta \\
& \qquad  \qquad \qquad \qquad + \{ \mbox{regular part} \}.
\end{split}
\end{equation*}

The regular parts appearing in the sum above will be included in the regular cubic terms; as for the singular parts, they make up $\mathcal{C}^S$, which we can finally define as follows:
\begin{equation}
\label{fuligule}   
\begin{split}
& \widetilde{\mathcal{F}}_+ (\mathcal{C}^S) (\xi) =  \frac{1}{4 \pi^{2}} \sum  \int e^{-it\Phi_{++-+}(\xi,\eta,\sigma,\zeta)}  \widetilde{f}_{+}  (\eta)  \widetilde{f}_{+} (\sigma) \widetilde{f}_{-} (\zeta)
\mathfrak{m}^S_{1121, ++-+} (\xi, \eta, \sigma, \zeta) \,d \eta\, d \sigma \,d \zeta ,  \\
& \widetilde{\mathcal{F}}_- (\mathcal{C}^S)(\xi) = - \frac{1}{4 \pi^{2}} \sum  \int  e^{-it\Phi_{--+-}(\xi,\eta,\sigma,\zeta)} \widetilde{f}_{-}  (\eta)  \widetilde{f}_{-} (\sigma) \widetilde{f}_{+} (\zeta)
\mathfrak{m}^S_{2212,--+-} (\xi, \eta, \sigma, \zeta) \, d \eta \,d \sigma \,d \zeta .
\end{split}
\end{equation}

\bigskip
\noindent \underline{The regular cubic terms $\mathcal{C}^R$} can be written in frequency space as
$$
\widetilde{\mathcal{F}}_\rho \mathcal{C}^R (\xi) =  \int e^{-it\Phi_{\lambda \mu \nu \rho}(\xi,\eta,\sigma,\zeta)}  \widetilde{f}_{\lambda}  (\eta)  \widetilde{f}_{\mu} (\sigma) \widetilde{f}_{\nu} (\zeta)
\mathfrak{m}^V_{jklm, \lambda \mu \nu \rho} (\xi, \eta, \sigma, \zeta) \,d \eta\, d \sigma \,d \zeta 
$$
with $V$ localized.

\bigskip
\noindent \underline{The remainder terms $\mathcal{T}$} are order 4 and higher.

\subsection{Evolution of the modulation parameters} \label{subsec:modulation}

\subsubsection{Geometrical decomposition}

We shall fix the modulation parameters $\gamma $ and $\omega$ by imposing the orthogonality conditions
\begin{equation} \label{orthogonalities}
\langle U, \sigma_3\Xi_j\rangle =0, \qquad j=0,1,2,3.
\end{equation}
This condition determines $\gamma$, $\omega$, $p$, $y$ and $U$ thanks to the following lemma.

\begin{lemma} \label{lem:renormalisation}

For any $\omega_0 \in (0,\omega^*)$, for $W\in L^2\times L^2$ with $\sigma_1 \overline W=W$ and $\| W\|_{L^2\times L^2}$ small enough, there exist unique parameters $\gamma,p,y\in \mathbb R$ and $\omega \in (0,\omega^*)$ with 
$$|\gamma|+|p|+|y|+|\omega-\omega_0|\lesssim \| W\|_{L^2\times L^2}$$
such that
\begin{equation} \label{renormalisation-at-Phiomega0}
\begin{pmatrix} \Phi_{\omega_0}\\ \Phi_{\omega_0}\end{pmatrix}+W = \begin{pmatrix}  e^{i(px+\gamma)} ( \Phi_\omega +U_1) \\ e^{-i(px+\gamma)} ( \Phi_\omega +U_2) \end{pmatrix} (x+y)
\end{equation}
where $U$ satisfies \eqref{orthogonalities} and $\sigma_1 \overline U = U$. Moreover, $\gamma$, $p$, $y$ and $\omega$ are $\mathcal C^\infty$ functions of $W$ in $L^2\times L^2$.

\end{lemma}

\begin{proof}

Consider the map 
$$
\Theta:(W,\gamma,p,y,\omega)\mapsto (\langle U, \sigma_3 \Xi_j \rangle)_{j=1,2,3,4}
$$ 
where $U$ is given by \eqref{renormalisation-at-Phiomega0}. Denoting $(L^2)^2_{real}$ for the subspace of $(L^2)^2$ made up of functions $U$ satisfying the reality condition $\sigma_1 \overline U = U$, this map is $\mathcal C^\infty$ from $(L^2)^2_{real} \times \mathbb R^3 \times (0,\omega^*)$ into $\mathbb R^4$. Moreover, it satisfies, by a direct computation:
$$
\left. \frac{\partial \Theta}{\partial (\gamma,p,y,\omega)} \right|_{0,0,0,0,\omega_0}
=\begin{pmatrix} 0 & 0 & 0 & - c_{\omega_0} \\ 
- c_{\omega_0}& 0 & 0 & 0 \\
0 & \| \Phi_{\omega_0} \|_{L^2}^2 & 0 & 0 \\
0 & 0 & \| \Phi_{\omega_0} \|_{L^2}^2 & 0 \end{pmatrix}.
$$
Since $c_\omega = \frac{d}{d\omega}\int \Phi^2_\omega dx \neq 0$ by assumption, the result of Lemma \ref{lem:renormalisation} then follows from the implicit function Theorem.
\end{proof}

\begin{lemma} \label{lem:equivalence-projection-discrete-orthogonalites}

Any $U\in L^2$ satisfies the orthogonality conditions \eqref{orthogonalities} if and only if $U=P_eU$.

\end{lemma}

\begin{proof}

Assume $U=P_eU$. Then using \eqref{S-v-dft-eq16} and \eqref{defFT} we have for $j=0,1,2,3$,
$$
\la U, \sigma_3\Xi_j \ra = \frac{1}{\sqrt{2 \pi}} \sum_{\epsilon = \pm } \epsilon \int \widetilde f(\xi) \overline{\la \sigma_3 \Xi_j, \psi_\epsilon (\cdot, \xi) \ra} d \xi. 
$$
Using $  \mathcal H^2\psi_\epsilon=(\uo+\xi^2)^2\psi_\epsilon$ and $\mathcal H^{*2} \sigma_3 \Xi_j=0$ by Lemma \ref{kernel} we have for all $\xi$
\begin{equation} \label{thym citrone}
\la \sigma_3 \Xi_j, \psi_\epsilon (\cdot, \xi) \ra=(\uo+\xi^2)^{-2} \la \mathcal H^{*2} \sigma_3 \Xi_j, \psi_\epsilon (\cdot, \xi) \ra =0
\end{equation}
(the integration by parts being possible thanks to the rapid decay of $\Xi_j$). Hence $\la U, \sigma_3\Xi_j \ra =0$ so that $U$ satisfies \eqref{orthogonalities}.

Assume conversely that $U$ satisfies \eqref{orthogonalities} and decompose $U=P_dU+P_eU$. Since $P_eU$ satisfies \eqref{orthogonalities} by the above discussion, then $\langle P_d U,\sigma_3\Xi_j\rangle=0$ for $j=0,1$. We have $P_d U\in \mbox{Span}\{\Xi_j\}_{j=0,1,2,3}$ by Lemma \ref{kernel}. Since the matrix of the projections on the generalized kernel
\begin{equation}\label{def:Matrice-M0} 
M_0:=(\langle \Xi_{j},\sigma_3 \Xi_k\rangle)_{0\leq j,k\leq 3}=\begin{pmatrix} 0 & c_\omega &0&0 \\ c_\omega & 0 &0 &0 \\ 0&0&0&-\| \Phi\|_{L^2}^2 \\ 0&0&-\| \Phi\|_{L^2}^2 & 0 \end{pmatrix}
\end{equation}
is non-degenerate, we infer  $P_dU=0$. Hence $U=P_eU$ as desired.

\end{proof}

If $|p-\underline{p}|$  and $|\omega - \underline{\omega}$ are small, then on the one hand $U$ is close to $\underline{U}$, and on the other hand $\underline{P_d}$ is close to $P_d$. Therefore, we infer from $P_d U =0$ that $\underline{P_d} \underline{U}$ should be small; this is quantified in the following lemma.

\begin{lemma}

For $(p,\omega)$ close to $(\underline{p},\underline{\omega})$, for any $U\in L^2$ with $U=P_e U$, the coefficients $(a_j)$ in \eqref{livres} satisfy
\begin{equation}\label{decomposition-bd-coefficients-discrets}
\sum_{j=0}^3 |a_j| \lesssim \left[ |p - \underline{p}| + |\omega-\underline{\omega}| \right] \int |\Uue| e^{-\mu |x|}\,dx,
\end{equation}
where $\underline{U}$ is defined in~\eqref{defUsouligne}, and where $\mu$ is a positive constant.
\end{lemma}

\begin{proof}
We start from the identity
$$
\langle \sum a_j \underline{\Xi_j} , \sigma_3 \Xi_i \rangle = \langle \underline{U} - \underline{U_e} , \sigma_3 \Xi_i \rangle.
$$
On the one hand, the left-hand side can be written
\begin{align*}
\langle \sum a_j \underline{\Xi_j} , \sigma_3 \Xi_i \rangle & = \langle \sum a_j \Xi_j , \sigma_3 \Xi_i \rangle + \langle \sum a_j [ \underline{\Xi_j} - \underline{\Xi}_j ] , \sigma_3 \Xi_i \rangle \\
& = M_0 \left( a_0,a_1,a_2,a_3 \right)^\top + O(|\omega - \underline{\omega}| \sum |a_j|)
\end{align*}
since $|\Xi_j-\underline{\Xi_j}|\lesssim |\omega-\underline{\omega}|$. On the other hand, the right-hand side can be written, with the help of the orthogonality conditions $\langle U,\sigma_3 \Xi_j\rangle=0$ and $\langle \Uue,\sigma_3 \underline{\Xi_j} \rangle=0$  following from Lemma \ref{lem:equivalence-projection-discrete-orthogonalites}, 
\begin{align*}
\langle \underline{U} - \underline{U_e} , \sigma_3 \Xi_i \rangle & = \langle \underline{U} - U - \underline{U_e} , \sigma_3 \Xi_i \rangle \\
& = \langle (\underline{U_e} + \underline{U_d}) (1 - e^{i(\underline{p}-p) \sigma_3 x}) , \sigma_3 \Xi_i \rangle - \langle \underline{U_e} , \sigma_3 [ \Xi_i - \underline{\Xi_i} ] \rangle.
\end{align*}
Using that $|\Xi_j-\underline{\Xi_j}|\lesssim |\omega-\underline{\omega}|e^{-2\mu|x|}$, this can be bounded as follows
$$
\left \langle \underline{U} - \underline{U_e} , \sigma_3 \Xi_i \rangle  \right| \lesssim \left[ |p - \underline{p}| + |\omega-\underline{\omega}| \right] \int |\Uue| e^{-\mu |x|}\,dx + |p - \underline{p}|  \sum |a_j|.
$$
Combining the above identities, we obtain the inequality
$$
\left| M_0 \left( a_0,a_1,a_2,a_3 \right)^\top \right| \lesssim \left[ |\omega - \underline{\omega}| + |p-\underline{p}| \right] \left[ \sum |a_j| +  \int |\Uue| e^{-\mu |x|}\,dx \right],
$$
from which the desired bound follows by invertibility of $M_0$.
\end{proof}

\subsubsection{The modulation equations}

\begin{lemma}
\label{echasseblanche}
The parameters $\gamma,\omega,p,y $ are differentiable with time and are solutions of the dynamical system
\begin{align}
 \label{id:modulation-technical-identity} M \begin{pmatrix} \omega-p^2-\dot \gamma \\ -i\dot \omega \\ i(2p-\dot y) \\ \dot p \end{pmatrix} = N, \qquad \mbox{with} \qquad N = - \begin{pmatrix} \langle \mathcal N(U),\sigma_3 \Xi_0\ra \\ \langle \mathcal N(U),\sigma_3 \Xi_1\ra \\  \langle \mathcal N(U),\sigma_3 \Xi_2\ra \\  \langle \mathcal N(U),\sigma_3 \Xi_3\ra \end{pmatrix}
\end{align}
and the matrix $M$ can be decomposed as
$$
M=M_0+M_1(U), \quad M_1= \begin{pmatrix}\la U,\Xi_0\ra & -\langle U,\sigma_3 \partial_\omega \Xi_0\rangle & - \langle  U,\sigma_3 \partial_x \Xi_0 \rangle & \la U,\Xi_3 \ra \\ 
\la U,\Xi_1\ra & -\langle U,\sigma_3 \partial_\omega \Xi_1\rangle & -\langle U,\sigma_3 \partial_x \Xi_1\rangle & \la U,x \Xi_1 \ra \\
\la U,\Xi_2\ra & -\langle U,\sigma_3 \partial_\omega \Xi_2\rangle & -\langle U,\sigma_3 \partial_x \Xi_2\rangle & \la U, x \Xi_2 \ra\\
\la U,\Xi_3\ra & -\langle U,\sigma_3 \partial_\omega \Xi_3\rangle & -\langle U,\sigma_3 \partial_x \Xi_3\rangle & \la U, x\Xi_3 \ra \end{pmatrix}
$$
where $M_0$ is given by \eqref{def:Matrice-M0}.
\end{lemma}

\begin{proof}
We will prove this claim assuming $v_0\in \mathcal C^\infty_c(\mathbb R)$ what will justify all computations below, and the result for general $v_0$ then follows from a standard approximation argument. By smoothness of the flow of \eqref{NLS}, and using Lemma \ref{lem:renormalisation}, one obtains that $\gamma,\omega,p,y \in C^1([0,T),\mathbb R)$. Let $j\in \{0,1,2,3\}$. Differentiating the orthogonality condition \eqref{orthogonalities} gives
$$
\langle i \partial_t U,\sigma_3\Xi_j\rangle+i\dot \omega \langle  U,\sigma_3\partial_\omega\Xi_j\rangle=0.
$$
By using \eqref{NLS-vectorial-renormalised-moving} this shows that
\begin{equation} \label{id:modulation-technical-0}
0=-\langle \mathcal{H}_\omega U,\sigma_3\Xi_j\rangle +\langle \mathcal{N}(U),\sigma_3\Xi_j\rangle +\langle \Mod ,\sigma_3\Xi_j \rangle+i\dot \omega \langle  U,\sigma_3\partial_\omega\Xi_j\rangle.
\end{equation}
Using Lemma \ref{kernel} and \eqref{orthogonalities} we have 
\begin{equation} \label{id:modulation-technical-1}
\langle \mathcal{H}_\omega U,\sigma_3\Xi_j\rangle =0 \qquad \mbox{for }j=0,1.
\end{equation}
By \eqref{id:definition-Mod}, it follows that
$$
\langle \Mod , \sigma_3 \Xi_j \rangle_{j = 1,2,3,4} = \left[ M_0 + M_1(U) \right] \begin{pmatrix} \omega-p^2-\dot \gamma \\ -i\dot \omega \\ i(2p-\dot y) \\ \dot p \end{pmatrix} ,
$$
which implies immediately the desired statement.
\end{proof}

\section{The bootstrap argument}

\label{sectionbootstrap}

\subsection{The trapped regime}

We pick two parameters $0 < \alpha \ll \nu \ll 1$, and for $T\in [0,\infty]$ we consider the function space $X_{T,\underline{\omega}}$ of functions $U$ with $U=P_e U$ associated to the norm
\begin{equation*}
\| U \|_{X_{T,\underline{\omega}}} = \| \widetilde{f} \|_{L^\infty_t ([0, T); L^\infty_\xi)} + \| \langle t \rangle^{-\alpha} \widetilde{f} \|_{L^\infty_t ([0, T); H^1_\xi)}\end{equation*}
(the dependence of the norm on $\omega$ is through the distorted Fourier transform, which is associated to the operator $\mathcal{H}_\omega$).

\begin{definition}[Trapped solutions] \label{def:trapped-solution}
Let $\omega_0,\underline{\omega}\in (0,\omega^*)$, $T\in [0,\infty]$, $C_1\gg 1$ and $\epsilon_1=C_1\epsilon \gg \epsilon>0$ and $1\gg \nu \gg \alpha >0$. We say that a solution $v$ of~\eqref{NLS} with data 
$$
v_0=\Phi_{\omega_0}+u_0,
$$
is trapped around $e^{i\underline p x} \Phi_{\underline{\omega}}$ on $[0,T)$ if
\begin{equation}
\label{bd:bootstrap-initial-smallness}\| u_0\|_{H^1}+\| \langle x \rangle u_0\|_{L^2} \leq \epsilon,
\end{equation}
and $v$ satisfies the following. It is defined on $[0,T)$ and there exists $(\gamma,p,y,\omega) \in C([0,T),\mathbb R^3 \times (0,\omega^*))$ satisfying
\begin{equation}
\label{bd:bootstrap-omega2} |p(t) - \underline p| + | \omega(t)-\underline{\omega}| \leq \epsilon_1 \langle t \rangle^{-1-\nu}, \qquad \forall t\in [0,T),
\end{equation}
such that $U$ defined by \eqref{def:U} satisfies the orthogonality conditions \eqref{orthogonalities} and the estimate
\begin{equation}  
\label{eqbootstrap}
\| U \|_{X_{T,\underline{\omega}}} < \epsilon_1 .
\end{equation}
\end{definition}

Note that by Lemma \ref{lem:renormalisation} we can always assume that $u_0 = u(0)$ and $\omega_0=\omega(0)$.
We now record immediate consequences of the bootstrap assumption.

\begin{lemma} 
\label{ibis}
Assume that $v$ is trapped on $[0,T)$. Then for all $t\in [0,T)$:
\begin{itemize}
\item[(i)] (Cancellation at zero frequency)
\begin{equation} \label{id:annulation-tildef(0)}
\widetilde f_+(t,0)=\widetilde f_-(t,0)=0.
\end{equation}
\item[(ii)] (Decay of the discrete parameters $a_j$) 
\begin{equation} \label{bd:estimation:a0(t)-a1(t)} 
\mbox{if $j=0,1,2,3$,} \qquad |a_j(t)| \lesssim \epsilon_1^2 \langle t \rangle^{-2-\nu+\alpha}
\end{equation}
\item[(iii)] (Global decay of the solution)
\begin{equation} \label{bd:globaldecay}
\| u(t,\cdot) \|_{L^\infty} \lesssim  \epsilon_1  \langle t \rangle ^{-1/2} 
\end{equation}
\item[(iv)] (Local decay of the solution)
\begin{equation} \label{bd:localdecay}
\| \langle x \rangle^{-1} u(t,\cdot) \|_{L^\infty} \lesssim  \epsilon_1  \langle t \rangle ^{-1+\alpha} 
\end{equation}
\item[(v)] (Uniform weighted $L^2$ norm on the Fourier side)
\begin{equation} \label{controlH1}
\| u(t,\cdot)\|_{H^1_x}+\| \langle \xi \rangle \widetilde{f} (t,\xi) \|_{L^2_\xi} \lesssim \epsilon_1.
\end{equation}
\end{itemize}
\end{lemma}

\begin{proof} 
\noindent \underline{Proof of (i)}. We have $v_0\in H^1\cap L^{2,1}$ by \eqref{bd:bootstrap-initial-smallness}, and we know from Proposition \ref{pr:cauchy} that solutions to \eqref{NLS} exist locally in $H^1\cap L^{2,1}$. The weighted energy estimate
$$
\frac{d}{dt} \left(\int \langle x \rangle^2 |v|^2 dx\right)^{\frac 12} =\frac{ \Im \int \partial_x (\langle x \rangle^2) v \overline{\partial_x v}}{\left(\int \langle x \rangle^2 |v|^2 dx\right)^{\frac 12}} \lesssim \left( \int |\partial_x v|^2 dx\right)^{\frac 12},
$$
(where we used the Cauchy-Schwarz inequality) and the boundedness of the kinetic energy $\int |\partial_x v|^2dx\lesssim 1$ (by \eqref{controlH1}, which is showed below to hold true as long as $v\in L^{2,1}$) then imply the bound $\| \langle x \rangle v \|_{L^2}\lesssim \langle t \rangle$. Hence $v\in L^1$, so that $u\in L^1$ which implies \eqref{id:annulation-tildef(0)} by (ii) in Proposition \ref{tourterelle}.

\medskip

\noindent \underline{Proof of (ii), (iii) and (iv)}. For the second assertion, we have $u=U_1$ with $U=\Pue U+\sum_{j=0}^3 a_j\underline{\Xi_j}$ and $\Pue U=e^{it\mathcal H}f$. The local decay Lemma \ref{bergeronnette} and \eqref{eqbootstrap} give
\begin{equation} \label{bd:local-decay-radiation-projection-essentiel}
\| \langle x \rangle^{-1}\Pue U\|_{L^\infty}\lesssim \epsilon_1\langle t \rangle^{-1+\alpha}.
\end{equation}
Injecting this estimate and \eqref{bd:bootstrap-omega2} in \eqref{decomposition-bd-coefficients-discrets} proves \eqref{bd:estimation:a0(t)-a1(t)}.

In turn, injecting \eqref{bd:estimation:a0(t)-a1(t)} and \eqref{bd:local-decay-radiation-projection-essentiel} in \eqref{livres} shows the fourth assertion \eqref{bd:localdecay}. The third assertion follows similarly, using the global decay Lemma \ref{aigrette}.

\medskip

\noindent \underline{Proof of (v)}. We first use the mass conservation $\int |v|^2=\int |v(0)|^2$ for \eqref{NLS}. The orthogonality \eqref{orthogonalities} implies the Pythagorean expansion $\int |v|^2=\int \Phi_\omega^2+\int |u|^2$, so that
\begin{equation} \label{grand-mat}
\int |u|^2 \,dx=\int |u(0)|^2 \,dx+ \int \Phi_{\omega(0)}^2\,dx-\int \Phi_{\omega}^2\,dx.
\end{equation}
Next, we use the conservation of the Hamiltonian $H(v)=H(v(0))$ for \eqref{NLS}. Integrating by parts and using the Taylor expansion $F(|v|^2)=F(\Phi_\omega^2)+F'(\Phi_\omega^2)\Phi_\omega (u+\bar u)+O(|u|^2)$,
$$
H(v)=H(\Phi_\omega) +\int |\partial_x u|^2 \,dx-\int (\partial_{x}^2 \Phi_\omega+F'(\Phi_\omega^2)\Phi_\omega)(u+\bar u)\,dx +O\left(\int |u|^2\right).
$$
The second term vanishes by the soliton equation \eqref{eq:soliton} and the orthogonality \eqref{orthogonalities}
$$
\int (\partial_{x}^2 \Phi_\omega+F'(\Phi_\omega^2)\Phi_\omega)(u+\bar u)\,dx =\omega \int \Phi_\omega(u+\bar u)\,dx =0.
$$
Therefore,
\begin{equation} \label{fregate}
\int |\partial_x u|^2\,dx= \int |\partial_x u(0)|^2\,dx+H(\Phi_{\omega(0)}) -H(\Phi_{\omega})+O\left(\int |u(0)|^2+|u|^2 \right).
\end{equation}
Combining \eqref{grand-mat} and \eqref{fregate} one obtains
$$
\| u\|_{H^1}^2 \lesssim \| u_0\|_{H^1}^2+|\omega-\omega(0)|.
$$
By \eqref{bd:bootstrap-initial-smallness} this implies $\| u\|_{H^1}^2\lesssim \epsilon^2+|\omega-\omega(0)|$. By \eqref{eq:modulation-gamma-rough} (whose proof is done shortly after, but does not use the bound \eqref{controlH1} we are currently proving) we have $|\omega-\omega(0)|\lesssim \epsilon_1^2$. Hence $\| u\|_{H^1}\lesssim \epsilon_1$. Using (iv) in Proposition \ref{tourterelle} and~\eqref{bd:bootstrap-omega2}, this shows \eqref{controlH1}.
\end{proof}

\subsection{Bootstrap}

The heart of the present article will be to show the following.

\begin{proposition}   \label{propbootstrap} Under the assumptions of Theorem~\ref{mainthm}, there exists $\nu>0$ such that, for any $0<\alpha\ll 1$ and $\omega_0\in (0,\omega^*)$, there exist $\epsilon_1\gg \epsilon_0>0$ such that the following holds true. Assume $v$ solves \eqref{NLS} and is trapped around $e^{i\underline px} \Phi_{\underline{\omega}}$ on $[0,T)$ for some $(\underline p,\underline{\omega}) \in \mathbb{R} \times (0,\omega^*)$ and $T \in (0,\infty)$. Then the maximal time of existence of $v$ is strictly greater than $T$, and there exists $(\underline{p'},\underline{\omega'}) \in \mathbb{R} \times (0,\omega^*)$ such that $v$ is trapped around $e^{i\underline{p}'x}\Phi_{\underline{\omega}'}$ on $[0,T)$ with:
\begin{align}
\label{bd:bootstrap-omega2-improved} & |p(t) - \underline{p}'| + | \omega(t)-\underline{\omega}'| \leq \frac {\epsilon_1} 2 \langle t \rangle^{-1-\nu}, \qquad \forall t\in [0,T],\\
\label{bd:bootstrapU-improved}  & \| U \|_{X_{\underline{\omega}',T}} < \frac {\epsilon_1} 2 . 
\end{align}
\end{proposition}

\begin{proof}[Proof of Proposition \ref{propbootstrap}]  The proof of Proposition~\ref{propbootstrap} will occupy the rest of this article. We will explain here how the results of sections~\ref{sectionmodulation},~\ref{SectionWeightedQuadratic}, \ref{SectionWeightedCubic} and~\ref{SectionPointwise} can be put together to prove Proposition \ref{propbootstrap}.

\medskip

\noindent \underline{Control of the modulation parameters}
Consider $(\underline{p},\underline{\omega},\gamma,p,y,\omega,U)$ as in the statement of Proposition \ref{propbootstrap}. The first step is to set $(\underline p',\underline{\omega}')=(p(T),\omega(T))$, and to use Proposition~\ref{lemmamodulation} to show that \eqref{bd:bootstrap-omega2-improved} is satisfied.

\medskip

\noindent \underline{Control of the radiation: the bootstrap within the bootstrap.} The second step is to show that~\eqref{bd:bootstrapU-improved} is satisfied, but it cannot follow immediately from Propositions~\ref{PropositionWeightedQuadratic},~\ref{PropositionWeightedCubic},~\ref{PropositionPointwise} for the following reason: the function $U$ must be measured in $X_{\underline{\omega}',T}$, while these propositions would give a bound in $X_{\underline{\omega},T}$.

For the rest of this proof, we use the distorted Fourier transform associated to the linearization around the soliton $\Phi_{\underline{\omega'}}$. We then introduce a new bootstrap argument: consider
$$
T' = \sup \left\{ t \in [0,T], \; \| U \|_{X_{t,\underline{\omega'}}} < \frac{\epsilon_1}{2} \right\}.
$$
By the local well-posedness result in Appendix~\ref{CauchyTheory} and Lemma~\ref{lem:renormalisation}, the set on the right-hand side is not empty, and thus $T'>0$ is well defined.

We claim that $T'=T$; arguing by contradiction, assume that $T'<T$. Duhamel's formula for the profile ${f} $ is given by~\eqref{dtf}, but an attentive examination of the right-hand side there reveals a loss of deriatives (or weight, in Fourier) in the term $e^{-it\mathcal{H}} \Mod_U$, preventing a naive time integration of the right-hand side.

\medskip

\noindent \underline{Absorbing the derivative loss in the phase.} Anticipating on Proposition~\ref{propModU}, we split
$$
e^{-it\rho(\uo+\xi^2)} \widetilde{\Mod_U} = \tau(t) \xi \widetilde{f} + \widetilde{\mathcal{M}}(\xi).
$$
After setting
$$
\theta(t) = \int_0^t \tau(s)\,ds,
$$
Duhamel's formula for the profile~\eqref{dtf} can now be written
$$
i \partial_t \left[ e^{i \theta(t) \xi} \widetilde{f}_\rho(\xi) \right] = e^{i [ \theta(t) \xi - i t \rho(\uo+\xi^2) ] } \left[ \widetilde{ \mathcal{N}} + \widetilde{ {\Mod_\Phi}} + \widetilde{ \mathcal{E}}  \right] + e^{i \theta(t) \xi} \widetilde{\mathcal{M}},
$$
which becomes, after expanding the nonlinear term $\mathcal{N}$,
\begin{equation}
\label{courlis}
i \partial_t \left[ e^{i \theta(t) \xi} \widetilde{f}_\rho(\xi) \right] = e^{i \theta(t) \xi  } \left[\widetilde{ \mathcal{Q}^R} + \widetilde{ \mathcal{C}^S} + \widetilde{ \mathcal{C}^R} + \widetilde{ \mathcal{T}} + \widetilde{ \mathcal{R}} + \widetilde{ \mathcal{M}} \right] +  e^{i [ \theta(t) \xi - i t \rho(\uo+\xi^2) ] } \left[ \widetilde{ {\Mod_\Phi}} + \widetilde{ \mathcal{E}}  \right].
\end{equation}

\medskip

\noindent \underline{Control of the $H^1_\xi$ norm of the radiation} We will now combine the proofs given by propositions~\ref{PropositionWeightedQuadratic} and~\ref{PropositionWeightedCubic}:
\begin{itemize}
\item Proposition~\ref{PropositionWeightedQuadratic} gives the bound 
$$\left\| \partial_\xi \int_0^t e^{i \theta(s) \xi} \widetilde{\mathcal{Q}^R} \,ds \right\|_{L^\infty_{[0,T']} L^2_\xi} \lesssim \epsilon_1^2.
$$
\item Proposition~\ref{PropositionWeightedCubic} shows that 
\begin{align*}
& \left\| \partial_\xi \int_0^t e^{i \theta(s) \xi} [\widetilde{\mathcal{C}} + \widetilde{\mathcal{T}} + \widetilde{\mathcal{R}}+ \widetilde{\mathcal{M}}] \,ds \right\|_{L^\infty_{[0,T']} L^2_\xi}\\
& \qquad \qquad \qquad + \left\|\partial_\xi \int_0^t e^{i [ \theta(t) \xi - i t \rho(\uo+\xi^2) ] } [ \widetilde{\Mod_\Phi} + \widetilde{\mathcal{E}} ] \right\|_{L^\infty_{[0,T']} L^2_\xi} \lesssim \epsilon_1^2 \langle t \rangle^\alpha.
\end{align*}
\end{itemize}

After applying $\partial_\xi$ to~\eqref{courlis} and integrating in time, the two estimates above show that
$$
\left\| \partial_\xi \left[ e^{i  \theta(t) \xi} \widetilde{f}(t) \right] \right\|_{L^2} \lesssim \| \partial_\xi \widetilde{f}_0 \|_{L^2} + \epsilon_1^2 \langle t \rangle^\alpha \lesssim \epsilon + \epsilon_1 \langle t \rangle^\alpha.
$$
Using in addition~\eqref{controlH1} and the fact that $|\theta|\lesssim \epsilon_1$ from Proposition \ref{propModU}, this implies that
\begin{align*}
\| \widetilde{f} \|_{H^1} & \lesssim \| \widetilde{f} \|_{L^2} + \left\| \partial_\xi \widetilde{f}(t) \right\|_{L^2} \lesssim
\| \widetilde{f} \|_{L^2} + \left\| \partial_\xi [e^{i \theta(t) \xi} \widetilde{f}(t)] \right\|_{L^2} + |\theta(t)| \|\xi  \widetilde{f}(t) \|_{L^2} \lesssim \epsilon + \epsilon_1^2 \langle t \rangle^\alpha.
\end{align*}

\medskip

\noindent \underline{Pointwise control of the radiation in Fourier and conclusion.}
Proposition~\ref{PropositionPointwise} gives the bound
$$
\| \widetilde{f} \|_{L^\infty_{[0,T']} L^\infty_\xi}  \lesssim \epsilon + \epsilon_1^2.
$$
Combining this with the $H^1_\xi$ bound, we obtain
$$
\| U \|_{X_{\underline{\omega'},T'}} \lesssim \epsilon + \epsilon_1^2 \qquad \Longrightarrow \qquad \| U \|_{X_{\underline{\omega'},T'}} < \frac{\epsilon_1}{4}
$$
since we chose $\epsilon < \epsilon_0 \ll 1$ and $\frac{\epsilon_1}{\epsilon} = C_1 \ll 1$. Thus, it is possible to prolong the solution up to some time $T'+\delta$, with $\delta>0$, in such a way that $\| U \|_{X_{\underline{\omega'},T'+\delta}} < \frac{\epsilon_1}{2}$. This contradicts the definition of $T'$, and thus proves the claim.

\end{proof}

\subsection{How Proposition~\ref{propbootstrap} implies Theorem~\ref{mainthm}}

\begin{proof}[Proof of Theorem  \ref{mainthm}]
The proof relies on a standard continuous induction argument. Keeping the notations of Proposition \ref{propbootstrap}, we define
$$
T_\infty=\sup \ \{T>0 \mbox{ such that } v \mbox{ is trapped on }[0,T)\}
$$
(we say that $v$ is trapped on $[0,T)$ if there exists $(\underline p,\underline{\omega})$ such that $v$ is trapped on $[0,T)$ around $e^{i\underline{p}x}\Phi_{\underline{\omega}}$). First observe that the set $\{T>0 \mbox{ such that } v \mbox{ is trapped on }[0,T)\}$ is not empty, by the local well-posedness result recalled in Appendix~\ref{CauchyTheory}.

\medskip

\noindent \underline{Propagation of the trap until $T_\infty$.} We claim that there exists $(\underline{p_\infty},\underline{\omega_\infty})$ such that $v$ is trapped around $e^{\underline{p_\infty}x}\Phi_{\underline{\omega_\infty}}$ on $[0,T_\infty)$.
To prove this claim, let $T_n\geq 0$ be a strictly increasing sequence of times converging to $T_\infty$. Then for each $n$, $v$ is trapped on $[0,T_n)$ and we let $(\underline{p_n},\underline{\omega_n})$, $\gamma_n$, $p_n$, $y_n$, $\omega_n$ and $U_n$ be the corresponding parameters and radiation given by Definition \ref{def:trapped-solution}. Up to extraction, by \eqref{bd:bootstrap-omega2}, we can assume that there exists $(\underline{p_\infty},\underline{\omega_\infty}) \in \mathbb{R} \times (0,\omega^*)$ with
\begin{equation}
\label{bd:bootstrap-omegainfty1} |\underline{p_\infty}| + |\underline{\omega_\infty}-\omega_0| \leq \epsilon_1
\end{equation}
 such that $\underline{\omega_n}\to \underline{\omega_\infty}$ as $n\to \infty$. Fix now any $0\leq T<T_\infty$, and let $n,m$ large enough so that $T_m,T_n>T$. By the uniqueness of $(\gamma,p,y,\omega)$ given by Lemma \ref{lem:renormalisation} and the estimate \eqref{controlH1}, we have that $(\gamma_m,p_m,y_m,\omega_m) = (\gamma_n,p_n,y_n,\omega_n)$ and $U_m=U_n$ coincide on $[0,T]$.

Hence there exists $(\gamma_\infty,p_\infty,y_\infty,\omega_\infty) \in C([0,T_\infty),\mathbb R^3 \times (0,\omega_*))$,  and $U_\infty$ given by \eqref{def:U} satisfying \eqref{orthogonalities} such that for any $0\leq T<T_\infty$, $(\gamma_n,p_\infty,y_\infty,\omega_\infty) = (\gamma_\infty,p_\infty,y_\infty,\omega_\infty)$ and $U_n=U_\infty$ on $[0,T]$ for all $n$ large enough. By \eqref{bd:bootstrap-omega2} with $\underline{\omega}=\underline{\omega_n} $ and $\omega=\omega_n$, taking $n\to \infty$ we obtain:
\begin{equation}
\label{bd:bootstrap-omegainfty2}  | p_\infty(t)-\underline{p_\infty}| + | \omega_\infty(t)-\underline{\omega_\infty}| \leq \epsilon_1 \langle t \rangle^{-1-\nu}, \qquad \forall t\in [0,T_\infty).
\end{equation}
Let $0\leq T<T_\infty$. Then by \eqref{eqbootstrap} with $\underline{\omega}=\underline{\omega_n}$ and $U=U_n=U_\infty$ we obtain $\| U_\infty \|_{X_{T,\underline{\omega_n}}}\leq \epsilon_1$. We have $\| U_\infty \|_{X_{T,\underline{\omega_n}}}\rightarrow \| U_\infty \|_{X_{T,\underline{\omega_\infty}}}$ since $\underline{\omega_n}\to \underline{\omega_\infty}$ so that $\| U_\infty \|_{X_{T,\underline{\omega_{\infty}}}}\leq \epsilon_1$. This implies
\begin{equation}
\label{eqbootstrapinfty} \| U_\infty \|_{X_{T_\infty,\underline{\omega_{\infty}}}}\leq \epsilon_1.
\end{equation}
Combining \eqref{bd:bootstrap-omegainfty1}, \eqref{bd:bootstrap-omegainfty2} and \eqref{eqbootstrapinfty} shows the claim of Step 1.

\medskip

\noindent \underline{Trapped for all times}. We claim that $T_\infty=\infty$. By contradiction, assume that $T_\infty<\infty$. Then by the result of Step 1, $v$ is trapped around some $e^{i\underline{p_\infty} x}\Phi_{\underline{\omega_\infty}}$ on $[0,T_\infty)$. Applying Proposition \ref{propbootstrap}, we obtain that the maximal time of existence of $v$ is strictly greater than $T_\infty$, and that there exist $\underline{\omega_\infty'},\underline{p}_\infty'$ such that $v$ is trapped around $e^{i\underline{p_\infty}' x}\Phi_{\underline{\omega_\infty'}}$ on $[0,T_\infty)$, and that  \eqref{bd:bootstrap-omega2-improved} and \eqref{bd:bootstrapU-improved} are satisfied. These norms are strict improvements of \eqref{bd:bootstrap-omega2} and \eqref{eqbootstrap}. By the local well-posedness result recalled in Appendix~\ref{CauchyTheory}, the solution can be prolonged locally in time, and by Lemma~\ref{lem:renormalisation}, it can be decomposed into the sum of the soliton and the radiation (renormalized by the symmetries). Therefore, $v$ is trapped around $e^{i\underline{p_\infty}'x}\Phi_{\underline{\omega}_\infty'}$ on $[0,T_\infty+\delta]$ for some $\delta>0$. This contradicts the definition of $T_\infty$! Hence $T_\infty=\infty$ as claimed.

\medskip
\noindent \underline{End of the proof}. Combining the first two steps, the solution is global and there exist $\underline{\omega},\underline{p}$ such that it is trapped around $e^{i\underline{p}x}\Phi_{\underline{\omega}}$ on $ [0,\infty)$. Item (i) of Theorem \ref{mainthm} then follows from Lemmas \ref{propfirst} and \ref{lemmamodulation} for the modulation parameters, and \eqref{bd:globaldecay} for the radiation. Item (ii) follows from Proposition \ref{pr:modified-scattering}, the identity \eqref{id:decomposition-v} and the asymptotics of item (i) for $\omega,\gamma,p,y$. Item (iii) follows from combining the continuity of the flow, see Proposition \ref{pr:cauchy}, and the various convergence estimates as $t\to \infty$ for $\omega,\gamma,p,y,\widetilde f$ that are uniform in the initial data.
\end{proof}

\subsection{First consequences of the bootstrap assumption}

In this subsection, we prove some rather immediate consequences of the bootstrap hypothesis, first on the convergence of the modulation parameters, and then on the decay of nonlinear terms.

\begin{lemma}
\label{propfirst}
If the solution is trapped on $[0,T)$,
\begin{align}
\label{eq:modulation-gamma-rough}& | \omega - \dot \gamma - p^2| + |\dot \omega| + |\dot y - 2p| + |\dot p| \lesssim \epsilon_1^2 \langle t\rangle^{-2+2\alpha}.
\end{align}
\end{lemma}

\begin{proof} The right-hand side of \eqref{id:modulation-technical-identity} is $O(\epsilon_1^2 \langle t\rangle^{-2+2\alpha})$ due to the improved local decay of solutions~\eqref{bd:localdecay} and the decay of the discrete parameters~\eqref{bd:estimation:a0(t)-a1(t)} under the bootstrap assumption. Since $c_\omega \geq c^*$ for some $c^*>0$ for $\omega$ near $\omega_0$, the matrix $M$ appearing in \eqref{id:modulation-technical-identity} is invertible with $M^{-1}$ bounded by some universal constant.
This leads to the desired bounds.
\end{proof}

\begin{proposition}
\label{propModU}
If the solution is trapped on $[0,T)$,
$$
e^{- it \rho (\uo+ \xi^2)} \widetilde{\Mod_{\Uu}}  =  \tau(t) \xi \widetilde f(\xi) + \widetilde{ \mathcal{M}}(\xi)
$$
where
\begin{align*}
& | \tau(t) | \lesssim \epsilon_1 \langle t \rangle^{-1-\nu} \\
& \| \langle \xi \rangle \widetilde{\mathcal{M}} \|_{L^2} + \|  \partial_\xi \widetilde{\mathcal{M}} \|_{L^2}  \lesssim \epsilon_1^2\langle t \rangle^{-1-\nu+\alpha}.
\end{align*}
\end{proposition}
\begin{proof}
Recall that
$$
\Mod_{\Uu}= \left[(\omega-p^2-\dot \gamma)+(2p-\dot y)(p-\pu)-(\pu-p)^2+\underline{\omega} - \omega \right]\sigma_3 \Uu +i[2p-\dot y+2(\pu-p)]\partial_x \Uu.
$$
By Proposition~\ref{propsigmadx},
\begin{align*}
&e^{- it \rho (\uo+ \xi^2)} \widetilde{\mathcal{F}}_\rho \sigma_3 \underline{U} = O_{{L^{2,1} \cap H^1} }( \epsilon_1 t^\alpha) \\
& e^{- it \rho (\uo+ \xi^2)} \widetilde{\mathcal{F}}_\rho \partial_x \underline{U} = i \xi \widetilde{f} + O_{L^{2,1} \cap H^1} (\epsilon_1 t^\alpha)
\end{align*}
Using in addition Lemma~\ref{propfirst} and the fact that $v$ is trapped, and finally setting
$$
\tau(t) = \dot y-2p+2(p-\pu)
$$
gives the desired statement.
\end{proof}

\begin{proposition} \label{grebehuppe}

Define $\mathcal{D}$ by
$$
\mathcal{D} = \underline{P_e} \left[ \mathcal{Q}^R + \mathcal{C}^R + \mathcal{T}  + \mathcal{R} + e^{-it\underline{\mathcal{H}}} \Mod_\Phi + e^{-it\underline{\mathcal{H}}} \mathcal{E} \right],
$$
so that the equation satisfied by $f$ is
$$
i\partial_t f =  \mathcal{D} +  \underline{P_e} \left[ \mathcal{C}^S + e^{-it\underline{\mathcal{H}}} \Mod_{\underline{U}} \right].
$$
Then, if the solution is trapped on $[0,T]$, the term $\mathcal{D}$ enjoys the decay
\begin{equation} \label{bd:mathcalD-L21}
\| \langle \xi \rangle \widetilde{\mathcal{D}}(\xi) \|_{L^2} \lesssim \epsilon_1^2 \langle t \rangle^{-3/2}.
\end{equation}

\end{proposition}

\begin{remark} The terms which were left out of $\mathcal{D}$ do not decay sufficiently fast. For $\mathcal{C}^S$, it decays like $\langle t \rangle^{-1}$, wich is not even integrable. By the previous proposition, the term $\widetilde{\Mod_U}$ decays at an integrable rate, but only like $\langle t \rangle^{-1-\nu}$.
\end{remark}

\begin{proof} We will examine each term in the definition of $\mathcal{D}$.

\medskip

\noindent \underline{Bound for $\widetilde{\mathcal{Q}^R}$, the case $|t|<1$.} 
By Proposition~\ref{propmuquad}, one can write (dropping all indices for ease of reading)
\begin{equation}
\label{rougegorge}
\langle \xi \rangle \widetilde{\mathcal{Q}^R}(t,\xi) = \langle \xi \rangle \int e^{-it\Phi(\xi,\eta,\sigma)} \widetilde{f}(\eta) \widetilde{f}(\sigma) \mathfrak{m}(\xi,\eta,\sigma) \,d\eta \,d\sigma 
\end{equation}
We claim that for any $g,h$,
\begin{equation} \label{voilier}
\left\|  \langle \xi \rangle \int g(\eta) h(\sigma) |\mathfrak{m}(\xi,\eta,\sigma)| \,d\eta \,d\sigma  \right\|_{L^2_\xi}  \lesssim \| \xi g\|_{L^2_\xi}  \| \xi h\|_{L^2_\xi}.
\end{equation}
Indeed, the bound on $\mathfrak{m}$ in Proposition~\ref{propmuquad} implies that
\begin{equation}
\label{boundmuetasigma}
\left| \frac{\langle \xi \rangle \mathfrak{m}(\xi,\eta,\sigma)}{\eta \sigma} \right| \lesssim \frac{\langle \xi\rangle}{\langle \eta \rangle \langle \sigma \ra} \sum_{\pm}\frac{1}{\la \xi\pm \eta \pm \sigma \ra^2}
\end{equation}
where in the above sum the signs $\pm$ take all possible values. We treat the $\frac{1}{\la \xi- \eta -\sigma \ra}$ case for concreteness. Assuming $g,h\geq 0$ and introducing $\check g(\xi)=\frac{|\xi|}{\langle \xi \rangle}g(\xi)$ we have to bound
\begin{align*}
 \int \check g(\eta) \check h(\sigma) \frac{ \langle \xi \rangle}{\la \xi- \eta-\sigma \ra^2}\,d\eta \,d\sigma  & =  \int \check g(\xi-\sigma+\zeta) \check h(\sigma) \frac{ \langle \xi \rangle}{\la \zeta \ra^2}\,d\zeta \,d\sigma \\
 & \leq   \int \check g(\xi-\sigma+\zeta) \check h(\sigma)\left(  \frac{ \langle \xi-\sigma+\zeta \rangle}{\la \zeta \ra^2}+  \frac{ \langle \sigma \rangle}{\la \zeta \ra^2}+\frac{1}{\la \xi \ra}\right) \,d\zeta \,d\sigma
\end{align*}
where we used the inequality $\frac{ \langle \xi \rangle}{\la \zeta \ra^2}\lesssim \frac{ \langle \xi-\sigma+\zeta \rangle}{\la \zeta \ra^2}+  \frac{ \langle \sigma \rangle}{\la \zeta \ra^2}+\frac{1}{\la \xi \ra}$.
The first term in the right-hand side is bounded by the Minkowski, Young and Cauchy-Schwarz inequalities
\begin{align*}
\left\| \int \check g(\xi-\sigma+\zeta) \check h(\sigma) \frac{ \langle \xi-\sigma+\zeta \rangle}{\la \zeta \ra^2}\,d\zeta \,d\sigma \right\|_{L^2_\xi}&\lesssim \int_{\zeta}\frac{d\zeta}{\la\zeta\ra^2} \left\| \int_\sigma  \langle \xi-\sigma+\zeta \rangle \check g(\xi-\sigma+\zeta) \check h(\sigma)d\sigma \right\|_{L^2_\xi}\\
& \lesssim \| \la \xi \ra \check g\|_{L^2}\| \check h\|_{L^1} \lesssim \| \xi g\|_{L^2}\| \xi h\|_{L^2}.
\end{align*}
The second term is symmetric to the first one and so is bounded similarly, and the third term is bounded by $\| \check g\|_{L^1}\| \check h\|_{L^1}\lesssim \| \xi g\|_{L^2}\| \xi h\|_{L^2}$. This shows \eqref{voilier}, which, injected in \eqref{rougegorge}, gives
$$
\|\langle \xi \rangle \widetilde{\mathcal{Q}^R}(t,\xi) \|_{L^2}  \lesssim \left\| \xi \widetilde{f} \right\|_{L^2}^2 \lesssim \epsilon_1^2
$$
for any $t$ (even though this bound will only be used for $|t|<1$).

\medskip

\noindent \underline{Bound for $\widetilde{\mathcal{Q}^R}$, the case $t>1$.} At this point, it is helpful to observe here that  the case $|t|>1$ can be dealt with similarly to the case $|t|<1$. Indeed, integrating by parts in $\eta$ and $\sigma$ in~\eqref{rougegorge} yields
\begin{align*}
& \langle \xi \rangle \int e^{-it\Phi(\xi,\eta,\sigma)} \widetilde{f}(\eta) \widetilde{f}(\sigma) \mathfrak{m}(\xi,\eta,\sigma) \,d\eta \,d\sigma \\
& \qquad \qquad = \pm \frac{\langle \xi \rangle}{4 t^2} \int e^{-it\Phi(\xi,\eta,\sigma)} \frac{\partial_\xi \widetilde{f}(\eta)}{|\eta|} \frac{\partial_\xi \widetilde{f}(\sigma)}{|\sigma|}  \mathfrak{m}(\xi, \eta,\sigma) \,d\eta \,d\sigma + \{ \mbox{simpler terms} \}.
\end{align*}
Using \eqref{voilier}, we see that
$$
\left\| \frac{\langle \xi \rangle}{ t^2} \int e^{-it\Phi(\xi,\eta,\sigma)} \frac{\partial_\xi \widetilde{f}(\eta)}{|\eta|} \frac{\partial_\xi \widetilde{f}(\sigma)}{|\sigma|}  \mathfrak{m} \,d\eta \,d\sigma \right\|_{L^2} \lesssim t^{-2} \| \partial_\xi \widetilde{f} \|_{L^2}^2 \lesssim \epsilon_1^2 t^{2\alpha-2}.
$$

Combining the cases $|t|<1$ and $|t|>1$ gives the bound
$$
\left\| \xi \widetilde{\mathcal{Q}^R} \right\|_{L^2}  \lesssim \epsilon_1^2 \langle t \rangle^{2\alpha-2}.
$$

\medskip

\noindent \underline{Bound for $\widetilde{\mathcal{C}^R}$.} By Proposition~\ref{gobemouche}, we can write $\langle \xi \rangle \widetilde{\mathcal{C}^R}$ as a linear combination of terms of the type
$$
\langle \xi \rangle \int e^{it\Phi(\xi,\eta,\sigma,\zeta)} \widetilde{f}(\eta) \widetilde{f}(\sigma) \widetilde{f}(\zeta) \mathfrak{m}(\xi,\eta,\sigma,\zeta) \, d\eta \, d\sigma \, d\zeta,
$$
up to simpler terms which we disregard. 

Just like in the estimate for $\widetilde{\mathcal{Q}^R}$ in Section~\ref{subsecnonres}, the cases $|t| < 1$ and $|t| > 1$ can be dealt with in a nearly identical fashion, and we shall focus on the case $|t|>1$. Integrating by parts in $\eta$, $\sigma$ and $\zeta$, the above becomes
$$
\frac{1}{t^3} \int e^{it\Phi(\xi,\eta,\sigma,\zeta)}\partial_\xi \widetilde{f}(\eta) \partial_\xi \widetilde{f}(\sigma)\partial_\xi \widetilde{f}(\zeta) \frac{\langle \xi \rangle \mathfrak{m}(\xi,\eta,\sigma,\zeta)}{\eta  \sigma \zeta} \, d\eta \, d\sigma \, d\zeta + \{ \mbox{easier terms} \}.
$$
By Proposition~\ref{gobemouche},
\begin{align*}
&\left| \frac{\langle \xi \rangle \mathfrak{m}(\xi,\eta,\sigma,\zeta)}{\eta \sigma \zeta} \right| \lesssim \frac{\la \xi \ra}{\la \eta \ra \la \sigma \ra \la \zeta \ra}\sum_{\pm}\frac{1}{\la \xi\pm \eta \pm \sigma\pm \zeta \rangle}
\end{align*}
(where it is understood that the above sum is performed on all combination of signs).

With this bound in hand, it is possible to argue just like in the estimate of $\langle \xi \rangle \widetilde{\mathcal{Q}^R}$ to show that
$$
\left\| \frac{1}{t^3} \int e^{it\Phi(\xi,\eta,\sigma,\zeta)}\partial_\xi \widetilde{f}(\eta) \partial_\xi \widetilde{f}(\sigma)\partial_\xi \widetilde{f}(\zeta) \frac{\langle \xi \rangle \mathfrak{m}(\xi,\eta,\sigma,\zeta)}{\eta  \sigma \zeta} \, d\eta \, d\sigma \, d\zeta \right\|_2 \lesssim \frac{1}{t^3} \| \partial_\xi \widetilde f \|_2^3 \lesssim \epsilon_1^3 t^{3\alpha - 3},
$$
from which the desired bound follows easily.

\medskip

\noindent \underline{Bound for $\widetilde{\mathcal{T}}$.} The term $\mathcal{T}$ can be written under the form $e^{-it \mathcal{H}} G(x,U)$, where $G$ is a smooth function such that $G(x,0) = G'(x,0) = G''(x,0) = G'''(x,0) =0$, and $G$ enjoys uniform (in $x$) bounds. Therefore, by Proposition~\ref{tourterelle}, and Lemma \ref{ibis}
$$
\| \langle \xi \rangle \widetilde{\mathcal{F}} [e^{-it \mathcal{H}} G(x,U)] \|_{L^2} 
= \| \langle \xi \rangle \widetilde{\mathcal{F}}  G(x,U) \|_{L^2} 
\lesssim \| G(x,U) \|_{H^1} \lesssim \| U \|_{\infty}^3 \| U \|_{H^1} \lesssim \epsilon_1^4 t^{-3/2}.
$$

\medskip

\noindent \underline{Bound for $\widetilde{\mathcal{R}}$}. This term is easily to bound, but is made up of summands with different behaviors. For the sake of illustration, we treat the one that decays the slowest, namely $\mathcal{Q}(\underline{U_e},a_j \Xi_j)$. Taking the distorted Fourier transform and applying the inequalities~\eqref{voilier} and~\eqref{bd:estimation:a0(t)-a1(t)} gives
\begin{equation}
\label{durbec}
\left\| \langle \xi \rangle \widetilde{\mathcal{F}} \mathcal{Q}(\underline{U_e},a_j \Xi_j) \right\|_{L^2} \lesssim |a_j| \left\| \xi  \widetilde{f} \right\|_{L^2} \left\| \xi \widetilde{\Xi_j} \right\|_{L^2} \lesssim \epsilon_1^3 \langle t \rangle^{-2-\nu+\alpha}.
\end{equation}

\medskip

\noindent \underline{Bound for $\widetilde{\Mod}$.} By definition of $\Mod_\Phi$ and Lemma~\ref{propfirst}, 
$$
\| \langle \xi \rangle \widetilde{\Mod}_\Phi \|_{L^2} \lesssim | \omega - p^2 - \dot \gamma| + |\dot \omega| + |2p - \dot y| + |\dot p| \lesssim \epsilon_1^2 \langle t \rangle^{-2 + 2 \alpha}.
$$

\medskip

\noindent \underline{Bound for $\widetilde{\mathcal{E}}$.} This error term consists of two summands
\begin{align*}
\mathcal E & =  (V_{\underline{\omega}} -  e^{i(p-\pu)\sigma_3 x} V_\omega  e^{i(\pu-p)\sigma_3 x}) \Uu + e^{i(p-\pu)\sigma_3 x} \mathcal N( e^{i(\pu-p)\sigma_3 x} \Uu)  -\mathcal N(\Uu) \\
& = \mathcal E_1+\mathcal E_2 .
\end{align*}
With the help of \eqref{id:relation-distorted-fourier-and-projectors} and Lemma \ref{bartavelle2}, the distorted Fourier transform of the first part of $\mathcal{E}$ can be written as
\begin{align*}
&\widetilde{\mathcal F}_\rho \mathcal{E}_1 = \widetilde{\mathcal F}_\rho \Pue \left( (V_{\underline{\omega}} - e^{i(p-\underline{p}) \sigma_3 x} V_\omega  e^{i(\underline{p}-p) \sigma_3 x} ) \underline{U} \right)\\
& \qquad = \widetilde{\mathcal F}_\rho \Pue \left( (V_{\underline{\omega}} - e^{i(p-\underline{p}) \sigma_3 x} V_\omega  e^{i(\underline{p}-p) \sigma_3 x} ) \underline{U_d} \right) + \widetilde{\mathcal F}_\rho \Pue \left( (V_{\underline{\omega}} - e^{i(p-\underline{p}) \sigma_3 x} V_\omega  e^{i(\underline{p}-p) \sigma_3 x} ) \underline{U_e} \right) \\
& \qquad = \widetilde{\mathcal F}_\rho \Pue \left( (V_{\underline{\omega}} - e^{i(p-\underline{p}) \sigma_3 x} V_\omega  e^{i(\underline{p}-p) \sigma_3 x} ) \underline{U_d} \right) + \sum_{\lambda \in \{ \pm \}} \int e^{ it \lambda (\uo+\eta^2)} \mathfrak{s}^{\omega,\underline{\omega}}_{\rho \lambda}(\xi,\eta) \widetilde{f}_\lambda(\eta) \,d\eta
\end{align*}
with $\mathfrak{s}=\mathfrak{s}^{\omega,\underline{\omega}}_{\rho \lambda}$ satisfying 
\begin{equation} \label{bound-mathcalE:bd:mathfraks}
\left| \partial_\xi^a \partial_\eta^b \frac{\mathfrak{s} (\xi,\eta)}{\eta} \right| \lesssim_{a,b} \frac{|\omega-\underline{\omega}| + |p-\underline{p}|}{\langle \eta \rangle}\sum_{\pm}  \frac{1}{\langle \xi \pm \eta \rangle^2}.
\end{equation}
The bound for the term involving $\underline{U_d}$ is easily obtained, and we focus now on the integral term accounting for $\underline{U_e}$. Integrating by parts in $\eta$,
\begin{align*}
 \int e^{ it \lambda (\uo+\eta^2)} \mathfrak{s} (\xi,\eta) \widetilde{f}_\lambda(\eta) \,d\eta & =-\frac{1}{2it}\int_{\mathbb R} e^{ it \lambda (\uo+\eta^2)} \partial_\eta \left( \frac{\mathfrak{s}}{\eta} \right) \widetilde{f} (\eta) \,d\eta\\
& \qquad \qquad -\frac{1}{2it}\int_{\mathbb R} e^{ it \lambda (\uo+\eta^2)}  \frac{\mathfrak{s} (\xi,\eta) }{\eta} \partial_\eta  \widetilde{f} (\eta) \,d\eta \ 
\end{align*}
(the boundary terms vanish since $\widetilde f(0)=0$). If $t \geq 1$, we then bound using \eqref{bound-mathcalE:bd:mathfraks}, then $\la \xi \ra \lesssim \la \eta \ra +\la \eta\pm \xi \ra$, \eqref{bd:bootstrap-omega2} and the Young inequality for convolution:
\begin{align*}
& \left\|  \la \xi \ra \int e^{ it \lambda (\uo+\eta^2)} \mathfrak{s} (\xi,\eta) \widetilde{f}_\lambda(\eta) \,d\eta \right\|_{L^2}  \lesssim \frac{|\omega-\underline{\omega}| + |p-\underline{p}|}{t} \left\| \sum_{\pm }\int_{\mathbb R} \frac{|\widetilde f(\eta)|+|\partial_\xi \widetilde f(\eta)|}{\langle \eta\rangle}\frac{\la \xi \ra}{\langle \eta\pm \xi\rangle^2} d\eta \right\|_{L^2_\xi} \\
&\qquad\qquad\qquad\qquad \lesssim \epsilon_1 t^{-2-\nu}\left\|  \sum_{\pm }\int_{\mathbb R} \frac{|\widetilde f(\eta)|+|\partial_\xi \widetilde f(\eta)|}{\langle  \eta \pm \xi \rangle^2}+ \frac{|\widetilde f(\eta)|+|\partial_\xi \widetilde f(\eta)|}{\la \eta \ra \langle  \eta \pm \xi \rangle}  d\eta \right\|_{L^2_\xi}\\
&\qquad\qquad\qquad\qquad \lesssim  \epsilon_1 t^{-2-\nu} \left(\| |\widetilde f|+|\partial_\xi \widetilde f| \|_{L^2} +\left\| \frac{|\widetilde f|+|\partial_\xi \widetilde f|}{\la \xi \ra} \right\|_{L^1}\right)\ \lesssim \epsilon_1^2 t^{-2-\nu+\alpha};
\end{align*}
the estimate for $|t| \leq 1$ is similar. Overall, we proved that
\begin{equation}
\label{boundE1}
\left\| \langle \xi \rangle \widetilde{\mathcal{E}_1} \right\|_{L^2} \lesssim \epsilon_1^2 \langle t \rangle^{-2-\nu+\alpha}.
\end{equation}

To bound the second part of $\mathcal{E}$, we use Proposition~\ref{tourterelle} to obtain
\begin{equation}
\label{boundE2}
\begin{split}
\left\| \langle \xi \rangle \widetilde{\mathcal{E}_2} \right\|_{L^2} & = \left\| \langle \xi \rangle \widetilde{\mathcal{F}} \left[ e^{i(p-\underline p) \sigma_3 x} \mathcal{N}(e^{i(\underline p - p) \sigma_3 x} \underline{U}) - \mathcal{N} (\underline{U}) \right] \right\|_{L^2} \\
& \lesssim \left\| e^{i(p-\underline p) \sigma_3 x} \mathcal{N}(e^{i(\underline p - p) \sigma_3 x} \underline{U}) - \mathcal{N} (\underline{U}) \right\|_{H^1} \\
& \lesssim |p - \underline{p}| \left\| \mathcal{N}(\underline{U}) \right\|_{H^1} + \{ \mbox{similar term} \} \\
& \lesssim \epsilon_1 \langle t \rangle^{-1-\nu} \epsilon_1^2 t^{-1} = \epsilon_1^3 \langle t \rangle^{-2-\nu};
\end{split}
\end{equation}
here, we used that the slowest decaying term in $H^1$ in $\mathcal{N}(U)$ is $\mathcal{C}^S$, which decays like $\epsilon_1^3 \langle t \rangle^{-1}$.

\end{proof}

\section{Control of the modulation parameters}
\label{sectionmodulation}

The aim of this section will be to prove the following proposition.

\begin{proposition}
\label{lemmamodulation}
Under the assumptions of Proposition~\ref{propbootstrap}, $\omega$ and $p$ can be continued up to time $T$, and furthermore
\begin{equation} \label{eq:modulation-omega-improved} 
 | \omega(t)-\omega(T)| + |p(t) - p(T)| \lesssim \epsilon_1^2 \langle t \rangle^{-1-\nu}
\end{equation}
if $t \in [0,T]$.
\end{proposition}

The proof of this proposition will occupy the rest of this section.

\subsection{Refined equation} 

With the notations of Lemma~\ref{echasseblanche}, we have, since $v$ is trapped and by Lemma~\ref{ibis},
$$
M = \underline{M_0} + O(\epsilon_1 \langle t \rangle^{-1 + \alpha}),
$$
where $\underline{M_0}$ stands for the matrix $M_0$ evaluated at the parameter $\underline{\omega}$. As for the nonlinear term, it can be bounded by
$$
N = - ( \langle \mathcal N(U),\sigma_3 \Xi_j \ra)_{j=0,1,2,3}  = O(\epsilon_1^2 \langle t \rangle^{-2+2\alpha})
$$
since $v$ is trapped and by Lemma~\ref{ibis}. The leading order term in $N$ can be isolated as follows:
\begin{align*}
N & = - ( \langle \mathcal N(U),\sigma_3 \Xi_j \ra)  = - ( \langle \mathcal N(\underline{U}),\sigma_3 \Xi_j \ra) + O(\epsilon_1^2 \langle t \rangle^{-3}) \\
& = - ( \langle \mathcal N(\underline{U_e}),\sigma_3 \Xi_j \ra) + O(\epsilon_1^2 \langle t \rangle^{-3}) \\
& = - ( \langle \mathcal B(\underline{U_e}),\sigma_3 \Xi_j \ra) + O(\epsilon_1^2 \langle t \rangle^{-3+3\alpha}),
\end{align*}
where $\mathcal{B}$ stands for the quadratic terms in $\mathcal{N}$:
$$
\mathcal{B}(U) = 
\begin{pmatrix} V_{++} U_1^2 + V_{--} U_2^2 + V_{+-} U_1 U_2 \\ - V_{++} U_2^2 - V_{--} U_1^2 - V_{+-} U_1 U_2 \end{pmatrix}.
$$
Using these estimates on $M$ and $N$, it follows from inverting the equation
$$
M (\omega - p^2 - \dot \gamma, \, -i\dot \omega, \, i(2p-\dot y), \, \dot p)^\top = N
$$
that
\begin{align*}
& \dot \omega = - c_{\underline{\omega}}^{-1} \langle \mathcal{B}(\underline{U_e}), \sigma_3 \Xi_0 \rangle + O(\epsilon_1^2 \langle t \rangle^{-3+3\alpha}) \\
& \dot p = - \| \Phi_{\underline{\omega}} \|_2^{-2} \langle \mathcal{B}(\underline{U_e}), \sigma_3 \Xi_2 \rangle + O(\epsilon_1^2 \langle t \rangle^{-3+3\alpha}).
\end{align*}

\subsection{Decomposition of the quadratic term} Integrating the above ODE between the times $t$ and $T$, we infer that $\omega(t)-\omega(T)$ as well as $p(t) - p(T)$ can be written as linear combinations of terms of the type
$$
\int_0^T \int_{-\infty}^{\infty} V_{ij} (\underline{U_e})_i (\underline{U_e})_j \,dx \, dt
$$
where the $V_{ij}$ are Schwartz functions and $i,j=1,2$. Changing to the Fourier space description, Proposition~\ref{propmuquadscal} translates the above into terms of the type
$$
\int_t^T  M(s) \,ds, \qquad \mbox{with} \qquad M(s) = \int e^{-is \Phi_{\lambda \mu}(\eta,\sigma)} \widetilde{f}_\lambda(\eta) \widetilde{f}_\mu(\sigma) \mathfrak{m}_{\lambda \mu}(\eta,\sigma) \,d\eta \, d\sigma.
$$
where
$$ 
 \Phi_{\lambda \mu}(\eta,\sigma) = \lambda (\uo+ \eta^2) + \mu (\uo+ \sigma^2).
$$
Thus, for such terms, it will suffice to prove the bound
\begin{equation}
\label{aimM0}
\left| \int_t^T M(s) \,ds \right| \lesssim \epsilon_1^2 \langle t \rangle^{-1-\nu}.
\end{equation}

Before delving into the estimates, let us count powers of time to understand where the difficulty lies. 
We see that integrations by parts in $\eta$ and $\sigma$ would gain a factor $t^{-2}$, but would lead to using the bound for $\| \partial_\xi \widetilde{f} \|_{L^2}$ twice, which loses $t^{2\alpha}$. Overall, the powers of $t$ become $t^{-2+2\alpha}$, which, after integrating in time, decays like $t^{-1+2\alpha}$, which is not enough. A small integrability gain (in $t$) is needed, which will be achieved by considering resonances.

First of all, the boundedness of $\mathfrak{m}(\eta,\sigma)$ combined to the Cauchy-Schwarz inequality gives
$$
| M(t) | \lesssim \| \widetilde{f} \|_{L^1}^2 \lesssim \| f \|_{H^1}^2 \lesssim \epsilon_1^2. 
$$
Therefore, it suffices to prove~\eqref{aimM0} for $t>1$, which will be assumed henceforth.
Cutoff functions will be used to split $M(t,\xi)$ into three different terms, which will be estimated separately. First, we localize time on the scale $2^m$, and introduce the high-frequency cutoff for a fixed small but universal constant $\kappa>0$
$$
\varphi^{h}(t,\eta,\sigma) = \sum_{m \geq 0} \varphi_m(t) \varphi_{> - m \kappa} (\eta) \varphi_{> -m \kappa} (\sigma),
$$
which restricts to the region where $|\eta|$ and $|\sigma|$ are $\gtrsim t^{-\kappa}$. Note that we are using here the \underline{inhomogeneous} dyadic decomposition of Section~\ref{dyadicdec}, for which indices such as $m$ run over $\mathbb{N}_0$.
On the support of the low-frequency cutoff,
$$
\varphi^{l}(t,\eta,\sigma) = 1 - \varphi^{h}(t,\eta,\sigma),
$$
either $\eta$ or $\sigma$ is small: $\min(|\eta|,|\sigma|)\lesssim t^{-\kappa}$.

We will also distinguish between a resonant and a non-resonant part, by localizing, in addition to time, the frequencies $\xi,\eta,\sigma$ on the scales $2^j$, $2^{k_1}$ and $2^{k_2}$ respectively, and finally the phase $\Phi$ on a scale $2^J$, with
$$
J = \max(k_1,k_2) - \kappa m.
$$
The resonant and non-resonant cutoff functions are then given by 
$$
\varphi^r(t,\xi,\eta,\sigma) = \sum_{k_1,k_2,m} \varphi_m(t) \varphi_{k_1}(\eta) \varphi_{k_2}(\sigma) \varphi_0(2^{-J} \Phi)  \quad \mbox{and} \quad \varphi^{nr}(t,\xi,\eta,\sigma) = 1 - \varphi^r(t,\xi,\eta,\sigma).
$$
On the support of the resonant cut-off $\varphi^r$ we have $|\Phi|\lesssim (|\eta|+|\sigma|)t^{-\kappa}$.

We can now split $M$ into a low frequency, a resonant, and a non-resonant term as follows (from this point on, we will systematically omit the dependence of all functions and symbols on the time variable, in order to make notations lighter)
\begin{align*}
M(t) & = \int e^{-it \Phi(\eta,\sigma)} \widetilde{f}(\eta) \widetilde{f}(\sigma) \mathfrak{m}(\eta,\sigma) \,d\eta \, d\sigma \\
& =  \int e^{-it \Phi(\eta,\sigma)} \widetilde{f}(\eta) \widetilde{f}(\sigma) \varphi^l(\eta,\sigma) \mathfrak{m}(\eta,\sigma) \,d\eta \, d\sigma \\
& \qquad + \int e^{-it \Phi(\eta,\sigma)} \widetilde{f}(\eta) \widetilde{f}(\sigma) \varphi^h(\eta,\sigma) \varphi^{r}(\eta,\sigma) \mathfrak{m}(\eta,\sigma) \,d\eta \, d\sigma \\
& \qquad + \int e^{-it \Phi(\eta,\sigma)} \widetilde{f}(\eta) \widetilde{f}(\sigma) \varphi^h(\eta,\sigma) \varphi^{nr}(\eta,\sigma) \mathfrak{m}(\eta,\sigma) \,d\eta \, d\sigma \\
& = M^l (t) + M^{r} (t) + M^{nr}(t).
\end{align*}

We will show in \eqref{improved-modulation:bd:Ml} that for all $\kappa>0$ we have $ |M^l(t)| |\lesssim  \epsilon_1^2 t^{-2-\kappa/4}$, in \eqref{improved-modulation:bd:Mr} that for $\kappa$ small enough $| M^{r}(t)  |\lesssim  \epsilon_1^2 t^{-2-\kappa/4}$, and in \eqref{improved-modulation:bd:Mnr} that for $\kappa$ small enough $| \int_{t}^T M^{nr}(s)ds|\lesssim \epsilon_1^2t^{-5/4}$, provided $\alpha $ is small enough depending on $\kappa$. This will prove \eqref{aimM0} if $\nu$ is chosen smaller than $\kappa/4$.

\subsection{The low frequency term $M^l$.} We claim that
\begin{equation} \label{improved-modulation:bd:Ml}
|M^l(t)| \lesssim \epsilon_1^2 t^{-2-\frac{\kappa}{4}}.
\end{equation}
Indeed, integrating by parts in $\eta$ and $\sigma$ gives 
\begin{align*}
M^l(t) & = \pm \frac{1}{4t^2} \int e^{-it\Phi(\eta,\sigma)}  \partial_\xi \widetilde{f}  (\eta)  \partial_\xi \widetilde{f} (\sigma)
\frac{\mathfrak{m}(\eta,\sigma)}{\eta \sigma}\varphi^l(\eta,\sigma) \,d \eta\, d \sigma \\
& \qquad \pm \frac{1}{4t^2} \int e^{-it\Phi(\eta,\sigma)} \widetilde{f}  (\eta)  \partial_\xi \widetilde{f} (\sigma)
\frac{\mathfrak{m}(\eta,\sigma)}{\eta \sigma}  \partial_\eta \varphi^l(\eta,\sigma)  \,d \eta\, d \sigma + \{ \mbox{similar or simpler terms} \} \\
& = M^{l,1}(t) + M^{l,2}(t) + \{ \mbox{similar or simpler terms} \}.
\end{align*}
Notice that the symbol $\frac{\mathfrak{m}(\eta,\sigma)}{\eta \sigma}$ is bounded, but not necessarily smooth at $\eta=0$ or $\sigma=0$; however, the integrations by parts above do not produce boundary terms since $\widetilde{f}(0)=0$.

By proposition~\ref{propmuquadscal}, $\left| \frac{\mathfrak{m}(\eta,\sigma)}{\eta \sigma} \right| \lesssim \frac{1}{\langle \eta \rangle \langle \sigma \rangle}$. On the support of the low frequency cutoff $\varphi^l$, either $\eta$ or $\sigma$ is $\lesssim t^{- \kappa}$. Hence by Cauchy-Schwarz and \eqref{eqbootstrap}:
$$
|M^{l,1}|\lesssim \frac{1}{t^2}\| \partial_\xi \widetilde f \|_{L^2}^2 \left\|  \frac{1}{\langle \eta \rangle \langle \sigma \rangle}\right\|_{L^2(\min(|\eta|,|\sigma|)\lesssim t^{-\kappa})} \lesssim \epsilon_1^2 t^{-2-\frac{\kappa}{2}+2\alpha}.
$$
The term $M^{l,2}$ can be dealt with identically, after noting that
$$
\left| \partial_\eta \varphi^\ell(\eta,\sigma) \right| \lesssim \frac{\varphi_{\sim -\kappa m}(\eta)}{|\eta|},
$$
and using the Hardy inequality and the fact that $\widetilde f(0)=0$ to bound $\left\| \frac{\widetilde{f}(\eta)}{\eta} \right\|_{L^2}$ by $\| \partial_\xi \widetilde f \|_{L^2}$. This shows \eqref{improved-modulation:bd:Ml} up to choosing $\alpha$ small enough.

\subsection{The resonant term $M^r$} 
We claim that
\begin{equation} \label{improved-modulation:bd:Mr}
|M^r(t)| \lesssim \epsilon_1^2 t^{-2-\frac{\kappa}{4}}.
\end{equation}
To show it, we integrate
\begin{align*}
M^r(t) & = \pm \frac{1}{4t^2} \int e^{-it \Phi(\eta,\sigma)} \partial_\xi \widetilde{f}(\eta) \partial_\xi \widetilde{f}(\sigma) \varphi^h \varphi^{r} \frac{ \mathfrak{m}(\eta,\sigma)}{\eta \sigma}  \,d\eta \, d\sigma \\
& \qquad \pm \frac{1}{4t^2} \int e^{-it \Phi(\eta,\sigma)} \widetilde{f}(\eta) \partial_\xi \widetilde{f}(\sigma) \partial_\eta [ \varphi^h \varphi^{r}] \frac{\mathfrak{m}(\eta,\sigma)}{\eta \sigma}  \,d\eta \, d\sigma \\
& \qquad + \{ \mbox{similar and simpler terms} \}\\
& = M^{r,1}(t) + M^{r,2} (t) + \{ \mbox{similar and simpler terms} \}.
\end{align*}

To bound $M^{r,1}$, we start by estimating the $L^2$ norm of $\varphi^r (\eta,\sigma) \varphi_{k_1} (\eta) \varphi_{k_2}(\sigma)$, which is the square root of the measure of the set $\{ |\Phi| \lesssim 2^{\max(k_1,k_2)}  t^{-\kappa}, \, |\eta| \sim 2^{k_1}, \, |\sigma| \sim 2^{k_2} \}$:
$$
\| \varphi^r(\eta,\sigma) \varphi_{k_1} (\eta) \varphi_{k_2}(\sigma) \|_{L^2_{\eta,\sigma}} \lesssim 
\left\{
\begin{array}{ll}
0 & \mbox{if $2^{k_1} \not \sim 2^{k_2}$} \\
2^{k_1/2} t^{-\frac{\kappa}{2}} & \mbox{if $2^{k_1} \sim 2^{k_2} \gg 1$} \\
t^{- \frac{\kappa}{3}} & \mbox{if $2^{k_1} \sim 2^{k_2} \lesssim 1$}.
\end{array}
\right.
$$
Therefore,
$$
\left\| \frac{\varphi^r(\eta,\sigma)}{\langle \eta \rangle \langle \sigma \rangle} \right\|_{L^2_{\eta,\sigma}} \lesssim \sum_{k_1,k_2} 2^{-k_1-k_2} \|  \varphi^r(\eta,\sigma) \varphi_{k_1} (\eta) \varphi_{k_2}(\sigma) \|_{L^2_{\eta,\sigma}}  \lesssim t^{-\frac{\kappa}{3} }.
$$
Using the bound $\left| \frac{\mathfrak{m}(\eta,\sigma)}{\eta \sigma} \right| \lesssim \frac{1}{\langle \eta \rangle \langle \sigma \rangle}$ and the Cauchy-Schwarz inequality, we obtain
$$
| M^{r,1}(t) | \lesssim t^{-2} \left\| \partial_\xi \widetilde f(\eta)\partial_\xi \widetilde f(\sigma)\right\|_{L^2_{\eta,\sigma}} \left\| \frac{\varphi^r(\eta,\sigma)}{\langle \eta \rangle \langle \sigma \rangle} \right\|_{L^2_{\eta,\sigma}} \lesssim \epsilon_1^2 t^{-2-\frac{\kappa}{3}-2\alpha}
$$

To deal with $M^{r,2}$, it suffices to notice that
$$
| \partial_\eta \varphi^h(\eta,\sigma) | \lesssim \frac{1}{|\eta|}  \qquad \mbox{and} \qquad | \partial_\eta \varphi^{r}(\eta,\sigma) | \lesssim  \langle t \rangle^{\kappa}.
$$
Therefore, the term involving $\partial_\eta \varphi^h$ can be treated with the help of Hardy's inequality, while the term including $\partial_\eta \varphi^{r}$ can be bounded through a further integration by parts in $\eta$. This further integration will gain a factor $\frac{1}{t}$; it will also give factors such as $\frac{\varphi^{h}(\eta,\sigma)}{\eta}$ and $\partial_\eta \frac{\varphi^{h}(\eta,\sigma)}{\eta}$, which only lose small powers of time
$$
\left| \frac{\varphi^{h}(\eta,\sigma)}{\eta} \right| \lesssim t^{\kappa}, \qquad \left| \partial_\eta \frac{\varphi^{h}(\eta,\sigma)}{\eta} \right| \lesssim t^{2\kappa},
$$
and are therefore harmless. Notice that the singularities of $\frac{\mathfrak{m}(\eta,\sigma)}{\eta \sigma}$ at $\eta =0$ or $\sigma=0$ are cancelled by the cutoff function $\varphi^h$. This shows \eqref{improved-modulation:bd:Mr}.

\subsection{The non-resonant term $M^{nr}$} We claim that
\begin{equation} \label{improved-modulation:bd:Mnr}
|M^{nr}|\lesssim \epsilon_1^2 t^{-\frac 54}.
\end{equation}
We treat this term through integration by parts in time, which using Proposition \ref{grebehuppe} yields

\begin{align*}
\int_t^T M^{nr}(s) \,ds & = i \int e^{-iT \Phi(\eta,\sigma)} \widetilde{f}(\eta) \widetilde{f}(\sigma) \frac{\varphi^h (\eta,\sigma)\varphi^{nr}(\eta,\sigma)}{\Phi(\eta,\sigma)} \mathfrak{m}(\eta,\sigma) \,d\eta \, d\sigma \\
& \qquad - i \int e^{-it \Phi(\eta,\sigma)}  \widetilde{f}(\eta) \widetilde{f}(\sigma) \frac{\varphi^h (\eta,\sigma)\varphi^{nr}(\eta,\sigma)}{\Phi(\eta,\sigma)} \mathfrak{m}(\eta,\sigma) \,d\eta \, d\sigma  \\
& \qquad + \int_t^T \int e^{-is \Phi(\eta,\sigma)} \widetilde{\mathcal{D}}(\eta) \widetilde{f}(\sigma) \frac{\varphi^h(\eta,\sigma) \varphi^{nr}(\eta,\sigma)}{\Phi(\eta,\sigma)} \mathfrak{m}(\eta,\sigma) \,d\eta \, d\sigma\,ds\\
& \qquad + \int_t^T \int e^{-is \Phi(\eta,\sigma)} \widetilde{\mathcal{C}^S}(\eta) \widetilde{f}(\sigma) \frac{\varphi^h(\eta,\sigma) \varphi^{nr}(\eta,\sigma)}{\Phi(\eta,\sigma)} \mathfrak{m}(\eta,\sigma) \,d\eta \, d\sigma \,ds\\
& \qquad + \int_t^T \int e^{-is \Phi(\eta,\sigma)} \widetilde{\mathcal{F}}[e^{-it\underline{\mathcal{H}} }\Mod_{\Uu} ](\eta) \widetilde{f}(\sigma) \frac{\varphi^h(\eta,\sigma) \varphi^{nr}(\eta,\sigma)}{\Phi(\eta,\sigma)} \mathfrak{m}(\eta,\sigma) \,d\eta \, d\sigma\,ds\\
& \qquad + \{ \mbox{similar and simpler terms} \}\\
& = N^1 + N^2 + N^3 + N^4 +N^5 + \{ \mbox{similar and simpler terms} \}.
\end{align*}

\medskip

\noindent \underline{The term $N^1$}. To deal with this term, we integrate by parts in $\eta$ and $\sigma$ to obtain
\begin{align*}
N^1 & = \frac{1}{T^2} \int e^{-iT \Phi(\eta,\sigma)} \partial_\xi \widetilde{f}(\eta) \partial_\xi \widetilde{f}(\sigma) \frac{\varphi^h(\eta,\sigma) \varphi^{nr}(\eta,\sigma)}{\Phi(\eta,\sigma)} \frac{\mathfrak{m}(\eta,\sigma)}{\eta \sigma} \,d\eta \, d\sigma \\
& \qquad + \frac{1}{T^2} \int e^{-iT \Phi(\eta,\sigma)} \widetilde{f}(\eta) \partial_\xi \widetilde{f}(\sigma) \partial_\eta \frac{\varphi^h(\eta,\sigma) \varphi^{nr}(\eta,\sigma)}{\Phi(\eta,\sigma)} \frac{\mathfrak{m}(\eta,\sigma)}{\eta \sigma} \,d\eta \, d\sigma \\
& \qquad + \{ \mbox{similar and simpler terms} \} \\
& = N^{1,1} + N^{1,2} + \{ \mbox{similar and simpler terms} \}
\end{align*}
To treat $N^{1,1}$, we use that $\frac{\varphi^h(\eta,\sigma) \varphi^{nr}(\eta,\sigma)}{\Phi(\eta,\sigma)} \lesssim \langle t \rangle^{\kappa}$ as well as $\left| \frac{\mathfrak{m}(\eta,\sigma)}{\eta \sigma} \right| \lesssim \frac{1}{\langle \eta \rangle \langle \sigma \rangle}$ to obtain, thanks to the Cauchy-Schwarz inequality,
$$
|N^{1,1}| \lesssim \frac{1}{T^2} \| \partial_\xi \widetilde{f} \|_{L^2}^2 T^{\kappa} \lesssim \epsilon_1^2 T^{\kappa -2}.
$$
To bound $N^{1,2}$, we use the inequalities
$$
| \partial_\eta \varphi^{h}(\eta,\sigma)| \lesssim t^{\kappa} \qquad \mbox{and} \qquad \left|\partial_\eta \frac{\varphi^{nr}(\eta,\sigma)}{\Phi(\eta,\sigma)} \right| \lesssim t^{2\kappa}.
$$
to obtain
$$
|N^{1,2}| \lesssim \frac{1}{T^2} T^{2 \kappa } \| \partial_\xi \widetilde{f} \|_{L^2} \| \widetilde{f} \|_{L^2} \lesssim \epsilon_1^2 T^{2\kappa + \alpha -2} 
$$

\medskip

\noindent \underline{The term $N_2$.} It can be treated identically.

\medskip

\noindent \underline{The term $N_3$.} To treat this term, we integrate by parts in $\sigma$, which yields 
\begin{align*}
N^3 & =  \pm \frac{1}{2} \int_t^T \frac{1}{s} \int e^{-is \Phi(\eta,\sigma)} \widetilde{\mathcal{D}}(\eta)\partial_\xi \widetilde{f}(\sigma) \frac{\varphi^h(\eta,\sigma) \varphi^{nr}(\eta,\sigma)}{\Phi(\eta,\sigma)} \frac{\mathfrak{m}(\eta,\sigma)}{\sigma} \,d\eta \, d\sigma \\
& \qquad \pm \frac{1}{2} \int_t^T \frac{1}{s} \int e^{-is \Phi(\eta,\sigma)} \widetilde{\mathcal{D}}(\eta)\widetilde{f}(\sigma) \partial_\sigma \frac{\varphi^h(\eta,\sigma) \varphi^{nr}(\eta,\sigma)}{\Phi(\eta,\sigma)} \frac{\mathfrak{m}(\eta,\sigma)}{\sigma} \,d\eta \, d\sigma \\
& \qquad + \{ \mbox{ simpler terms} \} \\
& = N^{3,1} + N^{3,2} + \{ \mbox{similar and simpler terms} \}
\end{align*}

In order to estimate $N^{3,1}$ we use the bounds $\left| \frac{\mathfrak{m}(\eta,\sigma)}{\eta \sigma} \right| \lesssim \frac{1}{\langle \eta \rangle \langle \sigma \rangle}$, the bound $\left| \frac{\varphi^{nr}(\eta,\sigma)}{\Phi(\eta,\sigma)} \right| \lesssim t^{\kappa}$ and the Cauchy-Schwarz inequality to write
$$
N^{3,1} \lesssim \int_t^T \frac{1}{s} \| \widetilde{\mathcal{D}} \|_{L^2} \| \partial_\xi \widetilde{f} \|_{L^2} t^{\kappa } \,ds.
$$
By \eqref{bd:mathcalD-L21}, this can further be bounded by
$$
N^{3,1} \lesssim \epsilon_1^3 \int_t^T s^{-1 - \frac{3}{2} +\alpha+\kappa}\,ds \lesssim \epsilon_1^3 t^{-\frac{3}{2} +  \alpha+\kappa}.
$$
The term $N^{3,2}$ can be bounded similarly, using this time that $\left| \partial_\sigma \frac{\varphi^h(\eta,\sigma) \varphi^{nr}(\eta,\sigma)}{\Phi(\eta,\sigma)} \right| \lesssim s^{2 \kappa}$.

\medskip

\noindent \underline{The term $N_4$: terms of $\delta$ type} The singular cubic term is made up of terms whose kernel involves a $\delta$ function, and terms whose kernel involves a principal value. We start with the former kind, and will turn to the latter in the next paragraph. The contribution of terms involving a $\delta$ function is a sum of terms of the type
\begin{align*}
&N_\delta =  \int_t^T \int e^{-is \Psi(\sigma,\eta',\sigma',\zeta')} \widetilde{f}(\eta') \widetilde{f}(\sigma') \widetilde{f}(\zeta') \widetilde{f}(\sigma) \\
 &  \qquad  \qquad  \qquad  \qquad \qquad \qquad \frac{\varphi^h(\eta,\sigma) \varphi^{nr}(\eta,\sigma)}{\Phi(\eta,\sigma)} \mathfrak{n}(\eta,\eta',\sigma',\zeta') \mathfrak{m}(\eta,\sigma) \,d\eta' \, d\sigma' \, d\zeta' \, d\sigma \,ds,
\end{align*}
where
$$
\Psi(\sigma,\eta',\sigma',\zeta') = \pm (\uo + \sigma^2) \pm (\uo + (\eta')^2) \pm (\uo + (\sigma')^2) \pm (\uo + (\zeta')^2),
$$
and where it is understood that $\eta = \eta'+\sigma'+\zeta'$; though other sign combinations are possible, they will be ignored to simplify notations. Notice the new factor $\mathfrak{n}(\eta,\eta',\sigma',\zeta')$ stemming from the cubic spectral distribution; it has all its derivatives bounded, except when one of the frequencies $\eta,\eta',\sigma',\zeta'$ vanishes, in which case it might be discontinuous. 

We now introduce new cutoff functions, $\varphi^H$ and $\varphi^L$, which are defined by
$$
\varphi^L(t,\xi) = \varphi_{<0}(t^{1/2} \xi) \quad \mbox{and} \quad \varphi^H(t,\xi) = \varphi_{\geq 0}(t^{1/2} \xi),
$$
and use them to split the contributions of the functions $\widetilde{f}(\eta')$, $\widetilde{f}(\sigma')$, and $\widetilde{f}(\zeta')$: we will denote
\begin{equation}
\label{foulque1}
\widetilde{f}(\xi) = \widetilde{f^L}(\xi) + \widetilde{f^H}(\xi) = \varphi^L(t,\xi) \widetilde{f}(\xi) + \varphi^H(t,\xi) \widetilde{f}(\xi).
\end{equation}
Note that
\begin{equation}
\label{foulque2}
\begin{split}
& \| \widetilde{f^L} \|_{L^1} \lesssim \epsilon_1 t^{-\frac{3}{4}+\alpha}, \quad  \| \widetilde{f^L} \|_{L^2} \lesssim \epsilon_1t^{-1/2+\alpha}, \\
&  \| \partial_\xi \widetilde{f^H} \|_{L^2} \lesssim \epsilon_1 t^\alpha, \quad \left\| |\xi|^{-1} \partial_\xi \widetilde{f^H} \right\|_{L^1} \lesssim \epsilon_1 t^{\frac{1}{4} + \alpha}, \quad \left\| \frac{\langle \xi \rangle}{|\xi|} \partial_\xi \widetilde{f^H} \right\|_{L^2} \lesssim \epsilon_1 t^{\frac 12 + \alpha};
\end{split}
\end{equation}
this follows from the bootstrap hypothesis and the fact that $\widetilde{f}(0) =0$.

Abandoning, for the sake of clarity in the notation, the dependence of the phase and symbols on the variables, the above becomes
\begin{align*}
N_\delta & = \int_t^T \int e^{-is \Psi} \widetilde{f^L}(\eta') \widetilde{f^L} (\sigma') \widetilde{f^L}(\zeta') \widetilde{f}(\sigma)  \frac{\varphi^h \varphi^{nr}}{\Phi} \mathfrak{n} \mathfrak{m} \,d\eta' \, d\sigma' \, d\zeta' \, d\sigma \,ds \\ 
& \qquad \qquad +  \int_t^T \int e^{-is \Psi} \widetilde{f^H}(\eta') \widetilde{f^H} (\sigma') \widetilde{f^H}(\zeta') \widetilde{f}(\sigma)  \frac{\varphi^h \varphi^{nr}}{\Phi} \mathfrak{n} \mathfrak{m} \,d\eta' \, d\sigma' \, d\zeta' \, d\sigma \,ds \\
& \qquad \qquad + \{ \mbox{cross terms} \} \\
& = N^{L}_\delta + N^{H}_\delta + \{ \mbox{cross terms}\} ,
\end{align*}
where the cross terms involve both high- and low-frequency contributions, and can be dealt with following the pattern established below for $N^{L}_\delta$ and $N^H_\delta$.

In order to bound $N^L_\delta$, we integrate by parts in $\sigma$ to obtain
$$
N^L_\delta = \pm \frac{1}{2} \int_t^T \frac{1}{s} \int e^{-is \Psi} \widetilde{f^L}(\eta') \widetilde{f^L} (\sigma') \widetilde{f^L}(\zeta') \partial_\xi \widetilde{f}(\sigma)  \frac{\varphi^h \varphi^{nr}}{\Phi} \frac{\mathfrak{n} \mathfrak{m}}{\sigma} \,d\eta' \, d\sigma' \, d\zeta' \, d\sigma \,ds + \{ \mbox{easier terms} \},
$$
Using the Cauchy-Schwarz inequality, the support condition on $\widetilde{f^L}$ and the bounds $\left| \frac{\mathfrak n \mathfrak{m}}{\sigma} \right| \lesssim \frac{1}{\langle \sigma \rangle}$ and $ \left| \frac{\varphi^h \varphi^{nr}}{\Phi} \right| \lesssim t^{\kappa}$, the leading term in $N^L_\delta$ can be bounded by
$$
\int_t^T s^{-1} s^{\kappa} \| \widetilde{f^L} \|_{L^1}^3 \| \partial_\xi \widetilde{f} \|_{L^2} \,d\sigma \lesssim \epsilon_1^4 \int_t^T s^{-\frac{13}{4}+4 \alpha+\kappa} \,ds \lesssim \epsilon_1^4 t^{-\frac{9}{4} +4 \alpha+\kappa}.
$$
The "easier term" in $N^L_\delta$ corresponds to the case where the derivative $\partial_\sigma$ hits the symbol; it is easily estimated thanks to the bound $\left| \partial_\sigma \frac{\varphi^h \varphi^{nr}}{\Phi} \frac{\mathfrak{n} \mathfrak{m}}{\sigma} \right| \lesssim t^{2\kappa}$.

In order to bound $N^H_\delta$, we integrate by parts in all the integration frequency variables, which yields
\begin{align*}
N^H_\delta & = \pm \frac{1}{2} \int_t^T \frac{1}{s^4} \int e^{-is \Psi} \partial_\xi \widetilde{f^H}(\eta') \partial_\xi \widetilde{f^H} (\sigma')  \partial_\xi\widetilde{f^H}(\zeta') \partial_\xi   \widetilde{f}(\sigma)  \frac{\varphi^h \varphi^{nr}}{\Phi} \frac{\mathfrak{n} \mathfrak{m}}{\sigma \eta' \sigma' \zeta'} \,d\eta' \, d\sigma' \, d\zeta' \, d\sigma \,ds \\
& \qquad \qquad + \{ \mbox{easier terms} \},
\end{align*}
By the arguments employed to bound $N^L_\delta$, the leading term in $N^H_\delta$ can be bounded by
$$
 \int_t^T s^{-4} s^{\kappa} \| |\xi|^{-1} \partial_\xi \widetilde{f^H} \|_{L^1}^3 \| \langle \xi \rangle^{-1} \partial_\xi \widetilde{f} \|_{L^1} \,ds \lesssim \epsilon_1^4 \int_t^T s^{-\frac{13}{4} + 4\alpha} \,ds \lesssim \epsilon_1^4 t^{-\frac{9}{4} + 4 \alpha+\kappa}.
$$
The "easier terms" in $N^H_\delta$ once again correspond to the case where the derivatives hit the symbol, and they are easily estimated.

\medskip

\noindent \underline{The term $N_4$: terms of $\operatorname{p.v.}$ type}. Such terms are given by expressions of the type
\begin{align*}
N_{\operatorname{p.v.}} = & \int_t^T \int e^{-is \Psi(\sigma,\eta',\sigma',\zeta')} \widetilde{f}(\eta') \widetilde{f}(\sigma') \widetilde{f}(\zeta') \widetilde{f}(\sigma)\\
& \qquad \quad  \frac{\varphi^h(\eta,\sigma) \varphi^{nr}(\eta,\sigma)}{\Phi(\eta,\sigma)} \mathfrak{n}(\eta,\eta',\sigma',\zeta') \mathfrak{m}(\eta,\sigma) \frac{\widehat{\phi}(\eta-\eta'-\sigma'-\zeta')}{\eta-\eta'-\sigma'-\zeta'} \,d\eta \,d\eta' \, d\sigma' \, d\zeta' \, d\sigma \,ds.
\end{align*}
We observe first that $\mathfrak{n}$ is given by a tensor product: it consists of multipliers which can be absorbed by the input functions $\widetilde{f}$ without modifying the estimates they satisfy; therefore, we will consider in the following that $\mathfrak{n} \equiv 1$.
Proceeding as in the previous paragraph, this can be split into
\begin{align*}
N_{\operatorname{p.v.}} & = \int_t^T \int e^{-is \Psi} \widetilde{f^L}(\eta') \widetilde{f^L}(\sigma') \widetilde{f^L}(\zeta') \widetilde{f}(\sigma) \frac{\varphi^h \varphi^{nr}}{\Phi} \mathfrak{m} \frac{\widehat{\phi}(\eta-\eta'-\sigma'-\zeta')}{\eta-\eta'-\sigma'-\zeta'} \,d\eta \,d\eta' \, d\sigma' \, d\zeta' \, d\sigma \,ds \\
&  \quad + \int_t^T \int e^{-is \Psi} \widetilde{f^H}(\eta') \widetilde{f^H}(\sigma') \widetilde{f^H}(\zeta') \widetilde{f}(\sigma) \frac{\varphi^h \varphi^{nr}}{\Phi} \mathfrak{m} \frac{\widehat{\phi}(\eta-\eta'-\sigma'-\zeta')}{\eta-\eta'-\sigma'-\zeta'} \,d\eta \,d\eta' \, d\sigma' \, d\zeta' \, d\sigma \,ds \\
&  \quad + \{ \mbox{cross terms} \} \\
& = N^L_{\operatorname{p.v.}} + N^H_{\operatorname{p.v.}} + \{ \mbox{cross terms} \}.
\end{align*}

To estimate $N^{L}_{\operatorname{p.v.}}$, we need to introduce some further notation: let $\mathcal{H}$ denote the truncated Hilbert transform with kernel $\frac{\widehat{\phi}(\xi)}{\xi}$ (which is bounded on $L^2$), and let furthermore $\Theta$ denote the phase $\Psi$ from which the summand $|\eta'|^2$ is removed. Integrating by parts in $\sigma$, the term $N^L_{\operatorname{p.v.}}$ can be written
\begin{align*}
N^L_{\operatorname{p.v.}} & = \pm \frac{1}{2} \int_t^T \frac{1}{s} \int e^{-is \Theta} [ \mathcal{H} e^{\pm i t |\cdot|^2} \widetilde{f^L}] (\eta - \sigma' -\zeta') \widetilde{f^L}(\sigma') \widetilde{f^L}(\zeta') \partial_\xi \widetilde{f}(\sigma) \frac{\varphi^h \varphi^{nr}}{\Phi} \frac{\mathfrak{m}}{\sigma} \,d\eta \, d\sigma' \, d\zeta' \, d\sigma \,ds \\
& \qquad + \{ \mbox{easier terms} \}.
\end{align*}
By the same arguments as in the previous paragraph, and using furthermore that, since $ \widetilde{f^L}$ has compact support and $\widehat{\phi}$ decays rapidly,
\begin{equation}
\label{foulque3}
\left\| \mathcal{H} e^{\pm i t |\cdot|^2} \widetilde{f^L} \right\|_{L^1} \lesssim \left\|   \widetilde{f^L} \right\|_{L^2},
\end{equation}

the leading term in the above right-hand side can be estimated by
\begin{align*}
\int_t^T s^{-1} s^{\kappa} \| \mathcal{H}e^{\pm i t |\cdot|^2} \widetilde{f^L} \|_{L^1} \|  \widetilde{f^L} \|_{L^1}^2 \| \partial_\xi \widetilde{f} \|_{L^2} \,ds & \lesssim \int_t^T s^{-1 + \kappa} \|  \widetilde{f^L} \|_{L^2} \|  \widetilde{f^L} \|_{L^1}^2 \| \partial_\xi \widetilde{f} \|_{L^2} \,ds\\
& \lesssim \epsilon_1^4 \int_t^T s^{-3+4 \alpha+\kappa} \,ds \lesssim \epsilon_1^4 t^{-2 + 4 \alpha+\kappa},
\end{align*}

while the details for the "easier terms" are omitted.

To estimate $N^H_{\operatorname{p.v.}}$, we integrate by parts in $\sigma,\eta',\sigma',\zeta'$, which leads to
\begin{align*}
N^H_{\operatorname{p.v.}} & = \pm \frac{1}{8} \int_t^T \frac{1}{s^4} \int e^{-is \Theta} [ \mathcal{H} (\cdot)^{-1} e^{\pm it |\cdot|^2} \partial_\xi \widetilde{f^H}] (\eta - \sigma' -\zeta')\partial_\xi \widetilde{f^H}(\sigma')\partial_\xi  \widetilde{f^H}(\zeta') \partial_\xi \widetilde{f}(\sigma) \\
&\qquad\qquad\qquad\qquad \frac{\varphi^h \varphi^{nr}}{\Phi} \frac{\mathfrak{m}}{\sigma' \zeta' \sigma} \,d\eta \, d\sigma' \, d\zeta' \, d\sigma \,ds + \{ \mbox{easier terms} \}.
\end{align*}
By the same arguments as in the previous paragraph and~\eqref{foulque2}, the leading term in the above right-hand side can be estimate by
$$
 \int_t^T s^{-4} s^{\kappa} \| \mathcal{H} (\cdot)^{-1} e^{\pm it |\cdot|^2} \partial_\xi \widetilde{f^H} \|_{L^1} \| |\xi|^{-1} \partial_\xi \widetilde{f^H} \|_{L^1}^2 \| \langle \xi \rangle^{-1} \partial_\xi \widetilde{f} \|_{L^1} \,ds.
$$
The truncated Hilbert transform is not bounded on $L^1$, but, since its integration kernel is rapidly decaying, it is bounded on weighted $L^2$ spaces, which allows to estimate, by~\eqref{foulque2},
\begin{equation}
\label{foulque4}
\begin{split}
\| \mathcal{H} (\cdot)^{-1} e^{\pm it |\cdot|^2} \partial_\xi \widetilde{f^H} \|_{L^1}& \lesssim \| \mathcal{H} (\cdot)^{-1} e^{\pm it |\cdot|^2} \partial_\xi \widetilde{f^H} \|_{L^2(\langle \cdot \rangle^2 dx)} \\
& \lesssim \| (\cdot)^{-1} e^{\pm it |\cdot|^2} \partial_\xi \widetilde{f^H} \|_{L^2(\langle \cdot \rangle^2 dx)} \lesssim \epsilon_1 t^{\frac{1}{2}+\alpha}.
\end{split}
\end{equation}
Coming back to the expression above, we get
\begin{align*}
& \int_t^T s^{\kappa-4} \| \mathcal{H} (\cdot)^{-1} e^{\pm it |\cdot|^2} \partial_\xi \widetilde{f^H} \|_{L^1} \| |\xi|^{-1} \partial_\xi \widetilde{f^H} \|_{L^1}^2 \| \langle \xi \rangle^{-1} \partial_\xi \widetilde{f}\|_{L^1} \,ds \lesssim \epsilon_1^4 t^{-2 + \kappa}.
\end{align*}

\noindent \underline{The term $N_5$}. This term is very similar to $N_1$. By Proposition~\ref{propModU},
$$
e^{- it \rho (\uo+ \xi^2)} \widetilde{\Mod_{\Uu}}  =  \tau(t) \xi \widetilde f(\xi) + \widetilde{ \mathcal{M}}(\xi).
$$
Inserting this formula in the expression for $N_5$ and integrating by parts in $\eta$ and $\sigma$ gives
\begin{align*}
N^5 & = \int_t^T \frac{ds}{s^2} \int e^{-is \Phi(\eta,\sigma)} i \tau \partial_\xi \widetilde{f}(\eta) \partial_\xi \widetilde{f}(\sigma) \frac{\varphi^h(\eta,\sigma) \varphi^{nr}(\eta,\sigma)}{\Phi(\eta,\sigma)} \frac{\mathfrak{m}(\eta,\sigma)}{ \sigma} \,d\eta \, d\sigma \\
& \qquad + \int_t^T \frac{ds}{s^2} \int e^{-is \Phi(\eta,\sigma)} \left[ \partial_\xi  \widetilde{\mathcal{M}}(\eta)  + \tau \widetilde{f}(\eta) \right] \partial_\xi \widetilde{f}(\sigma) \frac{\varphi^h(\eta,\sigma) \varphi^{nr}(\eta,\sigma)}{\Phi(\eta,\sigma)} \frac{\mathfrak{m}(\eta,\sigma)}{\eta \sigma} \,d\eta \, d\sigma \\
& \qquad + \int_t^T \frac{ds}{s^2}  \int e^{-is \Phi(\eta,\sigma)} \left[ \widetilde{\mathcal{M}}(\eta)  + \tau \widetilde{f}(\eta) \right] \partial_\xi \widetilde{f}(\sigma) \partial_\eta \frac{\varphi^h(\eta,\sigma) \varphi^{nr}(\eta,\sigma)}{\Phi(\eta,\sigma)} \frac{\mathfrak{m}(\eta,\sigma)}{\eta \sigma} \,d\eta \, d\sigma \\
& \qquad + \{ \mbox{similar and simpler terms} \} \\
& = N^{5,1} + N^{5,2} + N^{5,3} + \{ \mbox{similar and simpler terms} \}
\end{align*}
To treat $N^{5,1}$, we use that $\frac{\varphi^h(\eta,\sigma) \varphi^{nr}(\eta,\sigma)}{\Phi(\eta,\sigma)} \lesssim \langle t \rangle^{\kappa}$ as well as $\left| \frac{\mathfrak{m}(\eta,\sigma)}{ \sigma} \right| \lesssim \frac{1}{ \langle \sigma \rangle \langle \eta - \sigma \rangle^2}$ and the decay of $\tau(s)$ to obtain, thanks to the Cauchy-Schwarz inequality,
$$
|N^{5,1}| \lesssim \int_t^T  \| \partial_\xi \widetilde{f} \|_{L^2}\| \partial_\xi \widetilde{f} \|_{L^2} s^{\kappa} s^{-1-\nu} \,\frac{ds}{s^2} \lesssim \epsilon_1^2 t^{\kappa+2\alpha -2-\nu}.
$$
The term $N^{5,2}$ can be treated almost identically, using the decay of $\partial_\xi \widetilde{M}$ instead of that of $\tau$, and the bound $\left| \frac{\mathfrak{m}(\eta,\sigma)}{\eta \sigma} \right| \lesssim \frac{1}{\langle \eta \rangle \langle \sigma \rangle \langle \eta - \sigma \rangle^2}$. Finally, the term $N^{5,3}$ is once again of the same type, the difference being an additional loss of $t^\kappa$ stemming from the bound $|\partial_\eta \frac{\varphi^{nr}}{\Phi} | \lesssim t^{2\kappa}$. This gives
$$
|N^{5,3}| \lesssim \epsilon_1^2 t^{2\kappa+2\alpha -2-\nu},
$$
and we find overall
$$
|N^{5}|\lesssim \epsilon_1^2 t^{2\kappa+2\alpha -2-\nu}.
$$

\section{Weighted estimates for the quadratic terms}

\label{sectionquadratic}
\label{SectionWeightedQuadratic}

The aim of this section will be to prove the following proposition.

\begin{proposition}
\label{PropositionWeightedQuadratic}
Under the assumptions of Proposition~\ref{propbootstrap}, the regular quadratic interaction term $\mathcal{Q}^R$ can be bounded by
\begin{equation} \label{bd:mathcalQR-H1}
\left\| \partial_\xi \int_0^t e^{i \theta(s) \xi} \widetilde{\mathcal{Q}^R}(\xi) \,ds \right\|_{L^2} \lesssim \epsilon_1^2.
\end{equation}
if $0\leq t <T$.
\end{proposition}

\subsection{Preliminary steps} Expanding the expression in the right-hand side of~\eqref{bd:mathcalQR-H1} through Leibniz' rule, the term where $\partial_\xi$ hits $e^{i \theta(s) \xi}$ is much easier to control than the term where $\widetilde{\mathcal{Q}^R}$ is differentiated. Therefore, we will simply prove
$$
\left\|  \int_0^t e^{i \theta(s) \xi} \partial_\xi \widetilde{\mathcal{Q}^R}(\xi) \,ds \right\|_{L^2} \lesssim \epsilon_1^2.
$$ 

From the definition \eqref{defQR} of $\mathcal{Q}_R$, we have to estimate
\begin{align}
\label{decompositionQR} \partial_\xi \widetilde{\mathcal{Q}^R}_\rho(\xi) & = -it 2 \rho \xi \sum \int e^{-it\Phi_{\lambda \mu  \rho }(\xi,\eta,\sigma)}  \widetilde{f}_{\lambda}  (\eta) \widetilde{f}_{\mu} (\sigma)
\mathfrak{m}^V_{jkl, \lambda \mu \rho } (\xi, \eta, \sigma) \,d \eta\, d \sigma \\
\nonumber & + \sum \int e^{-it\Phi_{\lambda \mu  \rho }(\xi,\eta,\sigma)}  \widetilde{f}_{\lambda}  (\eta) \widetilde{f}_{\mu} (\sigma)
\partial_\xi \mathfrak{m}^V_{jkl, \lambda \mu \rho } (\xi, \eta, \sigma) \,d \eta\, d \sigma 
\end{align}
where
$$
\Phi_{\lambda \mu  \rho} =  \rho(\uo+ \xi^2) - \lambda  (\uo+ \eta^2)  - \mu  (\uo+ \sigma^2)
$$
and where we recall \eqref{hibou2}
\begin{equation} \label{hibou10}
  \left| \partial_\xi^a \partial_\eta^b \partial_\xi^c \frac{\mathfrak{m}}{\eta \sigma} \right|  \lesssim \frac{1}{\langle \eta \rangle \langle \sigma \rangle}  \sum_{\pm} \frac{1}{\langle \xi \pm \eta \pm \sigma \rangle^2}.
\end{equation}
We start by estimating the first term in the right-hand side of \eqref{decompositionQR}, which is the worst due to the $\xi t$ prefactor and as the symbol is smooth. The second term is much easier to deal with and will be considered after. Omitting indices, we shall simply write
$$
I(t,\xi) = t \xi   \int e^{-it\Phi(\xi,\eta,\sigma)}  \widetilde{f}  (\eta) \widetilde{f} (\sigma)
\mathfrak{m} (\xi, \eta, \sigma) \,d \eta\, d \sigma
$$
and aim at proving
$$
\left\| \int_0^t I(s,\xi) \,ds \right\|_{L^2} \lesssim \epsilon_1^2.
$$
For $|t|<1$, this bound follows from Proposition~\ref{grebehuppe} (or, more precisely, its proof). Therefore, it will suffice to prove that
\begin{equation} \label{aimI}
\left\| \int_1^t I(s,\xi) \,ds \right\|_{L^2} \lesssim \epsilon_1^2.
\end{equation}

Before delving into the estimates, let us quickly count powers to understand where the difficulty lies. 
\begin{itemize}
\item The prefactor $\xi$ gives a derivative loss, which will have to be recovered, either through integration by parts (in $\eta$, $\sigma$, $t$), or by showing that an analog of the Leibniz rule applies (due to the localization properties of the quadratic spectral distribution).
\item Counting powers of $t$, we see that integrations by parts in $\eta$ and $\sigma$ would gain a factor $t^{-2}$, but would lead to using the bound for $\| \partial_\xi \widetilde{f} \|_{L^2}$ twice, which loses $t^{2\alpha}$. Overall, the powers of $t$ become $t^{1-2+2\alpha}$, which, after integrating in time, fails to recover $t^\alpha$. A small integrability gain (in $t$) is needed, which will be achieved by considering resonances.
\end{itemize}

Just like we did for $M_0$ above, we will split $I(t,\xi)$ into three different pieces thanks to various cutoff functions. First, the low and high frequency cutoff functions are given by
$$
\varphi^{h}(\eta,\sigma) = \sum_{m \geq 0} \varphi_m(t) \varphi_{> -\kappa m} (\eta) \varphi_{> -\kappa m} (\sigma), \quad \mbox{and} \quad
\varphi^{l}(\eta,\sigma) = 1 - \varphi^{h}(\eta,\sigma),
$$
where $\kappa>0$ is a small universal constant. On the support of $\varphi^l$, either $|\eta|\lesssim t^{-\kappa}$ or $|\sigma|\lesssim t^{-\kappa}$ is small.

The resonant and non-resonant cutoff functions are defined by
$$
\varphi^r(\xi,\eta,\sigma) = \sum_{j,m} \varphi_m(t) \varphi_j(\xi) \varphi_{k_1}(\eta) \varphi_{k_2}(\sigma) \varphi_0(2^{-J} \Phi)  \quad \mbox{and} \quad \varphi^{nr}(\xi,\eta,\sigma) = 1 - \varphi^r(\xi,\eta,\sigma).
$$
with
$$
J = \max(j,k_1,k_2) -\kappa m.
$$
On the support of $\varphi^{r}$ the phase is small $|\Phi|\lesssim (\la \xi \ra+\la \eta \ra+\la \sigma \ra)t^{-\kappa} $.

This gives a decomposition of $I$ into a low-frequency part $I^l$, a resonant part $I^r$, and a non-resonant part $I^{nr}$
\begin{align*}
I (t,\xi) & = t \xi   \int \varphi^l(\eta,\sigma) e^{-it\Phi(\xi,\eta,\sigma)}  \widetilde{f}  (\eta) \widetilde{f} (\sigma)
\mathfrak{m}(\xi, \eta, \sigma)  \,d \eta\, d \sigma  \\
 & \quad \quad \quad + t \xi   \int \varphi^h(\eta,\sigma) \varphi^r(\xi,\eta,\sigma) e^{-it\Phi(\xi,\eta,\sigma)}  \widetilde{f}  (\eta) \widetilde{f} (\sigma)
\mathfrak{m}(\xi, \eta, \sigma)   \,d \eta\, d \sigma \\
& \quad \quad \quad +  t \xi \int \varphi^h(\eta,\sigma) \varphi^{nr}(\xi,\eta,\sigma) e^{-it\Phi(\xi,\eta,\sigma)}  \widetilde{f}  (\eta) \widetilde{f} (\sigma)
\mathfrak{m}(\xi, \eta, \sigma)   \,d \eta\, d \sigma \\
& = I^l(t,\xi) + I^r(t,\xi) + I^{nr}(t,\xi).
\end{align*}

We will show in \eqref{improved-quadratic-Il} that for all $\kappa>0$, $\| I^l(t,\cdot) \|_{L^2} \lesssim  \epsilon_1^2 t^{-1-\kappa/4}$, in \eqref{improved-quadratic-Ir} that for $\kappa $ small enough $\|I^r(t,\cdot) \|_{L^2}\lesssim  \epsilon_1^2 t^{-1-\kappa/4}$, and in \eqref{improved-quadratic-Inr} that for all $\kappa$ small enough, $\| \int_1^t \| I^{nr}(s,\cdot)ds \|_{L^2}\lesssim \epsilon_1^2$ provided $\alpha $ is small enough. This will prove \eqref{aimI}. Note that our cut-off parameter $\kappa$ in this section is different, and can be chosen independently, from the one of Section \ref{sectionmodulation}. Here, we shall choose $0<\kappa\ll \nu$.

\subsection{The low-frequency term} We claim that
\begin{equation} \label{improved-quadratic-Il}
\left\| I^l(t,\cdot) \right\|_{L^2} \lesssim \epsilon_1^2 t^{-1-\frac \kappa 4 }.
\end{equation}

To prove this estimate, we integrate by parts in $\eta$ and $\sigma$ to obtain
\begin{align*}
I^l(t,\xi) & = \pm \frac{1}{4t} \int e^{-it\Phi}  \partial_\xi \widetilde{f}  (\eta)  \partial_\xi \widetilde{f} (\sigma)
\frac{\xi \mathfrak{m}}{\eta \sigma} \varphi^l \,d \eta\, d \sigma \\
& \qquad \pm \frac{1}{4t} \int  e^{-it\Phi} \widetilde{f}  (\eta)  \partial_\xi \widetilde{f} (\sigma)
\frac{\xi \mathfrak{m}}{\eta \sigma}  \partial_\eta \varphi^l \,d \eta\, d \sigma + \{ \mbox{similar or simpler terms} \} \\
& = I^{l,1}(t,\xi) + I^{l,2}(t,\xi) + \{ \mbox{similar or simpler terms} \}.
\end{align*}
Notice that the symbol $\frac{\xi \mathfrak{m}}{\eta \sigma}$ is bounded, but not necessarily smooth at $\eta=0$ or $\sigma=0$; however, the integrations by parts above do not produce boundary terms since $\widetilde{f}(0)=0$.

Considering $I^{l,1}$ first, it is the sum of two terms, for which $|\eta|$ and $|\sigma|$ respectively are $\lesssim t^{-\kappa}$. Since both terms are symmetric, we only treat the former case, namely $|\eta|\lesssim t^{-\kappa}$. On the integrand of this term, by \eqref{hibou10}
$$
\left| \frac{\xi \mathfrak{m}}{\eta \sigma} \right| \lesssim \frac{|\xi|}{\langle \sigma \rangle} \sum_{\pm} \frac{1}{\langle \xi \pm \sigma \rangle^2} \lesssim \frac{1}{\langle \sigma \rangle \langle \xi \rangle}  + \frac{1}{\langle \xi \pm \sigma \rangle^2}.
$$
Therefore, it can be bounded by
\begin{align*}
& \left\|  \frac{1}{4t} \int \varphi_{\lesssim -\kappa m}(\eta) \varphi^l e^{-it\Phi}  \partial_\xi \widetilde{f}  (\eta)  \partial_\xi \widetilde{f} (\sigma)
\frac{\xi \mathfrak{m}}{\eta \sigma} \,d \eta\, d \sigma \right\|_{L^2} \\
& \qquad \lesssim\frac 1t  \left\| \int |\varphi_{\lesssim -\kappa m}(\eta) |\, |\partial_\xi \widetilde{f}  (\eta) | \, | \partial_\xi \widetilde{f} (\sigma)| \left[ \frac{1}{\langle \sigma \rangle \langle \xi \rangle} + \frac{1}{\langle \xi \pm \sigma \rangle^2} \right] \,d \eta\, d \sigma \right\|_{L^2_\xi} \\
 & \qquad \lesssim \frac 1t \| \varphi_{\lesssim -\kappa m} \partial_\xi \widetilde{f} \|_{L^1} \| \partial_\xi \widetilde{f} \|_{L^2} \\
& \qquad \lesssim \frac 1t t^{-\frac \kappa 2} \| \partial_\xi \widetilde{f} \|_{L^2}^2 \\
& \qquad \lesssim t^{-1-\frac{\kappa}{2}+2\alpha} \epsilon_1^2.
\end{align*}
The second inequality above can be justified by the Minkowski, Young and Cauchy-Schwarz inequalities.

The term $I^{l,2}$ can be treated just like $I^{l,1}$ after noting that, thanks to Hardy's inequality,
$$
\| [ \partial_\eta \varphi_{>-\kappa m} (\eta) \widetilde{f}(\eta) \|_{L^2} \lesssim \left\| \varphi_{\sim -\kappa m } (\eta) \frac{1}{|\eta|}  \widetilde{f}(\eta) \right\|_{L^2} \lesssim \| \partial_\xi \widetilde{f}(\eta) \|_{L^2}.
$$

\subsection{The resonant term} We claim that
\begin{equation} \label{improved-quadratic-Ir}
\| I^r(t,\xi) \|_{L^2}  \lesssim \epsilon_1^2t^{-1-\frac{\kappa}{4}}.
\end{equation}

To obtain this bound, it will be convenient to split $I^r$ via a dyadic decompositon into
$$
I^{r}(t,\xi) = \sum_{jk_1 k_2 m} I_{jk_1k_2}^r (t,\xi)
$$
with
\begin{align*}
&I^{r}_{jk_1k_2}(t,\xi) = t \xi   \int e^{-it\Phi(\xi,\eta,\sigma)}  \widetilde{f}  (\eta) \widetilde{f} (\sigma)
\mathfrak{m}(\xi, \eta, \sigma) \varphi_{<J}(\Phi(\xi,\eta,\sigma)) \varphi^h(\eta,\sigma) \mathfrak{p}_{jk_1k_2}(\xi,\eta,\sigma) \,d \eta\, d \sigma \\
& \mathfrak{p}_{jk_1k_2}(\xi,\eta,\sigma) = \varphi_j(\xi) \varphi_{k_1}(\eta) \varphi_{k_2}(\sigma)
\end{align*}
Without loss of generality, we can assume that $2^{k_1} \geq 2^{k_2}$; we will distinguish two cases depending on the relative sizes of $2^{j}$ and $2^{k_1}$.

\medskip

\noindent \underline{Case 1: $2^j \gtrsim 2^{k_1}$.} This condition can immediately be improved to $2^j \sim 2^{k_1} \geq 2^{k_2}$, since otherwise $\Phi \gg 2^J$. After integrating by parts in $\eta$ and $\sigma$, $I_{jk_1k_2}^r(t,\xi)$ can be written
\begin{equation}
\label{macareux}
\begin{split}
I_{jk_1k_2}^r(t,\xi) & = \pm \frac{1}{t} \int e^{-it\Phi}  \partial_\xi \widetilde{f}  (\eta)  \partial_\xi \widetilde{f} (\sigma)
\frac{\xi \mathfrak{m} \mathfrak{p}}{\eta \sigma} \varphi_{< J}(\Phi) \varphi^h \,d \eta\, d \sigma \\
& \qquad \pm \frac{1}{t} \int e^{-it\Phi}  \widetilde{f}  (\eta)  \partial_\xi \widetilde{f} (\sigma)
\frac{\xi \mathfrak{m} \mathfrak{p}}{\eta \sigma} \partial_\eta [ \varphi_{< J}(\Phi)] \varphi^h \,d \eta\, d \sigma + \{ \mbox{similar or simpler terms} \} \\
& =  I_{jk_1k_2}^{r,1}(t,\xi) + I_{jk_1k_2}^{r,2}(t,\xi)   + \{ \mbox{similar or simpler terms} \}.
\end{split}
\end{equation}
(where we omitted the dependence of $\mathfrak{m}$, $\mathfrak{p}$, and $\Phi$ on $\xi,\eta,\sigma$ for simplicity). In order to bound  $I_{jk_1k_2}^{r,1}$ in $L^2$, we apply Schur's test to the kernel
$$
K(\xi,\eta) = \int e^{-it\Phi}  \partial_\xi \widetilde{f} (\sigma) \frac{\xi \mathfrak{m} \mathfrak{p}}{\eta \sigma} \varphi_{< J}(\Phi) \varphi^h \,d\sigma.
$$
By Proposition~\ref{propmuquad}, and since $\langle \xi \rangle \sim \langle \eta \rangle$, there holds $\left|  \frac{\xi \mathfrak{m} \mathfrak{p}}{\eta \sigma} \right| \lesssim 2^{-k_2}$. Furthermore, we can make use of the following estimate, uniform in $r,R \in \mathbb{R}$ and $k\in \mathbb N$:
\begin{equation}
\label{boundRr}
\left| \{ x \in \mathbb{R}, \quad |x^2-R| < r, \quad |x| \sim 2^k \} \right| \lesssim \min\left( 2^{-k} r, \ \sqrt{r}\right)
\end{equation}
to bound the size of the set of $\eta$ of size $\sim 2^{k_1}$ such that $|\Phi| < 2^J$ (for fixed $\xi$ and $\sigma$) by $O\left(\min \left(2^{J-k_1}, \ 2^{J/2}\right)\right)=O(t^{-\kappa/2})$. Therefore, using Cauchy-Schwarz in the $\sigma$ variable gives
$$
\int |K(\xi,\eta)| \,d\eta \lesssim 2^{-k_2} t^{-\kappa} \| \partial_\xi \widetilde{f} \|_{L^2} 2^{\frac{k_2}{2}} \lesssim 2^{-\frac{k_2}{2}} t^{-\frac \kappa 2+\alpha}\epsilon_1.
$$
Similarly,
$$
\int |K(\xi,\eta)| \,d\xi \lesssim 2^{-\frac{k_2}{2}} t^{-\frac{\kappa}{2}+\alpha} \epsilon_1.
$$
By Schur's test, this gives the bound
$$
\| I_{jk_1k_2}^{r,1}(t,\xi) \|_{L^2} \lesssim  2^{-\frac{k_2}{2}} t^{-1-\frac{\kappa}{2}+\alpha} \epsilon_1 \| \varphi_{\sim j} \partial_\xi \widetilde{f} \|_{L^2}.
$$
Using almost orthogonality to sum over $j\sim k_1$, we get
$$
\left\| \sum_{jk_1} I_{jk_1k_2}^{r,1}(t,\xi) \right\|_{L^2} \lesssim 2^{-\frac{k_2}{2}}t^{-1-\frac{\kappa}{2}+\alpha} \epsilon_1  \|  \partial_\xi \widetilde{f} \|_{L^2}  \lesssim 2^{-\frac{k_2}{2}}t^{-1-\frac{\kappa}{2}+2\alpha} \epsilon_1^2.
$$
Summing over $k_2$ gives the desired result
$$
\sum_{k_2} \left\| \sum_{jk_1} I_{jk_1k_2}^{r,1}(t,\xi) \right\|_{L^2} \,ds \lesssim \sum_{k_2} 2^{-\frac{k_2}{2}} t^{-1-\frac{\kappa}{2} +2\alpha} \epsilon_1^2 \lesssim t^{-1-\frac{\kappa}{2}+2\alpha} \epsilon_1^2.
$$

Turning to $I_{jk_1k_2}^{r,2}$, the derivative of $\varphi_{<J} (\Phi)$ can be bounded by
$$
|\partial_\eta \varphi_{< J}(\Phi)|  = |\partial_\eta \Phi \varphi_{< J}'(\Phi)|  \lesssim 2^{k_1} 2^{-J} \lesssim t^{\kappa}.
$$
One can then follow the same estimates as for $I_{jk_1k_2}^{r,1}$, but the sum over $m$ does not converge; this problem is solved by an additional integration by parts in $\eta$, after which the estimates go through, following the same estimates as for $I_{jk_1k_2}^{r,1}$. Notice that, thanks to the localization $\varphi^h$, a further integration by parts in time does not lead to a singularity at $\eta=0$; rather, it produces a factor $\frac{\varphi^h}{\eta t}$, which can be bounded by $\lesssim t^{-1 + \kappa}$.

\medskip

\noindent \underline{Case 2: $2^{k_1}  \gg 2^{j}$.} In order for $\Phi$ to be $O(2^J)$, this implies that $2^{k_1} \sim 2^{k_2} \gg 2^j$. We proceed just like in Case 1, writing  
\begin{equation}
\label{macareux2}
\begin{split}
I_{jk_1k_2}^{r}(t,\xi)  =   I_{jk_1k_2}^{r,1}(t,\xi) + I_{jk_1k_2}^{r,2}(t,\xi)   + \{ \mbox{similar or simpler terms} \},
\end{split}
\end{equation}
and estimate $I_{jk_1k_2}^{r,1}$ in the same way as above. Some estimates have to be adapted to this new situation, though: first, the bound on $\left|  \frac{\xi \mathfrak{m} \mathfrak{p}}{\eta \sigma} \right|$ is now $2^{j-2k_1}$. Second, using the bound~\eqref{boundRr}, we get the estimates
\begin{align*}
& |\{ \eta, |\Phi(\xi,\eta,\sigma) |< 2^J, |\eta| \sim 2^{k_1} \}| \lesssim 2^{J -k_1} \sim t^{-\kappa},\\
& |\{ \xi, |\Phi(\xi,\eta,\sigma)| < 2^J, |\xi|\sim 2^j \}| \lesssim \min \left( 2^{J -j}, \ 2^{\frac J2} \right) \lesssim 2^{k_1-j}t^{-\frac{\kappa}{2}}.
\end{align*}
Overall, we obtain
\begin{align*}
& \int |K(\xi,\eta)| \,d\eta \lesssim 2^{j - 2k_1} t^{-\kappa} \| \partial_\xi \widetilde f\|_{L^2} 2^{k_1/2} \lesssim 2^{j-
\frac{3k_1}{2}} t^{-\kappa+\alpha} \epsilon_1 \\
&  \int |K(\xi,\eta)| \,d\xi \lesssim 2^{j-2k_1} 2^{k_1-j}t^{-\kappa} \| \partial_\xi \widetilde f\|_{L^2} 2^{k_1/2} \epsilon_1 \lesssim 2^{-\frac{k_1}{2} }t^{-\kappa+\alpha} \epsilon_1,
\end{align*}
and thus, as $j\ll k_1$, by Schur's test,
\begin{equation}
\label{petrel}
\| I_{jk_1k_2}^{r,1} \|_{L^2} \lesssim t^{-1}2^{-\frac{k_1}{2}} t^{-\kappa+\alpha}\epsilon_1 \| \partial_\xi \widetilde{f} \|_{L^2}  \lesssim 2^{-\frac{k_1}{2}}  t^{-1-\kappa+2\alpha}\epsilon_1^2.
\end{equation}
Summing over $k_2\sim k_1$, $j < k_1$ and $k_1$, this gives the desired estimate, namely
$$
 \left\| \sum_{jk_1k_2 } I_{jk_1k_2}^{r,1}(t,\xi) \right\|_{L^2} \, ds \lesssim t^{-1-\kappa+2\alpha} \epsilon_1^2.
$$
There remains to estimate $I^{r,2}_{jk_1 k_2 }$; once again,
$$
\left| \partial_\eta \varphi_{< J}(\Phi) \right| = \left|  \partial_\eta \Phi \varphi_{< J}'(\Phi) \right|  = 2^{k_1-J} \lesssim  t^{\kappa}.
$$
Just like for $I^{r,1}$, this is not quite enough to sum, but an additional integration by parts in $\eta$ suffices to close the estimates.

\subsection{The non-resonant term}  \label{subsecnonres}

We claim that
\begin{equation} \label{improved-quadratic-Inr}
\left\| \int_1^t e^{i\theta(s)\xi} I^{nr}(s,\xi)ds \right\|_{L^2}  \lesssim \epsilon_1^2.
\end{equation}
Integrating by parts in time and using the decompositions in Propositions~\ref{propModU} and Proposition~\ref{grebehuppe}, we get that
\begin{align*}
\int_1^t e^{i\theta(s)\xi} I^{nr} (s,\xi) \,ds = &-  \int e^{i\theta(t)\xi -it\Phi} \widetilde{f}(\eta) \widetilde{f}(\sigma) t \xi \mathfrak{m} \frac{\varphi^h \varphi^{nr}}{i \Phi} \,d\eta\,d\sigma \\
& \qquad - \int_1^t s \int e^{i\theta(s)\xi -is\Phi} \left[ \widetilde{\mathcal{D}}(\eta) +\widetilde{\mathcal M}(\eta) \right]\widetilde{f}(\sigma)  \xi  \mathfrak{m} \frac{\varphi^h \varphi^{nr}}{\Phi} \,d\eta \, d\sigma \,ds \\
& \qquad -  \int_1^t s \int e^{i\theta(s)\xi-is\Phi} \widetilde{\mathcal{C^S}}(\eta) \widetilde{f}(\sigma)  \xi  \mathfrak{m} \frac{\varphi^h \varphi^{nr}}{ \Phi} \,d\eta \, d\sigma \,ds \\
& \qquad - \int_1^t s \tau(s) \int e^{i\theta(s)\xi-is\Phi} \widetilde{f}(\eta) \widetilde{f}(\sigma) \eta \xi  \mathfrak{m} \frac{\varphi^h \varphi^{nr}}{ \Phi} \,d\eta \, d\sigma \,ds \\
& \qquad - \int_1^t  \int e^{i\theta(s)\xi-is\Phi} \widetilde{f}(\eta) \widetilde{f}(\sigma) \xi \partial_s[ s \varphi^h \varphi^{nr}] \frac{\mathfrak{m}}{\Phi} \,d\eta \, d\sigma \,ds \\
&\qquad -  \int_1^t s \tau(s) \int e^{i\theta(s)\xi-is\Phi} \widetilde{f}(\eta) \widetilde{f}(\sigma) \xi^2 \mathfrak{m} \frac{\varphi^h \varphi^{nr}}{\Phi} \,d\eta \, d\sigma \,ds \\
& \qquad + \{ \mbox{similar or simpler terms} \} \\
 = & J^{0}+ J^{1} + J^{2} + J^{3} + J^4 + J^5 + \{ \mbox{similar or simpler terms} \}.
\end{align*}

\medskip

\noindent \underline{Estimate of $J^0$.} We claim that
$$
\| J^0 \|_{L^2} \lesssim \epsilon_1^2.
$$
For $t>1$ this will follow as a consequence of the following estimate:
\begin{equation} \label{kumquat}
\left\|t  \xi \int e^{-it\Phi}g(\eta)h(\sigma) \mathfrak m \frac{\varphi^h\varphi^{nr}}{\Phi}d\eta \, d\sigma \right\|_{L^2_\xi} \lesssim t^{-1+3\kappa}\left(\left\| \frac{g}{\langle \xi \rangle}\right\|_{L^2_{\xi}}+\left\| \frac{\partial_\xi g}{\langle \xi \rangle} \right\|_{L^2_\xi}\right)\| h\|_{H^1_\xi}
\end{equation}
for general functions $g,h$, and of \eqref{eqbootstrap}. For $t\leq 1$ the estimate is much easier to prove, so that we skip it. We now show \eqref{kumquat}; integrating by parts,
\begin{align*}
t  \xi \int e^{-it\Phi}g(\eta)h(\sigma) \mathfrak m \frac{\varphi^h\varphi^{nr}}{\Phi}d\eta \, d\sigma & = \frac{\iota \xi}{4t}  \int e^{-it\Phi}\frac{\partial_\xi g(\eta)}{\eta}\frac{\partial_\xi h(\sigma)}{\sigma} \mathfrak m \frac{\varphi^h\varphi^{nr}}{\Phi}d\eta d\sigma \\
&+ \frac{\iota \xi}{4t}  \int e^{-it\Phi}\frac{g(\eta)}{\eta}\frac{ h(\sigma)}{\sigma} \mathfrak m \partial_\eta \partial_\sigma \left( \frac{\varphi^h\varphi^{nr}}{\Phi}\right)d\eta d\sigma +\{ \mbox{easier}\} \\
&= I+II+\{ \mbox{easier}\}
\end{align*}
for some sign $\iota \in \{\pm 1\}$. The easier term contains the remaining cases: when a derivative hits $\frac{\mathfrak m}{\eta \sigma}$, and the mixed cases. It can be estimated exactly as $I$ and $II$ so that we skip it.

To estimate $I$ we use that $|\frac{ \varphi^{nr}}{\Phi}|\lesssim \frac{t^{\kappa}}{\langle \xi \rangle+\langle \eta\rangle+\langle \sigma \rangle}$ by the definition of $\varphi^{nr}$. Hence
$$
|I|\lesssim |\xi|t^{-1+\kappa} \int \frac{|\partial_\xi g(\eta)|}{|\eta|\langle \eta \rangle}\frac{|\partial_\xi h(\sigma)|}{|\sigma|} |\mathfrak m | d\eta d\sigma
$$
Applying \eqref{voilier} then shows
$$
\| I \|_{L^2_\xi}\lesssim t^{-1+\kappa}\left\| \frac{\partial_\xi g}{\langle \xi \rangle}\right\|_{L^2_\xi}\left\| \partial_\xi h\right\|_{L^2_\xi}.
$$

To estimate $II$, we use that
$$
\left| \partial_\eta \partial_\sigma (\frac{\varphi^h \varphi^{nr}}{\Phi}) \right| \lesssim \frac{t^{3\kappa}}{\langle \xi \rangle+\langle \eta\rangle+\langle \sigma \rangle}.
$$
To justify this bound, observe that, if $\langle \xi \rangle\sim 2^j$, $\langle \eta \rangle \sim 2^{k_1}$ and $\langle \sigma \rangle \sim 2^{k_2}$ and $|\Phi|\sim 2^J$, then on the support of $\varphi^{nr}$ one has $2^J \gtrsim 2^{j+k_1+k_2}t^{ -\kappa }$ and therefore
$$
\left| \partial_\eta \frac{\varphi^{nr}}{\Phi} \right| \lesssim 2^{-2J} 2^{k_1} \lesssim   2^{-2j-k_1-2k_2}t^{2\kappa}\lesssim \frac{t^{2\kappa}}{\langle \eta\rangle+\langle \sigma\rangle+\langle \xi \rangle}
$$
(and similarly for the $\partial_\sigma$ derivative). Hence
$$
|II|\lesssim t^{-1+3\kappa}|\xi| \int \frac{| g(\eta)|}{|\eta|\langle \eta \rangle}\frac{|h(\sigma)|}{|\sigma|} |\mathfrak m | d\eta d\sigma .
$$
Applying again \eqref{voilier} shows
$$
\| II \|_{L^2_\xi}\lesssim t^{-1+3\kappa}\left\| \frac{ g}{\langle \xi \rangle}\right\|_{L^2_\xi}\| h\|_{L^2_\xi}.
$$
Combining, we obtain \eqref{kumquat} as desired.

\noindent \underline{Estimate of $J^1$.} Denoting 
$$
\mathcal{J}^{1} = s \int e^{-is\Phi} \left[ \widetilde{D}(\eta) +\widetilde{\mathcal M}(\eta) \right] \widetilde{f}(\sigma)  \xi  \mathfrak{m} \frac{\varphi^h \varphi^{nr}}{ \Phi} \,d\eta \, d\sigma,
$$ 
for the integrand in $s$ in the definition of $J^{1}$ (so that $J^{1} = \int_1^t \mathcal{J}^{1} \,ds$), we claim that
$$
\| \mathcal{J}^{1} \|_{L^2} \lesssim \epsilon_1^3  \langle s \rangle^{- 1-\nu+2\kappa+\alpha},
$$
from which the desired estimate follows by time integration. The estimate will follow the pattern that was established for $\xi \widetilde{\mathcal{Q}^R}$ in Proposition~\ref{grebehuppe}.

We integrate by parts in $\sigma$ to obtain
\begin{align*}
\mathcal{J}^1 = &\pm \frac{1}{2}  \int e^{-is\Phi}  \left[ \widetilde{D}(\eta)  +\widetilde{\mathcal M}(\eta) \right]  \partial_\xi \widetilde{f}(\sigma) \frac{\xi \mathfrak{m}}{\sigma} \frac{\varphi^h \varphi^{nr}}{i\Phi} \,d\eta \, d\sigma \\
& \qquad \pm \frac{1}{2} \int e^{-is\Phi} \left[ \widetilde{D}(\eta) +\widetilde{\mathcal M}(\eta) \right]  \widetilde{f}(\sigma) \frac{\xi  \mathfrak{m}}{\sigma} \partial_\sigma \left[ \frac{\varphi^h \varphi^{nr}}{i\Phi} \right] \,d\eta \, d\sigma  + \{ \mbox{easier terms} \} \\
= & \mathcal{J}^{1,1} + \mathcal{J}^{1,2} + \{ \mbox{easier terms} \}.
\end{align*}
We estimate $\mathcal{J}^{1,1}$ with the help of inequality~\eqref{voilier}, of the bound $\frac{\varphi^h \varphi^{nr}}{\Phi} \lesssim \langle t \rangle^{\kappa}$, of Proposition~\ref{propModU} and of Proposition~\ref{grebehuppe}. This gives
$$
\left\| \mathcal{J}^{1,1} \right\|_{L^2} \lesssim \langle s \rangle^{\kappa} \| \partial_\xi \widetilde f \|_{L^2} \left\| \xi  [ \widetilde{D} + \widetilde{\mathcal M} ]  \right\|_{L^2} \lesssim \epsilon_1^3 \langle s \rangle^{- 1-\nu+\kappa+\alpha}.
$$

Turning to $J^{1,2}$, we can once again follow the same argument, using that
$$
\left| \partial_\sigma \left[ \frac{\varphi_{\geq J}(\Phi)}{\Phi} \right] \right| \lesssim \frac{\langle t \rangle^{ \kappa}}{|\sigma|} + \langle t \rangle^{2\kappa}.
$$
to obtain
$$
\| \mathcal{J}^{1,2} \|_{L^2} \lesssim  \epsilon_1^3  \langle s \rangle^{-1-\nu+2\kappa+\alpha}.
$$

\medskip

\noindent \underline{Estimate of $J^{2}$: the $\delta$ term.} 
We write $\mathcal{J}^2$ for the integrand (in $s$) of $J^2$, in other words
$$
\mathcal{J}^2 =  \int e^{-is\Phi} \widetilde{\mathcal{C^S}}(\eta) \widetilde{f}(\sigma)  \xi  \mathfrak{m} \frac{\varphi^h \varphi^{nr}}{\Phi} \,d\eta \, d\sigma
$$

The singular cubic terms fall into two categories: those involving $\delta(p)$ and those involving $\operatorname{p.v.} \frac{1}{p}$. We call $\eta'$, $\sigma'$ and $\zeta'$ the input frequencies in the cubic term, and $\eta$ the output frequency. Focusing first on the $\delta$ term, we are facing
\begin{align*}
\mathcal{J}^2_\delta = &  \int s\xi e^{-is \Psi(\xi,\sigma,\eta',\sigma',\zeta')} \mathfrak{m}(\xi,\eta,\sigma) \mathfrak{n}(\eta,\eta',\sigma',\zeta') \widetilde{f}(\sigma) \widetilde{f}(\eta') \widetilde{f}(\sigma') \widetilde{f}(\zeta') \frac{\varphi^h \varphi^{nr}(\xi,\eta,\sigma) }{\Phi(\xi,\eta,\sigma) }\\
& \qquad \qquad  \qquad \qquad \qquad\qquad \qquad  \qquad \qquad \delta(\eta \pm \eta' \pm \sigma' \pm \zeta') \,d\eta' \,d\sigma' \,d\zeta' \,d\eta \,d\sigma ,
\end{align*}
where we denoted $\mathfrak{n}$ for the factor such that $\mathfrak{m}^S(\eta,\eta',\sigma',\zeta') = \delta(p) \mathfrak{n}(\eta,\eta',\sigma',\zeta')$, and where the four-linear phase is given by
with
$$
\Psi(\xi,\sigma,\eta',\sigma',\zeta') = (\uo+ |\xi|^2) \pm (\uo + |\s|^2) \pm (\uo+|\eta'|^2) \pm (\uo + |\s'|^2) \pm (\uo + |\z'|^2).
$$
In the above expression, all variables all spelled out for clarity, but in the following, we will omit them as is our customary convention, which results in the formula
$$
\mathcal{J}^2_\delta =  \int s\xi e^{-is \Psi} \mathfrak{m} \mathfrak{n} \widetilde{f}(\sigma) \widetilde{f}(\eta') \widetilde{f}(\sigma') \widetilde{f}(\zeta') \frac{\varphi^h \varphi^{nr}}{\Phi}  \,d\eta' \,d\sigma' \,d\zeta'  \,d\sigma.
$$
We claim that
$$
\| \mathcal{J}^2_\delta(s,\cdot) \|_{L^2} \lesssim s^{-\frac{9}{4} + \kappa+4 \alpha},
$$
which gives the desired bound on $J^2$ after time integration.
Proceeding as in the previous section, we will now split $f$ into low and high frequencies, see~\eqref{foulque1} and~\eqref{foulque2}, which gives the decomposition
\begin{align*}
\mathcal{J}^2_\delta & =  \int s\xi e^{-is \Psi} \mathfrak{m} \mathfrak{n} \widetilde{f}(\sigma) \widetilde{f^L}(\eta') \widetilde{f^L}(\sigma') \widetilde{f^L}(\zeta') \frac{\varphi^h \varphi^{nr}}{\Phi}  \,d\eta' \,d\sigma' \,d\zeta'  \,d\sigma \\
& \qquad + \int s\xi e^{-is \Psi} \mathfrak{m} \mathfrak{n} \widetilde{f}(\sigma) \widetilde{f^H}(\eta') \widetilde{f^H}(\sigma') \widetilde{f^H}(\zeta') \frac{\varphi^h \varphi^{nr}}{\Phi}  \,d\eta' \,d\sigma' \,d\zeta'  \,d\sigma  \\
& \qquad + \{ \mbox{cross terms} \} \\
& = \mathcal{J}^{2,L}_\delta +  \mathcal{J}^{2,H}_\delta+ \{ \mbox{cross terms} \}.
\end{align*}
Here, the cross terms involve both high- and low-frequency contributions, and they can be bounded in a similar way to $\mathcal{J}^{2,L}_\delta$ and $\mathcal{J}^{2,H}_\delta$, which will be treated below.

In order to bound $\mathcal{J}^{2,L}_\delta$, we integrate by parts in $\sigma$, which yields
$$
 \mathcal{J}^{2,L}_\delta = \pm \frac{1}{2} \int e^{-is \Psi} \partial_\xi \widetilde{f}(\sigma)  \widetilde{f^L}(\eta') \widetilde{f^L}(\sigma') \widetilde{f^L}(\zeta') \frac{\xi \mathfrak{m} \mathfrak{n} \varphi^h \varphi^{nr}}{\sigma \Phi}  \,d\eta' \,d\sigma' \,d\zeta'  \,d\sigma + \{ \mbox{easier terms} \}.
$$
Since $|\eta| \lesssim 1$ and $|\Phi|\gtrsim (\la \xi \ra +\la \eta\ra+\la \sigma \ra)t^{-\kappa}$ there holds using \eqref{hibou10} that $\left| \frac{\xi \mathfrak{m} \mathfrak{n} \varphi^h \varphi^{nr}}{\sigma \Phi} \right|\lesssim \frac{t^{\kappa}}{\la \sigma \ra}\sum_{\pm}\frac{1}{\la \xi \pm \eta\pm \sigma \ra^2}$. Using in addition~\eqref{foulque2}, this leads to the following bound for the leading term in $\mathcal{J}^{2,L}_\delta$
$$
\|  \mathcal{J}^{2,L}_\delta \|_{L^2} \lesssim s^{\kappa} \|  \partial_\xi \widetilde{f} \|_{L^2} \| \widetilde{f^L} \|_{L^1}^3 \lesssim \epsilon_1^4 s^{-\frac 94 +\kappa+ 4\alpha}.
$$

Turning to $\mathcal{J}^{2,H}_{\delta}$, we integrate by parts in all the Fourier variables to obtain
\begin{align*}
\mathcal{J}^{2,H}_\delta & = \pm \frac{1}{16 s^3} \int e^{-is \Psi} \partial_\xi \widetilde{f}(\sigma)   \partial_\xi \widetilde{f^H}(\eta')  \partial_\xi  \widetilde{f^H}(\sigma')  \partial_\xi  \widetilde{f^H}(\zeta') \frac{\xi \mathfrak{m} \mathfrak{n} \varphi^h \varphi^{nr}}{\sigma \eta' \sigma' \zeta' \Phi}  \,d\eta' \,d\sigma' \,d\zeta'  \,d\sigma \\
& \qquad \qquad +  \{ \mbox{easier term} \}.
\end{align*}
Once again, the key point is to bound the symbol; by $|\Phi|\gtrsim (\la \xi \ra +\la \eta\ra+\la \sigma \ra)t^{-\kappa}$, \eqref{hibou10} and Proposition~\ref{S-v-csd}, we see that 
\begin{align*}
\left| \frac{\xi \mathfrak{m} \mathfrak{n} \varphi^h \varphi^{nr}}{\sigma \eta' \sigma' \zeta' \Phi} \right| & \lesssim \frac{ t^{\kappa} }{ |\eta'| |\sigma' ||\zeta'|\la \sigma \ra  }  \sum_{\pm} \frac{1}{\langle \xi \pm \eta \pm \sigma \rangle^2} .
\end{align*}
Therefore, we can estimate the leading term in $\mathcal{J}^{2,H}_\delta$ with the help of~\eqref{foulque2} by
\begin{align*}
\| \mathcal{J}^{2,H}_\delta \|_{L^2} 
& \lesssim s^{-3+\kappa} \| \partial_\xi \widetilde{f} \|_{L^2} \| |\xi|^{-1} \partial_\xi \widetilde{f^H} \|_{L^1}^3  \lesssim \epsilon_1^4 s^{-\frac{9}{4} +\kappa+ 4 \alpha},
\end{align*}
which is the desired bound.

\medskip

\noindent \underline{Estimate of $J^{2}$: the $p.v.$ term.} In that case, the time integrand becomes
\begin{align*}
\mathcal{J}^2_{p.v.} = & \int s\xi e^{-is \Psi(\xi,\sigma,\eta',\sigma',\zeta')} \mathfrak{m}(\xi,\eta,\sigma) \mathfrak{n}(\eta,\eta',\sigma',\zeta') \widetilde{f}(\sigma) \widetilde{f}(\eta') \widetilde{f}(\sigma') \widetilde{f}(\zeta') \frac{\varphi^h \varphi^{nr}(\xi,\eta,\sigma) }{\Phi(\xi,\eta,\sigma) }\\
& \qquad \qquad  \qquad \qquad \qquad\qquad \qquad  \qquad \qquad \frac{\widehat{\phi}
(\eta - \eta' - \sigma' - \zeta')}{\eta - \eta' - \sigma' - \zeta'} \,d\eta' \,d\sigma' \,d\zeta' \,d\eta \,d\sigma.
\end{align*}
We will prove that
$$
\left\| \mathcal{J}^2_{p.v.}(s,\cdot) \right\|_{L^2} \lesssim \epsilon_1^4 s^{-2 + \kappa+4 \alpha}.
$$
First, we need to manipulate this formula to make it more amenable to estimates. First, the symbol $\mathfrak{n}$ is a tensor product in the variables $\eta,\eta',\sigma',\zeta'$, which can be distributed on the functions $\widetilde{f}$ without altering the estimates they satisfy. For this reason, we shall assume that $\mathfrak{n}=1$. Second, we denote, as in the previous section $\mathcal{H}$ for the truncated Hilbert transform with kernel $\frac{\widehat{\phi}(\xi)}{\xi}$, and $\Theta$ for the phase $\Psi$ from which the summand $1+|\eta'|^2$ is removed. Then, the above can be written
$$
\mathcal{J}^2_{p.v.} =  \int s\xi e^{-is \Theta}\widetilde{f}(\sigma) \left[ \mathcal{H} e^{\pm i s (1+|\cdot|^2)} \widetilde{f} \right] (\eta - \sigma' - \zeta') \widetilde{f}(\sigma') \widetilde{f}(\zeta')  \frac{ \mathfrak{m} \varphi^h \varphi^{nr} }{\Phi } \,d\eta \,d\sigma' \,d\zeta'  \,d\sigma
$$
The next step is to split between high- and low-frequency contributions
\begin{align*}
\mathcal{J}^2_{p.v.} & = \int s\xi e^{-is \Theta}\widetilde{f}(\sigma) \left[ \mathcal{H} e^{\pm i s (1+|\cdot|^2)} \widetilde{f^L} \right] (\eta - \sigma' - \zeta') \widetilde{f^L}(\sigma') \widetilde{f^L}(\zeta')  \frac{ \mathfrak{m} \varphi^h \varphi^{nr} }{\Phi } \,d\eta \,d\sigma' \,d\zeta'  \,d\sigma \\
& \qquad + \int s\xi e^{-is \Theta}\widetilde{f}(\sigma) \left[ \mathcal{H} e^{\pm i s (1+|\cdot|^2)} \widetilde{f^H} \right] (\eta - \sigma' - \zeta') \widetilde{f^H}(\sigma') \widetilde{f^H}(\zeta')  \frac{ \mathfrak{m} \varphi^h \varphi^{nr} }{\Phi } \,d\eta \,d\sigma' \,d\zeta'  \,d\sigma \\
& \qquad + \{ \mbox{cross terms} \} \\
& = \mathcal{J}^{2,L}_{p.v.} + \mathcal{J}^{2,H}_{p.v.} +  \{ \mbox{cross terms} \} 
\end{align*}

The low-frequency part can be estimated after an integration by parts in $\sigma$:
\begin{align*}
\mathcal{J}^{2,L}_{p.v.} & = \pm \frac{i}{2} \int e^{-is \Theta}\partial_\xi \widetilde{f}(\sigma) \left[ \mathcal{H} e^{\pm i s (1+|\cdot|^2)} \widetilde{f^L} \right] (\eta - \sigma' - \zeta') \widetilde{f^L}(\sigma') \widetilde{f^L}(\zeta')  \frac{\xi \mathfrak{m} \varphi^h \varphi^{nr} }{\sigma \Phi } \,d\eta \,d\sigma' \,d\zeta'  \,d\sigma \\
& \qquad \qquad \qquad \qquad \qquad + \{ \mbox{easier terms} \}
\end{align*}

Proceeding similarly to the case of the $\delta$ term $J^2_\delta$, using~\eqref{foulque2},~\eqref{foulque3} and the fact that $\widehat \phi$ is Schwartz class we can estimate
\begin{align*}
\| \mathcal{J}^{2,L}_{p.v.} \|_{L^2} \lesssim s^\kappa \| \mathcal{H} e^{\pm i s (1+|\cdot|^2)} \widetilde{f^L} \|_{L^1} \|  \widetilde{f^L} \|_{L^1}^2 \| \partial_\xi \widetilde{f} \|_{L^2} \lesssim \| \widetilde{f^L} \|_{L^2} \|  \widetilde{f^L} \|_{L^1}^2  \| \partial_\xi \widetilde{f} \|_{L^2} \lesssim \epsilon_1^4 s^{-2+\kappa+4\alpha}.
\end{align*}

Finally, in order to estimate the high frequeny contribution, we integrate by parts in $\eta'$, $\sigma'$, $\zeta'$ and $\sigma$ to obtain
\begin{align*}
\mathcal{J}^{2,H}_{p.v.} & = \pm \frac{1}{16 s^3} \int e^{-is \Theta} \partial_\xi \widetilde{f}(\sigma) \left[ \mathcal{H} (\cdot)^{-1} e^{\pm i s (1+|\cdot|^2)}\partial_\xi \widetilde{f^H} \right] (\eta - \sigma' - \zeta') \partial_\xi \widetilde{f^H}(\sigma') \partial_\xi \widetilde{f^H}(\zeta') \\
& \qquad \qquad \qquad \qquad \qquad \qquad \frac{\xi \mathfrak{m} \varphi^h \varphi^{nr} }{\sigma \zeta' \sigma' \Phi } \,d\eta \,d\sigma' \,d\zeta'  \,d\sigma + \{ \mbox{easier terms} \}.
\end{align*}

Proceeding similarly to the case of the $\delta$ term $J^2_\delta$, we can estimate, with the help of~\eqref{foulque2} and~\eqref{foulque4}
\begin{align*}
& \| \mathcal{J}^{2,H}_{p.v.} \|_{L^2} \\
& \qquad \lesssim s^{\kappa} \| \partial_\xi \widetilde{f} \|_{L^2} \left\| |\cdot|^{-1} \partial_\xi \widetilde{f^H} \right\|_{L^1}^2 \left[ \left\| \mathcal{H} (\cdot)^{-1} e^{\pm i s (1+|\cdot|^2)} \partial_\xi \widetilde{f^H} \right\|_{L^1} + \left\| \mathcal{H} e^{\pm i s (1+|\cdot|^2)} \partial_\xi \widetilde{f^H} \right\|_{L^2} \right] \\
& \qquad \lesssim \epsilon_1^4 s^{-2+\kappa+4\alpha},
\end{align*}
which is the desired estimate!

\medskip

\noindent \underline{Estimate of $J^{3}$.} It can be estimated in a very similar way to $J^0$. Indeed, thanks to \eqref{kumquat} and Proposition~\ref{propModU}
$$
\int_1^t s |\tau(s)| \left\|  \int e^{-is\Phi} \widetilde{f}(\eta) \widetilde{f}(\sigma) \eta \xi  \mathfrak{m} \frac{\varphi^h \varphi^{nr}}{\Phi} \,d\eta \, d\sigma \right\|_{L^2_\xi} \,ds \lesssim \epsilon_1 \int_1^t s^{-2-\nu+3\kappa}\| \widetilde f\|_{H^1_\xi}^2 ds\lesssim \epsilon_1^3.
$$

\medskip

\noindent \underline{Estimate of $J^{4}$.} For this term, it suffices to integrate by parts in $\eta$ and $\sigma$, using that
$$
\left| \frac{1}{\Phi} \partial_s [ s \varphi^h \varphi^{nr} ] \right| \lesssim s^{\kappa}.
$$

\noindent \underline{Estimate of $J^5$} It follows from the following variant of~\eqref{kumquat}:
$$
\left\| t  \xi^2 \int e^{-it\Phi}g(\eta)h(\sigma) \mathfrak m \frac{\varphi^h\varphi^{nr}}{\Phi}d\eta \, d\sigma \right\|_{L^2_\xi} \lesssim t^{-1+3\kappa}\left(\left\| \frac{g}{\langle \xi \rangle}\right\|_{L^2_{\xi}}+\left\| \frac{\partial_\xi g}{\langle \xi \rangle} \right\|_{L^2_\xi}\right)\| h\|_{H^1_\xi}
$$
Its proof is nearly identical to that of~\eqref{kumquat}; the only difference being that we use the inequality $|\frac{ \varphi^{nr}}{\Phi}|\lesssim \frac{t^{\kappa}}{\langle \xi \rangle+\langle \eta\rangle+\langle \sigma \rangle} \lesssim \frac{t^\kappa}{\langle \xi \rangle}$ instead of $\frac{t^\kappa}{\langle \eta \rangle}$, and similarly for derivatives of $\frac{ \varphi^{nr}}{\Phi}$.

\subsection{End of the proof}

In view of \eqref{aimI} which has been proven in the previous subsections, there only remains to estimate the second term in the right-hand side of \eqref{decompositionQR}. It is actually much simpler to estimate. Indeed, integrating by parts and omitting indices we get 
\begin{align*}
\int e^{-it\Phi(\xi,\eta,\sigma) }  \widetilde{f}  (\eta) \widetilde{f} (\sigma)
\partial_\xi \mathfrak{m} \,d \eta\, d \sigma = -\frac{\lambda \mu}{4t^2} \int e^{-it\Phi } \partial_\eta \partial_\sigma \left( \widetilde{f}  (\eta) \widetilde{f}(\sigma) \frac{\partial_\xi \mathfrak{m} }{\eta \sigma}\right) \,d \eta\, d \sigma 
\end{align*}
(note that there are no boundary terms as $\widetilde f(0)=0$). By \eqref{hibou10} we have
$$
|\partial_\eta \partial_\sigma ( \widetilde{f}  (\eta) \widetilde{f}(\sigma) \frac{\partial_\xi \mathfrak{m} }{\eta \sigma})|\lesssim \frac{|\widetilde f(\eta)|+|\partial_\xi \widetilde f(\eta)|}{\la \eta\ra}\frac{|\widetilde f(\sigma)|+|\partial_\xi \widetilde f(\sigma)|}{\la \sigma \ra}  \sum_{\pm} \frac{1}{\langle \xi \pm \eta \pm \sigma \rangle^2}
$$
so that the Minkowski and Cauchy-Schwarz inequalities give
$$
\left\|  \int e^{-it\Phi } \partial_\eta \partial_\sigma \left( \widetilde{f}  (\eta) \widetilde{f}(\sigma) \frac{\partial_\xi \mathfrak{m} }{\eta \sigma}\right) \,d \eta\, d \sigma \right\|_{L^2_\xi} \lesssim \| |\widetilde f|+|\partial_\xi \widetilde f| \|_{L^2}^2\lesssim \epsilon_1^2 \la t\ra^{2\alpha}.
$$
Hence the second term in the right-hand side of \eqref{decompositionQR} satisfies
 $$
\left\|  \int_0^t dt \sum \int e^{-it\Phi_{\lambda \mu  \rho }(\xi,\eta,\sigma)}  \widetilde{f}_{\lambda}  (\eta) \widetilde{f}_{\mu} (\sigma)
\partial_\xi \mathfrak{m}^V_{jkl, \lambda \mu \rho } (\xi, \eta, \sigma) \,d \eta\, d \sigma \right\|_{L^2_\xi}\lesssim \int_0^t \epsilon_1^2 \la t\ra^{2\alpha-2}dt\lesssim \epsilon_1^2.
 $$
Injecting \eqref{aimI} and the above inequality in \eqref{decompositionQR} shows the desired estimate \eqref{bd:mathcalQR-H1}.

\section{Weighted estimates for the remaining terms}
\label{sectionremaining}
\label{SectionWeightedCubic}

The aim of this section is to prove the following estimate.

\begin{proposition}
\label{PropositionWeightedCubic}
Under the assumptions of Proposition~\ref{propbootstrap}, there holds
\begin{equation}
\begin{split}
\label{courlis2}
& \left\| \partial_\xi \int_0^t e^{i \theta(s) \xi} [\widetilde{\mathcal{C}} + \widetilde{\mathcal{T}} + \widetilde{\mathcal{R}}+ \widetilde{\mathcal{M}}] \,ds \right\|_{L^\infty_{[0,t]} L^2_\xi}\\
& \qquad \qquad \qquad + \left\|\partial_\xi \int_0^t e^{i [ \theta(s) \xi - i s \rho(\uo+\xi^2) ] } [ \widetilde{\Mod_\Phi} + \widetilde{\mathcal{E}} ] \,ds \right\|_{L^\infty_{[0,t]} L^2_\xi} \lesssim \epsilon_1^2 \langle t \rangle^\alpha.
\end{split}
\end{equation}
for $0\leq t < T$. 
\end{proposition}

\subsection{First reduction} Considering the left-hand side of~\eqref{courlis2}, and expanding the terms there through Leibniz' rule, we examine first the terms for which $\partial_\xi$ hits $e^{i \theta(s) \xi}$, namely
$$
\int_0^t i \theta(s)  e^{i \theta(s) \xi} [\widetilde{\mathcal{C}} + \widetilde{\mathcal{T}} + \widetilde{\mathcal{R}}+ \widetilde{\mathcal{M}}] \,ds + \int_0^t  i \theta(s) e^{i [ \theta(s) \xi - i s \rho(\uo+\xi^2) ] } [ \widetilde{\Mod_\Phi} + \widetilde{\mathcal{E}} ] \,ds.
$$
These terms are easily estimated in $L^2$ by
\begin{align*}
\int_0^t \left[ \|\widetilde{\mathcal{C}} \|_{L^2} +  \|\widetilde{\mathcal{T}} \|_{L^2} +  \|\widetilde{\mathcal{R}} \|_{L^2} +  \|\widetilde{\mathcal{M}} \|_{L^2} + \| \widetilde{\Mod_\Phi} \|_{L^2} + \| \widetilde{\mathcal{E}} \|_{L^2} \right] \,ds \lesssim \epsilon_1^2 \langle \log t \rangle.
\end{align*}
Therefore, it will suffice to prove that
\begin{align*}
& \int_0^t \left[ \| \partial_\xi \widetilde{\mathcal{C}} \|_{L^2} + \|\partial_\xi  \widetilde{\mathcal{T}} \|_{L^2} + \| \partial_\xi  \widetilde{\mathcal{R}} \|_{L^2} + \| \partial_\xi  \widetilde{\mathcal{M}}\|_{L^2} \right] \,ds \\
& \qquad \qquad \qquad +  \int_0^t \left\| \partial_\xi \left[ e^{ - i s \rho(\uo+\xi^2)  } \widetilde{\Mod_\Phi} \right] \right\|_{L^2}\,ds + \int_0^t \left\| \partial_\xi \left[ e^{ - i s \rho(\uo+\xi^2)  } \widetilde{\mathcal{E}} ) \right] \right\|_{L^2} \,ds \lesssim \epsilon_1^2 \langle t \rangle^\alpha,
\end{align*}
which will be our aim in this section. For the term involving $\widetilde{\mathcal{M}}$, this is a direct consequence of~\eqref{propModU}, we now turn to the other terms.

\subsection{The commutation lemma}

The following lemma is a variant of lemmas 4.3 and 4.4 in \cite{CP}. The aim of this lemma is to commute $\partial_\xi$ with the following trilinear operator
$$
\mathcal{T}_{b(p)} (f_1, f_2, f_3) (\xi) =   \int e^{ -it \Phi_{\lambda \mu \nu \rho}}  f_1 (\eta) f_2(\sigma) f_3 (\zeta) b(p )\, d\eta \, d \sigma \, d \zeta ,
$$
where $\Phi_{\lambda \mu \nu \rho}(\xi,\eta,\sigma,\zeta)$ was defined in~\eqref{fauconcrecerelle}, $b$ is a distribution, and
$$
p = \alpha \xi - \beta \eta - \gamma \sigma - \delta \zeta.
$$
As is always our convention, the parameters $\alpha,\beta,\gamma,\delta,\lambda,\mu,\nu,\rho$ belong to $\{ \pm 1 \}$.

\begin{lemma}  \label{goeland}
If $\lambda + \mu + \nu = \rho$, then for $f_j \in \mathcal{S}$, 
	\begin{equation*} 
	\begin{split}
	&\alpha \partial_\xi \mathcal{T}_{b(p)} (f_1, f_2, f_3) (\xi)\\
	& \qquad  = \rho \lambda \beta \mathcal{T}_{b(p)} ( \partial_\xi f_1, f_2, f_3) (\xi) + \rho \mu \gamma \mathcal{T}_{b(p)} (  f_1, \partial_\xi f_2, f_3) (\xi) + \rho \nu \delta \mathcal{T}_{b(p)} (  f_1, f_2, \partial_\xi f_3) (\xi) \\ 
	& \qquad \qquad \qquad - 2it \rho \mathcal{T}_{pb(p)} (  f_1, f_2,   f_3) (\xi) .
	\end{split}
	\end{equation*}
\end{lemma}

\begin{proof}
	Differentiating $\mathcal{T}_{b} (f_1, f_2, f_3)$ with respect to $\xi$ gives
	\begin{equation}   \label{foudebassan}
	\partial_\xi \mathcal{T}_{b(p)} (f_1, f_2, f_3) (\xi) =  \mathcal{T}_{\partial_\xi b} (f_1, f_2, f_3) (\xi) - 2it  \rho \mathcal{T}_{\xi b(p)} (f_1, f_2, f_3) (\xi) .  \\
	\end{equation}
	Since
	$$
\alpha \partial_\xi b (p) = - \beta \partial_\eta b(p) = -\gamma \partial_\sigma b(p) = -\delta \partial_\zeta b(p),
	$$ 
	there holds
	$$
\alpha \partial_\xi b (p) = \alpha \rho (\lambda + \mu + \nu) \partial_\xi b (p) = - \rho(\lambda \beta \partial_\eta b(p)+ \mu \gamma \partial_\sigma b(p) + \nu \delta \partial_\zeta b(p)).
	$$
	Substituting the above right-hand side for $\partial_\xi b(p)$ in~\eqref{foudebassan}, and integrating by parts in $\eta, \sigma$ or $\zeta$ as appropriate gives
	\begin{align*}
	&\alpha \partial_\xi \mathcal{T}_{b(p)} (f_1, f_2, f_3) = \rho \lambda \beta \mathcal{T}_{b(p)} (\partial_\xi f_1, f_2, f_3) + \rho \mu \gamma \mathcal{T}_{b(p)} ( f_1, \partial_\xi f_2, f_3) + \rho \nu \delta \mathcal{T}_{b(p)} ( f_1, f_2,\partial_\xi  f_3) \\
	& \qquad \qquad \qquad \qquad \qquad - 2it  \rho \alpha \mathcal{T}_{\xi b(p)} (f_1, f_2, f_3) (\xi) + 2it \rho \beta \mathcal{T}_{\eta b(p)}(f_1, f_2, f_3) \\
	& \qquad \qquad \qquad \qquad \qquad \qquad + 2it \rho \gamma \mathcal{T}_{ \sigma b(p)} (f_1, f_2, f_3) + 2it \rho \delta \mathcal{T}_{ \zeta b(p)} (f_1, f_2, f_3).
	\end{align*}
	Gathering the terms in the last two lines into $2it\rho \mathcal{T}_{pb(p)}$ gives the desired result.
\end{proof}

\subsection{The singular cubic terms} In this subsection we prove

\begin{equation} \label{bd:partial-xi-CS-L2}
\left\| \partial_\xi  \widetilde{\mathcal{C}^S}(t)  \right\|_{L^2} \lesssim \epsilon_1^3 \langle t \rangle^{-1+\alpha}.
\end{equation}

\noindent  \underline{Deriving the relevant expressions to be estimated}.
Out of the two components of $\mathcal{C}^S$ in~\eqref{fuligule}, we will only estimate
\begin{equation}
\label{fauvette}
 \widetilde{\mathcal{F}}_+ (\mathcal{C}^S) (\xi) =  \frac{1}{4 \pi^{2}} \sum  \int e^{-it\Phi_{++-+}(\xi,\eta,\sigma,\zeta)}  \widetilde{f}_{+}  (\eta)  \widetilde{f}_{+} (\sigma) \widetilde{f}_{-} (\zeta)
\mathfrak{m}^S_{1121, ++-+} (\xi, \eta, \sigma, \zeta) \,d \eta\, d \sigma \,d \zeta
\end{equation}
since the other component can be treated identically. Recall that, by Lemma~\ref{gobemouche},
\begin{equation*}  
\begin{split}
\mathfrak{m}^{S}_{1121, ++-+} (\xi, \eta, \sigma, \zeta) 
& =  L\lambda \mu \nu \sum_{\alpha \beta \gamma \delta \epsilon}  a^\epsilon_{  \substack{  1121, +++ - \\ \alpha \beta \gamma \delta }} (\xi, \eta, \sigma, 
\zeta) \left[ \pi \delta(p) +  \sqrt{2\pi} \epsilon p.v. \frac{\widehat{\phi} (p)}{ip}  \right]   \\
& =  \sum_{\alpha \beta \gamma \delta \epsilon }  \mathfrak{a}^{+, 1, \epsilon}_\beta(\eta)  \mathfrak{a}^{+, 1, \epsilon}_\gamma(\sigma)  \mathfrak{a}^{-, 2, \epsilon}_\delta(\zeta) \overline{\mathfrak{a}^{+, 1, \epsilon}_\alpha}(\xi) \left[\pi \delta(p) +  \sqrt{2\pi} \epsilon p.v. \frac{\widehat{\phi} (p)}{ip}  \right] ,
\end{split}
\end{equation*}
where $p =  \alpha \xi -  \beta \eta -   \gamma \sigma - \delta \zeta$.

Therefore,~\eqref{fauvette} can be written as a linear combination of terms of the type 
$$
\mathcal{T}_{b(p)}(\mathfrak{a}(\xi)\widetilde f(\xi),\, \mathfrak{a}(\xi)\widetilde f(\xi),\,\mathfrak{a}(\xi)\widetilde  f(\xi)), \qquad \mbox{with} \qquad b(p) =  \delta(p) \quad \mbox{or} \quad   p.v. \frac{\widehat{\phi} (p)}{ip},
$$
where we dropped all indices for simplicity. By Lemma~\ref{goeland},
\begin{align*}
&\partial_\xi \mathcal{T}_{b(p)}(\mathfrak{a}(\xi)\widetilde  f(\xi),\,\mathfrak{a}(\xi)\widetilde  f(\xi),\,\mathfrak{a}(\xi)\widetilde  f(\xi)) = \rho \lambda \alpha \beta  \mathcal{T}_{b(p)}(\partial_\xi[\mathfrak{a}(\xi)\widetilde  f(\xi)],\,\mathfrak{a}(\xi)\widetilde  f(\xi),\,\mathfrak{a}(\xi)\widetilde  f(\xi)) \\
& \qquad \qquad \qquad \qquad \qquad - 2it \rho \alpha \mathcal{T}_{pb(p)}(\mathfrak{a}(\xi)\widetilde  f(\xi),\,\mathfrak{a}(\xi)\widetilde  f(\xi),\,\mathfrak{a}(\xi)\widetilde  f(\xi)) + \{ \mbox{similar terms} \}.
\end{align*}
The second term on the right-hand side is actually a regular cubic term, and as such it will be bounded in the next subsection. We turn to the first term on the right-hand side, for which we will assume that $\partial_\xi$ hits $\widetilde{f}(\xi)$ rather than $\mathfrak{a}(\xi)$, since this is the harder case. 

\medskip

\noindent \underline{Multilinear operators}. A simplifying feature which we encounter here is the following: the integrand of $\widetilde{\mathcal{F}}_+(\mathcal{C}^S)$ tensorizes, as a function of $\xi,\eta,\sigma,\zeta$. Therefore, after switching from the distorted to the flat Fourier transform (which can be achieved through the wave operator), $\widetilde{\mathcal{F}}_+(\mathcal{C}^S)$ can be expressed in terms of the operators
\begin{align*}
& U(f_1,f_2,f_3)(x) = \widehat{\mathcal{F}}^{-1}_{\xi \to x} \int \widehat{f_1}(\eta) \widehat{f_2}(\sigma) \widehat{f_3}(\zeta) \delta(p) \,d\eta \,d\sigma \,d\zeta \\
& V(f_1,f_2,f_3)(x) = \widehat{\mathcal{F}}^{-1}_{\xi \to x} \int \widehat{f_1}(\eta) \widehat{f_2}(\sigma) \widehat{f_3}(\zeta) \operatorname{p.v.} \frac{\widehat{\phi}(p)}{p} \,d\eta \,d\sigma \,d\zeta,
\end{align*}
where 
$$
p = \alpha \xi - \beta \eta - \gamma \sigma - \delta \zeta.
$$
Since these operators correspond to convolution in frequency space, they are given by multiplication opeartors in physical space:
\begin{align*}
& U(f_1,f_2,f_3)(x) = 2\pi f_1(\alpha \beta x) f_2 (\alpha \gamma x) f_3(\alpha \delta x) \\
& V(f_1,f_2,f_3)(x) = (2\pi)^{3/2} Z(\alpha x) f_1(\alpha \beta x) f_2 (\alpha \gamma x) f_3(\alpha \delta x), \qquad \mbox{with} \quad \widehat{Z}(\xi) = \frac{\widehat{\phi}(\xi)}{\xi}.
\end{align*}
Since $Z \in L^\infty$, the operators $U$ and $V$ enjoy H\"older bounds
$$
\| U(f_1,f_2,f_3) \|_q + \| V(f_1,f_2,f_3) \|_q \lesssim \| f_1 \|_{p_1} \| f_2 \|_{p_2} \| f_3 \|_{p_3}.
$$
whenever $1\leq q,p_1,p_2,p_3\leq \infty$ satisfy $\frac{1}{q}=\frac{1}{p_1}+\frac{1}{p_2}+\frac{1}{p_3}$.
\medskip

\noindent \underline{The estimates}.
Using the wave operator defined in~\eqref{grimpereau} and the notations defined in the previous paragraph, we can write $\mathcal{T}_b(\partial_\xi[\mathfrak{a}(\xi)\widetilde  f(\xi)],\,\mathfrak{a}(\xi)\widetilde  f(\xi),\,\mathfrak{a}(\xi)\widetilde  f(\xi))$ as a linear combination of terms of the type
$$
e^{it(1+\xi^2)}\mathfrak{a}(\xi) \widehat{\mathcal{F}} U(\mathcal{W}  e^{it \mathcal{H}} \mathfrak{a}(\widetilde{D}) \widetilde{\mathcal{F}}^{-1} \partial_\xi \widetilde f(\xi), \,  \mathcal{W} e^{it \mathcal{H}} \mathfrak{a}(\widetilde{D}) f, \, \mathcal{W} e^{it \mathcal{H}} \mathfrak{a}(\widetilde{D}) f),
$$
and the same expression with $V$ replacing $U$. Since both can be estimated identically, we focus on the expression involving $U$ above. Applying successively Plancherel's theorem, the H\"older estimate, Lemma~\ref{aigrette} and Corollary~\ref{pluvier},
\begin{align*}
&  \left\|e^{it(1+\xi^2)}\mathfrak{a}(\xi) \widehat{\mathcal{F}} U(\mathcal{W}  e^{it \mathcal{H}} \mathfrak{a}(\widetilde{D})  \widetilde{\mathcal{F}}^{-1}\partial_\xi \widetilde f(\xi), \,  \mathcal{W} e^{it \mathcal{H}} \mathfrak{a}(\widetilde{D}) f, \, \mathcal{W} e^{it \mathcal{H}} \mathfrak{a}(\widetilde{D}) f) \right\|_{L^2}  \\
& \qquad \lesssim  \left\| \mathcal{W}  e^{it \mathcal{H}} \mathfrak{a}(\widetilde{D})  \widetilde{\mathcal{F}}^{-1}\partial_\xi \widetilde f(\xi) \right\|_2 \left\| \mathcal{W} e^{it \mathcal{H}} \mathfrak{a}(\widetilde{D}) f \right\|_\infty \left\| \mathcal{W} e^{it \mathcal{H}} \mathfrak{a}(\widetilde{D}) f) \right\|_\infty \\
&  \qquad \lesssim \epsilon_1^3 \langle t \rangle^{\alpha} \langle t \rangle^{-\frac{1}{2}} \langle t \rangle^{-\frac{1}{2}} = \epsilon_1^3 \langle t \rangle^{-1+\alpha}.
\end{align*}

\subsection{The regular cubic terms} \label{subsecregcub}
We want to bound here the contribution of the regular cubic terms by
\begin{equation}  \label{bd:partial-xi-CR-L2}
\| \partial_\xi \widetilde{\mathcal{C}^R} \|_{L^2}\lesssim \epsilon_1^3 t^{3\alpha - 2}.
\end{equation}
We can write $\partial_\xi \widetilde{\mathcal{C}^R}$ as a linear combination of terms of the type
$$
t \xi \int e^{it\Phi(\xi,\eta,\sigma,\zeta)} \widetilde{f}(\eta) \widetilde{f}(\sigma) \widetilde{f}(\zeta) \mathfrak{m}(\xi,\eta,\sigma,\zeta) \, d\eta \, d\sigma \, d\zeta,
$$
up to simpler terms which we disregard. The above expression was already estimated in $L^2$ in Lemma~\ref{grebehuppe}, or rather its proof. There, it was proved that it enjoys the bound $\epsilon_1^3 t^{3\alpha - 2}$, from which the desired estimate follows!

\subsection{The higher order and remainder terms}

We claim that
\begin{equation} \label{bd:partial-xi-mathcalT-mathcalR}
\| \partial_\xi \widetilde{\mathcal T}\|_{L^2}\lesssim \epsilon_1^4 \langle t \rangle^{-1} \qquad \mbox{and} \qquad \| \partial_\xi \widetilde{\mathcal R}\|_{L^2}\lesssim \epsilon_1^3 \langle t\rangle^{-1-\nu+\alpha}.
\end{equation}
To prove these bounds, we do not need to handle carefully resonances on the Fourier side as rough estimates on the physical side suffice.

Recall \eqref{id:definition-Nu}, the definition of the nonlinear terms $N(u,\bar u)$. Recall $u=U_1$, $\bar u=U_2$ and $U=\Uue+\Uud$. We introduce
\begin{align*}
& T(u,\bar u)=N(\underline{U_{e,1}},\underline{U_{e,2}})-V_{++} \underline{U_{e,1}^2} - V_{--}  \underline{U_{e,2}^2}- V_{+-}  \underline{U_{e,1}} \underline{U_{e,2}} \\
&\qquad \qquad \qquad  - V_{++-}  \underline{U_{e,1}^2} \underline{U_{e,2}} - V_{+--}  \underline{U_{e,1}} \underline{U_{e,2}^2} - V_{+++}  \underline{U_{e,1}^3} + V_{---}  \underline{U_{e,2}^3},\\
& R(u,\bar u)=N(u,\bar u)-N(\underline{U_{e,1}},\underline{U_{e,2}}),
\end{align*}
so that
$$
\mathcal T= e^{-it\mathcal H}\left( T(u,\bar u),T(\bar u,u)\right)^\top \qquad \mbox{and}\qquad \mathcal R= e^{-it\mathcal H}\left( R(u,\bar u),R(\bar u,u)\right)^\top .
$$
Now for a general function $g$, we have by applying (iv) in Proposition \ref{tourterelle}
\begin{align*}
\left\| \partial_\xi \left(e^{-it\rho(1+\xi^2)} \widetilde{\mathcal F}_{\rho} g\right)\right\|_{L^2} & \lesssim \left\| t\xi \widetilde{\mathcal F}_{\rho} g\right\|_{L^2}+\left\| \partial_\xi \widetilde{\mathcal F}_{\rho} g\right\|_{L^2} \\
&\lesssim t \| g\|_{H^1}+\| \langle x\rangle g \|_{L^2}.
\end{align*}
Therefore
\begin{align}
\label{bd:mathcalT-by-physical-side} & \| \partial_\xi \widetilde{\mathcal T}\|_{L^2}\lesssim t \| T(u,\bar u)\|_{H^1}+\| \langle x\rangle T(u,\bar u)\|_{L^2}, \\
\label{bd:mathcalR-by-physical-side} & \| \partial_\xi \widetilde{\mathcal R}\|_{L^2}\lesssim t \| R(u,\bar u)\|_{H^1}+\| \langle x\rangle R(u,\bar u)\|_{L^2}.
\end{align}
Applying (iv) in Proposition \ref{tourterelle} again, then using $\widetilde{\mathcal F}_\rho\Uue = e^{it\rho(1+\xi^2)} \widetilde f_\rho$ we have
\begin{equation} \label{bd:Ue-weighted-L2}
\| \langle x \rangle \Uue\|_{L^2}\lesssim \| \widetilde{\mathcal F} \Uue \|_{H^1}\lesssim  \| t \xi \widetilde f\|_{L^2}+ \| \widetilde f\|_{H^1} \lesssim \epsilon_1 \langle t\rangle
\end{equation}
where we used \eqref{eqbootstrap} and \eqref{controlH1} for the last inequality.

From the very definition of $T(u,\bar u)$ and $R(u,\bar u)$ we have the pointwise estimates
\begin{align*}
& |T(u,\bar u)|\lesssim |\Uue|^5+V|\Uue|^4, \qquad |\partial_x T(u,\bar u)|\lesssim |\partial_x \Uue| |\Uue|^4+V'|\partial_x \Uue||\Uue|^3,\\
& |R(u,\bar u)|\lesssim |\Uud|(|\Uud|+|\Uue|), \quad |\partial_x R(u,\bar u)|\lesssim  |\partial_x \Uud|(|\Uud|+|\Uue|)+|\partial_x \Uue| |\Uud|
\end{align*}
for $V$ and $V'$ two exponentially decreasing functions. Hence
\begin{align}
\nonumber \| T(u,\bar u)\|_{H^1} &\lesssim \| \Uue \|_{H^1}\| \Uue\|_{L^\infty}^4+\| \Uue \|_{H^1}\| \langle x \rangle^{-1} \Uue\|_{L^\infty}^3 \\
\label{bd:T-H1}&\qquad \qquad \quad \lesssim \epsilon_1 \epsilon_1^4 \langle t \rangle^{-2}+\epsilon_1 \epsilon_1^3 \langle t \rangle^{-3+3\alpha} \ \lesssim \epsilon_1^4 \langle t \rangle^{-2},
\end{align}
where we used \eqref{controlH1}, \eqref{bd:globaldecay}, \eqref{bd:estimation:a0(t)-a1(t)} and \eqref{bd:localdecay}, and
\begin{align}
\nonumber \| \langle x \rangle T(u,\bar u)\|_{L^2} &\lesssim \| \langle x \rangle \Uue \|_{L^2}\| \Uue\|_{L^\infty}^4+\| \langle x \rangle^{-1} \Uue\|_{L^\infty}^4 \\
\label{bd:T-weighted-L2}&\qquad \qquad \quad \lesssim  \epsilon_1\langle t\rangle \epsilon_1^4 \langle t \rangle^{-2}+ \epsilon_1^4 \langle t \rangle^{-4+4\alpha} \ \lesssim \epsilon_1^4 \langle t \rangle^{-1},
\end{align}
where we used in addition \eqref{bd:Ue-weighted-L2}. Similarly, using $\Uud=a_0\underline{\Xi_0}+a_1\underline{\Xi_1}$, \eqref{bd:estimation:a0(t)-a1(t)} and \eqref{controlH1}
\begin{align}
\nonumber & \| R(u,\bar u)\|_{H^1} +\| \langle x \rangle R\|_{L^2}\lesssim (|a_0|+|a_1|)(|a_0|+|a_1|+\| \Uue\|_{H^1})\\
\label{bd:R-L2-H1} &\qquad \qquad \qquad \quad  \qquad \qquad \quad \lesssim  \epsilon_1^2 \langle t\rangle^{-2-\nu+\alpha} \left( \epsilon_1^2 \langle t\rangle^{-2-\nu+\alpha}+\epsilon_1 \right)  \ \lesssim \epsilon_1^3 \langle t \rangle^{-2-\nu+\alpha}.
\end{align}
Injecting \eqref{bd:T-H1} and \eqref{bd:T-weighted-L2} in \eqref{bd:mathcalT-by-physical-side}, and \eqref{bd:R-L2-H1} in \eqref{bd:mathcalT-by-physical-side}, we obtain the desired estimate \eqref{bd:partial-xi-mathcalT-mathcalR}.

\subsection{The modulation term} In this subsection, we prove that
\begin{equation} \label{bd:partial-xi-Mod-L2}
\left\| \partial_\xi \left[ e^{-i\rho t (\uo+\xi^2)} \widetilde{\Mod_\Phi} \right] \right\|_{L^2}\lesssim \epsilon_1^3 \langle t\rangle^{-2 + 2\alpha- \nu}.
\end{equation}
We now use \eqref{id:relation-distorted-fourier-and-projectors} and then the fourth assertion in Lemma~\ref{tourterelle} to obtain
\begin{align*}
& \left\| \partial_\xi \left[ e^{-i\rho t (\uo+\xi^2)} \widetilde{\Mod_\Phi} \right] \right\|_{L^2} = \left\| \partial_\xi \left[ e^{-i\rho t (\uo+\xi^2)}\widetilde{\mathcal F} \underline{P_e} {\Mod_\Phi} \right] \right\|_{L^2} \\
& \qquad \qquad \leq t \left\| \xi \widetilde{\mathcal F} \underline{P_e}  {\Mod_\Phi} \right\|_{L^2} + \left\| \partial_\xi \widetilde{\mathcal F} \underline{P_e}  {\Mod_\Phi} \right\|_{L^2} \lesssim t \| \underline{P_e}  \Mod_\Phi \|_{H^1} + \|\underline{P_e}  \Mod_\Phi \|_{L^{2,1}}.
\end{align*}
Using now~\eqref{projectedmodulationterm} which gives the projection of $\Mod_\Phi$, followed by Lemma~\ref{propfirst} and the boostrap hypothesis~\eqref{bd:bootstrap-omega2} which control the modulation parameters,
\begin{align*}
&\|  \underline{P_e}  {\Mod_\Phi} \|_{H^1} + \|  \underline{P_e}  {\Mod_\Phi} \|_{L^{2,1}} \\
& \qquad \lesssim \left[ |\omega-p^2-\dot \gamma| + |\dot \omega| + |2p-\dot y| + |\dot p| \right] \sum_{j=0}^3 \left[ \left\| e^{i(\pu-p)\sigma_3 x} \Xi_2-\underline{\Xi_2} \right\|_{H^1} +\left\| e^{i(\pu-p)\sigma_3 x} \Xi_2-\underline{\Xi_2} \right\|_{L^{2,1}} \right] \\
& \qquad \lesssim \left[ |\omega-p^2-\dot \gamma| + |\dot \omega| + |2p-\dot y| + |\dot p| \right] |\underline{p} - p| \\
& \qquad \lesssim \epsilon_1^3 \langle t \rangle^{-3 + 2\alpha - \nu}.
\end{align*}
Combining the above estimates gives the desired bound~\eqref{bd:partial-xi-Mod-L2}.

\subsection{The error term} In this subsection, we prove that
\begin{equation}
\label{bound:partial-xi-tildemathcalE} \| \partial_\xi (e^{-it\rho(\uo+\xi^2)}\widetilde{\mathcal E}) \|_{L^2}\lesssim  \epsilon_1^2 \langle t\rangle^{-1-\nu+\alpha}.
\end{equation}

Recall the decomposition
\begin{align*}
\mathcal E & =  (V_{\underline{\omega}} -  e^{i(p-\pu)\sigma_3 x} V_\omega  e^{i(\pu-p)\sigma_3 x}) \Uu + e^{i(p-\pu)\sigma_3 x} \mathcal N( e^{i(\pu-p)\sigma_3 x} \Uu)  -\mathcal N(\Uu) \\
& = \mathcal E_1+\mathcal E_2 .
\end{align*}

\noindent \underline{Estimates for $\mathcal E_1$.} Using first $U=\Uud+\Uue$ and $\widetilde {\mathcal F}\Pud =0$, then applying Lemma \ref{bartavelle2} we have for $\rho=\pm$
\begin{align*}
 \widetilde{\mathcal F}_\rho \mathcal E_1 & = \widetilde{\mathcal F}_\rho  \left(   (V_{\underline{\omega}} -  e^{i(p-\pu)\sigma_3 x} V_\omega  e^{i(\pu-p)\sigma_3 x}) \Uud \right)+  \widetilde{\mathcal F}_\rho  \left(  (V_{\underline{\omega}} -  e^{i(p-\pu)\sigma_3 x} V_\omega  e^{i(\pu-p)\sigma_3 x}) \Uue \right) \\
 & = \widetilde{\mathcal F}_\rho  \left(  (V_{\underline{\omega}} -  e^{i(p-\pu)\sigma_3 x} V_\omega  e^{i(\pu-p)\sigma_3 x})\Uud \right)+ \sum_{\lambda \in \{ \pm \}} \int e^{ it \lambda (1+\eta^2)} \mathfrak{s}^{\omega,\underline{\omega}}_{\rho \lambda}(\xi,\eta) \widetilde{f}_\lambda(\eta) \,d\eta \\
 & = \widetilde{\mathcal{F}}_\rho \mathcal{E}_1' + \widetilde{\mathcal{F}}_\rho \mathcal{E}_1''
\end{align*}
with $\mathfrak{s}=\mathfrak{s}^{\omega,\underline{\omega}}_{\rho \lambda}$ satisfying
$$
 \left| \partial_\xi^a \partial_\eta^b \frac{\mathfrak{s} (\xi,\eta)}{\eta} \right| \lesssim_{a,b} \frac{|\omega-\underline{\omega}|}{\langle \eta \rangle}\sum_{\pm}  \frac{1}{\langle \xi \pm \eta \rangle^2} .
 $$

Differentiating with respect to $\xi$, we obtain
\begin{align*}
& \partial_\xi \left(  e^{-it\rho(\uo+\xi^2)}\widetilde{\mathcal{F}}_\rho \mathcal E_{1} \right) \\
& \quad = -2i t\rho \xi e^{-it\rho(\uo+\xi^2)} \widetilde{\mathcal F}_\rho \mathcal E_{1}+e^{-it\rho(\uo+\xi^2)}\partial_\xi  \widetilde{\mathcal{F}}_\rho \mathcal{E}_1' +e^{ -it \rho(\uo+\xi^2) } \sum_{\lambda \in \{ \pm \}} \int e^{ it \lambda (\uo+\eta^2)}\partial_\xi \mathfrak{s} (\xi,\eta) \widetilde{f}_\lambda(\eta) \,d\eta\\
& \quad = -2i t\rho \xi e^{-it\rho(\uo+\xi^2)} \widetilde{\mathcal F}_\rho \mathcal E_{1}+ I+II.
\end{align*}
The first term is bounded thanks to~\eqref{boundE1}, which yields
$$
t \| \xi \widetilde{\mathcal{E}_1} \|_{L^2} \lesssim \epsilon_1^2 \langle t \rangle^{-1-\nu +\alpha}.
$$
To bound the second term, we use item (iv) of Proposition \ref{tourterelle}, the decay of the discrete parameters~\eqref{bd:estimation:a0(t)-a1(t)} and the bootstrap hypothesis~\eqref{bd:bootstrap-omega2} to obtain
$$
  \| I \|_{L^{2}}  \lesssim \|   (V_{\underline{\omega}} -  e^{i(p-\pu)\sigma_3 x} V_\omega  e^{i(\pu-p)\sigma_3 x}) \Uud  \|_{L^{2,1}} \lesssim \epsilon_1^3 \langle t\rangle^{-3-2\nu+\alpha}.
$$
Finally, the term $II$ can be bounded exactly as $\mathcal{E}_1$ in the estimates leading to~\eqref{boundE1}, giving 
$$
\| II\|_{L^2}\lesssim \epsilon_1^2  \langle t\rangle^{-2-\nu+\alpha}.
$$
Overall, we find that
\begin{equation}
\label{bound:partial-xi-tildemathfrakE-1-inter2}  \|  \partial_\xi (e^{-it\rho(1+\xi^2)}\widetilde{\mathcal E}_1)   \|_{L^2}\lesssim  \epsilon_1^2 \langle t\rangle^{-1-\nu+\alpha}.
\end{equation}

\medskip

\noindent \underline{Estimates for $\mathcal E_2$.} We recall by \eqref{id:definition-mathcalNU} that $ \mathcal{N}(U)=(\mathcal{N}_1(U),\mathcal{N}_2(U))$ where $\mathcal N_1 (U)= N(U_1,U_2)$. An explicit computation gives
\begin{align*}
 \mathcal E_{2,1} &= e^{i(p-\pu )x} \mathcal N_1(e^{i(\pu-p)\sigma_3x}\Uu)-\mathcal N_1(\Uu) \\
 & = F'\left(\Phi^2+e^{i(\pu-p)x}\Phi \Uu_1 +e^{i(p-\pu)x}\Phi \Uu_2+\Uu_1\Uu_2\right)(e^{i(p-\pu)x}\Phi+\Uu_1) \\
 & \qquad -F'\left(\Phi^2+\Phi \Uu_1 +\Phi \Uu_2+\Uu_1\Uu_2\right)(\Phi+\Uu_1)-e^{2i(p-\pu)x}F''(\Phi^2)\Phi^2 \Uu_2+F''(\Phi^2)\Phi^2 \Uu_2 \\
 & \qquad +F'(\Phi^2) \Phi - e^{i(p-\underline p) x} F'(\Phi^2) \Phi.
\end{align*}
The formula for $ \mathcal E_{2,2}$ is similar since $\mathcal{N}_2(U)=-N(U_2,U_1)$. Performing a Taylor expansion in the above formula, using the exponential localization of $\Phi$, it follows that (since $|U| \lesssim 1$)
$$
| \mathcal E_{2}|\lesssim |p-\pu| |U|^2 e^{-\mu x} \quad \mbox{and} \quad | \partial_x \mathcal E_{2}|\lesssim |p-\pu||U| (|U|+|\partial_x U|)e^{-\mu x}
$$
for some $\mu>0$. Therefore, the improved local decay \eqref{bd:localdecay}, \eqref{bd:bootstrap-omega2} and \eqref{controlH1} imply
$$
\| \langle x \rangle \mathcal E_2\|_{L^2}\lesssim \epsilon_1^3 \langle t \rangle^{-3-\nu+2\alpha} \quad \mbox{and}\quad \| \partial_x \mathcal E_2\|_{L^2}\lesssim \epsilon_1^3 \langle t \rangle^{-2-\nu+\alpha}
$$
Using item (iv) of Proposition \ref{tourterelle}, these bounds imply
\begin{equation} \label{bound:partial-xi-tildemathfrakE-1-inter3000}
\| \widetilde{\mathcal F} \mathcal E_2\|_{H^1}\lesssim \epsilon_1^3 \langle t \rangle^{-3-\nu+2\alpha} \quad \mbox{and}\quad \| \langle \xi \rangle \widetilde{\mathcal F} \mathcal E_2\|_{L^2}\lesssim \epsilon_1^3 \langle t \rangle^{-2-\nu+\alpha} 
\end{equation}
Using $|\partial_\xi (e^{it\rho(\uo+\xi^2)t}\widetilde{\mathcal F} \mathcal E_2)|\leq 2t|\xi| |\widetilde{\mathcal F} \mathcal E_2|+|\partial_\xi \widetilde{\mathcal F} \mathcal E_2|$ these bounds imply 
\begin{equation}
\label{bound:partial-xi-tildemathfrakE-1-inter3}  \|  \partial_\xi (e^{-it\rho(\uo+\xi^2)}\widetilde{\mathcal E}_2)   \|_{L^2}\lesssim  \epsilon_1^3 \langle t\rangle^{-1-\nu+\alpha},
\end{equation}
which was the desired bound.

\section{Pointwise estimate on the Fourier side}
\label{sectionpointwise}
\label{SectionPointwise}

The aim of this section is to prove the following estimate.

\begin{proposition} \label{PropositionPointwise}
Under the assumptions of Proposition~\ref{propbootstrap}, there holds if $t \in (0,T)$
$$
\| \widetilde{f}(t) \|_{L^\infty} \lesssim \epsilon + \epsilon_1^2.
$$
Furthermore, there exists $W^\infty_\rho(\xi) = O(\epsilon + \epsilon_1^2)$ such that
$$
\left| \widetilde{f}_\rho(t,\xi) - W^\infty_\rho(\xi) e^{-i \frac{L}{2} |W^\infty_\rho|^2 \log t} \right| \lesssim \epsilon_1^2 \langle t \rangle^{-\alpha} \qquad \mbox{if $t >  |\xi|^{(2\alpha - \frac{1}{2})^{-1}}$}.
$$
\end{proposition}
 
\subsection{Heuristic Asymptotics}
\label{SS-v-HA}

\subsubsection{The oscillatory integrals and their phases}
We consider the asymptotics as $t \rightarrow \infty$ for the following model operators
\begin{equation} \label{SS-v-HA-eq1}
I_\delta (t,\xi) := \int e^{-it \Phi_{\lambda \mu \nu \rho} (\xi, \eta, \sigma, \zeta)} F (\xi, \eta, \sigma, \zeta) \delta (p)  \, d \eta \,d \sigma \, d \zeta , 
\end{equation}
\begin{equation} \label{SS-v-HA-eq2}
I_{\operatorname{p.v.}} (t,\xi) := \int e^{-it \Phi_{\lambda \mu \nu \rho} (\xi, \eta, \sigma, \zeta)} F (\xi, \eta, \sigma, \zeta) \frac{\widehat{\phi} (p)}{p}  \,d \eta \,d \sigma \,d \zeta , 
\end{equation}
where 
\begin{equation}
p = \alpha \xi -  \beta \eta -   \gamma \sigma - \delta \zeta
\end{equation}
(recall that $\lambda, \mu ,\nu, \rho,\alpha, \beta, \gamma,\delta \in \{ \pm 1 \}$). The signs $\lambda,\mu,\nu,\rho$ are furthermore restricted to satisfy
$$
\lambda = - \mu = \nu = \rho.
$$

Under this condition, the phase can be written
\begin{equation} \label{SS-v-HA-eq3}
\begin{split}
\Phi_{\lambda \mu \nu \rho} (\xi, \eta, \sigma, \zeta)
& =  \rho(\uo+ \xi^2) - \lambda  (\uo+ \eta^2)  - \mu  (\uo+ \sigma^2)  - \nu  (\uo+ \zeta^2)  \\
& =  \rho \left[ \xi^2 - \eta^2  + \sigma^2  -  (p - \alpha \xi +  \beta \eta +  \gamma \sigma  )^2 \right] , \\
\end{split}
\end{equation} 
which we sometimes simply denote $\Phi = \Phi (\xi, \eta, \sigma, \zeta)$, keeping the dependence on the various parameters implicit. 

The gradient of the phase with respect to $\eta$ and $\sigma$ is
\begin{equation*}
\begin{split}
\nabla_{\eta, \sigma} \Phi (\xi, \eta, \sigma, \zeta)  = -2 \rho
\begin{pmatrix}
  \eta +  \beta (p - \alpha \xi + \beta \eta + \gamma \sigma)  \\
- \sigma  +  \gamma (p - \alpha \xi + \beta \eta + \gamma \sigma)   
\end{pmatrix}.
\end{split}
\end{equation*} 
It vanishes if
\begin{equation*}
\begin{split}
& \eta = - \beta (p - \alpha \xi + \beta \eta + \gamma \sigma) , \\
& \sigma  =  \gamma (p - \alpha \xi + \beta \eta + \gamma \sigma)  . \\ 
\end{split}
\end{equation*} 
Multiplying the first equation by $\gamma$ and the second by $\beta$, and then comparing the resulting equations gives $\eta  = - \beta \gamma \sigma$, which implies
\begin{equation}  
\begin{split}
&  \eta = - \beta  (p - \alpha \xi  ) ,
\end{split}
\end{equation}
and hence
\begin{equation}  
\begin{split}
&  \sigma = - \beta \gamma \eta = \gamma (p - \alpha \xi  ) , \\
& \zeta = - \delta (p - \alpha \xi + \beta \eta + \gamma \sigma) = - \delta  (p - \alpha \xi  ) .  \\
\end{split}
\end{equation}
Therefore, the stationary point with respect to $\eta$, $\sigma$ is given by
\begin{equation}   \label{SS-v-HA-eq8}
\begin{split}
&  \eta_S = -  \beta (p - \alpha \xi  ) , \\
&  \sigma_S =  \gamma  (p - \alpha \xi  ) , \\
& \zeta_S  = -  \delta (p - \alpha \xi  ) .  \\
\end{split}
\end{equation}
Moreover, for any $\xi,\eta,\sigma,\zeta$,
\begin{equation} \label{SS-v-HA-eq9}
\begin{split}
\text{Hess}_{\eta, \sigma} \Phi (\xi, \eta, \sigma, \zeta) 
= -2 \rho
\begin{pmatrix}
2  &  \beta \gamma  \\
 \beta \gamma   &   0
\end{pmatrix}.
\end{split}
\end{equation} 
Hence 
\begin{align*}
& \det \text{Hess}_{\eta, \sigma} \Phi (\xi, \eta, \sigma, \zeta) = -4 \\
& \text{sign}\, \text{Hess}_{\eta, \sigma} \Phi (\xi, \eta, \sigma, \zeta) = 0 
\end{align*} 
(denoting $\text{sign} M$ the number of positive minus the number of negative eigenvalues of a matrix $M$).

\subsubsection{Asymptotics for $I_\delta$.} In this case, $p =0$ and the stationary point is given by
\begin{equation*} 
\begin{split}
&  \eta_{S0} =  \alpha \beta \xi    , \\
&  \sigma_{S0} = - \alpha \gamma \xi   , \\
& \zeta_{S0} = \alpha \delta \xi,
\end{split}
\end{equation*}
and
\begin{equation*}
\Phi  (\xi, \eta_{S0}, \sigma_{S0}, \zeta_{S0}) = 0 . 
\end{equation*}
By the stationary phase lemma, 
\begin{align*}
I_\delta (t,\xi) & \sim \frac{2 \pi}{t} \frac{e^{i \frac{\pi}{4} \text{sign} \, \text{Hess}_{\eta, \sigma} \Phi }}{|\det  \text{Hess}_{\eta, \sigma} \Phi |^{1/2}} e^{-it \Phi} F(\xi, \eta_{S0}, \sigma_{S0}, \zeta_{S0}) \qquad \text{ as } t \rightarrow +\infty \\
& = \frac{\pi}{t} F(\xi, \eta_{S0}, \sigma_{S0}, \zeta_{S0}) .
\end{align*}

\subsubsection{Asymptotics for $I_{p.v.}$.} In this case we need to consider $p \neq  0$. Then
\begin{equation*}
\Phi  (\xi, \eta_S, \sigma_S, \zeta_S) =- 2 \alpha \rho \xi p + O(p^2). 
\end{equation*}
We apply the stationary phase lemma for fixed $p$, which yields
\begin{equation*}
I_{\operatorname{p.v.}} (t,\xi) 
 \sim \frac{ \pi}{t}   F(\xi, \eta_{S0}, \sigma_{S0}, \zeta_{S0}) \int  e^{-it 2 \alpha \rho \xi p } \frac{\widehat{\phi} (p)}{p} dp \qquad \mbox{as $t\to \infty$}.
\end{equation*}
Since $\widehat{\mathcal{F}} \big( \frac{\widehat{\phi} (p)}{p} \big) = \frac{1}{\sqrt{2 \pi}} \widehat{\mathcal{F}} (1/x) * \phi = -\frac{i}{2} \text{sign} * \phi $, 
\begin{equation*} 
I_{\operatorname{p.v.}} (t,\xi) \sim - \frac{\pi}{t}  \left( i \sqrt{\frac{\pi}{2}} \right)  F(\xi, \eta_{S0}, \sigma_{S0}, \zeta_{S0})  \text{sign} (\rho  \alpha \xi ) \quad \mbox{as $t \to \infty$}.
\end{equation*}

\subsubsection{Derivation of the asymptotic ODE} \label{tengmalm}
We focus here on the algebra, and argue heuristically that singular cubic terms are the only ones to affect the large-time behavior of the solution. More precisely, we claim that, to leading order
\begin{equation}
\label{passereau}
i\partial_t \widetilde{f}_\rho(\xi) \sim  \widetilde{\mathcal{F}}_\rho (e^{-it\mathcal{H}}V_{++-} (U_1^2 U_2  e_1 -U_2^2 U_1  e_2 )) (\xi).
\end{equation}
We derive first (again, heuristically) an asymptotic formula for the above right-hand side, starting from formula~\eqref{fuligule} which gives the Fourier space representation of the cubic term. By formula~\eqref{beccroise}, the cubic spectral distribution has a nontrivial singular part for specific values of $\lambda,\mu,\nu$. Thus, the formula simplifies to become 
\begin{equation*}   
\begin{split}
& \widetilde{\mathcal{F}}_+ (e^{-it\mathcal{H}}V_{++-} (U_1^2 U_2  e_1 - U_2^2 U_1  e_2 )) (\xi) \\
& \qquad \qquad =  \frac{1}{4 \pi^{2}} \sum  \int e^{-it\Phi_{+-++}(\xi,\eta,\sigma,\zeta)}  \widetilde{f}_{+}  (\eta)  \widetilde{f}_{-} (\sigma) \widetilde{f}_{+} (\zeta)
\mu^S_{1211, +-++} (\xi, \eta, \sigma, \zeta) \,d \eta\, d \sigma \,d \zeta ,  \\
& \qquad  \qquad \qquad \qquad + \{ \mbox{regular part} \} \\
& \widetilde{\mathcal{F}}_- (e^{-it\mathcal{H}}V_{++-} (U_1^2 U_2  e_1 - U_2^2 U_1  e_2 ) ) (\xi) \\
& \qquad \qquad =  - \frac{1}{4 \pi^{2}} \sum  \int  e^{-it\Phi_{-+--}(\xi,\eta,\sigma,\zeta)} \widetilde{f}_{-}  (\eta)  \widetilde{f}_{+} (\sigma) \widetilde{f}_{-} (\zeta)
\mu^S_{2122,-+--} (\xi, \eta, \sigma, \zeta) \, d \eta \,d \sigma \,d \zeta \\
& \qquad  \qquad \qquad \qquad + \{ \mbox{regular part} \}.
\end{split}
\end{equation*}
Formula~\eqref{beccroise} yields the expression
 \begin{align*}
& \mu^S_{1211, +-++}(\xi,\eta,\sigma,\zeta) = \mu^S_{2122,-+--}(\xi,\eta,\sigma,\zeta) \\
& \qquad \qquad \qquad =  - L \sum_{\epsilon \alpha \beta \gamma \delta} \overline{\mathfrak{a}^\epsilon_\alpha (\xi)} \mathfrak{a}^\epsilon_\b (\eta) \mathfrak{a}^\epsilon_\g(\s) \mathfrak{a}^\epsilon_\d (\z) \left[ \pi \delta (p) + \epsilon \sqrt{2 \pi} \operatorname{p.v.} \frac{\widehat{\phi}(p)}{ip} \right]
 \end{align*}
 where the constant $L$ is the limit at infinity of $V_{+-+}$, namely
 $$
 L = F''(0).
 $$
 
If $(\lambda,\mu,\nu,\rho) = (+,+,-,+)$ or $(-,-,+,-)$, then the asymptotic formulas for $I_\delta$ and $I_{\operatorname{p.v.}}$ give
\begin{align*}
& \widetilde{\mathcal{F}}_\rho (e^{it\mathcal{H}}V_{++-} (U_1^2 U_2  e_1 - U_2^2 U_1  e_2 ))(\xi)  \\
& \;\; \sim - \frac{L}{4t}  \sum_{\epsilon \alpha \beta \gamma \delta} \overline{\mathfrak{a}^\epsilon_\alpha (\xi)} \mathfrak{a}^\epsilon_\b (\a \b \x) \mathfrak{a}^\epsilon_\g(-\a \g \x) \mathfrak{a}^\epsilon_\d (\a \d \x) \widetilde{f}_{\rho}  (\a \b \x)  \widetilde{f}_{-\rho} (- \a \g \x) \widetilde{f}_{\rho} (\a \d \x) \left[ 1 - \epsilon \a \r \operatorname{sign} \x \right].
\end{align*}
The factor $\left[ 1 - \epsilon \a \r \operatorname{sign} \x \right]$ vanishes unless $\alpha = - \epsilon \rho$ for $\xi>0$, and $\alpha = \epsilon \rho$ for $\xi<0$, so that the above becomes, as $t \to \infty$,
$$
\widetilde{\mathcal{F}}_\r (e^{it\mathcal{H}}V_{++-} (- U_1^2 U_2  e_1 + U_2^2 U_1  e_2 ))(\xi) \sim - \frac{L}{2t} \Theta_\rho(\xi)  \\
$$
where
\begin{align*}
& \Theta_\rho(\xi) = \\
& \quad \left\{ \begin{array}{ll}\mbox{if $\xi> 0$,} &  \displaystyle \sum \overline{\mathfrak{a}^\epsilon_{-\epsilon \rho} (\xi)} \mathfrak{a}^\epsilon_\b (-\epsilon \r \b \x) \mathfrak{a}^\epsilon_\g(\epsilon \r \g \x) \mathfrak{a}^\epsilon_\d (-\epsilon \r \d \x) \widetilde{f}_{\rho}  (-\epsilon \r \b \x)  \widetilde{f}_{-\rho} (\epsilon \r \g \x) \widetilde{f}_{\rho} (-\epsilon \r \d \x),
 \\
\mbox{if $\xi < 0$,} & \displaystyle \sum \overline{\mathfrak{a}^\epsilon_{\epsilon \rho} (\xi)} \mathfrak{a}^\epsilon_\b (\epsilon \r \b \x) \mathfrak{a}^\epsilon_\g(-\epsilon \r \g \x) \mathfrak{a}^\epsilon_\d (\epsilon \r \d \x) \widetilde{f}_{\rho}  (\epsilon \r \b \x)  \widetilde{f}_{-\rho} (-\epsilon \r \g \x) \widetilde{f}_{\rho} (\epsilon \r \d \x).  \end{array}\right. 
\end{align*}

We claim that 
\begin{equation} \label{id:formula-sum-EDO-Fourier}
\Theta_\rho(\xi)= - |\widetilde f_\rho(\xi)|^2\widetilde f_\rho(\xi)
\end{equation}
which we prove by investigating the different cases below.

\medskip

\noindent \underline{Computation in the case $\xi>0$ and $\rho=+$.} Using \eqref{id:coefficients-a-singular-cubic-term}, the factor $ \overline{\mathfrak{a}^\epsilon_{-\epsilon } (\xi)} \mathfrak{a}^\epsilon_\b (-\epsilon \b \x) \mathfrak{a}^\epsilon_\d(-\epsilon \d \x) $ vanishes unless $(\e,\b,\g)=(-,+,+)$ in which case it equals $1$. Therefore, using once again \eqref{id:coefficients-a-singular-cubic-term}, followed by~\eqref{fauconpelerin},
$$
\Theta_+ (\xi)=  \widetilde f_+(\xi)^2 \sum_{\gamma=\pm} \mathfrak{a}_\gamma^-(-\gamma \xi) \widetilde f_-(-\gamma \xi) = \widetilde f_+(\xi)^2 \left(s(\xi)\widetilde f_-(-\xi)+r(\xi)\widetilde f_-(\xi) \right) = - |\widetilde f_+(\xi)|^2 \widetilde f_+(\xi).
$$
\medskip

\noindent \underline{Computation in the case $\xi<0$ and $\rho=+$.} The factor $\overline{\mathfrak{a}^\epsilon_{\epsilon \rho} (\xi)} \mathfrak{a}^\epsilon_\b (\epsilon \r \b \x) \mathfrak{a}^\epsilon_\d(\epsilon \r \d \x)$ vanishes unless $(\e,\b,\g)=(+,+,+)$ in which case it equals $1$. Thus
\begin{align*}
\Theta_+(\xi) & =\widetilde f_+(\xi)^2 \sum_{\g=\pm} \mathfrak{a}^+_\g(-\g \xi)\widetilde f_-(-\g \xi)  = \widetilde f_+(\xi)^2\left(s(-\xi)\widetilde f_-(-\xi)+r(-\xi)\widetilde f_-(\xi) \right) \\
& = - |\widetilde f_+(\xi)|^2 \widetilde f_+(\xi) .
\end{align*}

\noindent \underline{Computation in the case $\rho=-$.} Changing the sign of $\rho$ in the formula giving $\Theta_\rho$ essentially corresponds to switching the cases $\xi>0$ and $\xi<0$. This observation gives the desired result in the case $\rho=-$.

\subsection{A stationary phase lemma}

The following lemma makes rigorous the heuristic arguments used in the previous section to compute the asymptotic behavior of $I_\delta$ and $I_{\operatorname{p.v.}}$.

\begin{proposition} \label{capucin} Recall that $I_\delta$ and $I_{\operatorname{p.v.}}$ are defined in \eqref{SS-v-HA-eq1} and  \eqref{SS-v-HA-eq2}. Assume that $\rho = \lambda = - \mu = \nu$. If
$$
F(\xi,\eta,\sigma,\zeta) = g_1(\eta) g_2(\sigma) g_3(\zeta),
$$
with
$$
\| g_j(t) \|_{L^\infty} + \| \la \xi \ra g_j(t) \|_{L^2} + \la t \ra^{-\alpha} \| \partial_\xi g_j(t) \|_{L^2} \lesssim 1, \qquad j=1,2,3,
$$
then, for $t \in [0,1]$ and for any $\xi$,
$$
|I_\delta (t,\xi)| + |I_{\operatorname{p.v.}}(t,\xi)| \lesssim 1
$$
while if $t >1$ and $\kappa>0$,
\begin{align*}
&I_\delta(t,\xi) = \frac{\pi}{t} g_1( \alpha \beta\xi)  g_2(-\alpha \gamma \xi) g_3 ( \alpha \delta \zeta) + O\left( t^{-15/14}\right) \\
&I_{\operatorname{p.v.}}(t,\xi) = - \frac{\pi}{t} i \sqrt{\frac{\pi}{2}} g_1( \alpha \beta\xi)  g_2(-\alpha \gamma \xi) g_3 ( \alpha \delta \zeta) \operatorname{sign}(\rho \alpha \xi) \\
& \qquad \qquad \qquad \qquad \qquad \qquad\qquad \qquad + O \left( t^{-15/14} + t^{\kappa-1} \langle \sqrt t \xi \rangle^{-1} + t^{\alpha -1} \langle \sqrt t \xi \rangle^{-3/2} \right).
\end{align*}
\end{proposition}

\begin{proof} \underline{The case $0 \leq t \leq 1$.} It is easily dealt with, thanks to the following formulas (for which we assume $\alpha = \beta = \gamma = \delta = \rho =1$ for simplicity)
\begin{align*}
& I_\delta(t,\xi) = (2\pi)^2 e^{-it\xi^2} \widehat{\mathcal{F}} \left[ e^{-it\partial_x^2} \widecheck{g_1} \, e^{it\partial_x^2} \widecheck{g_2}  \, e^{-it\partial_x^2} \widecheck{g_3}\right] (\xi) \\
& I_{\operatorname{p.v.}}(t,\xi) = (2\pi)^{5/2} e^{-it\xi^2} \widehat{\mathcal{F}} \left[ Z \, e^{-it\partial_x^2} \widecheck{g_1} \, e^{it\partial_x^2} \widecheck{g_2} \,  e^{-it\partial_x^2} \widecheck{g_3} \right] (\xi), \qquad Z = \widehat{\mathcal{F}} \frac{\widehat{\phi}(p)}{p}.
\end{align*}
Since $Z \in L^\infty$, we can bound, under the assumptions of the proposition,
\begin{align*}
|I_\delta(t,\xi)| + |I_{\operatorname{p.v.}}(t,\xi)| & \lesssim \left\|  e^{-it\partial_x^2} \widecheck{g_1} \, e^{it\partial_x^2} \widecheck{g_2}  \, e^{-it\partial_x^2} \widecheck{g_3} \right\|_{L^1} \\
& \lesssim \| e^{-it\partial_x^2} \widecheck{g_1} \|_{L^2} \| e^{it\partial_x^2} \widecheck{g_2} \|_{L^2} \| e^{-it\partial_x^2} \widecheck{g_3} \|_{L^\infty} \lesssim \| g_1 \|_{L^2} \| g_2 \|_{L^2} \| g_3 \|_{H^1} \lesssim 1.
\end{align*}

\medskip

\noindent \underline{The case $t \geq 1$.} In this case, we shall rely on the auxiliary lemmas~\eqref{engoulevent1} and~\ref{engoulevent2}, which are stated and proved below. Lemma~\ref{engoulevent1} with $p=0$ immediately gives the desired estimate for $I_\delta$.

As far as $I_{\operatorname{p.v.}}$ is concerned, a more careful treatment is necessary. The first step is to split
$$
I_{\operatorname{p.v.}}(t,\xi) = \int \dots \varphi_0(t^{10} p) \,d\eta\,d\sigma\,d\zeta +  \int \dots \varphi_{\geq 1}(t^{10} p) \,d\eta\,d\sigma\,d\zeta = I_{\operatorname{p.v.}}^1(t,\xi)  + I_{\operatorname{p.v.}}^2(t,\xi) .
$$

To bound $I_{\operatorname{p.v.}}^1(t,\xi)$, we observe that it can be written
$$
I_{\operatorname{p.v.}}^1(t,\xi) = \int e^{-i\rho t [\xi^2 - \sigma^2 - \zeta^2]} \left[ e^{-i\rho t \eta^2} g_1(\eta) -  e^{-i\rho t \eta_0^2} g_1(\eta_0) \right] g_2(\sigma) g_3(\zeta) \varphi_0(t^{10} p) \frac{\widehat{\phi}(p)}{p} \,dp\,d\sigma\,d\zeta ,
$$
where we view $p,\sigma,\zeta$ as the independent integration variables while
$$
\eta = \beta [ \alpha \xi -\gamma \sigma - \delta \zeta - p] \quad \mbox{and} \quad \eta_0 = \eta (p=0) = \beta [ \alpha \xi -\gamma \sigma - \delta \zeta ].
$$
Since $|e^{-i\rho t \eta^2} g_1(\eta) -  e^{-i\rho t \eta_0^2} g_1(\eta_0)| \lesssim t \langle \eta \rangle |g_1(\eta)| |p| + \| g_1\|_{H^1} |p|^{1/2}$ on the support of the integrand, we can use the Cauchy-Schwarz inequality to bound
$$
|I^1_{\operatorname{p.v.}}(t,\xi)| \lesssim t \| g_1 \|_{L^{2,1}} t^{-5} + \| g_1 \|_{H^1} t^{-5} \lesssim t^{-4+\alpha}.
$$

Turning to $I_{\operatorname{p.v.}}^2(t,\xi)$, we use Lemma~\ref{engoulevent1} to obtain the equality
\begin{align*}
| I_{\operatorname{p.v.}}^2(t,\xi) | & = \frac{\pi}{t} \int e^{-i \rho t(2\alpha \xi p -p^2)} g_1(\eta_S) g_2(\sigma_S) g_3(\zeta_S)  \varphi_{\geq 1}(t^{10} p) \frac{\widehat{\phi}(p)}{p} \,dp \\
& \qquad \qquad \qquad  \qquad \qquad \qquad + O \left( t^{-14/13} \int \varphi_{\geq 1}(t^{10} p)\left|  \frac{\widehat{\phi}(p)}{p} \right| \,dp \right) \\
& = \frac{\pi}{t} \int e^{-i \rho t(2\alpha \xi p -p^2)} g_1(\eta_S) g_2(\sigma_S) g_3(\zeta_S)  \varphi_{\geq 1}(t^{10} p) \frac{\widehat{\phi}(p)}{p} \,dp + O(t^{-15/14}).
\end{align*}
In the leading term in the above right-hand side, it is possible to remove the cutoff function $\varphi_{\geq 1}(t^{10} p)$, with a justification similar to the bound for $I^1_{\operatorname{p.v.}}$. This gives
$$
| I_{\operatorname{p.v.}}^2(t,\xi) | = \frac{\pi}{t} \int e^{-i \rho t(2\alpha \xi p -p^2)} g_1(\eta_S) g_2(\sigma_S) g_3(\zeta_S) \frac{\widehat{\phi}(p)}{p} \,dp + O(t^{-15/14}).
$$
There remains to apply Lemma~\ref{engoulevent2} to obtain the desired statement for $ I_{\operatorname{p.v.}}(t,\xi)$.
\end{proof}

\begin{lemma} 
\label{engoulevent1}
Consider the expression
$$
J_p(t) = \int e^{-it (\xi^2 - \eta^2 + \sigma^2 - \zeta^2)}  g_1(\eta) g_2(\sigma) g_3(\zeta) \delta(\alpha \xi - \beta \eta - \gamma \sigma - \delta \zeta - p) \,d\eta \,d\sigma \,d\zeta.
$$
Under the assumptions of Proposition~\ref{capucin} and if $t \geq 1$,
$$
J_p(t) = e^{-it(2\alpha \xi p - p^2)} \frac{\pi}{t} g_1(\eta_S) g_2 (\sigma_S) g_3(\zeta_S) + O\left( t^{-1-\frac{1}{3}(\frac 14 - \alpha)}\right)
$$
where
$$
\eta_S = \beta(\alpha \xi - p), \quad \sigma_S = - \gamma ( \alpha \xi - p) \quad \mbox{and} \quad \zeta_S = \delta (\alpha \xi - p).
$$
\end{lemma}

\begin{proof}
On the support of the integrand of $J_p$, the phase can be expressed as
$$
\xi^2 - \eta^2 + \sigma^2 - \zeta^2 = 2 \alpha \xi p - p^2 - 2 AB \qquad 
\mbox{with} \quad
\left\{
\begin{array}{l}
A = \delta \zeta - (\alpha \xi - p) \\ B = \delta \zeta + \gamma \sigma.
\end{array}
\right.
$$
Therefore, $J_p$ can be written
$$
J_p(t) = e^{-it(2 \alpha \xi p - p^2)} H_p(t)
$$
where
$$
H_p(t) = \int e^{2itAB} g_1(\eta) g_2(\sigma) g_3(\zeta) \,dA \,dB \qquad \mbox{and} \qquad 
\left\{
\begin{array}{l}
\eta = \beta ( \alpha \xi - p - B) \\
\sigma = \gamma(B-A - (\alpha \xi -p)) \\
\zeta = \delta (A + \alpha \xi -p).
\end{array}
\right.
$$
We now split $H_p$ into three pieces, by localizing $A$ and $B$ around $0$ on a scale $\sim t^{\delta - \frac{1}{2}}$, where $\delta >0$ will be chosen subsequently:
\begin{align*}
H_p(t) & = \int \varphi_0(t^{\frac 12 - \delta} A) \varphi_0(t^{\frac 12 - \delta} B) \dots \,dA \,dB +  \int \varphi_{\geq 1}(t^{\frac 12 - \delta} A) \dots \,dA \,dB \\
& \qquad \qquad \qquad +  \int \varphi_{0}(t^{\frac 12 - \delta} A) \varphi_{\geq 1}(t^{\frac 12 - \delta} B) \dots \,dA \,dB \\
& = H_p^1(t) + H_p^2(t) + H_p^3(t).
\end{align*}

The leading contribution in $H_p^1$ is given by setting $\eta,\sigma,\zeta$ to be $\eta_S,\sigma_S,\zeta_S$, which is the value they assume when $A=B=0$, which yields
\begin{align*}
& g_1(\eta_S) g_2(\sigma_S) g_3(\zeta_S) \int e^{2itAB} \varphi_0(t^{\frac 12 - \delta} A) \varphi_0(t^{\frac 12 - \delta} B) \,dA \,dB \\
& \quad = \sqrt{2\pi} g_1(\eta_S) g_2(\sigma_S) g_3(\zeta_S) t^{\delta - \frac 12} \int \widehat{\varphi_0} (-2 t^{\frac 12 + \delta} B) \varphi_0(t^{\frac 12 - \delta} B) \,dB \\
& \quad = \sqrt{2\pi}  g_1(\eta_S) g_2(\sigma_S) g_3(\zeta_S) t^{\delta - \frac 12} \int \widehat{\varphi_0} (-2 t^{\frac 12 + \delta} B) \,dB + O(t^{-2}) \\
& \quad = g_1(\eta_S) g_2(\sigma_S) g_3(\zeta_S)  \frac{\pi}{t} + O(t^{-2}).
\end{align*}

When subtracting this leading order contribution, the error term in $H_p^1$ is
\begin{align*}
& \int e^{2itAB} \varphi_0(t^{\frac 12 - \delta} A) \varphi_0(t^{\frac 12 - \delta} B)  \left[ g_1(\eta) g_2(\sigma) g_3(\zeta) - g_1(\eta_S) g_2(\sigma_S) g_3(\zeta_S) \right] \,dA \,dB \\
& \quad =  \int e^{2itAB} \varphi_0(t^{\frac 12 - \delta} A) \varphi_0(t^{\frac 12 - \delta} B)  \left[ g_1(\eta) - g_1(\eta_S) \right] g_2(\sigma) g_3(\zeta) \,dA \,dB + \{ \mbox{similar terms} \}.
\end{align*}
With the help of the Cauchy-Schwarz inequality, the term in the right-hand side can be estimated by
\begin{align*}
& \int \varphi_0(t^{\frac 12 - \delta} A) \varphi_0(t^{\frac 12 - \delta} B) |\eta - \eta_S|^{\frac 12} \| \partial_\xi g_1 \|_{L^2} \| g_2 \|_\infty \| g_3 \|_\infty \,dA \,dB \\
& \qquad \qquad \qquad \qquad \qquad \qquad \qquad \qquad \qquad \lesssim t^{\delta - \frac{1}{2}} \cdot  t^{\frac{3\delta}{2} - \frac 34} \cdot t^\alpha = t^{\frac{5 \delta}{2} +\alpha - \frac 54}.
\end{align*}

To estimate $H_p^2$, we integrate by parts in $B$ to obtain
$$
H_p^2(t) = - \int  \frac{e^{2itAB}}{2itA} \varphi_{\geq 1}(t^{\frac 12 - \delta} A) \partial_B g_1(\eta) g_2(\sigma) g_3(\zeta) \,dA \,dB + \{ \mbox{similar term} \}.
$$
Applying successively the Cauchy-Schwarz inequality, the Plancherel theorem, and Lemma~\ref{aigrette} giving dispersive estimates,
\begin{align*}
|H_p^2(t)| & \lesssim \frac 1t \|g_3\|_\infty \left\| \frac{ \varphi_{\geq 1}(t^{\frac 12 - \delta} A)}{A} \right\|_{L^2_A} \left\| \int e^{it (\eta^2 + \sigma^2)} \partial_B g_1(\eta) g_2(\sigma) \,dB \right\|_{L^2_A} \\
& \lesssim \frac 1t \|g_3\|_\infty t^{\frac 14 - \frac \delta 2} \| \partial_\xi g_1 \|_{L^2} \| e^{it\partial_x^2} \widehat{g_2} \|_{L^\infty} \\
& \lesssim t^{-1} \cdot t^{\frac 14 - \frac \delta 2} \cdot t^\alpha \cdot t^{-\frac 12} = t^{-\frac{5}{4} + \alpha - \frac \delta 2},
\end{align*}
under the assumptions made on $g_1, g_2, g_3$. 

Turning to $H_p^3$, one can proceed similarly, integrating by parts in $A$ rather than $B$, which leads to 
\begin{align*}
H_p^3(t) & = - \int  \frac{e^{2itAB}}{2itB} \varphi_{\geq 1}(t^{\frac 12 - \delta} B)g_1(\eta) \partial_A \left[ \varphi_0(t^{\frac 12 - \delta} A) g_2(\sigma) g_3(\zeta) \right] \,dA \,dB.
\end{align*}
The same arguments lead to the bound
\begin{align*}
| H_p^3(t) | \lesssim t^{-1} \cdot t^{\frac 14 - \frac \delta 2} \left\| \int e^{it (\sigma^2 + \zeta^2)} \partial_A \left[ \varphi_0(t^{\frac 12 - \delta} A)  g_2(\sigma) g_3(\zeta) \right]  \,dB \right\|_{L^2_B}. 
\end{align*}
Focusing on the last factor in the line above, it can be expanded as
\begin{align*}
& \left\| \int e^{it (\sigma^2 + \zeta^2)} \left[ t^{\frac 12 - \delta} \varphi_0'(t^{\frac 12 - \delta} A)  g_2 g_3 + \varphi_0(t^{\frac 12 - \delta} A)  \partial_A g_2 g_3 + \varphi_0(t^{\frac 12 - \delta} A)  g_2 \partial_A g_3 \right]  \,dB \right\|_{L^2_B}\\
& \qquad \lesssim t^{\frac12 -\delta} \| e^{it\partial_x^2} \widehat{g_2} \|_{L^\infty} \| \varphi'_0(t^{\frac 12 - \delta} A) g_3 \|_{L^2} + \| \partial_A g_2 \|_{L^2} \| e^{it\partial_x^2} \widehat{\mathcal{F}} [\varphi_0(t^{\frac 12 - \delta} A) g_3 \|_{L^\infty} \\
& \qquad \qquad \qquad \qquad +  \| e^{it\partial_x^2} \widehat{g_2} \|_{L^\infty} \| \partial_A g_3 \|_{L^2} \\
& \qquad \lesssim t^{\frac 12 - \delta} \cdot t^{-\frac 1 2} \cdot t^{\frac{\delta}{2} - \frac 14} + t^\alpha \left[ t^{-\frac 34} \| \varphi_0(t^{\frac 12 - \delta} A) g_3 \|_{H^1} + t^{-\frac 12} \| g_3 \|_{L^\infty} \right] + t^{-\frac 1 2} \cdot t^\alpha. \\
& \qquad \lesssim t^{-\frac 14 - \frac \delta 2},
\end{align*}
which implies that
$$
|H^3_p(t)| \lesssim t^{-1 - \delta}.
$$

Combining all the previous estimates, we find
$$
\left| J_p(t) - e^{-it(2\alpha \xi p - p^2)} \frac{\pi}{t} g_1(\eta_S) g_2 (\sigma_S) g_3(\zeta_S) \right| \lesssim t^{-1 - \delta } +  t^{-\frac{5}{4} + \alpha + \frac {5\delta} 2} .
$$
Optimizing over $\delta$ gives the desired result.
 \end{proof}

\begin{lemma}
\label{engoulevent2} Assuming that
$$
F(\xi) = 0 \qquad \mbox{and} \qquad \|F\|_{L^\infty} + t^{-\alpha} \| F \|_{H^1} \lesssim 1,
$$
consider the expression
$$
K(t,\xi) = \int e^{-i \rho t(2 \xi p -p^2)} \frac{\widehat{\phi}(p)}{p} F(p) \,dp.
$$
Then if $t \geq 1$ and $\kappa>0$,
$$
K(t,\xi) = - \frac{i}{2} F(0) \operatorname{sign} [ \rho \xi ] + O(\langle \sqrt{t} \xi \rangle^{\kappa-1} + t^{\alpha - \frac 12} + t^\alpha \langle \sqrt t \xi \rangle^{-3/2} ).
$$
\end{lemma}

\begin{proof}
Changing variables to $X = \sqrt t \xi$ and $P = \sqrt t p$, the formula for $K_p$ becomes
$$
K(t,\xi) = \int e^{-i \rho 2 XP +i \rho P^2} \frac{\widehat{\phi}(t^{-1/2}P)}{P} F(t^{-1/2}P) \,dP.
$$

We now apply a dyadic decomposition in $P$ and write $K$ as
$$
K(t,\xi) = \sum_{j=0}^\infty \int \varphi_j(P) e^{-i \rho 2 XP +i \rho P^2} \frac{\widehat{\phi}(t^{-1/2}P)}{P} F(t^{-1/2}P) \,dP = \sum_{j=0}^\infty K_j(t,\xi).
$$

\medskip

\noindent \underline{Estimate of $K_0$.}
The term $K_0$ will provide the leading order contribution. To separate this leading order contribution, it is convenient to split $K_0$ as follows (recalling $\widehat{\phi}(0) = \frac{1}{\sqrt{2\pi}}$)
\begin{align*}
K_0(t,\xi) & = \frac{F(0)}{\sqrt{2\pi}}  \int \varphi_0(P) e^{-i \rho 2 XP} \frac{1}{P} \,dP + \int \varphi_0(P) e^{-i \rho 2 XP} \left[ e^{i \rho P^2} - 1 \right] \frac{\widehat{\phi}(t^{-1/2}P)}{P} F(t^{-1/2}P) \,dP \\
&\qquad  + \int \varphi_0(P) e^{-i \rho 2 XP} \frac{1}{P} \left[ \widehat{\phi}(t^{-1/2}P) F(t^{-1/2}P) - \widehat{\phi}(0) F(0) \right] \,dP \\
& = K_0^0(t,\xi) + K_0^1(t,\xi) + K_0^2(t,\xi).
\end{align*}
The first term, $K_0^0$, will provide the asymptotics: with the help of the formula $\widehat{\mathcal{F}} \left[ \frac {\varphi_0(P)} P \right] = - \frac{i}{2} \widehat{\varphi_0} * \operatorname{sign}$,
\begin{align*}
K_0^0(t,\xi) & = F(0) \widehat{\mathcal{F}} \left[ \frac{\varphi_0(P)}{P} \right] (2 \rho X)
= - \frac i 2 F(0) \left[ \operatorname{sign} * \widehat{\varphi_0} \right] (2 \rho X) \\
& = - i \sqrt{\frac{\pi}{2}} F(0) \operatorname{sign} (\rho X) + O(\langle X \rangle^{-1}).
\end{align*}
The second term, $K_0^1$, is obviously $O(1)$ for $|X|<1$; for $|X|>1$ it can be bounded by integration by parts:
$$
|K_0^1(t,\xi)| \lesssim \frac{1}{|X|} \left| \int e^{-2i \rho XP} \partial_P \left[\varphi_0(P) (e^{i P^2} - 1) \frac{\widehat{\phi}(t^{-1/2}P)}{P} F(t^{-1/2}P)  \right] \, dP \right| \lesssim \frac{1}{|X|}.
$$
As for the third term, $K_0^2$, we rely on the regularity of $F$ to obtain
$$
|K_0^2(t,\xi)| \lesssim t^{-1/2} \| \phi F \|_{H^1}  \int \varphi_0(P) |P|^{-1/2} \,dP \lesssim t^{\alpha - \frac{1}{2}}.
$$

\medskip

\noindent \underline{Estimate of $K_j$, $j \geq 1$.}
Two cases need to be distinguished. If $|X| \not \sim 2^{j}$, then the phase $\Phi_X(P) = -2XP + P^2$ is not stationary, with derivative $| \Phi_X'(P) | \sim |X| + 2^j$ on the support of $K_j$. Therefore, an integration by parts in $P$ leads to the estimate
$$
|K_j(t,\xi) | \lesssim (|X| + 2^j)^{-1}
$$

If $|X| \sim 2^j$ with $j \geq 1$, then the stationary point of the phase $\Phi_X(P)$, namely $P=X$ might belongs to the support of the integrand of $K_j$. Then it is convenient to rescale by $2^j$ by setting $P' = 2^{-j}P$ and $X' = 2^{-j}X$, which yields
$$
K_j(t,\xi) = \int \varphi_1(P') e^{-i 2^{2j} \Phi_{X'}(P')} G(P') \,dP', \qquad G(P') =  \frac{\widehat{\phi}(t^{-1/2} 2^j P')}{P'} F(t^{-1/2} 2^j P') \varphi_{\sim 1}(P');
$$
notice here that
$$
G(X') = 0 \quad \mbox{and} \quad \|G \|_{L^\infty} + t^{-\alpha} \| G \|_{H^1} \lesssim 1.
$$

We now introduce a scale $R$, which will be fixed later, and split $K_j$ into
$$
K_j(t,\xi) = \int \dots \varphi_0(R^{-1}(X'-P')) \,dP' +  \int \dots \varphi_{\geq 1}(R^{-1}(X'-P')) \,dP' = K_j^1(t,\xi) + K_j^2(t,\xi).
$$
The term $K_j^1(t,\xi)$ can be estimated by taking advantage of the cancellation of the integrand at $X'$:
$$
| K_j^1(t,\xi) | \lesssim \| G \|_{H^1} \int \varphi_0(R^{-1}(X'-P')) |X'-P'|^{1/2} \,dP' \lesssim  t^\alpha R^{3/2}.
$$
An integration by parts allows to write $K_j^2$ as
\begin{align*}
K_j^2(t,\xi) & = -i 2^{-2j} \int e^{-i 2^{2j} \Phi_{X'}(P')}\varphi_1(P')  \partial_{P'} G(P') \frac{\varphi_{\geq 1}(R^{-1}(X'-P'))}{2X'-2P'} \,dP' \\
& \qquad -i 2^{-2j} \int e^{-i 2^{2j} \Phi_{X'}(P')} \varphi_1(P') G(P') \partial_{P'} \frac{\varphi_{\geq 1}(R^{-1}(X'-P'))}{2X'-2P'}\,dP' \\
& \qquad + \{ \mbox{easier term} \}
\end{align*}
By the Cauchy-Schwarz inequality and the Hardy inequality, both terms on the right-hand side can be bounded by
$$
2^{-2j} \| G \|_{H^1} \left\| \frac{\varphi_{\geq 1}(R^{-1}(X'-P'))}{2X'-2P'} \right\|_{L^2} \lesssim 2^{-2j} R^{-1/2} t^\alpha.
$$
There remains to to optimize over $R$ to obtain
$$
|K_j(t,\xi)| \lesssim t^\alpha R^{3/2} + 2^{-2j} R^{-1/2} t^\alpha \lesssim t^\alpha 2^{-3j/2} \sim t^\alpha |X|^{-3/2} \mathbf{1}_{|X|>1}.
$$
\end{proof}

\subsection{Proof of the pointwise bound in distorted Fourier space}

Applying the stationary phase lemma~\ref{capucin} to the singular cubic terms, and following the algebra  leads to the following statement.

\begin{lemma}[Asymptotics for the singular cubic term] \label{guifette1} If the solution is trapped up to time $T$, then there holds for any time $t \in (0,T)T$ that
\begin{itemize}
\item[(i)]The distorted Fourier transform of the singular cubic term is bounded
$$
|\widehat{\mathcal{F}}_\rho \mathcal{C}^S(t,\xi) | \lesssim \epsilon_1^3 \qquad \mbox{if $t \in [0,1]$},
$$
\item[(ii)] If $|\xi| > t^{-\frac 12 + 2 \alpha}$, the Fourier transform of the singular cubic term enjoys the asymptotics
$$
\widehat{\mathcal{F}}_\rho \mathcal{C}^S(t,\xi) = \frac{L}{2t} |\widetilde{f}_\rho(t,\xi)|^2 \widetilde{f}_\rho(t,\xi) + \epsilon_1^3 O \left( t^{-1-\alpha} \right).
$$
\end{itemize}
\end{lemma}

\begin{proof} On the one hand, the stationary phase lemma~\ref{capucin} gives a formula for the leading order term in $\widehat{\mathcal{F}}_\rho \mathcal{C}^S(t,\xi)$ as $t \to \infty$ as well as error estimates. On the other hand, the computations in Section~\ref{tengmalm} show that this formula reduces to $- |\widetilde{f}_\rho(t,\xi)|^2 \widetilde{f}_\rho(t,\xi)$.
\end{proof}

\begin{lemma} [Integrable bound pointwise] \label{guifette2} If the solution is trapped up to time $T$, all the nonlinear terms except $\mathcal{C}^S$ enjoy integrable bounds in $\mathcal{F} L^\infty$. More precisely, if $t \in (0,T)$,
\begin{equation}  
\left\| \widetilde{\mathcal{Q}^R} \right\|_{L^\infty} + \left\| \widetilde{\mathcal{C}^R} \right\|_{L^\infty} + \left\| \widetilde{\mathcal{T}} \right\|_{L^\infty} + \left\| \widetilde{\mathcal{R}} \right\|_{L^\infty} + \left\| \widetilde{\mathcal{E}} \right\|_{L^\infty} + \left\| \widetilde{\Mod}_\Phi \right\|_{L^\infty} \lesssim \epsilon_1^2 \langle t \rangle^{-1-\frac{\nu}{2}}.
\end{equation}
\end{lemma}

\begin{proof} The terms involving only the continuous spectrum are $\mathcal{Q}^R$, $\mathcal{C}^R$ and $\mathcal{T}$. Out of these terms, $\mathcal{Q}^R$ is the term that decays the slowest, and it is the term we will focus on. As usual, we drop indices and write it as a sum of terms of the type
$$
\int e^{-it\Phi(\xi,\eta,\sigma)} \widetilde{f}(\eta) \widetilde{f}(\sigma) \mathfrak{m}(\xi,\eta,\sigma) \,d\eta \,d\sigma.
$$
The case $|t|<1$ is obvious, and we will only consider the case $|t|>1$. Integrating by parts in $\eta$ and $\sigma$ turns the above into
$$
\frac{1}{t^2} \int e^{-it\Phi} \partial_\xi \widetilde{f}(\eta) \partial_\xi \widetilde{f}(\sigma) \frac{\mathfrak{m}(\xi,\eta,\sigma)}{\eta \sigma} \,d\eta \,d\sigma + \{ \mbox{easier terms} \}.
$$
Applying the Cauchy-Schwarz inequality in $\eta$ and $\sigma$ gives
$$
\left\| \frac{1}{t^2} \int e^{-it\Phi} \partial_\xi \widetilde{f}(\eta) \partial_\xi \widetilde{f}(\sigma) \frac{\mathfrak{m}(\xi,\eta,\sigma)}{\eta \sigma} \,d\eta \,d\sigma \right\|_{L^\infty} \lesssim \frac{1}{t^2} \| \partial_\xi \widetilde{f} \|_{L^2}^2 \lesssim t^{2\alpha -2} \epsilon_1^2.
$$

There remains to treat the terms stemming from the modulation. Using \eqref{bd:estimation:a0(t)-a1(t)} to bound $ \widetilde{\mathcal{R}} $, \eqref{boundE1}, \eqref{boundE2} and \eqref{bound:partial-xi-tildemathcalE} and the Sobolev embedding for $\tilde{\mathcal E}$, and \eqref{eq:modulation-gamma-rough} for $\widetilde{\Mod}_\Phi$:
\begin{align*}
& \left\| \widetilde{\mathcal{R}} \right\|_{L^\infty} + \left\| \widetilde{\mathcal{E}} \right\|_{L^\infty} + \left\| \widetilde{\Mod}_\Phi \right\|_{L^\infty} \\
& \qquad \qquad \lesssim \epsilon_1 |a_0| + \epsilon_1 |a_1| + |\dot \omega| + |\omega -p^2- \dot \gamma| +|2p-\dot y|+|\dot p|+\epsilon_1^2\langle t\rangle^{-1-\nu+\alpha} \lesssim\epsilon_1^2 \langle t \rangle^{-1-\frac{\nu}{2}}.
\end{align*}

\end{proof}

With the help of the above results, we are now able to prove the uniform bound on $\widetilde{f}$.

\subsection{Proof of Proposition~\ref{PropositionPointwise}}

If $|\xi| < \langle t \rangle^{-\frac 12 + \alpha}$, the $L^\infty$ bound of $\widetilde{f}$ can be bounded using that $\widetilde{f}(0)=0$ together with the bound on $\| \widetilde{f} \|_{H^1}$:
\begin{equation}
\label{guillemot}
|\widetilde{f}(\xi)| \lesssim \left| \int_0^\xi \partial_\xi \widetilde{f} \,d\xi \right| \lesssim \langle t \rangle^{-\frac{1}{4} + \frac  \alpha 2} \| \partial_\xi \widetilde{f} \|_{L^2} \lesssim \langle t \rangle^{-\frac{1}{4} +\frac {3\alpha} 2} (\epsilon + \epsilon_1^2).
\end{equation}

By Proposition \ref{propModU} and the Sobolev embedding we have 
$$
e^{- it \rho (1+ \xi^2)} \widetilde{\Mod_{\Uu}}  =  \tau(t) \xi \widetilde f(\xi) + O(\epsilon_1^2 \langle t \rangle^{-1-\frac{\nu}{2}}).
$$
Based on this identity and on lemmas~\ref{guifette1} and~\ref{guifette2} the evolution \eqref{dtf} of $\tilde f$ becomes
$$
i\partial_t \tilde f_\rho= \frac{L}{2t}|\tilde f_\rho|^2\tilde f_\rho+\tau(t)\xi \tilde f_\rho+O(\epsilon_1^2\langle t\rangle^{-1-\alpha})
$$
for $|\xi|>t^{-1/2+2\alpha}$, and for $|\xi|\leq 1$ there holds $\partial_t \tilde f_\rho =\epsilon_1^2$ for all $t$. The scalar $Y = \widetilde{f}_\rho(t,\xi)$ (with $\rho \in \{ \pm \}$) satisfies the differential inequalities
$$
\left\{
\begin{array}{ll}
\dot Y = O(\epsilon_1^2) & \mbox{if $|\xi|\leq 1$ for any $t$,}\\
i \dot Y = \frac{L}{2t} |Y|^2 Y +\tau(t)\xi Y + O(\epsilon_1^2 \langle t \rangle^{-1 -\alpha}) & \mbox{if $|\xi| > \langle t \rangle^{-\frac 12 +2 \alpha}$},
\end{array}
\right.
$$
which implies that
$$
\left\{
\begin{array}{ll}
\frac{d}{dt} |Y| = O(\epsilon_1^2) & \mbox{if $|\xi|\leq 1$ for any $t$,}\\
\frac{d}{dt} |Y| =  O(\epsilon_1^2 t^{-1 -\alpha}) & \mbox{if $|\xi| > t^{-\frac 12 + 2 \alpha}$}.
\end{array}
\right.
$$
If $|\xi| > 1$, this implies immediately the existence of $Z_\infty = O(\epsilon + \epsilon_1^2)$ such that
$$
 ||Y| - Z_\infty| \lesssim \epsilon_1^2 \langle t \rangle^{-\alpha} \qquad \mbox{for any $t$}.
$$
If $|\xi| < 1$, we use~\eqref{guillemot} to bound $Y$ on as long as $t < |\xi|^{(\alpha - \frac{1}{2})^{-1}}$, and then integrate the differential inequality to obtain the existence of $Z_\infty = O(\epsilon + \epsilon_1^2)$ such that
$$
|Y(t)| = O(\epsilon + \epsilon_1^2) \;\; \mbox{for any $t$} \qquad \mbox{and} \qquad ||Y| - Z_\infty| \lesssim \epsilon_1^2 \langle t \rangle^{-\alpha} \;\; \mbox{if $t >  |\xi|^{(2 \alpha - \frac{1}{2})^{-1}}$}.
$$

Based on these bounds, we can now write the differential inequality satisfied by $Y$ as
$$
\left\{
\begin{array}{ll}
\dot Y = O(\epsilon_1^2) & \mbox{for any $t$}\\
i \dot Y = \frac{L}{2t} Z_\infty^2 Y +\tau(t) \xi Y+ O(\epsilon_1^2 \langle t \rangle^{-1 -\alpha}) & \mbox{if $t>1$ and $|\xi| > \langle t \rangle^{-\frac 12 + 2 \alpha}$},
\end{array}
\right.
$$
For the unknown $W = Y e^{i \frac{L}{2} Z_\infty^2 \log t +i\xi \theta(t)}$, this implies that
$$
i \dot W =  O(\epsilon_1^2 \langle t \rangle^{-1 -\alpha}),
$$
from which the desired statement follows by integration.

\section{Modified scattering for distorted and flat Fourier transforms}
\label{sectionmodified}

\begin{proposition} \label{pr:modified-scattering}

Assume that $v$ is trapped for all times $t\geq 0$ in the sense of Definition \ref{def:trapped-solution}, and that there exists $\widetilde g\in L^2$ such that
\begin{equation} \label{id:hp-distorted-modified-scattering}
\widetilde f_\rho = e^{  i\frac{L}{2}|\widetilde g_\rho|^2 \ln t}\widetilde g_\rho +o_{L^{2}}(1) \qquad \mbox{as }t\to \infty 
\end{equation}
for $\rho=\pm$. Then
\begin{equation} \label{id:result-modified-scattering}
u(t)=\mathcal F^{-1} \left( e^{ i\left((\uo+\xi^2)t+ \frac{L}{2}|\widetilde g_+|^2 \ln t\right)}\widetilde g_+ \right)+o_{L^2}(1) \qquad \mbox{as }t\to \infty.
\end{equation}

\end{proposition}

\begin{proof}

We have $u=U_1$. We decompose $U=\Uue+\Uud$, where $\| \Uud(t)\|_{L^2}\to 0$ as $t\to \infty$ because of \eqref{bd:estimation:a0(t)-a1(t)}. In addition,
$\widetilde{\mathcal F}\Uue=e^{  i \rho(\uo+\xi^2)t}\widetilde f $. Hence we have as $t\to \infty$
$$
u(t,x)=\widetilde{\mathcal F}^{-1}\left(e^{  i \rho(\uo+\xi^2)t} \widetilde f_\rho \right)_1(x)+o_{L^2}(1).
$$
For $\widetilde \chi$ a smooth cut-off function such that $\widetilde \chi(\xi)=1$ for $|\xi|\geq 2$ and $\widetilde \chi(\xi)=0 $ for $|\xi|\leq 1$ we let $\widetilde \chi_\delta(\xi)=\widetilde \chi(\xi/\delta)$. We decompose $\widetilde f=\widetilde\chi_\delta \widetilde f+(1-\widetilde\chi_\delta) \widetilde f$ and apply item (i) of Proposition \ref{tourterelle} to obtain as $\delta \to 0$
$$
 \widetilde{\mathcal F}^{-1}\left( e^{  i \rho(\uo+\xi^2)t} \widetilde f_\rho  \right) =\widetilde{\mathcal F}^{-1}\left(e^{  i \rho(\uo+\xi^2)t} \widetilde \chi_\delta \widetilde f_\rho \right)+o_{L^2}(1) 
$$
uniformly for $t\in [1,\infty)$. Similarly, by Parseval, as $\delta\to 0$,
$$
\mathcal F^{-1}\left( e^{ i\left((\uo+\xi^2)t+ \frac{L}{2}|\widetilde g_+|^2 \ln t\right)}\chi_\delta \widetilde g_+  \right)= \mathcal F^{-1}\left(e^{ i\left((\uo+\xi^2)t+ \frac{L}{2}|\widetilde g_+|^2 \ln t\right)}  \widetilde g_+ \right)+o_{L^2}(1) 
$$
uniformly for $t\in [1,\infty)$. Combining, we obtain that in order to show \eqref{id:result-modified-scattering} it is sufficient to show that for any $\delta>0$,
\begin{equation} \label{bd:modified-scattering-main}
\widetilde{\mathcal F}^{-1}\left(e^{  i \rho(\uo+\xi^2)t} \widetilde \chi_\delta \widetilde f_\rho \right)_1 =\mathcal F^{-1}\left( e^{ i\left((\uo+\xi^2)t+ \frac{L}{2}|\widetilde g_+|^2 \ln t\right)}\chi_\delta \widetilde g_+  \right)+o_{L^2}(1) 
\end{equation}
as $t\to \infty$. The remaining part of the proof is devoted to showing \eqref{bd:modified-scattering-main}.

Applying successively the formula \eqref{inverseFT} for the inverse Fourier transform, then \eqref{alouette} and the fact that $\psi^S_{-,1}=0$ to decompose the eigenfunctions, we get
\begin{align}
\nonumber \widetilde{\mathcal F}^{-1}\left(e^{  i\rho(\uo+\xi^2)t} \widetilde \chi_\delta \widetilde f_\rho \right)_1(x)  & = \frac{1}{\sqrt{2\pi}} \sum_\rho \rho \int e^{  i\rho(\uo+\xi^2)t} \widetilde \chi_\delta \widetilde f_\rho \psi_{\rho,1}(x,\xi) \,d\xi\\
\nonumber & = \frac{1}{\sqrt{2\pi}} \int e^{  i(\uo+\xi^2)t} \widetilde \chi_\delta \widetilde f_+ \psi_{+,1}^S(x,\xi) \,d\xi\\
\nonumber  & \quad + \frac{1}{\sqrt{2\pi}} \int e^{  i(\uo+\xi^2)t} \widetilde \chi_\delta \widetilde f_+ \psi_{+,1}^R(x,\xi) \,d\xi \\
\nonumber    & \quad + \frac{1}{\sqrt{2\pi}} \int e^{  -i(\uo+\xi^2)t}\widetilde \chi_\delta \widetilde f_- \psi_{-,1}^R(x,\xi) \,d\xi \\
\label{id:modified-scattering-decomposition}    &= u^S_{+}(t,x)+ u^R_{+}(t,x)+ u^R_{-}(t,x).
\end{align}

\noindent \underline{The singular terms.} We claim that
\begin{equation} \label{id:modified-scattering-uS}
u^S_+(t,x) = \mathcal F^{-1} \left( e^{ i\left((\uo+\xi^2)t+ \frac{L}{2}|\widetilde g_+|^2 \ln t\right)}\widetilde \chi_\delta \widetilde g_+ \right)+o_{L^2}(1) 
\end{equation}
as $t\to \infty$. We now prove this claim. We have by \eqref{alouette} and Lemma \ref{heroncendre}
$$
 \psi_{+,1}^S(x,\xi)=e^{i\xi x}+ r(|\xi|)  \chi_{-\sgn \xi }(x) e^{-i\xi x}+\left[ (s(|\xi|)-1)\chi_{\sgn \xi}(x)+\chi_{+}(x)+\chi_-(x)-1)\right]e^{i\xi x}.
$$
Hence
\begin{align}
\nonumber u^S_+(t,x) &= \mathcal F^{-1} \left( e^{ i(\uo+\xi^2)t}\widetilde \chi_\delta \widetilde f_+ \right)(x)\\
\nonumber &+\frac{1}{\sqrt{2\pi}} \int e^{ i((\uo+\xi^2)t-x \xi)} \widetilde \chi_\delta (\xi)\widetilde f_+(\xi)r(|\xi|)  \chi_{-\sgn \xi }(x)  \,d\xi \\
\nonumber &+\frac{1}{\sqrt{2\pi}} \int  e^{ i((\uo+\xi^2)t+x \xi)} \widetilde \chi_\delta (\xi)\widetilde f_+(\xi) (s(|\xi|)-1)\chi_{\sgn \xi}(x) \,d\xi \\
\nonumber &+\frac{1}{\sqrt{2\pi}} \left(\chi_{+}(x)+\chi_-(x)-1\right)  \int  e^{ i((\uo+\xi^2)t+x \xi)} \widetilde \chi_\delta (\xi)\widetilde f_+(\xi) \,d\xi  \\
\label{id:modified-scattering-decomposition-uS}  &=  \mathcal F^{-1} \left( e^{ i(\uo+\xi^2)t} \widetilde \chi_\delta (\xi) \widetilde f_+ \right)(x)+ I+II+III.
\end{align}
We have by \eqref{id:hp-distorted-modified-scattering} and Plancherel as $t\to \infty$
\begin{equation} \label{id:modified-scattering-asymptotic-uS-1} 
 \mathcal F^{-1} \left( e^{ i(\uo+\xi^2)t} \widetilde \chi_\delta (\xi) \widetilde f_+\right)= \mathcal F^{-1} \left( e^{ i\left((\uo+\xi^2)t+ \frac{L}{2}|\widetilde g_+|^2 \ln t\right)}\widetilde \chi_\delta \widetilde g_+ \right)+o_{H^1}(1) .
\end{equation}
To estimate $I$ we split $I=(2\pi)^{-1/2}(I_-+I_+)$ where
$$
I_-(t,x)= \chi_{+}(x) \int_{-\infty}^0 e^{  i( (\uo+\xi^2)t-x\xi)} \widetilde \chi_\delta (\xi) \widetilde f_+(\xi)  \overline{r(\xi)}  \,d\xi  , \quad I_+(t,x)= \chi_-(x)\int_0^\infty ...
$$
Let $x\in \mathbb R$. If $x\leq -2$ then $I_-(t,x)=0$ because of the support of $\chi_+$, so we now assume $x\geq -2$. The integrand in $I_-$ is zero unless $\xi\leq -\delta$ so we can write $e^{  i( (\uo+\xi^2)t-x\xi}= \frac{1}{2i\xi t-x}\partial_\xi e^{  i( (\uo+\xi^2)t-x\xi}$ and the denominator does not vanish for $t$ large. We integrate by parts and find
\begin{align*}
I_-(t,x) & = \chi_{+}(x)  \int_{-\infty}^0 e^{  i( (\uo+\xi^2)t-x\xi)} \frac{i}{2\xi t-x} \partial_\xi \left( \widetilde \chi_\delta (\xi) \widetilde f_+(\xi) \overline{r(\xi)}\right)  \,d\xi + \{ \mbox{easier}\}.
\end{align*}
By \eqref{estimatesrs} we have $\left| \partial_\xi \left( \widetilde \chi_\delta (\xi) \widetilde f_+(\xi) \overline{r(\xi)}\right)\right|\lesssim \la \xi\ra^{-1} \widetilde \chi_{\delta/2}(\xi)(|\widetilde f|+|\partial_\xi \widetilde f|)$. Therefore, the first term is estimated by the Minkowski inequality
\begin{align*}
& \left\| \chi_{+}(x)  \int_{-\infty}^0 e^{  i( (\uo+\xi^2)t-x\xi)} \frac{i}{2\xi t-x} \partial_\xi \left( \widetilde \chi_\delta (\xi) \widetilde f_+(\xi) \overline{r(\xi)}\right)  \,d\xi \right\|_{L^2} \\
&\quad \lesssim \int_{-\infty}^0 \left| \partial_\xi \left( \widetilde \chi_\delta (\xi) \widetilde f_+(\xi) \overline{r(\xi)}\right)\right| \| \frac{\chi_+(x)}{2\xi t-x}\|_{L^2_x}d\xi \\
&\qquad \lesssim \int_{-\infty}^0 \la \xi\ra^{-1} \widetilde \chi_{\delta/2}(\xi)(|\widetilde f|+|\partial_x \widetilde f|)   \frac{1}{\sqrt{\delta t-2}}d\xi \\
&\quad \qquad \lesssim \frac{\| \widetilde f\|_{H^1}}{\sqrt{\delta t-2}}  \ \to 0
\end{align*}
as $t\to \infty$. The easier term, when the $\partial_\xi$ derivative hits $\frac{1}{2t\xi-x}$, is easier to estimate and we skip it. Therefore
$$
\| I_-\|_{L^2}\to 0 \qquad \mbox{as $t \to \infty$}.
$$
The term $I_+$ can be dealt with the exact same way, using that $\xi\geq \delta$ in the integrand and that $\chi_-$ localises on the set $x\leq 2$. Hence
$$
\| I\|_{L^2}\to 0 \qquad \mbox{as $t \to \infty$}.
$$
The term $II$ is analogous to $I$, since $|s(|\xi|)-1|\lesssim \la \xi \ra^{-1}$, so that, again by the same reasoning
$$
\| II\|_{L^2}\to 0 \qquad \mbox{as $t \to \infty$}.
$$
The third term $III$ has support within $[-2,2]$ as $\chi_++\chi_{-}-1$ vanishes for $|x|\geq 2$. For $|x|<2$ by the improved local decay of Lemma \ref{bergeronnette}
$$
|III|\lesssim \frac{1}{t} \| \widetilde \chi_\delta (\xi)\widetilde f_+(\xi) \|_{H^1_\xi}\lesssim \epsilon_1 t^{\alpha-1}.
$$
Hence
$$
\| III \|_{L^2}\to 0 \qquad \mbox{as $t \to \infty$}.
$$
Injecting that both $I$, $II$ and $III$ converge to $0$ in $L^2$ in \eqref{id:modified-scattering-decomposition-uS} shows \eqref{id:modified-scattering-uS}.

\medskip

\noindent \underline{The remainder terms.} We claim that
\begin{equation} \label{bd:u+-u-:modified-scattering}
\| u^R_{+}(t,\cdot )+ u^R_{-}(t,\cdot ) \|_{L^2}\to 0
\end{equation}
as $t\to \infty$. To show it, we first recall that $\psi_{+,1}^R$ is given by \eqref{alouette}. Using \eqref{guimauve} and 
\eqref{carambar}, one obtains that $u^R_+(t,\cdot)$ is of the form
$$
 \int e^{  i(\uo+\xi^2)t} \widetilde \chi_\delta (\xi) \widetilde f_+(\xi) \left( m_1(x,\xi)e^{i\xi x}+m_2(x,\xi)e^{-i\xi x}\right) \,d\x
$$
where $|\partial_x^k\partial_\xi^l  m_j|\lesssim e^{-\beta |x|} \langle \xi \rangle^{-l-1}$ for $j=1,2$. By the decay of Lemma \ref{aigrette}, this term is bounded by a multiple of $\epsilon_1t^{\alpha-1} e^{-\beta |x|} $, so that
$$
\| \int e^{  i(\uo+\xi^2)t} \widetilde \chi_\delta \widetilde f_+ \left(\widetilde m_1(x,\xi)e^{i\xi x}+\widetilde m_2(x,\xi)e^{-i\xi x}\right) \,d\x \|_{L^2}\lesssim \epsilon_1 t^{\alpha-1}\to 0.
$$
The term $u^R_{-}(t,\cdot ) $ can be bounded exactly as $u^R_{+}(t,\cdot )$. Hence \eqref{bd:u+-u-:modified-scattering}.

\noindent \underline{End of the proof}. Injecting \eqref{id:modified-scattering-uS} and \eqref{bd:u+-u-:modified-scattering} in \eqref{id:modified-scattering-decomposition} shows the desired identity \eqref{bd:modified-scattering-main}.

\end{proof}

\begin{appendix}

\section{Cauchy theory}
\label{CauchyTheory}

We here briefly give a local well-posedness result for \eqref{NLS} in the space $H^1\cap L^{2,1}$ for the sake of completeness. This is a standard result by Ginibre and Velo \cite{GV}. Solutions satisfy equation \eqref{NLS} in an integral sense, see \cite{GV} for details.

\begin{proposition} \label{pr:cauchy}

Assume $F\in C^3(\mathbb R)$. Then given any $M>0$, there exists $T(M)>0$ such that for any $v_0\in H^1\cap L^{2,1}$ with $\| v_0\|_{H^1\cap L^{2,1}}\leq M$, there exists a unique solution $v\in C([0,T(M)],H^1\cap L^{2,1})$ to \eqref{NLS}. Moreover, the map $v_0\mapsto v$ is continuous from the ball of $H^1\cap L^{2,1}$ of radius $M$ into $C([0,T(M)],H^1\cap L^{2,1})$.

\end{proposition}

\begin{proof}
It follows \cite{GV}, Section 2. The desired solution is a fixed point of the integral equation
\begin{equation} \label{id:NLS-integral-equation}
u(t)=e^{-it\partial_x^2}u_0-i\int_0^t e^{-i(t-s)\partial_x^2}F'(|u(s)|^2)u(s)ds.
\end{equation}
The linear Schr\"odinger group satisfies the following continuity bound on $H^1\cap L^{2,1}$
$$
\| e^{-it\partial_x^2} u\|_{H^1}+\| e^{-it\partial_x^2} u\|_{L^{2,1}}\lesssim \langle t \rangle \left( \|  u\|_{H^1}+\|  u\|_{L^{2,1}}\right)
$$
which can be proved by standard Fourier analysis. Moreover, by a standard application of the Sobolev embedding, the application $u\mapsto F'(|u|^2)u$ maps continuously $H^1\cap L^{2,1}$ into itself, and is uniformly Lipschitz on bounded sets. One can then solve \eqref{id:NLS-integral-equation} by appealing to the Banach-Picard fixed point Theorem. We refer to \cite{GV} for details.
\end{proof}
\end{appendix}

\bibliographystyle{abbrv}
\bibliography{references}

\end{document}